\documentclass{article}
\usepackage{amsmath,amsthm,amssymb,amscd}
\usepackage[arrow,matrix]{xy}
\usepackage{mathtools}
\usepackage{enumerate}
\usepackage{ascmac}
\usepackage{graphicx}
\usepackage{mathrsfs}
\usepackage{typearea}
\typearea{12}

\theoremstyle{definition}
\newtheorem{thm}{Theorem}[section]
\newtheorem{conj}[thm]{Conjecture}
\newtheorem{lemma}[thm]{Lemma}
\newtheorem{corollary}[thm]{Corollary}
\newtheorem{remark}[thm]{Remark}
\newtheorem{dfn}[thm]{Definition}

\newtheorem{prop}[thm]{Proposition}

\newtheorem{exa}[thm]{Example}
\newtheorem{ass}[thm]{Assumption}

\def\Ad{\mathop{\mathrm{Ad}}\nolimits}
\def\Attr{\mathop{\mathrm{Attr}}\nolimits}

\def\Frac{\mathop{\mathrm{Frac}}\nolimits}
\def\Pic{\mathop{\mathrm{Pic}}\nolimits}

\def\DCoh{\mathop{D^{b}\mathrm{Coh}}\nolimits}
\def\DCohT{\mathop{D^{b}\mathrm{Coh}^{\bb{T}}}\nolimits}

\def\Spec{\mathop{\mathrm{Spec}}\nolimits}
\def\Stab{\mathop{\mathrm{Stab}}\nolimits}
\def\Supp{\mathop{\mathrm{Supp}}\nolimits}
\def\dim{\mathop{\mathrm{dim}}\nolimits}
\def\rank{\mathop{\mathrm{rk}}\nolimits}
\def\wt{\mathop{\mathrm{wt}}\nolimits}

\def\Ker{\mathop{\mathrm{Ker}}\nolimits}

\def\Hom{\mathop{\mathrm{Hom}}\nolimits}
\def\RHom{\mathop{R\mathrm{Hom}}\nolimits}
\def\Ext{\mathop{\mathrm{Ext}}\nolimits}
\def\ext{\mathop{\mathrm{ext}}\nolimits}
\def\End{\mathop{\mathrm{End}}\nolimits}
\def\det{\mathop{\mathrm{det}}\nolimits}
\def\gr{\mathop{\mathrm{gr}}\nolimits}

\newcommand{\mf}[1]{{\mathfrak{#1}}}
\newcommand{\mb}[1]{{\mathbf{#1}}}
\newcommand{\bb}[1]{{\mathbb{#1}}}
\newcommand{\mca}[1]{{\mathcal{#1}}}
\newcommand{\mr}[1]{{\mathrm{#1}}}
\newcommand{\msc}[1]{{\mathscr{#1}}}

\newcommand{\da}[1]{\big\downarrow\raise.5ex\rlap{$\scriptstyle#1$}}

\begin{document}

\title{Elliptic canonical bases for toric hyper-K\"ahler manifolds}
\author{Tatsuyuki Hikita\footnote{thikita@kurims.kyoto-u.ac.jp}}
\date{}

\maketitle

\begin{abstract}
Lusztig defined certain involutions on the equivariant $K$-theory of Slodowy varieties and gave a characterization of certain bases called canonical bases. In this paper, we give a conjectural generalization of these involutions and $K$-theoretic canonical bases to conical symplectic resolutions which have good Hamiltonian torus actions and state several conjectures related to them which we check for toric hyper-K\"ahler manifolds. We also propose an elliptic analogue of these bar involutions. As a verification of our proposal, we explicitly construct elliptic lifts of $K$-theoretic canonical bases and prove that they are invariant under elliptic bar involutions for toric hyper-K\"ahler manifolds. 
\end{abstract}

\section{Introduction}

Lusztig \cite{Lu1,Lu2} defined certain modules of affine Hecke algebras called periodic modules and defined the notion of canonical bases in them. In \cite{Lu3,Lu4}, Lusztig gave a geometric construction of these modules in terms of equivariant $K$-theory of Springer resolutions or Slodowy varieties, and gave a geometric characterization of (signed) canonical bases using certain involution called bar involution. Basic properties of canonical bases including their existence and relation to modular representation theory of semisimple Lie algebras in large enough characteristic were also conjectured by Lusztig and proved by Bezrukavnikov-Mirkovi\'c \cite{BM}.  

\subsection{$K$-theoretic canonical bases}

One of the aim of this paper is to give an analogue of the notion of bar involutions and canonical bases to equivariant $K$-theory of conical symplectic resolutions which have Hamiltonian torus actions with finitely many fixed points. For $ADE$ type quiver varieties, analogous bar involutions and canonical bases were proposed by Lusztig in \cite{Lu5}, and constructed by Varagnolo-Vasserot in \cite{VV}. In all the previous works, bar involutions are defined by composing several automorphisms and it seems complicated at least to the author. Our main observation in this paper is that the bar involution for Springer resolutions (and conjecturally for other known cases satifying our assumptions) have a very simple characterization by using the $K$-theoretic analogue of the stable bases introduced by Maulik-Okounkov \cite{MO}. 

In this introduction, we give a rough definition of the $K$-theoretic bar involution and the $K$-theoretic canonical bases. Let $X$ be a conical symplectic resolution with conical action of $\bb{S}\coloneqq\bb{C}^{\times}$. Throughout this paper, we always assume that the symplectic form has weight 2 and there exists a Hamiltonian torus action $H\curvearrowright X$ commuting with the $\bb{S}$-action such that the fixed point set $X^{H}$ is finite. 

The $K$-theoretic stable basis (see e.g. \cite{O1,OS}) depends on some data called chamber $\mf{C}\subset\bb{X}_{\ast}(H)\otimes_{\bb{Z}}\bb{R}$, polarization $T^{1/2}\in K_{H\times\bb{S}}(X)$, and slope $s\in\Pic(X)\otimes_{\bb{Z}}\bb{R}$. Associated with these data, one can give a characterization of certain element $\Stab^{K}_{\mf{C},T^{1/2},s}(p)\in K_{H\times\bb{S}}(X)$ in the equivariant $K$-theory of $X$ corresponding to each fixed point $p\in X^{H}$. These elements form a basis in the localized equivariant $K$-theory $K_{H\times\bb{S}}(X)_{\mr{loc}}\coloneqq K_{H\times\bb{S}}(X)\otimes_{K_{H\times\bb{S}}(\mr{pt})}\Frac(K_{H\times\bb{S}}(\mr{pt}))$. 

Let $v\in K_{\bb{S}}(\mr{pt})\cong\bb{Z}[v,v^{-1}]$ be the element corresponding to the natural representation of $\bb{S}$. We want to define $K_{H}(\mr{pt})$-linear and $K_{\bb{S}}(\mr{pt})$-antilinear (i.e. $\beta^{K}(v\cdot-)=v^{-1}\beta^{K}(-)$) map $\beta^{K}_{X,T^{1/2},s}=\beta^{K}:K_{H\times\bb{S}}(X)_{\mr{loc}}\rightarrow K_{H\times\bb{S}}(X)_{\mr{loc}}$ by the formula
\begin{align}\label{roughbar}
\beta^{K}(\Stab^{K}_{\mf{C},T^{1/2},s}(p))=\pm v^{?}\Stab^{K}_{-\mf{C},T^{1/2},s}(p)
\end{align}
for any $p\in X^{H}$. For a more precise formula about the normalization, we refer to the main body of the paper. We only note here that in order to write down the formula, one needs the information about the tangent bundle of the symplectic dual $X^{!}$ of $X$ in the sense of Braden-Licata-Proudfoot-Webster \cite{BLPW3}.  

One of the drawback of this approach is that it is not clear if $\beta^{K}$ is actually an involution, and it preserves the integral form $K_{H\times\bb{S}}(X)\subset K_{H\times\bb{S}}(X)_{\mr{loc}}$. In fact, these requirements give a strong restriction on the normalization in (\ref{roughbar}). One of the main conjecture in this paper is that in the explicitly specified normalization, these properties for $\beta$ hold. 

\begin{conj}
The above $\beta^{K}$ is an involution and preserves $K_{H\times\bb{S}}(X)\subset K_{H\times\bb{S}}(X)_{\mr{loc}}$. Moreover, this does not depend on the choice of $\mf{C}$. 
\end{conj}

We note that the $K$-theoretic bar involution does depend on the choice of polarization and slope. When $X$ is a Springer resolution, we will check that this bar involution essentially coincides with the bar involution defined by Lusztig for a specific choice of polarization and slope. 

Once we have defined the bar involution, we can give a characterization of the (signed) canonical basis following Lusztig \cite{Lu3,Lu4}. Let $\bb{D}_{X}:K_{H\times\bb{S}}(X)\rightarrow K_{H\times\bb{S}}(X)$ be the Serre duality. We define an inner product $(-:-):K_{H\times\bb{S}}(X)\times K_{H\times\bb{S}}(X)\rightarrow\Frac(K_{H\times\bb{S}}(\mr{pt}))$ by 
\begin{align}\label{usualinnerprod}
(\mca{F}:\mca{G})\coloneqq[R\Gamma(X,\mca{F}\otimes^{\bb{L}}_{X}\mca{G})],
\end{align}
and another inner product $(-||-):K_{H\times\bb{S}}(X)\times K_{H\times\bb{S}}(X)\rightarrow\Frac(K_{H\times\bb{S}}(\mr{pt}))$ by 
\begin{align}\label{innerprod}
(\mca{F}||\mca{G})\coloneqq (\mca{F}:\bb{D}_{X}\beta^{K}\mca{G}).
\end{align}
Using this, we define 
\begin{align}\label{signedcanonicalbasis}
\bb{B}^{\pm}_{X,T^{1/2},s}&\coloneqq\left\{m\in K_{H\times\bb{S}}(X)\middle| \beta^{K}(m)=m, (m||m)\in 1+v^{-1}K_{H}(\mr{pt})[\![v^{-1}]\!]\right\}. 
\end{align}
We note that for any $\lambda\in\mathbb{X}^{\ast}(H)$ and $m\in\mathbb{B}^{\pm}_{X,T^{1/2},s}$, we have $\pm\lambda m\in\mathbb{B}^{\pm}_{X,T^{1/2},s}$. In Proposition~\ref{expansion}, we give a conjectural Kazhdan-Lusztig type algorithm to compute $\bb{B}^{\pm}_{X,T^{1/2},s}$, and in particular, fix the sign in it which gives a subset $\bb{B}_{X,T^{1/2},s}\subset\bb{B}^{\pm}_{X,T^{1/2},s}$ such that $\bb{B}^{\pm}_{X,T^{1/2},s}=\bb{B}_{X,T^{1/2},s}\sqcup-\bb{B}_{X,T^{1/2},s}$. Another main conjecture in this paper is the following. 

\begin{conj}
There exists an $H\times\bb{S}$-equivariant tilting bundle $\mca{T}_{T^{1/2},s}$ on $X$ such that $\bb{B}_{X,T^{1/2},s}$ coincides with the set of $K$-theory classes of indecomposable summands of $\mca{T}_{T^{1/2},s}$ up to shifts by $\bb{X}^{\ast}(H)$. In particular, $\bb{B}_{X,T^{1/2},s}$ is a basis of $K_{H\times\bb{S}}(X)$ as a $\bb{Z}[v,v^{-1}]$-module.
\end{conj}

When $X$ is a Springer resolution and for the specific choice of the data as above, this conjecture follows from a result of Bezrukavnikov-Mirkovi\'c \cite{BM} combined with our comparison result Proposition~\ref{prop_Springer_comparison}. In this paper, we also check it for toric hyper-K\"ahler manifolds, see Corollary~\ref{cor_toric_tilting}. We remark that this tilting bundle essentially coincides with the one constructed by McBreen-Webster \cite{MW} and \v{S}penko-Van den Bergh \cite{SV}.  

We remark that by the definition of $K$-theoretic canonical basis, the endomorphism ring $\mca{A}_{T^{1/2},s}\coloneqq\End(\mca{T}_{T^{1/2},s})$ of the conjectural tilting bundle has nonnegative grading with respect to $\bb{S}$. In particular, this is Koszul by the Kaledin's argument in \cite{BM}. We will formulate a conjecture on the highest weight category structure on the equivariant module category of the Koszul dual of $\mca{A}_{T^{1/2},s}$, see Conjecture~\ref{categorical_stable_basis}. A natural duality functor on it should lift the $K$-theoretic bar involution to an involution on the derived category $\DCoh^{H\times\bb{S}}(X)$ of $H\times\bb{S}$-equivariant coherent sheaves on $X$, see Conjecture~\ref{anti-involution}. 

We also note that by varying the slope parameters, we would obtain a family of tilting bundles and hence $t$-structures on $\DCoh^{H\times\bb{S}}(X)$. This should be a part of the data defining the real variations of stability conditions in the sense of Anno-Bezrukavnikov-Mirkovi\'c \cite{ABM}, see Conjecture~\ref{conj_real_variation}. We will give a conjecture on the wall-crossing of $K$-theoretic canonical bases, which is also given by a Kazhdan-Lusztig type algorithm, see Conjecture~\ref{conj_wall_crossing}. 

\subsection{Elliptic bar involutions}

The second and the main aim of this paper is to give an elliptic analogue of the $K$-theoretic bar involution. Since Aganagic-Okounkov \cite{AO} defined the elliptic analogue of the stable basis, it seems natural to consider the elliptic analogue of (\ref{roughbar}) replacing $K$-theoretic stable bases by elliptic stable bases. 

Let $\Stab^{AO}_{\mf{C},T^{1/2}}(p)$ be the elliptic stable basis associated with $p\in X^{H}$ in the sense of Aganagic-Okounkov \cite{AO}. This is a section of some line bundle on equivariant elliptic cohomology of $X$ extended by adding K\"ahler parameters. In order to obtain an involution, it seems natural to shift the K\"ahler parameters by $v^{\det T^{1/2}}$ and then multiply it by $\Theta(N^{!}_{p^{!},-})$, the Thom class of the negative part of the tangent bundle of $X^{!}$ at the fixed point $p^{!}\in (X^{!})^{H^{!}}$ corresponding to $p\in X^{H}$ under symplectic duality. In order to consider its $K$-theory limits, we also twist it by $v^{?}\cdot\sqrt{\det T^{1/2}\cdot\det T^{1/2,!}_{p^{!}}}^{-1}$. Let us denote by $\mf{S}_{X,\mf{C}}(p)$ the resulting renormalized elliptic stable basis. Here, we omit the polarization from the notation but they depends on the choice. 

We want to define the elliptic bar involution $\beta^{ell}_{X}=\beta^{ell}$ by the formula 
\begin{align*}
\beta^{ell}(\mf{S}_{X,\mf{C}}(p))=(-1)^{\frac{\dim X}{2}}\mf{S}_{X,-\mf{C}}(p)
\end{align*}
for any $p\in X^{H}$. Here, $\beta^{ell}$ should be considered as an involution on some space of meromorphic functions on $\Spec\left(K_{H\times\bb{S}}(X)\otimes_{\bb{Z}}\bb{C}[\Pic(X)^{\vee}]\right)$. We also conjecture that $\beta^{ell}$ does not depend on the choice of $\mf{C}$, and hence it is an involution. We will check this for toric hyper-K\"ahler manifolds in Corollary~\ref{cor_ell_bar_invariance}. 

We remark that in the above normalization, following symmetry under the symplectic duality is expected (cf. \cite{RSVZ1,RSVZ2}): 
\begin{align*}
\mf{S}_{X,\mf{C}}(p_{1})|_{p_{2}}=\pm\mf{S}_{X^{!},\mf{C}^{!}}(p_{2}^{!})|_{p_{1}^{!}}.
\end{align*}
Here, we identified the equivariant parameters for $X$ and K\"ahler parameters for $X^{!}$ and vice versa, as predicted by symplectic duality. In some sense, this symmetry explains the naturality of the above normalization. By taking the $K$-theory limits, this also explains the appearance of symplectic duality in the normalization of our $K$-theoretic bar involutions. For more detail, see section 4.3.

The main problem toward a definition of the notion of elliptic canonical basis is to generalize other conditions such as asymptotic norm one property or triangular property of $K$-theoretic canonical bases to the elliptic case. Although we do not know how to deal with this problem in general, we give a candidate for the elliptic canonical bases for toric hyper-K\"ahler manifolds by explicitly constructing an elliptic lift of $K$-theoretic canonical bases which are invariant under the elliptic bar involution.

\subsection{Elliptic canonical bases for toric hyper-K\"ahler manifolds}

Let us explain the main result of this paper briefly. Let $1\rightarrow S\rightarrow T\rightarrow H\rightarrow 1$ be an exact sequence of algebraic tori. We will identify the character lattice $\bb{X}^{\ast}(T)\cong\bb{Z}^{n}$ with its dual by taking the standard pairing $(-,-)$ on $\bb{Z}^{n}$. The toric hyper-K\"ahler manifold $X\coloneqq\mu^{-1}(0)^{\mr{ss}}/S$ is defined by the Hamiltonian reduction of $T^{\ast}\bb{C}^{n}$ by $S$. Here, $\mu$ is the moment map for the $S$-action and the GIT stability parameter is taken generically. We assume that $X$ is smooth. 

By considering the induced line bundles on the quotient, we obtain a natural homomorphism $\mca{L}:\bb{X}^{\ast}(T)\rightarrow\Pic^{H\times\bb{S}}(X)$. By considering $\bb{X}_{\ast}(S)\subset\bb{X}_{\ast}(T)$ as a subset of $\bb{X}^{\ast}(T)$ by using the pairing, we define the provisional elliptic canonical basis by the following formula. 

\begin{dfn}
For each $\lambda\in\bb{X}^{\ast}(T)$, we set 
\begin{align}\label{ecb}
\Theta_{X}(\lambda)\coloneqq(q;q)_{\infty}^{-\rank S}\sum_{\beta\in\bb{X}_{\ast}(S)}(-1)^{(\kappa,\beta)}q^{\frac{1}{2}(\beta,\beta)+(\lambda+\frac{1}{2}\kappa,\beta)}\mca{L}(\lambda+\beta)z^{\beta}.
\end{align}
Here, $(q;q)_{\infty}\coloneqq\prod_{m\geq 1}(1-q^{m})$, $\kappa=(1,1,\ldots,1)\in\bb{Z}^{n}$, and $z^{\beta}$ is the K\"ahler parameter corresponding to $\beta$. 
\end{dfn}

One can easily check that for $\beta\in\bb{X}_{\ast}(S)$ and $\alpha\in\bb{X}^{\ast}(H)\subset\bb{X}^{\ast}(T)$, we have 
\begin{align*}
\Theta_{X}(\lambda+\beta)&=(-1)^{(\kappa,\beta)}q^{-\frac{1}{2}(\beta,\beta)-(\lambda+\frac{1}{2}\kappa,\beta)}z^{-\beta}\Theta_{X}(\lambda),\\
\Theta_{X}(\lambda+\alpha)&=a^{\alpha}\Theta_{X}(\lambda).
\end{align*}
Here, $a^{\alpha}$ is the equivariant parameter corresponding to $\alpha\in\bb{X}^{\ast}(H)$. Therefore, the number of independent elements in $\{\Theta_{X}(\lambda)\}_{\lambda\in\bb{X}^{\ast}(T)}$ is at most the number of elements of $\Xi\coloneqq\bb{X}^{\ast}(T)/(\bb{X}_{\ast}(S)+\bb{X}^{\ast}(H))\cong\bb{X}^{\ast}(S)/\bb{X}_{\ast}(S)\cong\bb{X}_{\ast}(H)/\bb{X}^{\ast}(H)$, which turns out to be equal to the rank of $K(X)$. We also remark that the formula (\ref{ecb}) has the form of the theta function associated with the lattice $\bb{X}_{\ast}(S)$, where the pairing is induced from $\bb{X}^{\ast}(T)$. We point out that if we specialize all the equivariant parameters to 1 in (\ref{ecb}), we get the character of irreducible module corresponding to $\lambda\in\bb{X}^{\ast}(S)/\bb{X}_{\ast}(S)$ for the lattice vertex operator superalgebra $\mca{V}_{\bb{X}_{\ast}(S)}$ associated with $\bb{X}_{\ast}(S)$ for certain choice of conformal vector.

One can also define the elliptic canonical basis $\Theta_{X^{!}}(\lambda)$ for the symplectic dual $X^{!}$, which is defined analogously by the dual exact sequence of algebraic tori $1\rightarrow H^{\vee}\rightarrow T^{\vee}\rightarrow S^{\vee}\rightarrow1$. We note that we identify $\bb{X}^{\ast}(T)\cong\bb{X}^{\ast}(T^{\vee})$ by using the pairing. We also remark that the choice of GIT parameter for $X^{!}$ is given by the choice of chamber for $X$. The main result of this paper is the following, see Theorem~\ref{thm_Schur_Weyl} and Corollary~\ref{cor_ell_bar_invariance}. 

\begin{thm}
In the above setting, we have $\beta^{ell}_{X}(\Theta_{X}(\lambda))=\Theta_{X}(\lambda)$ and $\beta^{ell}_{X^{!}}(\Theta_{X^{!}}(\lambda))=\Theta_{X^{!}}(\lambda)$ for any $\lambda\in\bb{X}^{\ast}(T)$. We also have the following expansion of the elliptic stable bases into the elliptic canonical bases. 
\begin{align}\label{SWduality}
\pm\mf{S}_{X,\mf{C}}(p)=\sum_{\lambda\in\Xi}(-1)^{(\lambda,\kappa)}q^{\frac{1}{2}(\lambda,\lambda+\kappa)}\Theta_{X^{!}}(\lambda)|_{p^{!}}\cdot\Theta_{X}(\lambda).
\end{align}
\end{thm}

We note that in the sum of (\ref{SWduality}), each term does not depend on the choice of representatives for $\lambda$. We remark that this formula resembles the irreducible decomposition of the lattice vertex operator superalgebra $\mca{V}_{\bb{Z}^{n}}$ as a module for the commuting action of vertex operator subalgebras $\mca{V}_{\bb{X}_{\ast}(S)}$ and $\mca{V}_{\bb{X}^{\ast}(H)}$, which forms a dual pair in the sense that they are commutant to each other in $\mca{V}_{\bb{Z}^{n}}$. The author is not sure if this kind of phenomenon happens in general or not, but we hope that the results of this paper would give some hints for the investigation of elliptic canonical bases for other conical symplectic resolutions. 

The plan of this paper is as follows. In section 2, we check that our definition of $K$-theoretic bar involution essentially coincides with Lusztig's definition for Springer resolutions. In section 3, we propose the definition of $K$-theoretic bar involutions and state several conjectures related to $K$-theoretic canonical bases. In section 4, we propose the definition of elliptic bar involutions. In section 5, we specialize to the case of toric hyper-K\"ahler manifolds and calculate $K$-theoretic canonical bases. We also prove all the conjectures stated in section 3 in these cases. In section 6, we prove our main result on the elliptic canonical bases for toric hyper-K\"ahler manifolds. 

\subsection{Acknowledgment}

The author thanks Tomoyuki Arakawa, Akishi Ikeda, Hiroshi Iritani, Syu Kato, Hitoshi Konno, Toshiro Kuwabara, Michael McBreen, Takahiro Nagaoka and Andrei Okounkov for valuable discussions related to this work. Especially, the author thanks Andrei Okounkov for his suggestion to consider the elliptic stable envelope. This work was supported by JSPS KAKENHI Grant Number 17K14163. 

\section{Reformulation}

In this section, we reformulate Lusztig's definition of the bar involution on the equivariant $K$-theory of Springer resolutions. We should remark that the Lusztig's definition works for Slodowy varieties associated with any nilpotent element. If the nilpotent element is regular in some Levi algebras, then we can use our approach to define $K$-theoretic bar involutions. The results of this section are included for motivational purposes and will not be used elsewhere in this paper. Hence we restrict ourselves to the case of Springer resolutions for simplicity. Comparison of our definition with Lusztig's definition for other Slodowy varieties should be straight-forward once some formulas for the $K$-theoretic stable bases analogous to Proposition~\ref{SZZ} are available. 

We first list some standard notations used in this section. Let $G$ be a semisimple algebraic group over $\bb{C}$ of adjoint type. We fix $B\subset G$ a Borel subgroup and $H\subset B$ a Cartan subgroup. We will denote by $\mf{g}, \mf{b},\mf{h}$ their Lie algebras. We use the convention that the nonzero $H$-weights appearing in $\mf{b}$ is negative. We denote by $\bb{X}^{\ast}(H)$ (resp. $\bb{X}_{\ast}(H)$) the character lattice (resp. cocharacter lattice) of $H$. We set $\mf{h}^{\ast}_{\bb{R}}:=\bb{X}^{\ast}(H)\otimes_{\bb{Z}}\bb{R}$ and $\mf{h}_{\bb{R}}:=\bb{X}_{\ast}(H)\otimes_{\bb{Z}}\bb{R}$. Let $\{\alpha_{i}\}_{i\in I}$ be the set of simple roots and $\{\alpha^{\vee}_{i}\}_{i\in I}$ be the set of simple coroots. Let $W$ be the Weyl group of $G$ and $s_{i}\in W$ be the simple reflection corresponding to $i\in I$. For $w\in W$, we denote by $l(w)$ the length of $w$. Let $w_{0}\in W$ be the longest element and $e\in W$ be the identity element. We denote by $R_{+}$ (resp. $R_{+}^{\vee}$) the set of positive roots (resp. positive coroots) and $\rho=\frac{1}{2}\sum_{\alpha\in R_{+}}\alpha\in\mf{h}^{\ast}_{\bb{R}}$ the half sum of positive roots. 

\subsection{Lusztig's bar involution}

We briefly recall Lusztig's bar involution on the equivariant $K$-theory of Springer resolutions. For more details, we refer to \cite{Lu3,Lu4}. 

Let $\mca{B}\coloneqq G/B$ be the flag variety and $X\coloneqq T^{\ast}\mca{B}\cong\{(gB,y)\in\mca{B}\times\mf{g}^{\ast}\mid \Ad(g)^{-1}(y)\in(\mf{g}/\mf{b})^{\ast}\}$ be the Springer resolution. The torus $\bb{T}\coloneqq H\times\bb{S}$ acts naturally on $X$ by $(h,\sigma)\cdot (gB,y)=(hgB,\sigma^{-2}\Ad(h)y)$ and this action preserves the subvariety $\mca{B}$. By the pushforward along zero section, we obtain an inclusion $K_{\bb{T}}(\mca{B})\subset K_{\bb{T}}(X)$. Let $\mr{pr}:X\rightarrow\mca{B}$ be the natural projection. 

Let $\msc{H}$ be the affine Hecke algebra associated with the Langlands dual of $G$. I.e., $\msc{H}$ is the $\bb{Z}[v,v^{-1}]$-algebra generated by $T_{w}$ ($w\in W$) and $\theta_{\lambda}$ ($\lambda\in\bb{X}^{\ast}(H)$) with the following relations: 
\begin{itemize}
\item $(T_{s_{i}}-v)(T_{s_{i}}+v^{-1})=0$ ($\forall i\in I$), 
\item $T_{w}T_{w'}=T_{ww'}$ if $l(ww')=l(w)+l(w')$ ($w,w'\in W$), 
\item $\theta_{\lambda}T_{s_{i}}-T_{s_{i}}\theta_{s_{i}(\lambda)}=(v-v^{-1})\theta_{\frac{[\lambda]-[s_{i}(\lambda)]}{1-[-\alpha_{i}]}}$ ($\forall i\in I$, $\lambda\in\bb{X}^{\ast}(H)$), 
\item $\theta_{\lambda}\theta_{\lambda'}=\theta_{\lambda+\lambda'}$ ($\forall \lambda,\lambda'\in\bb{X}^{\ast}(H)$), 
\item $\theta_{0}=1$. 
\end{itemize}
Here, for $\lambda\in\bb{X}^{\ast}(H)$, we denote by $[\lambda]\in\bb{Z}[\bb{X}^{\ast}(H)]\cong K_{H}(\mr{pt})$ the corresponding element and for $c=\sum_{\lambda\in \bb{X}^{\ast}(H)}c_{\lambda}[\lambda]\in K_{H}(\mr{pt})$, we set $\theta_{c}=\sum_{\lambda\in\bb{X}^{\ast}(H)}c_{\lambda}\theta_{\lambda}$. It is known (see e.g. \cite{CG,Lu3}) that $\msc{H}$ acts on $K_{\bb{T}}(\mca{B})$ and $K_{\bb{T}}(X)$, and the inclusion $K_{\bb{T}}(\mca{B})\subset K_{\bb{T}}(X)$ is compatible with the $\msc{H}$-actions. We do not recall its construction, but in order to fix the convention, we write down its action on the fixed point basis. Our convention mainly follows that of \cite{Lu3}. 

The fixed points of $\mca{B}$ and $X$ with respect to the $H$-action are parametrized by $W$. For each $w\in W$, we also denote by $w\coloneqq wB\in\mca{B}$ the corresponding fixed point. We denote by $\mca{O}_{w}\in K_{\bb{T}}(\mca{B})\subset K_{\bb{T}}(X)$ the $K$-theory class of the structure sheaf of the fixed point $w$. The $\msc{H}$-action on the fixed point basis is given by 
\begin{itemize}
\item $T_{s_{i}}\mca{O}_{w}=\frac{v-v^{-1}}{1-[-w(\alpha_{i})]}\mca{O}_{w}+\frac{v^{-1}[-w(\alpha_{i})]-v}{1-[-w(\alpha_{i})]}\mca{O}_{ws_{i}}$ ($\forall i\in I$)
\item $\theta_{\lambda}\mca{O}_{w}=[w(\lambda)]\cdot \mca{O}_{w}$ ($\forall \lambda\in\bb{X}^{\ast}(H)$)
\end{itemize}
We note that the action of $\theta_{\lambda}$ is given by tensor product of an equivariant line bundle $\mca{L}_{\lambda}\coloneqq (G\times\bb{C}_{\lambda})/B$ on $\mca{B}$ (or its pullback to $X$), where the action of $B$ is given by $b\cdot(g,x)=(gb^{-1},\lambda(b)x)$. 

We fix a Lie algebra automorphism $\varpi:\mf{g}\rightarrow\mf{g}$ such that $\varpi(h)=-h$ for any $h\in\mf{h}$. This induces an automorphism of $\mca{B}$ and $X$, which is denoted by the same letter. Lusztig's bar involution $\beta^{L}:K_{\bb{T}}(X)\rightarrow K_{\bb{T}}(X)$ is then defined by 
\begin{align*}
\beta^{L}\coloneqq(-v)^{l(w_{0})}T_{w_{0}}^{-1}\varpi^{\ast}\bb{D}_{X}.
\end{align*}
One can define an inner product $(-||-)_{L}:K_{\bb{T}}(X)\times K_{\bb{T}}(X)\rightarrow\Frac(K_{\bb{T}}(\mr{pt}))$ in the same way as (\ref{innerprod}) replacing $\beta^{K}$ by $\beta^{L}$. In order to compute the image of $K$-theoretic stable bases by $\beta^{L}$, the following lemma proved in \cite{Lu4} is useful. 

\begin{lemma}[\cite{Lu4}, Lemma 8.13]\label{L_orth}
Let $\mf{Z}_{1}\coloneqq\mr{pr}^{-1}(e)\subset X$. Then for any $f,f'\in\bb{X}^{\ast}(H)$ and $w,w'\in W$, we have
\begin{align*}
(fT_{w}^{-1}\mca{O}_{\mf{Z}_{1}}||f'T_{w'}^{-1}\mca{O}_{\mf{Z}_{1}})_{L}=v^{-2l(w_{0})}ff'^{-1}\delta_{w,w'}.
\end{align*}
\end{lemma}

We remark that in \cite{Lu4}, this formula is proved more generally for Slodowy varieties associated with nilpotent elements which are regular in some Levi subalgebras. 

\subsection{Stable bases}

Next we recall a result of Su-Zhao-Zhong \cite{SZZ} describing the $K$-theoretic stable bases for Springer resolutions. For the definition and basic properties of $K$-theoretic stable bases used in this paper, we refer to section 3.2. In particular, the sign convention is slightly different from \cite{O1,OS}. 

We first fix the data defining the $K$-theoretic stable basis. For the chamber, we take the negative Weyl chamber $\mf{C}=\{x\in\mf{h}_{\bb{R}}\mid\forall i\in I, \langle x,\alpha_{i}\rangle<0 \}$. For the polarization $T^{1/2}$, we take the pullback of the tangent bundle of $\mca{B}$ by $X\rightarrow\mca{B}$. For the slope, we take $s\in\rho-A^{+}_{0}\subset\Pic(X)\otimes_{\bb{Z}}\bb{R}\cong\mf{h}^{\ast}_{\bb{R}}$, where $A^{+}_{0}=\{x\in\mf{h}_{\bb{R}}^{\ast}\mid\forall\alpha^{\vee}\in R_{+}^{\vee}, 0<\langle x,\alpha^{\vee}\rangle<1\}$ is the fundamental alcove. We note that $\rho$ corresponds to an actual (nonequivariant) line bundle on $X$ and we take an $\bb{T}$-equivariant lift $\mca{L}_{\rho}$. The following description of the $K$-theoretic stable basis is essentially proved in type $A$ by Rim\'anyi-Tarasov-Varchenko \cite{RTV} and in general by Su-Zhao-Zhong \cite{SZZ}. 

\begin{prop}[\cite{SZZ}, Theorem~0.1]\label{SZZ}
For any $w\in W$, we have 
\begin{align*}
\Stab^{K}_{\mf{C},T^{1/2},s}(w)=(-1)^{l(w)}[\rho-w\rho]\cdot T_{w}^{-1}\mca{O}_{\mf{Z}_{1}}
\end{align*}
\end{prop}

\begin{proof}
We only need to compare the convention in \cite{SZZ} with ours. First, we note that the choice of the Borel subgroup in \cite{SZZ} is opposite to ours, hence the fixed point corresponding to $w$ in \cite{SZZ} is $ww_{0}$ in our notation. We also note that the line bundle corresponding to $\lambda\in\bb{X}^{\ast}(H)$ is $\mca{L}_{w_{0}\lambda}$ in our notation. Hence the chamber is the same  and the slope in \cite{SZZ} is taken in $-A^{+}_{0}$. Moreover, the polarization is opposite to ours. 
Finally, one can check that the affine Hecke algebra action denoted by $T_{w}$ in \cite{SZZ} is given by $v^{l(w)}\mca{L}_{\rho}T_{w_{0}ww_{0}}\mca{L}_{-\rho}$ in our notation. This does not depend on the choice of $\mca{L}_{\rho}$. 

One can easily check that 
\begin{align*}
(-1)^{l(w_{0})}\Stab^{K}_{\mf{C},T^{1/2}_{\mr{opp}},-\rho+s}(e)=(-v)^{l(w_{0})}[2\rho]\cdot\mca{O}_{\mf{Z}_{1}}. 
\end{align*}
Therefore, Theorem 0.1 of \cite{SZZ} is 
\begin{align*}
(-1)^{l(w_{0})-l(w)}\Stab^{K}_{\mf{C},T^{1/2}_{\mr{opp}},-\rho+s}(w)=\mca{L}_{\rho}T_{w}^{-1}\mca{L}_{-\rho}\cdot (-v)^{l(w_{0})}[2\rho]\cdot\mca{O}_{\mf{Z}_{1}} 
\end{align*}
in our notation. Here, the sign in the left hand side comes from our convention on the sign of $K$-theoretic stable basis. By using Lemma \ref{polarization} and Lemma \ref{slope}, we obtain 
\begin{align*}
\Stab^{K}_{\mf{C},T^{1/2},s}(w)&=v^{-l(w_{0})}(\det T^{1/2}_{w})^{-1}i^{\ast}_{w}\mca{L}_{\rho}\cdot\mca{L}_{-\rho}\otimes\Stab^{K}_{\mf{C},T^{1/2}_{\mr{opp}},-\rho+s}(w)\\
&=(-1)^{l(w)}[\rho-w\rho]\cdot T_{w}^{-1}\mca{O}_{\mf{Z}_{1}}
\end{align*}
as required.
\end{proof}

\subsection{Comparison}

We now prove the following formula which was our starting point of this work. 

\begin{prop}\label{prop_Springer_comparison}
For any $w\in W$, we have 
\begin{align}\label{bar_reformulation}
\beta^{L}(\Stab^{K}_{\mf{C},T^{1/2},s}(w))=(-v)^{3l(w_{0})}\Stab^{K}_{-\mf{C},T^{1/2},s}(w). 
\end{align}
\end{prop}

\begin{proof}
By Lemma \ref{L_orth} and Proposition \ref{SZZ}, we obtain 
\begin{align*}
(\Stab^{K}_{\mf{C},T^{1/2},s}(w):\bb{D}_{X}\beta^{L}\Stab^{K}_{\mf{C},T^{1/2},s}(w'))=v^{-2l(w_{0})}\delta_{w,w'}
\end{align*}
for any $w,w'\in W$. On the other hand, Lemma \ref{duality} and Lemma \ref{orthogonality} imply that 
\begin{align*}
(\Stab^{K}_{\mf{C},T^{1/2},s}(w):\bb{D}_{X}\Stab^{K}_{-\mf{C},T^{1/2},s}(w'))=(-v)^{l(w_{0})}\delta_{w,w'}. 
\end{align*}
By comparing the two formulas, the proposition follows since the pairing $(-:-)$ is perfect after localization and $\{\Stab^{K}_{\mf{C},T^{1/2},s}(w)\}_{w\in W}$ forms a basis of localized equivariant $K$-theory. 
\end{proof}

Our main observation in this paper is that except for the normalization, an analogue of the formula (\ref{bar_reformulation}) makes sense if one can define $K$-theoretic stable basis in order to characterize certain antilinear map. Therefore, we try to make this kind of formula into a definition of $K$-theoretic bar involution. Our remaining task is to fix the normalization in some way, which turns out to be related to the notion of symplectic duality introduced by Braden-Licata-Proudfoot-Webster in \cite{BLPW3}.

\section{Main conjectures}

In this section, we propose a general definition of $K$-theoretic canonical bases for conical symplectic resolutions which have good Hamiltonian torus actions. We also formulate several conjectures about them which will be proved for toric hyper-K\"ahler manifolds in section 5. 

\subsection{Symplectic duality} 

First we briefly recall basic definitions on symplectic resolution and symplectic duality in the form we need later. In this paper, we mainly follow the setting of \cite{BLPW3} for symplectic resolution. For symplectic duality, we partly follow the definition of 3d mirror symmetry in \cite{RSVZ2}. This is designed to give certain symmetry of elliptic stable bases under the duality, see Conjecture~\ref{conj_ell_stab_duality}.

Let $X$ be a connected smooth algebraic variety over $\bb{C}$ with algebraic symplectic form $\omega$ and an $\bb{S}$-action. We assume that the $\bb{S}$-weight of $\omega$ is $2$ and the $\bb{S}$-action is conical, which means that $\bb{S}$-weights appearing in the coordinate ring $\bb{C}[X]$ are nonnegative and the weight $0$ part consists only of constant functions. If the natural morphism $\pi:X\rightarrow\Spec(\bb{C}[X])$ is a resolution of singularity, $X=(X,\omega,\bb{S})$ is called conical symplectic resolution. We denote by $o\in\Spec(\bb{C}[X])$ the unique $\bb{S}$-fixed point and $L\coloneqq\pi^{-1}(o)$ the central fiber of $\pi$. 

In this paper, we always assume for simplicity that there does not exist another conical symplectic resolution $X'$ and symplectic vector space $V\neq0$ such that $X\cong X'\times V$. We also assume that there exists an effective action of another algebraic torus $H$ on $X$ which is Hamiltonian and commute with the $\bb{S}$-action such that the fixed point set $X^{H}$ is finite. We always take maximal $H$ among such torus actions. For $p\in X^{H}$, we will denote by $i_{p}:\{p\}\hookrightarrow X$ the natural inclusion. We set $\bb{T}\coloneqq H\times\bb{S}$. 

For any $p\in X^{H}$, we denote by $\Phi(p)$ the multiset of $H$-weights appearing in the tangent space $T_{p}X$ at $p$. We call it the multiset of equivariant roots at $p$. We also simply call an element in the union (as a set) $\overline{\Phi}\coloneqq\cup_{p\in X^{H}}\Phi(p)$ equivariant root for $X$. An equivariant root $\alpha$ define a hyperplane in $\mf{h}_{\bb{R}}=\bb{X}_{\ast}(H)\otimes_{\bb{Z}}\bb{R}$ by $H_{\alpha}\coloneqq\{\xi\in\mf{h}_{\bb{R}}\mid \langle \xi,\alpha\rangle=0\}$. A connected component $\mf{C}$ in the complement $\mf{h}_{\bb{R}}\setminus\cup_{\alpha\in\overline{\Phi}}H_{\alpha}$ is called chamber, and it gives a decomposition of the tangent space $T_{p}X=N_{p,+}\oplus N_{p,-}$ into attracting and repelling parts. We denote by $\Attr_{\mf{C}}(p)\coloneqq\{x\in X\mid\lim_{t\rightarrow0}\xi(t)\cdot x=p\}$ the attracting set of $p$ with respect to $\mf{C}$, where $\xi\in\bb{X}_{\ast}(H)$ is a one parameter subgroup of $H$ contained in $\mf{C}$. We note that this does not depend on the choice of $\xi$ in $\mf{C}$. The choice of chamber also gives a partial order $\preceq_{\mf{C}}$ on $X^{H}$ generated by $p\in\overline{\Attr_{\mf{C}}(p')}\implies p\preceq_{\mf{C}}p'$. 

We set $P\coloneqq\Pic(X)$ and $P^{\vee}\coloneqq\Hom_{\bb{Z}}(P,\bb{Z})$ its dual. We denote by $P_{\bb{R}}\coloneqq P\otimes_{\bb{Z}}\bb{R}$ and $\mf{A}\subset P_{\bb{R}}$ the ample cone of $X$. We will also need the notion of K\"ahler roots at each fixed point, but unfortunately, we do not know intrinsic definition of this notion. Our temporary definition is to take the equivariant roots at the corresponding fixed point for dual conical symplectic resolution. In this paper, we consider them as additional data and take a multiset $\Psi(p)$ of elements in $P^{\vee}$ for each $p\in X^{H}$. As a condition they should satisfy, we assume that the walls in the slope parameters where the $K$-theoretic stable basis $\Stab^{K}_{\mf{C},T^{1/2},s}(p)$ recalled in the next section jump are contained in the walls of the form $\{s\in P_{\bb{R}}\mid \langle s,\beta\rangle\in\bb{Z}\}$ for some $\beta\in\Psi(p)$\footnote{In this paper, the notations for the bar involutions are always equipped with some upper index, and simple $\beta$ is used to denote a K\"ahler root. We hope this notation does not cause any confusions. }. We also assume that $\mf{A}$ is a connected component of the complement of linear hyperplanes in $P_{\bb{R}}$ defined by the K\"ahler roots in the union $\overline{\Psi}\coloneqq\cup_{p\in X^{H}}\Psi(p)$. 

\begin{remark}
For quiver varieties, one can read off the information about $\Psi(p)$ without using symplectic duality by a conjecture of Dinkins-Smirnov \cite{DS}. However, this does depend on the presentation of $X$ as a GIT quotient. Hence we need to allow some factor of symplectic vector spaces and modify the statement of symplectic duality to include some information about how to present the conical symplectic resolutions. 

Since these differences only affect the overall constants in the $K$-theoretic bar involutions and canonical bases, we do not pursue this point further here. We only note that this freedom on the normalization is important for example when one try to compare our definition to the notion of global crystal bases of level 0 extremal weight modules of quantum affine algebras defined by Kashiwara \cite{K}. For the canonical bases in equivariant $K$-theory of $ADE$ type quiver varieties defined by Varagnolo-Vasserot \cite{VV}, this kind of comparison is given by Nakajima \cite{N}.
\end{remark}

We also take an equivariant lift $\mf{L}:P\rightarrow\Pic^{\bb{T}}(X)$, i.e., a section of the natural homomorphism $\Pic^{\bb{T}}(X)\rightarrow P$ given by forgetting the equivariant structures. In this paper, when we consider symplectic resolutions, they are always equipped with the above additional data such as $\mf{C}$, $\Psi(p)$, and $\mf{L}$. We now formulate a notion of dual pair between conical symplectic resolutions $X=(X,\mf{C},\mf{A},\Phi,\Psi,\mf{L})$ and $X^{!}=(X^{!},\mf{C}^{!},\mf{A}^{!},\Phi^{!},\Psi^{!},\mf{L}^{!})$ as follows. 

\begin{dfn}\label{dual_pair}
We say that a pair of conical symplectic resolutions $X$ and $X^{!}$ forms a dual pair if 
\begin{itemize}
\item There exists an order reversing bijection $(X^{H},\preceq_{\mf{C}})\cong((X^{!})^{H^{!}},\succeq_{\mf{C}^{!}})$. We denote by $p^{!}\in(X^{!})^{H^{!}}$ the fixed point corresponding to $p\in X^{H}$;
\item There exist isomorphisms $\bb{X}_{\ast}(H)\cong P^{!}$ and $P\cong\bb{X}_{\ast}(H^{!})$ such that under this identification, we have $\mf{C}=\mf{A}^{!}$, $\mf{A}=\mf{C}^{!}$, $\Phi(p)=\Psi(p^{!})$, and $\Psi(p)=\Phi(p^{!})$;
\item For any $\lambda\in P$, $\lambda^{!}\in P^{!}$, and $p\in X^{H}$, we have 
\begin{align}
\langle\wt_{H}i^{\ast}_{p}\mf{L}(\lambda),\lambda^{!}\rangle&=-\langle\wt_{H^{!}}i^{\ast}_{p^{!}}\mf{L}^{!}(\lambda^{!}),\lambda\rangle,\label{eqn_dual_pair_H}\\
\wt_{\bb{S}}i^{\ast}_{p}\mf{L}(\lambda)&=-\langle\wt_{H^{!}}\det N^{!}_{p^{!},-},\lambda\rangle,\label{eqn_dual_pair_S}\\
\wt_{\bb{S}^{!}}i^{\ast}_{p^{!}}\mf{L}^{!}(\lambda^{!})&=-\langle\wt_{H}\det N_{p,-},\lambda^{!}\rangle;\label{eqn_dual_pair_S_dual}
\end{align}
\item For any $p\in X^{H}$, we have 
\begin{align}\label{eqn_S_wt}
\wt_{\bb{S}}\det N_{p,-}+\frac{1}{2}\dim X=-\left(\wt_{\bb{S}^{!}}\det N^{!}_{p^{!},-}+\frac{1}{2}\dim X^{!}\right).
\end{align} 
\end{itemize}
\end{dfn}

Let $\{a_{1},\ldots,a_{e}\}$ be a basis of $\bb{X}^{\ast}(H)$ considered as elements of $K_{H}(\mr{pt})$ and $\{c_{1},\ldots,c_{e}\}$ be the dual basis of $\bb{X}_{\ast}(H)\cong P^{!}$. Similarly, let $\{z_{1},\ldots,z_{r}\}$ be a basis of $P^{\vee}\cong\bb{X}^{\ast}(H^{!})$ considered as elements of $K_{H^{!}}(\mr{pt})$ and $\{l_{1},\ldots,l_{r}\}$ be the dual basis of $P$. We also identify $\bb{S}^{!}=\bb{S}$ and their equivariant parameters by $v^{!}=v$. In these notations, the third condition in Definition~\ref{dual_pair} is equivalent to the following:
\begin{align}
i^{\ast}_{p}\mf{L}(\lambda)&=\prod_{i=1}^{e}a_{i}^{-\langle i^{\ast}_{p^{!}}\mf{L}^{!}(c_{i}),\lambda\rangle}\cdot v^{-\langle\det N^{!}_{p^{!},-},\lambda\rangle},\label{eqn_line_bundle}\\
i^{\ast}_{p^{!}}\mf{L}^{!}(\lambda^{!})&=\prod_{i=1}^{r}z_{i}^{-\langle i^{\ast}_{p}\mf{L}(l_{i}),\lambda^{!}\rangle}\cdot v^{-\langle\det N_{p,-},\lambda^{!}\rangle}.\label{eqn_line_bundle_dual}
\end{align}

\begin{exa}
Let $X=T^{\ast}(G/B)$ be the Springer resolution as in section 2. We recall that $G$ is of adjoint type. Let $G^{\vee}$ be the adjoint type semisimple algebraic group whose Lie algebra is the Langlands dual of $\mf{g}$. We claim that the pair $X$ and $X^{!}=T^{\ast}(G^{\vee}/B^{\vee})$ with certain choice of data forms a dual pair in the above sense. Here, we take a Borel subgroup $B^{\vee}\subset G^{\vee}$ as in section 2. 

In this case, it is well-known that the Picard group $P$ of $X$ is given by the weight lattice and the ample cone is given by the positive Weyl chamber in our convention. For example, we take the chamber $\mf{C}$ to be the negative Weyl chamber as in section 2.2. Since the ample cone of $X^{!}$ is the positive Weyl chamber, we twist the isomorphism $P^{!}\cong\bb{X}_{\ast}(H)$ by $-1$ so that $\mf{C}$ and $\mf{A}^{!}$ corresponds to each other. We take $\mf{C}^{!}$ to be the positive Weyl chamber and the isomorphism $P\cong\bb{X}_{\ast}(H^{!})$ to be the natural one. We take $\Psi(w)$ to be the set of all coroots for any $w\in W$. The data $\Psi^{!}$ is chosen similarly. For $\mf{L}:P\rightarrow\Pic^{\bb{T}}(X)$, we take $\mf{L}(\lambda)=v^{\langle2\rho^{\vee},\lambda\rangle}[-\lambda]\cdot\mca{L}_{\lambda}$ for any $\lambda\in P$, where $\mca{L}_{\lambda}$ is defined as in section 2.1 and $2\rho^{\vee}\in P^{\vee}$ is the sum of all positive coroots. We note that the shift $[-\lambda]$ is needed to make them $H$-equivariant. Similarly, we take $\mf{L}^{!}(\lambda^{!})=v^{\langle2\rho,\lambda^{!}\rangle}[-\lambda^{!}]\cdot\mca{L}^{!}_{\lambda^{!}}$. Here, we consider $2\rho$ as an element of $(P^{!})^{\vee}$. 

We identify $X^{H}\cong W$ and $(X^{!})^{H^{!}}\cong W$ by $w\leftrightarrow w^{-1}$. One can check that the partial order on $X^{H}$ given by $\mf{C}$ is the Bruhat order and the partial order on $(X^{!})^{H^{!}}$ given by $\mf{C}^{!}$ is the opposite Bruhat order. Therefore, this gives an order reversing bijection. The second condition in Definition~\ref{dual_pair} is obvious from our choice. We note that $\det N_{w,-}=v^{-2l(w)}[2\rho]$ and $\det N^{!}_{w^{-1},-}=v^{-2l(w_{0})+2l(w)}[-2\rho^{\vee}]$. This easily implies (\ref{eqn_dual_pair_S}), (\ref{eqn_dual_pair_S_dual}), and (\ref{eqn_S_wt}). Here, we note that $2\rho\in\bb{X}^{\ast}(H)$ is identified with $-2\rho\in(P^{!})^{\vee}$. We note that $\wt_{H}i^{\ast}_{w}\mf{L}(\lambda)=w\lambda-\lambda\in\bb{X}^{\ast}(H)$ and $\wt_{H^{!}}i^{\ast}_{w^{-1}}\mf{L}^{!}(\lambda^{!})=w^{-1}\lambda^{!}-\lambda^{!}\in\bb{X}^{\ast}(H^{!})$ for any $\lambda\in P$ and $\lambda^{!}\in P^{!}$. Since $\lambda^{!}\in P^{!}$ is identified with $-\lambda^{!}\in\bb{X}_{\ast}(H)$, (\ref{eqn_dual_pair_H}) follows from the obvious equation $\langle w\lambda-\lambda,-\lambda^{!}\rangle=-\langle w^{-1}\lambda^{!}-\lambda^{!},\lambda\rangle$. 
\end{exa}

We conjecture that if $X$ and $X^{!}$ are symplectic dual in the sense of Braden-Licata-Proudfoot-Webster \cite{BLPW3}, then they form a dual pair in the sense of Definition~\ref{dual_pair}. We will check this for toric hyper-K\"ahler manifolds in Proposition~\ref{prop_toric_dual_pair}. In this paper, we always assume that a symplectic resolution considered in this paper is equipped with a dual symplectic resolution in the above sense. 

\begin{ass}\label{ass_dual_pair}
For any choice of $\mf{C}$, the conical symplectic resolution $X=(X,\mf{C},\mf{A},\ldots)$ has a dual conical symplectic resolution $X^{!}=(X^{!},\mf{C}^{!},\mf{A}^{!},\ldots)$ such that $X$ and $X^{!}$ forms a dual pair in the sense of Definition~\ref{dual_pair}.
\end{ass}

We note that if $\mca{L}\in\Pic^{\bb{T}}(X)$ is a $\bb{T}$-equivariant line bundle and $l\in P$ is its underlying line bundle, then $\mca{L}\otimes\mf{L}(l)^{-1}$ is a trivial as a non-equivariant line bundle on $X$. The Assumption~\ref{ass_dual_pair} implies that 
\begin{align*}
w(\mca{L})\coloneqq\wt_{\bb{S}}i^{\ast}_{p}\mca{L}-\sum_{\beta\in\Psi_{+}(p)}\langle\beta,l\rangle
\end{align*}
does not depend on the choice of $p\in X^{H}$. Here, $\Psi_{+}(p)$ is the sub-multiset of $\Psi(p)$ consisting of K\"ahler roots at $p$ which are positive with respect to the ample cone $\mf{A}$. We note that $w:\Pic^{\bb{T}}(X)\rightarrow\bb{Z}$ gives a homomorphism and the image of $\mf{L}$ lands in the kernel of $w$. 

\subsection{$K$-theoretic standard bases}

In this section, we recall the definition of $K$-theoretic stable basis introduced in \cite{O1,OS}. In order to define this, we need to choose further data. For an element $\mca{F}\in K_{\bb{T}}(X)$, we denote by $\mca{F}^{\vee}\coloneqq [R\Hom(\mca{F},\mca{O}_{X})]\in K_{\bb{T}}(X)$ its dual. For a $\bb{T}$-equivariant vector bundle $\mca{V}$, we set $\bigwedge^{\bullet}_{-}\mca{V}\coloneqq\sum_{i\geq0}(-1)^{i}\bigwedge^{i}\mca{V}\in K_{\bb{T}}(X)$, where $\bigwedge^{i}\mca{V}$ is the $i$-th exterior product of $\mca{V}$. We take an element $T^{1/2}\in K_{\bb{T}}(X)$ called polarization satisfying $T^{1/2}+v^{-2}(T^{1/2})^{\vee}=T_{X}$, where $T_{X}$ is the $K$-theory class of the tangent bundle of $X$. We note that for any $\mca{G}\in K_{\bb{T}}(X)$, $T^{1/2}_{\mca{G}}\coloneqq T^{1/2}-\mca{G}+v^{-2}\mca{G}^{\vee}$ and $T^{1/2}_{\mr{opp}}\coloneqq v^{-2}(T^{1/2})^{\vee}$ are also polarization for $X$.

We also take a generic element $s\in P_{\bb{R}}\setminus\cup_{\beta\in\overline{\Psi}}\{s\in P_{\bb{R}}\mid \langle s,\beta\rangle\in\bb{Z}\}$ called slope. The map $\mf{L}$ naturally extends to give a fractional line bundle $\mf{L}(s)\in\Pic^{\bb{T}}(X)\otimes_{\bb{Z}}\bb{R}$ and each restriction at a fixed point $i^{\ast}_{p}\mf{L}(s)$ gives an element of $\bb{X}^{\ast}(\bb{T})\otimes_{\bb{Z}}\bb{R}$. For an element of the form $m=\sum_{\mu\in\mf{h}^{\ast}_{\bb{R}}}m_{\mu}\mu$, we denote by $\deg_{H}(m)\subset\mf{h}^{\ast}_{\bb{R}}$ the convex hull of $\{\mu\in\mf{h}^{\ast}_{\bb{R}}\mid m_{\mu}\neq0\}$. 

\begin{dfn}[\cite{O1,OS}]\label{stab_def}
A set of elements $\{\Stab^{K}_{\mf{C},T^{1/2},s}(p)\}_{p\in X^{H}}$ of $K_{\bb{T}}(X)$ is called stable basis if it satisfies the following conditions:
\begin{itemize}
\item $\Supp(\Stab^{K}_{\mf{C},T^{1/2},s}(p))\subset\sqcup_{p'\preceq_{\mf{C}}p}\Attr_{\mf{C}}(p')$;
\item $i^{\ast}_{p}\Stab^{K}_{\mf{C},T^{1/2},s}(p)=\sqrt{\frac{\det N_{p,-}}{\det T^{1/2}_{p}}}\cdot\bigwedge^{\bullet}_{-}N^{\vee}_{p,-}$;\footnote{We changed the sign and normalization from \cite{O1,OS}. We note that this only depends on the determinant of $T^{1/2}$.}
\item $\deg_{H}\left(i^{\ast}_{p'}\Stab^{K}_{\mf{C},T^{1/2},s}(p)\cdot i^{\ast}_{p}\mf{L}(s)\right)\subset \deg_{H}\left(i^{\ast}_{p'}\Stab^{K}_{\mf{C},T^{1/2},s}(p')\cdot i^{\ast}_{p'}\mf{L}(s)\right)$ for any $p'\preceq_{\mf{C}}p\in X^{H}$.
\end{itemize}
\end{dfn}

Here, the square root appearing in the normalization is well-defined by \cite{O1,OS}. We note that our normalization is different from \cite{O1,OS} by $\sqrt{\det T^{1/2}_{p,=0}}$ where $T^{1/2}_{p,=0}$ is the $H$-invariant part of $T^{1/2}_{p}\coloneqq i_{p}^{\ast}T^{1/2}$. The existence of $\sqrt{\det T^{1/2}_{p,=0}}$ follows from the equation $T^{1/2}_{p,=0}+v^{-2}(T^{1/2}_{p,=0})^{\vee}=0$. If the slope $s$ is sufficiently generic so that $\wt_{H}i^{\ast}_{p}\mf{L}(s)-\wt_{H}i^{\ast}_{p'}\mf{L}(s)\notin\bb{X}^{\ast}(H)$ for any $p\neq p'\in X^{H}$, then the $K$-theoretic stable basis is unique if it exists by \cite[Proposition~9.2.2]{O1}. The existence is claimed in some generality in \cite{O2} and proved for example when $X$ is a toric hyper-K\"ahler manifold, quiver variety \cite{AO}, or Springer resolution \cite{SZZ}. In this paper, we assume that the $K$-theoretic stable bases exist uniquely for any $X$ and the additional data we consider. 

\begin{ass}\label{ass_stab_existence}
The slope $s\in P_{\bb{R}}$ satisfies $\wt_{H}i^{\ast}_{p}\mf{L}(s)-\wt_{H}i^{\ast}_{p'}\mf{L}(s)\notin\bb{X}^{\ast}(H)$ for any $p\neq p'\in X^{H}$. Moreover, $\Stab^{K}_{\mf{C},T^{1/2},s}(p)$ exists for any $p\in X^{H}$. 
\end{ass}

We next collect some basic results on the $K$-theoretic stable bases for our reference since we have changed the convention slightly. 

\begin{lemma}[\cite{O1}, Exercise 9.1.2]\label{polarization}
For any $\mca{G}\in K_{\bb{T}}(X)$ and $p\in X^{H}$, we have 
\begin{align*}
\Stab^{K}_{\mf{C},T^{1/2}_{\mca{G}},s}(p)=v^{\rank \mca{G}}\det i^{\ast}_{p}\mca{G}\cdot\Stab^{K}_{\mf{C},T^{1/2},s+\det\mca{G}}(p).
\end{align*}
\end{lemma}

\begin{lemma}\label{slope}
For any $l\in P$ and $p\in X^{H}$, we have 
\begin{align*}
\Stab^{K}_{\mf{C},T^{1/2},s+l}(p)=(i^{\ast}_{p}\mf{L}(l))^{-1}\cdot\mf{L}(l)\otimes\Stab^{K}_{\mf{C},T^{1/2},s}(p).
\end{align*}
\end{lemma}

\begin{lemma}[\cite{OS}]\label{duality}
For any $p\in X^{H}$, we have 
\begin{align*}
\Stab^{K}_{\mf{C},T^{1/2},s}(p)^{\vee}=(-v)^{-\frac{1}{2}\dim X}\Stab^{K}_{\mf{C},T^{1/2}_{\mr{opp}},-s}(p).
\end{align*}
\end{lemma}

These formulas easily follow from the definition and the uniqueness of $K$-theoretic stable basis. Recall the inner product $(-:-)$ defined in (\ref{usualinnerprod}). The $K$-theoretic stable bases have the following orthogonality property with respect to $(-:-)$.  

\begin{lemma}[\cite{OS}, Proposition 1]\label{orthogonality}
For any $p,p'\in X^{H}$, we have 
\begin{align*}
\left(\Stab^{K}_{\mf{C},T^{1/2},s}(p):\Stab^{K}_{-\mf{C},T^{1/2}_{\mr{opp}},-s}(p')\right)=\delta_{p,p'}.
\end{align*}
\end{lemma}

In order to define $K$-theoretic bar involutions, we further renormalize the $K$-theoretic stable bases. For each $p\in X^{H}$, we set $$a_{p}(T^{1/2},s)\coloneqq\sum_{\beta\in\Psi_{+}(p)}\left(\lfloor\langle s,\beta\rangle\rfloor+\frac{1}{2}\right)-\frac{1}{4}\dim X+\frac{1}{2}w(\det T^{1/2})\in\frac{1}{2}\bb{Z}.$$  By numerical experiments, we expect that $a_{p}(T^{1/2},s)\in\bb{Z}$. This is equivalent to the following assumption which is also assumed in this paper. 

\begin{ass}\label{assumption_polarization}
There exists a polarization $T^{1/2}$ on $X$ such that 
\begin{align*}
w(\det T^{1/2})\equiv\frac{1}{2}\dim X+\frac{1}{2}\dim X^{!}\mod 2.
\end{align*}
\end{ass}

We note that if $T^{1/2}$ satisfies the above condition, then $T^{1/2}_{\mca{G}}$ also satisfies the condition for any $\mca{G}\in K_{\bb{T}}(X)$. We set 
\begin{align}\label{label_set}
\bb{F}\coloneqq\{(\lambda,p)\mid \lambda\in\bb{X}^{\ast}(H),p\in X^{H}\}.
\end{align}
This set will label the $K$-theoretic standard bases and $K$-theoretic canonical bases. 

\begin{dfn}\label{general_standard_basis}
For any $(\lambda,p)\in\bb{F}$, we set 
\begin{align*}
\mca{S}_{\mf{C},T^{1/2},s}(\lambda,p)\coloneqq(-1)^{\frac{1}{2}\dim X}v^{a_{p}(T^{1/2},s)}[\lambda]\cdot\Stab^{K}_{\mf{C},T^{1/2},s}(p).
\end{align*}
We call the set $\bb{B}^{\mr{std}}_{\mf{C},T^{1/2},s}\coloneqq\{\mca{S}_{\mf{C},T^{1/2},s}(\lambda,p)\}_{(\lambda,p)\in\bb{F}}\subset K_{\bb{T}}(X)$ \emph{$K$-theoretic standard basis} for $X$. 
\end{dfn} 

We will simply write $\mca{S}_{\mf{C},T^{1/2},s}(p)\coloneqq\mca{S}_{\mf{C},T^{1/2},s}(0,p)$. The standard basis should be considered as a $\bb{Z}[v,v^{-1}]$-basis for the equivariant $K$-theory of the full attracting set $\sqcup_{p\in X^{H}}\Attr_{\mf{C}}(p)$. One explanation of the seemingly strange normalization in this definition will be given in section 4.3 by considering its elliptic analogue. Here, we list some basic properties of the $K$-theoretic standard bases for our reference. 

\begin{lemma}\label{std_polarization}
For any $\mca{G}\in K_{\bb{T}}(X)$ and $(\lambda,p)\in\bb{F}$, we have 
\begin{align*}
\mca{S}_{\mf{C},T^{1/2}_{\mca{G}},s}(\lambda,p)=\mca{S}_{\mf{C},T^{1/2},s+\det\mca{G}}(\lambda+\wt_{H}\det i^{\ast}_{p}\mca{G},p).
\end{align*}
\end{lemma}

\begin{proof}
This follows from Lemma~\ref{polarization} and $a_{p}(T^{1/2}_{\mca{G}},s)=a_{p}(T^{1/2},s+\det\mca{G})-\rank\mca{G}-\wt_{\bb{S}}\det i^{\ast}_{p}\mca{G}$. 
\end{proof}

\begin{lemma}\label{std_slope}
For any $l\in P$ and $(\lambda,p)\in\bb{F}$, we have 
\begin{align*}
\mca{S}_{\mf{C},T^{1/2},s+l}(\lambda,p)=\mf{L}(l)\otimes\mca{S}_{\mf{C},T^{1/2},s}(\lambda-\wt_{H}i^{\ast}_{p}\mf{L}(l),p).
\end{align*}
\end{lemma}

\begin{proof}
This follows from Lemma~\ref{slope} and $a_{p}(T^{1/2},s+l)=a_{p}(T^{1/2},s)+\wt_{\bb{S}}\det i^{\ast}_{p}\mf{L}(l)$. 
\end{proof}

\begin{lemma}\label{std_duality}
For any $(\lambda,p)\in\bb{F}$, we have 
\begin{align*}
\mca{S}_{\mf{C},T^{1/2},s}(\lambda,p)^{\vee}=(-v)^{\frac{1}{2}\dim X}\mca{S}_{\mf{C},T^{1/2}_{\mr{opp}},-s}(-\lambda,p).
\end{align*}
\end{lemma}

\begin{proof}
This follows from Lemma~\ref{duality} and $a_{p}(T^{1/2}_{\mr{opp}},-s)=-a_{p}(T^{1/2},s)-\dim X$.
\end{proof}

\begin{lemma}\label{std_orthogonality}
For any $p,p'\in X^{H}$, we have 
\begin{align*}
\left(\mca{S}_{\mf{C},T^{1/2},s}(p):\mca{S}_{-\mf{C},T^{1/2}_{\mr{opp}},-s}(p')\right)=v^{-\dim X}\delta_{p,p'}.
\end{align*}
\end{lemma}

\begin{proof}
This follows from Lemma~\ref{orthogonality} and $a_{p}(T^{1/2}_{\mr{opp}},-s)=-a_{p}(T^{1/2},s)-\dim X$.
\end{proof}

\subsection{$K$-theoretic bar involution}\label{bar_involution_section}
 
Now we define the $K$-theoretic bar involution. Since $\{\mca{S}_{\mf{C},T^{1/2},s}(p)\}_{p\in X^{H}}$ forms a basis of $K_{\bb{T}}(X)_{\mr{loc}}$ over $\Frac(K_{\bb{T}}(\mr{pt}))$, we can define a $K_{H}(\mr{pt})$-linear map $\beta^{K}=\beta^{K}_{\mf{C},T^{1/2},s}:K_{\bb{T}}(X)_{\mr{loc}}\rightarrow K_{\bb{T}}(X)_{\mr{loc}}$ by the following conditions:
\begin{itemize}
\item $\beta^{K}(vm)=v^{-1}\beta^{K}(m)$ for any $m\in K_{\bb{T}}(X)_{\mr{loc}}$;
\item $\beta^{K}\left(\mca{S}_{\mf{C},T^{1/2},s}(p)\right)=(-v)^{\frac{\dim X}{2}}\mca{S}_{-\mf{C}, T^{1/2},s}(p)$ for any $p\in X^{H}$.
\end{itemize}
We note that $\beta^{K}_{-\mf{C},T^{1/2},s}\circ\beta^{K}_{\mf{C},T^{1/2},s}=id$. The following conjecture implies that $\beta^{K}$ is an involution, and hence we call it \emph{$K$-theoretic bar involution} associated with the data $\mf{C}$, $T^{1/2}$, and $s$. 

\begin{conj}\label{K-theoretic_bar_involution}
The $K$-theoretic bar involution $\beta^{K}_{\mf{C},T^{1/2},s}$ does not depend on the choice of $\mf{C}$. 
\end{conj}

For toric hyper-K\"ahler manifolds, this conjecture is proved in Corollary~\ref{cor_toric_indep_chamber_K}. Assuming this conjecture, we will sometimes omit $\mf{C}$ from the notation and write $\beta^{K}_{T^{1/2},s}=\beta^{K}_{\mf{C},T^{1/2},s}$. In this section, we prove certain triangular properties of $\beta^{K}$ with respect to the $K$-theoretic standard bases. We first define a partial order on the labeling set $\bb{F}$.

\begin{dfn}\label{dfn_poset}
For $(\lambda,p),(\lambda',p')\in\bb{F}$, we write $(\lambda,p)\leq_{\mf{C},s}(\lambda',p')$ if we have $\langle\lambda-\wt_{H}i^{\ast}_{p}\mf{L}(s),\xi\rangle\leq\langle\lambda'-\wt_{H}i^{\ast}_{p'}\mf{L}(s),\xi\rangle$ for any $\xi\in\mf{C}$.
\end{dfn}

\begin{lemma}\label{partial_order}
The relation $\leq_{\mf{C},s}$ defines a partial order on $\bb{F}$. Moreover, the number of elements $
(\lambda,p)\in\bb{F}$ satisfying $(\lambda',p')\leq_{\mf{C},s}(\lambda,p)\leq_{\mf{C},s}(\lambda'',p'')$ is finite for any $(\lambda',p'),(\lambda'',p'')\in\bb{F}$, i.e., the partial order $\leq_{\mf{C},s}$ is interval finite. 
\end{lemma}

\begin{proof}
If $(\lambda,p)\leq_{\mf{C},s}(\lambda',p')$ and $(\lambda',p')\leq_{\mf{C},s}(\lambda,p)$, then we have $\lambda-\wt_{H}i^{\ast}_{p}\mf{L}(s)=\lambda'-\wt_{H}i^{\ast}_{p'}\mf{L}(s)$. This implies that $\wt_{H}i^{\ast}_{p}\mf{L}(s)-\wt_{H}i^{\ast}_{p'}\mf{L}(s)\in\bb{X}^{\ast}(H)$ and hence $p=p'$ by Assumption~\ref{ass_stab_existence}. This proves the antisymmetry. The other properties are trivial to check. 

The second claim follows from the compactness of the set
$\bigcap_{\xi\in\mf{C}}\{\mu\in\mf{h}^{\ast}_{\bb{R}}\mid\langle\mu',\xi\rangle\leq\langle\mu,\xi\rangle\leq\langle\mu'',\xi\rangle \}$ for any $\mu',\mu''\in\mf{h}^{\ast}_{\bb{R}}$. 
\end{proof}

Using this partial order, we define a $\bb{Z}[v,v^{-1}]$-module of formal sums 
\begin{align*}
\bb{M}_{\mf{C},s}\coloneqq\left\{\sum_{(\lambda,p)\in\bb{F}}f_{\lambda,p}(v)S_{\lambda,p}\middle| \begin{aligned}f_{\lambda,p}(v)\in\bb{Z}[v,v^{-1}],
 \exists (\lambda_{1},p_{1}),\ldots, (\lambda_{m},p_{m})\in\bb{F}\\\mbox{ s.t. if }f_{\lambda,p}\neq0, \mbox{then }
 (\lambda,p)\geq_{\mf{C},s}(\lambda_{i},p_{i})\mbox{ for some }i
\end{aligned}
\right\}.
\end{align*}
For any $\mca{F}\in K_{\bb{T}}(X)$, we have $\mca{F}=\sum_{p\in X^{H}}v^{\dim X}\left(\mca{F}:\mca{S}_{-\mf{C},T^{1/2}_{\mr{opp}},-s}(p)\right)\cdot\mca{S}_{\mf{C},T^{1/2},s}(p)$ by Lemma~\ref{std_orthogonality} and the possible denominators appearing in the inner product are of the form $\bigwedge^{\bullet}_{-}(T^{\ast}_{p}X)$ for some $p\in X^{H}$. Therefore, by sending $\mca{S}_{\mf{C},T^{1/2},s}(\lambda,p)$ to $S_{\lambda,p}$ and expanding the rational function appearing in the coefficients into formal series in the positive or negative direction with respect to $\mf{C}$, we obtain two natural embeddings $\iota^{\pm}_{\mf{C},T^{1/2},s}:K_{\bb{T}}(X)\hookrightarrow\bb{M}_{\pm\mf{C},s}$. The following lemma together with Lemma~\ref{partial_order} implies that one can extend the $K$-theoretic bar involution to $\bb{M}_{\pm\mf{C},s}$. 

\begin{lemma}\label{R_triangular}
For each $(\lambda,p)\in\bb{F}$, we have
\begin{align*}
\iota^{\pm}_{\mf{C},T^{1/2},s}((-v)^{\mp\frac{\dim X}{2}}\mca{S}_{-\mf{C},T^{1/2},s}(\lambda,p))\in S_{\lambda,p}+\sum_{(\lambda',p')>_{\pm\mf{C},s}(\lambda,p)}\bb{Z}[v,v^{-1}]\cdot S_{\lambda',p'}.
\end{align*}
\end{lemma}

\begin{proof}
We first note that 
\begin{align}
v^{\dim X}\left(\mca{S}_{-\mf{C},T^{1/2},s}(p):\mca{S}_{-\mf{C},T^{1/2}_{\mr{opp}},-s}(p')\right)=\sum_{p''\in X^{H}}\frac{i^{\ast}_{p''}\Stab^{K}_{-\mf{C},T^{1/2},s}(p)\cdot i^{\ast}_{p''}\Stab^{K}_{-\mf{C},T^{1/2}_{\mr{opp}},-s}(p')}{\bigwedge^{\bullet}_{-}(N_{p'',+}^{\vee})\cdot\bigwedge^{\bullet}_{-}(N_{p'',-}^{\vee})}\label{inner_sum}
\end{align}
If we write $N_{p'',+}=\sum_{i}w_{i}$, $w_{i}\in\bb{X}^{\ast}(\bb{T})$, then we have 
\begin{align*}
\frac{i^{\ast}_{p''}\Stab^{K}_{-\mf{C},T^{1/2},s}(p'')\cdot i^{\ast}_{p''}\Stab^{K}_{-\mf{C},T^{1/2}_{\mr{opp}},-s}(p'')}{\bigwedge^{\bullet}_{-}(N_{p'',+}^{\vee})\cdot\bigwedge^{\bullet}_{-}(N_{p'',-}^{\vee})}&=v^{\frac{\dim X}{2}}\det N_{p'',+}\cdot\frac{\bigwedge^{\bullet}_{-}(N_{p'',+}^{\vee})}{\bigwedge^{\bullet}_{-}(N_{p'',-}^{\vee})}\\
&=(-v)^{\frac{\dim X}{2}}\prod_{i}\frac{1-w_{i}}{1-v^{2}w_{i}}.
\end{align*}
If we expand this in positive (resp. negative) direction with respect to $\mf{C}$, then the expansion start from $(-v)^{\frac{\dim X}{2}}$ (resp. $(-v)^{-\frac{\dim X}{2}}$). This formula also implies that if we set $\rho_{p''}\coloneqq\wt_{H}\det N_{p'',+}$, then for any $\xi\in\mf{C}$, we have 
\begin{align*}
\left\langle\xi,\deg_{H}\left(i^{\ast}_{p''}\Stab^{K}_{-\mf{C},T^{1/2},s}(p'')\cdot i^{\ast}_{p''}\Stab^{K}_{-\mf{C},T^{1/2}_{\mr{opp}},-s}(p'')\right)\right\rangle=\left[-\langle\xi,\rho_{p''}\rangle,\langle\xi,\rho_{p''}\rangle\right]
\end{align*}

On the other hand, by the definition of $K$-theoretic stable basis, we have
\begin{align}
\deg_{H}\left(i^{\ast}_{p''}\Stab^{K}_{-\mf{C},T^{1/2},s}(p)\right)&\subset\deg_{H}\left(i^{\ast}_{p''}\Stab^{K}_{-\mf{C},T^{1/2},s}(p'')\cdot\frac{i^{\ast}_{p''}\mf{L}(s)}{i^{\ast}_{p}\mf{L}(s)}\right),\label{deg_bound1}\\
\deg_{H}\left(i^{\ast}_{p''}\Stab^{K}_{-\mf{C},T^{1/2}_{\mr{opp}},-s}(p')\right)&\subset\deg_{H}\left(i^{\ast}_{p''}\Stab^{K}_{-\mf{C},T^{1/2}_{\mr{opp}},-s}(p'')\cdot\frac{i^{\ast}_{p''}\mf{L}(-s)}{i^{\ast}_{p'}\mf{L}(-s)}\right).\label{deg_bound2}
\end{align}
Therefore, we obtain 
\begin{align*}
\left\langle\xi,\deg_{H}\left(i^{\ast}_{p''}\Stab^{K}_{-\mf{C},T^{1/2},s}(p)\cdot i^{\ast}_{p''}\Stab^{K}_{-\mf{C},T^{1/2}_{\mr{opp}},-s}(p')\right)\right\rangle\subset\left[\left\langle\xi,-\rho_{p''}+\wt_{H}\frac{i^{\ast}_{p'}\mf{L}(s)}{i^{\ast}_{p}\mf{L}(s)}\right\rangle,\left\langle\xi,\rho_{p''}+\wt_{H}\frac{i^{\ast}_{p'}\mf{L}(s)}{i^{\ast}_{p}\mf{L}(s)}\right\rangle\right]
\end{align*}
This implies that if $S_{\lambda',p'}$ appears in the expansion of $\iota^{\pm}_{\mf{C},T^{1/2},s}((-v)^{\mp\frac{\dim X}{2}}\mca{S}_{-\mf{C},T^{1/2},s}(\lambda,p))$, then we have $\langle\xi,\lambda'-\lambda\rangle\geq\langle\xi,\wt_{H}i^{\ast}_{p'}\mf{L}(s)-\wt_{H}i^{\ast}_{p}\mf{L}(s)\rangle$ for any $\xi\in\pm\mf{C}$, which means $(\lambda,p)\leq_{\pm\mf{C},s}(\lambda',p')$.

For the coefficient of $S_{\lambda,p}$, we note that when $p''\neq p=p'$, the fractional shift appearing in (\ref{deg_bound1}) and (\ref{deg_bound2}) are opposite and not integral. Hence we have 
\begin{align*}
\left\langle\xi,\deg_{H}\left(i^{\ast}_{p''}\Stab^{K}_{-\mf{C},T^{1/2},s}(p)\cdot i^{\ast}_{p''}\Stab^{K}_{-\mf{C},T^{1/2}_{\mr{opp}},-s}(p)\right)\right\rangle\subset \left(-\langle\xi,\rho_{p''}\rangle,\langle\xi,\rho_{p''}\rangle\right).
\end{align*}
This implies that only $p''=p=p'$ part in (\ref{inner_sum}) contribute to the coefficient of $S_{\lambda,p}$. 
\end{proof}

\subsection{$K$-theoretic canonical basis}

Once we have defined the bar involution, we can follow Lusztig \cite{Lu3,Lu4} to define the notion of signed canonical basis by imposing bar invariance and asymptotic norm one property. Recall the inner product $(-||-):K_{\bb{T}}(X)\times K_{\bb{T}}(X)\rightarrow\Frac(K_{\bb{T}}(\mr{pt}))$ defined in (\ref{innerprod}). We note that this depends on the choice of $\mf{C}$, $T^{1/2}$, and $s$ through $\beta^{K}_{\mf{C},T^{1/2},s}$, but we simply omit the dependence from the notation if there is no chance of confusion. Recall that $L\subset X$ is the central fiber. We set 
\begin{align*}
\bb{B}^{\pm}_{L,T^{1/2},s}&\coloneqq\left\{m\in K_{\bb{T}}(L)\middle| \beta^{K}_{T^{1/2},s}(m)=v^{\dim X}m, (m||m)\in 1+v^{-1}K_{H}(\mr{pt})[v^{-1}]\right\},\\
\bb{B}^{\pm}_{X,T^{1/2},s}&\coloneqq\left\{m\in K_{\bb{T}}(X)\middle| \beta^{K}_{T^{1/2},s}(m)=m, (m||m)\in 1+v^{-1}K_{H}(\mr{pt})[\![v^{-1}]\!]\right\},
\end{align*}
where we expand the rational function in Laurent series of $v^{-1}$ with coefficients in $\Frac(K_{H}(\mr{pt}))$. We note that for any $\lambda\in\bb{X}^{\ast}(H)$ and $m\in\bb{B}^{\pm}_{X,T^{1/2},s}$, we have $\pm[\lambda]\cdot m\in\bb{B}^{\pm}_{X,T^{1/2},s}$. If we assume Conjecture~\ref{K-theoretic_bar_involution}, then this definition does not depend on the choice of $\mf{C}$. We call $\bb{B}^{\pm}_{L,T^{1/2},s}$ and $\bb{B}^{\pm}_{X,T^{1/2},s}$ \emph{signed $K$-theoretic canonical bases} for $L$ and $X$. 

If we assume that $\beta^{K}$ is an involution, we can formally construct a family of bar invariant and asymptotic norm one elements by Kazhdan-Lusztig type algorithm. More precisely, we can prove the following, whose proof provides such an algorithm.  
 
\begin{prop}\label{expansion}
Assume that $\beta^{K}_{\mf{C},T^{1/2},s}$ is an involution. 
\begin{enumerate}
\item For any $(\lambda,p)\in\bb{F}$, there exists unique $C^{T^{1/2},s}_{\lambda,p}\in S_{\lambda,p}+\sum_{(\lambda',p')>_{\mf{C},s}(\lambda,p)}v^{-1}\bb{Z}[v^{-1}]\cdot S_{\lambda',p'}\subset\bb{M}_{\mf{C},s}$ such that $v^{-\dim X}\cdot\beta^{K}_{\mf{C},T^{1/2},s}(C^{T^{1/2},s}_{\lambda,p})=C^{T^{1/2},s}_{\lambda,p}$. 
\item For any $(\lambda,p)\in\bb{F}$, there exists unique $E^{T^{1/2},s}_{\lambda,p}\in S_{\lambda,p}+\sum_{(\lambda',p')<_{\mf{C},s}(\lambda,p)}v^{-1}\bb{Z}[v^{-1}]\cdot S_{\lambda',p'}\subset\bb{M}_{-\mf{C},s}$ such that $\beta^{K}_{\mf{C},T^{1/2},s}(E^{T^{1/2},s}_{\lambda,p})=E^{T^{1/2},s}_{\lambda,p}$.
\end{enumerate}
\end{prop}

\begin{proof}
We only prove the first statement since the second statement can be proved similarly. We first prove the existence. We take any total order on $\bb{F}$ refining $\leq_{\mf{C},s}$ and prove inductively that for any $(\lambda',p')\geq_{\mf{C},s}(\lambda,p)$, there exists an element of the form 
\begin{align*}
C^{\leq(\lambda',p')}_{\lambda,p}=\sum_{(\lambda,p)\leq_{\mf{C},s}(\lambda'',p'')\leq_{\mf{C},s}(\lambda',p')}f^{\lambda'',p''}_{\lambda,p}(v)S_{\lambda'',p''}
\end{align*}
with $f^{\lambda,p}_{\lambda,p}(v)=1$ and $f^{\lambda'',p''}_{\lambda,p}(v)\in v^{-1}\bb{Z}[v^{-1}]$ for any $(\lambda'',p'')\neq(\lambda,p)$ such that 
\begin{align*}
v^{-\dim X}\cdot\beta^{K}_{\mf{C},T^{1/2},s}\left(C^{\leq(\lambda',p')}_{\lambda,p}\right)-C^{\leq(\lambda',p')}_{\lambda,p}\in\sum_{(\lambda'',p'')>_{\mf{C},s}(\lambda',p')}\bb{Z}[v,v^{-1}]\cdot S_{\lambda'',p''}.
\end{align*}
For $(\lambda',p')=(\lambda,p)$, we can take $C^{\lambda,p}_{\lambda,p}=S_{\lambda,p}$ by Lemma~\ref{R_triangular}. Assume that we have constructed $C^{<(\lambda',p')}_{\lambda,p}$ and write 
\begin{align*}
v^{-\dim X}\cdot\beta^{K}_{\mf{C},T^{1/2},s}\left(C^{<(\lambda',p')}_{\lambda,p}\right)-C^{<(\lambda',p')}_{\lambda,p}\in g(v)S_{\lambda',p'}+\sum_{(\lambda'',p'')>_{\mf{C},s}(\lambda',p')}\bb{Z}[v,v^{-1}]\cdot S_{\lambda'',p''}.
\end{align*}
for some $g(v)\in\bb{Z}[v,v^{-1}]$. Since $\beta^{K}_{\mf{C},T^{1/2},s}$ is an involution, we also obtain 
\begin{align*}
v^{-\dim X}\cdot\beta^{K}_{\mf{C},T^{1/2},s}\left(C^{<(\lambda',p')}_{\lambda,p}\right)-C^{<(\lambda',p')}_{\lambda,p}\in -g(v^{-1})S_{\lambda',p'}+\sum_{(\lambda'',p'')>_{\mf{C},s}(\lambda',p')}\bb{Z}[v,v^{-1}]\cdot S_{\lambda'',p''}.
\end{align*}
by Lemma~\ref{R_triangular}. By comparing the coefficient, we obtain $g(v)=-g(v^{-1})$. Hence there exists an element $f^{\lambda',p'}_{\lambda,p}(v)\in v^{-1}\bb{Z}[v^{-1}]$ such that $g(v)=f^{\lambda',p'}_{\lambda,p}(v)-f^{\lambda',p'}_{\lambda,p}(v^{-1})$. Then it suffices to take $C^{\leq(\lambda',p')}_{\lambda,p}=C^{<(\lambda',p')}_{\lambda,p}+f^{\lambda',p'}_{\lambda,p}(v)S_{\lambda',p'}$. This proves the existence of $C^{T^{1/2},s}_{\lambda,p}$. 

For the uniqueness, we assume that there is another $C'_{\lambda,p}$ satisfying the above conditions. Then we can expand 
\begin{align*}
C'_{\lambda,p}=C^{T^{1/2},s}_{\lambda,p}+\sum_{(\lambda',p')>_{\mf{C},s}(\lambda,p)}g^{\lambda',p'}_{\lambda,p}(v)C^{T^{1/2},s}_{\lambda',p'}
\end{align*}
with $g^{\lambda',p'}_{\lambda,p}(v)\in v^{-1}\bb{Z}[v^{-1}]$. By $v^{-\dim X}\cdot\beta^{K}_{\mf{C},T^{1/2},s}(C'_{\lambda,p})=C'_{\lambda,p}$, we also obtain 
\begin{align*}
C'_{\lambda,p}=C^{T^{1/2},s}_{\lambda,p}+\sum_{(\lambda',p')>_{\mf{C},s}(\lambda,p)}g^{\lambda',p'}_{\lambda,p}(v^{-1})C^{T^{1/2},s}_{\lambda',p'}. 
\end{align*}
Hence we have $g^{\lambda',p'}_{\lambda,p}(v)=g^{\lambda',p'}_{\lambda,p}(v^{-1})\in v^{-1}\bb{Z}[v^{-1}]\cap v\bb{Z}[v]=\{0\}$. This proves $C'_{\lambda,p}=C^{T^{1/2},s}_{\lambda,p}$. 
\end{proof}

\begin{conj}\label{K-theoretic_canonical_basis}
The $K$-theoretic bar involution $\beta^{K}_{\mf{C},T^{1/2},s}$ is an involution. For any $(\lambda,p)\in\bb{F}$, there exists $\mca{C}_{\mf{C},T^{1/2},s}(\lambda,p)\in K_{\bb{T}}(L)$ (resp. $\mca{E}_{\mf{C},T^{1/2},s}(\lambda,p)\in K_{\bb{T}}(X)$) such that $\iota^{+}_{\mf{C},T^{1/2},s}(\mca{C}_{\mf{C},T^{1/2},s}(\lambda,p))=C^{T^{1/2},s}_{\lambda,p}$ (resp. $\iota^{-}_{\mf{C},T^{1/2},s}(\mca{E}_{\mf{C},T^{1/2},s}(\lambda, p))=E^{T^{1/2},s}_{\lambda,p}$). Moreover, $\bb{B}_{L,T^{1/2},s}\coloneqq\{\mca{C}_{\mf{C},T^{1/2},s}(\lambda,p)\}_{(\lambda,p)\in\bb{F}}$ (resp. $\bb{B}_{X,T^{1/2},s}\coloneqq\{\mca{E}_{\mf{C},T^{1/2},s}(\lambda,p)\}_{(\lambda,p)\in\bb{F}}$) forms a basis of $K_{\bb{T}}(L)$ (resp. $K_{\bb{T}}(X)$) as a $\bb{Z}[v,v^{-1}]$-module. 
\end{conj}

For toric hyper-K\"ahler manifolds, this conjecture is proved in Proposition~\ref{prop_toric_bar_invariance_E(A)}, Lemma~\ref{lem_toric_bar_invariance_C(A)}, and Corollary~\ref{cor_toric_canonical_basis}. We note that $\mca{C}_{\mf{C},T^{1/2},s}(\lambda,p)=[\lambda]\cdot\mca{C}_{\mf{C},T^{1/2},s}(p)$ and $\mca{E}_{\mf{C},T^{1/2},s}(\lambda,p)=[\lambda]\cdot\mca{E}_{\mf{C},T^{1/2},s}(p)$ for any $(\lambda,p)\in\bb{F}$, where we set $\mca{C}_{\mf{C},T^{1/2},s}(p)\coloneqq\mca{C}_{\mf{C},T^{1/2},s}(0,p)$ and $\mca{E}_{\mf{C},T^{1/2},s}(p)\coloneqq\mca{E}_{\mf{C},T^{1/2},s}(0,p)$. We call $\bb{B}_{L,T^{1/2},s}$ and $\bb{B}_{X,T^{1/2},s}$ \emph{$K$-theoretic canonical bases} for $L$ and $X$ associated with the data $T^{1/2}$ and $s$. Conjecture~\ref{K-theoretic_canonical_basis} implies that the $K$-theoretic bar involutions preserve $K_{\bb{T}}(L)$ and $K_{\bb{T}}(X)$, which are already quite nontrivial from our definition. We should remark that $\mca{C}_{\mf{C},T^{1/2},s}(\lambda,p)$ and $\mca{E}_{\mf{C},T^{1/2},s}(\lambda,p)$ does depend on the choice of $\mf{C}$, but as we will show, the sets $\bb{B}_{L,T^{1/2},s}$ and $\bb{B}_{X,T^{1/2},s}$ will not depend on it (possibly up to sign). If we change $\mf{C}$, then the parametrization of the canonical basis by the fixed points will change. 

We also note that since $L$ is contained in the full attracting sets, the conjecture implies that the sum in the definition of $C^{T^{1/2},s}_{\lambda,p}$ is actually a finite sum. In particular, the above Kazhdan-Lusztig type algorithm gives a way to calculate $K$-theoretic canonical bases for $L$ if we know some formula for the $K$-theoretic stable bases. The sum in the definition of $E^{T^{1/2},s}_{\lambda,p}$ is always an infinite sum except for the trivial cases, but as Proposition~\ref{dual_basis} shows, $\bb{B}_{L,T^{1/2},s}$ and $\bb{B}_{X,T^{1/2},s}$ are dual basis with respect to $(-||-)$, hence one can also calculate $\bb{B}_{X,T^{1/2},s}$ if we know $\bb{B}_{L,T^{1/2},s}$. 

We also conjecture the following positivity property for the expansion of the $K$-theoretic canonical bases in terms of the $K$-theoretic standard bases. This is an analogue of the positivity of the Kazhdan-Lusztig polynomials. 

\begin{conj}\label{positivity}
For any $(\lambda,p)\in\bb{F}$, we have 
\begin{align*}
C^{T^{1/2},s}_{\lambda,p}&\in\sum_{(\lambda'.p')\geq_{\mf{C},s}(\lambda,p)}\bb{Z}_{\geq0}[(-v)^{-1}]\cdot S_{\lambda',p'},\\
E^{T^{1/2},s}_{\lambda,p}&\in\sum_{(\lambda'.p')\leq_{\mf{C},s}(\lambda,p)}\bb{Z}_{\geq0}[v^{-1}]\cdot S_{\lambda',p'}.
\end{align*}
\end{conj}

For toric hyper-K\"ahler manifolds, this conjecture is proved in Corollary~\ref{cor_toric_positivity2}. 

\subsection{Basic properties}

In this section, we prove some basic properties of the $K$-theoretic canonical bases assuming Conjecture~\ref{K-theoretic_canonical_basis}. We first check that $\bb{B}^{\pm}_{L,T^{1/2},s}=\bb{B}_{L,T^{1/2},s}\sqcup-\bb{B}_{L,T^{1/2},s}$ and $\bb{B}^{\pm}_{X,T^{1/2},s}=\bb{B}_{X,T^{1/2},s}\sqcup-\bb{B}_{X,T^{1/2},s}$. The asymptotic norm one property follows from the following lemma. 

\begin{lemma}\label{std_orthonormality}
For any $p,p'\in X^{H}$, we have 
\begin{align*}
\left(\mca{S}_{\mf{C},T^{1/2},s}(p)||\mca{S}_{\mf{C},T^{1/2},s}(p')\right)=\delta_{p,p'}.
\end{align*}
\end{lemma}

\begin{proof}
By Lemma~\ref{std_duality} and Lemma~\ref{std_orthogonality}, we have 
\begin{align*}
\left(\mca{S}_{\mf{C},T^{1/2},s}(p)||\mca{S}_{\mf{C},T^{1/2},s}(p')\right)&=\left(\mca{S}_{\mf{C},T^{1/2},s}(p):\bb{D}_{X}\beta^{K}_{\mf{C},T^{1/2},s}(\mca{S}_{\mf{C},T^{1/2},s}(p'))\right)\\
&=(-v)^{\frac{\dim X}{2}}\left(\mca{S}_{\mf{C},T^{1/2},s}(p):\mca{S}_{-\mf{C},T^{1/2},s}(p')^{\vee}\right)\\
&=v^{\dim X}\cdot\left(\mca{S}_{\mf{C},T^{1/2},s}(p):\mca{S}_{-\mf{C},T^{1/2}_{\mr{opp}},-s}(p')\right)\\
&=\delta_{p,p'}.
\end{align*}
\end{proof}

We denote by $\dag:K_{\bb{T}}(\mr{pt})\rightarrow K_{\bb{T}}(\mr{pt})$ the involution induced from the inverse for $H$. 

\begin{corollary}\label{asymptotic_orthonormality}
Assume Conjecture~\ref{K-theoretic_canonical_basis}. For any $p,p'\in X^{H}$, we have 
\begin{align*}
(\mca{C}_{\mf{C},T^{1/2},s}(p)||\mca{C}_{\mf{C},T^{1/2},s}(p'))&\in\delta_{p,p'}+v^{-1}K_{H}(\mr{pt})[v^{-1}],\\
(\mca{E}_{\mf{C},T^{1/2},s}(p)||\mca{E}_{\mf{C},T^{1/2},s}(p'))&\in\delta_{p,p'}+v^{-1}K_{H}(\mr{pt})[\![v^{-1}]\!].
\end{align*}
\end{corollary}

\begin{proof}
For the first statement, this follows from Lemma~\ref{std_orthonormality} since the expansion of $\mca{C}_{\mf{C},T^{1/2},s}(p)$ in terms of the standard basis is a finite sum. For the second statement, let us write 
\begin{align*}
\mca{E}_{\mf{C},T^{1/2},s}(p)=\sum_{p''\in X^{H}}f_{p,p''}\cdot\mca{S}_{\mf{C},T^{1/2},s}(p'')
\end{align*}
for some $f_{p,p''}\in\Frac(K_{\bb{T}}(\mr{pt}))$. The formal construction of $E^{T^{1/2},s}_{0,p}$ implies that if we expand $f_{p,p''}$ in $v^{-1}$, we have $f_{p,p''}\in\delta_{p,p''}+v^{-1}\Frac(K_{H}(\mr{pt}))[\![v^{-1}]\!]$. On the other hand, since $X^{\bb{S}}=L^{\bb{S}}$ is smooth and proper, we have $f_{p,p''}\in K_{H}(\mr{pt})(\!(v^{-1})\!)$. Hence we obtain $f_{p,p''}\in\delta_{p,p''}+v^{-1}K_{H}(\mr{pt})[\![v^{-1}]\!]$ and 
\begin{align*}
(\mca{E}_{\mf{C},T^{1/2},s}(p)||\mca{E}_{\mf{C},T^{1/2},s}(p'))=\sum_{p''}f_{p,p''}\cdot f^{\dag}_{p',p''}\in\delta_{p,p'}+v^{-1}K_{H}(\mr{pt})[\![v^{-1}]\!]
\end{align*}
by Lemma~\ref{std_orthonormality}. 
\end{proof}

We denote by $\overline{(-)}:K_{\bb{T}}(\mr{pt})\rightarrow K_{\bb{T}}(\mr{pt})$ the involution induced from the inverse map for $\bb{S}$

\begin{corollary}\label{signed_canonical_basis}
Assume Conjecture~\ref{K-theoretic_canonical_basis}. We have 
\begin{align*}
\bb{B}^{\pm}_{L,T^{1/2},s}&=\bb{B}_{L,T^{1/2},s}\sqcup-\bb{B}_{L,T^{1/2},s},\\
\bb{B}^{\pm}_{X,T^{1/2},s}&=\bb{B}_{X,T^{1/2},s}\sqcup-\bb{B}_{X,T^{1/2},s}.
\end{align*}
\end{corollary}

\begin{proof}
We only prove the second statement since the proof is the same for the first. By Corollary~\ref{asymptotic_orthonormality}, we have $\bb{B}_{X,T^{1/2},s}\sqcup-\bb{B}_{X,T^{1/2},s}\subset\bb{B}^{\pm}_{X,T^{1/2},s}$. Hence we only need to prove the other inclusion. 

Let $\mca{E}\in\bb{B}^{\pm}_{X,T^{1/2},s}$. Since $\bb{B}_{X,T^{1/2},s}$ is a basis of $K_{\bb{T}}(X)$, one can write $\mca{E}=\sum_{p\in X^{H}}f_{p}\mca{E}_{\mf{C},T^{1/2},s}(p)$ for some $f_{p}\in K_{\bb{T}}(\mr{pt})$. By $\beta_{\mf{C},T^{1/2},s}$-invariance, we obtain $\overline{f_{p}}=f_{p}$. Let $N$ be the maximal degree in $v$ of $f_{p}$, $p\in X^{H}$, and write $f_{p}=\sum_{i=-N}^{N}f_{p,i}v^{i}$ for some $f_{p,i}\in K_{H}(\mr{pt})$. By Corollary~\ref{asymptotic_orthonormality}, we obtain 
\begin{align*}
(\mca{E}||\mca{E})=\sum_{p\in X^{H}}f_{p,N}f^{\dag}_{p,N}v^{2N}+\cdots.	
\end{align*}
By the asymptotic norm one property of $\mca{E}$, we have $N=0$ and $\sum_{p\in X^{H}}f_{p,0}f^{\dag}_{p,0}=1$. This implies that $f_{p}=\pm[\lambda]\cdot\delta_{p,p'}$ for some $\lambda\in\bb{X}^{\ast}(H)$  and $p'\in X^{H}$, hence $\mca{E}\in\bb{B}_{X,T^{1/2},s}\sqcup-\bb{B}_{X,T^{1/2},s}$.  
\end{proof}

We note that Corollary~\ref{signed_canonical_basis} together with Conjecture~\ref{K-theoretic_bar_involution} implies that the $K$-theoretic canonical bases $\bb{B}_{L,T^{1/2},s}$ and $\bb{B}_{X,T^{1/2},s}$ does not depend on the choice of $\mf{C}$ up to sign.  Recall that under the assumption of Conjecture~\ref{K-theoretic_canonical_basis}, we have $\beta^{K}_{\mf{C},T^{1/2},s}=\beta^{K}_{-\mf{C},T^{1/2},s}$ and hence the signed $K$-theoretic canonical bases are the same for $\mf{C}$ and $-\mf{C}$. In particular, for each $(\lambda,p)\in\bb{F}$, there exists $(\lambda^{-},p^{-})\in\bb{F}$ such that $\mca{C}_{\mf{C},T^{1/2},s}(\lambda,p)=\pm\mca{C}_{-\mf{C},T^{1/2},s}(\lambda^{-},p^{-})$. This implies that 
\begin{align*}
\mca{C}_{\mf{C},T^{1/2},s}(\lambda,p)\in\pm\mca{S}_{-\mf{C},T^{1/2},s}(\lambda^{-},p^{-})+\sum_{(\lambda',p')<_{\mf{C},s}(\lambda^{-},p^{-})}v^{-1}\bb{Z}[v^{-1}]\cdot\mca{S}_{-\mf{C},T^{1/2},s}(\lambda',p').
\end{align*}
By the bar invariance, we also obtain 
\begin{align*}
\mca{C}_{\mf{C},T^{1/2},s}(\lambda,p)\in\pm(-v)^{-\frac{\dim X}{2}}\mca{S}_{\mf{C},T^{1/2},s}(\lambda^{-},p^{-})+\sum_{(\lambda',p')<_{\mf{C},s}(\lambda^{-},p^{-})}v^{-\frac{\dim X}{2}+1}\bb{Z}[v]\cdot\mca{S}_{\mf{C},T^{1/2},s}(\lambda',p').
\end{align*}
If we further assume Conjecture~\ref{positivity}, then the sign above must be positive. This proves the following statement. 

\begin{prop}\label{degree_bound}
Assume Conjecture~\ref{K-theoretic_canonical_basis}. For each $(\lambda,p)\in\bb{F}$, there exists $(\lambda^{-},p^{-})\in\bb{F}$ such that 
\begin{align*}
\mca{C}_{\mf{C},T^{1/2},s}(\lambda,p)=\sum_{(\lambda,p)\leq_{\mf{C},s}(\lambda',p')\leq_{\mf{C},s}(\lambda^{-},p^{-})}P^{\lambda',p'}_{\lambda,p}(v)\cdot \mca{S}_{\mf{C},T^{1/2},s}(\lambda',p'),
\end{align*}
where $P^{\lambda,p}_{\lambda,p}(v)=1$, $P^{\lambda^{-},p^{-}}_{\lambda,p}(v)=\pm(-v)^{-\frac{\dim X}{2}}$, and 
\begin{align*}
P^{\lambda',p'}_{\lambda,p}(v)\in v^{-1}\bb{Z}[v^{-1}]\cap v^{-\frac{\dim X}{2}+1}\bb{Z}[v]
\end{align*}
for $(\lambda,p)<_{\mf{C},s}(\lambda',p')<_{\mf{C},s}(\lambda^{-},p^{-})$. If we further assume Conjecture~\ref{positivity}, then we have $P^{\lambda^{-},p^{-}}_{\lambda,p}(v)=(-v)^{-\frac{\dim X}{2}}$.
\end{prop}

Using this, one can also give a characterization of $\mathbb{B}_{L,T^{1/2},s}$ which is similar to Kashiwara's theory of global crystal bases. 

\begin{corollary}\label{global_crystal_basis}
Assume Conjecture~\ref{K-theoretic_canonical_basis}. We have
\begin{align*}
K_{\bb{T}}(L)\cap\left\{\mca{S}_{\mf{C},T^{1/2},s}(\lambda,p)+\sum_{(\lambda',p')\in\bb{F}}\left(v^{-1}\bb{Z}[v^{-1}]\cap v^{-\frac{\dim X}{2}}\bb{Z}[v]\right)\cdot\mca{S}_{\mf{C},T^{1/2},s}(\lambda',p')\right\}=\left\{\mca{C}_{\mf{C},T^{1/2},s}(\lambda,p)\right\}. 
\end{align*}
\end{corollary}

\begin{proof}
Let $\mca{C}$ be an element of the LHS. Since $\mathbb{B}_{L,T^{1/2},s}$ is a $\bb{Z}[v,v^{-1}]$-basis of $K_{\bb{T}}(L)$, we can write 
\begin{align*}
\mca{C}=\sum_{(\lambda',p')\in\bb{F}}f_{\lambda',p'}(v)\cdot\mca{C}_{\mf{C},T^{1/2},s}(\lambda',p')
\end{align*}
for some $f_{\lambda',p'}(v)\in\bb{Z}[v,v^{-1}]$. By Proposition~\ref{degree_bound}, the maximal and minimal degree of $f_{\lambda',p'}(v)$ must be 0. By comparing the constant term of the coefficient of $\mca{S}_{\mf{C},T^{1/2},s}(\lambda',p')$, we obtain $f_{\lambda',p'}(v)=\delta_{(\lambda,p),(\lambda',p')}$ and hence $\mca{C}=\mca{C}_{\mf{C},T^{1/2},s}(p)$. 
\end{proof}

We next compute the pairing of $\mathbb{B}_{L,T^{1/2},s}$ and $\mathbb{B}_{X,T^{1/2},s}$. We first list some basic properties of the pairing. Recall that $\dag:K_{\bb{T}}(\mr{pt})\rightarrow K_{\bb{T}}(\mr{pt})$ is the involution induced from the inverse for $H$ and $\overline{(-)}:K_{\bb{T}}(\mr{pt})\rightarrow K_{\bb{T}}(\mr{pt})$ is the involution induced from the inverse for $\bb{S}$.

\begin{lemma}\label{inner_symmetry}
For any $\mca{F},\mca{G}\in K_{\bb{T}}(X)_{\mr{loc}}$, we have  $(\mca{G}||\mca{F})=(\mca{F}||\mca{G})^{\dag}$. If we denote by $(-||-)_{\mr{opp}}$ the inner product defined by $\beta^{K}_{\mf{C},T^{1/2}_{\mr{opp}},-s}$, then we have $(\mca{F}^{\vee}||\mca{G}^{\vee})_{\mr{opp}}=v^{\dim X}\cdot(\mca{F}||\mca{G})^{\vee}$
\end{lemma}

\begin{proof}
We only need to check the formulas for $\mca{F}=\mca{S}_{\mf{C},T^{1/2},s}(p)$ and $\mca{G}=\mca{S}_{\mf{C},T^{1/2},s}(p')$.  By Lemma~\ref{std_orthonormality}, this is obvious for the first one and the second one follows from Lemma~\ref{std_duality}. 
\end{proof}

\begin{lemma}\label{inner_bar}
Assume that $\beta^{K}_{\mf{C},T^{1/2},s}$ is an involution. For any $\mca{F},\mca{G}\in K_{\bb{T}}(X)_{\mr{loc}}$, we have $(\beta^{K}(\mca{F})||\beta^{K}(\mca{G}))=v^{\dim X}\overline{(\mca{F}||\mca{G})}$. 
\end{lemma}

\begin{proof}
We only need to check the formula for $\mca{F}=\mca{S}_{\mf{C},T^{1/2},s}(p)$ and $\mca{G}=\mca{S}_{\mf{C},T^{1/2},s}(p')$. By the assumption and Lemma~\ref{std_orthonormality} applied to the opposite chamber, we obtain $(\mca{S}_{-\mf{C},T^{1/2},s}(p)||\mca{S}_{-\mf{C},T^{1/2},s}(p'))=\delta_{p,p'}$. Then the statement follows from the definition of $\beta^{K}$.  
\end{proof}

\begin{prop}\label{dual_basis}
Assume Conjecture~\ref{K-theoretic_canonical_basis}. We have $(\mca{C}_{\mf{C},T^{1/2},s}(p)||\mca{E}_{\mf{C},T^{1/2},s}(p'))=\delta_{p,p'}$ for any $p,p'\in X^{H}$.
\end{prop}

\begin{proof}
By Conjecture~\ref{K-theoretic_canonical_basis}, $(\mca{C}_{\mf{C},T^{1/2},s}(p)||\mca{E}_{\mf{C},T^{1/2},s}(p'))\in K_{\bb{T}}(\mr{pt})$ since the support of $\mca{C}_{\mf{C},T^{1/2},s}(p)$ is proper. By Proposition~\ref{expansion} and Lemma~\ref{std_orthonormality}, we obtain $(\mca{C}_{\mf{C},T^{1/2},s}(p)||\mca{E}_{\mf{C},T^{1/2},s}(p'))\in\delta_{p,p'}+v^{-1}K_{H}(\mr{pt})[v^{-1}]$. By Lemma~\ref{inner_bar}, we also obtain $(\mca{C}_{\mf{C},T^{1/2},s}(p)||\mca{E}_{\mf{C},T^{1/2},s}(p'))=\overline{(\mca{C}_{\mf{C},T^{1/2},s}(p)||\mca{E}_{\mf{C},T^{1/2},s}(p'))}$, hence we need to have $(\mca{C}_{\mf{C},T^{1/2},s}(p)||\mca{E}_{\mf{C},T^{1/2},s}(p'))=\delta_{p,p'}$. 
\end{proof}

If we define the map $\partial:K_{\bb{T}}(\mr{pt})\rightarrow \bb{Z}[v,v^{-1}]$ by $\partial(\sum_{\mu\in\bb{X}^{\ast}(H)}f_{\mu}(v)\cdot[\mu])=f_{0}(v)$, then Proposition~\ref{dual_basis} implies that $\bb{B}_{L,T^{1/2},s}$ and $\bb{B}_{X,T^{1/2},s}$ forms a dual basis with respect to $\partial(-||-)$. As a corollary, we also obtain the following result on the expansion of $\bb{B}^{\mr{std}}_{\mf{C},T^{1/2},s}$ in terms of $\bb{B}_{X,T^{1/2},s}$. 

\begin{corollary}\label{std_to_canonical}
Assume Conjecture~\ref{K-theoretic_canonical_basis}. For any $(\lambda',p')\in\bb{F}$, we have 
\begin{align*}
\mca{S}_{\mf{C},T^{1/2},s}(\lambda',p')=\sum_{(\lambda,p)\leq_{\mf{C},s}(\lambda',p')}P^{\lambda',p'}_{\lambda,p}(v)\cdot \mca{E}_{\mf{C},T^{1/2},s}(\lambda,p),
\end{align*}
where $P^{\lambda',p'}_{\lambda,p}(v)$ is the same as Proposition~\ref{degree_bound} and $P^{\lambda',p'}_{\lambda,p}(v)=0$ if $(\lambda',p')\nleq_{\mf{C},s}(\lambda^{-},p^{-})$. 
\end{corollary}

\begin{proof}
This follows from Proposition~\ref{degree_bound} and Proposition~\ref{dual_basis}. 
\end{proof}

If we change the polarization or shift the slope, then the $K$-theoretic canonical bases change as follows.

\begin{lemma}\label{can_pol}
Assume Conjecture~\ref{K-theoretic_canonical_basis}. We have $\mca{C}_{\mf{C},T^{1/2}_{\mca{G}},s}(\lambda,p)=\mca{C}_{\mf{C},T^{1/2},s+\det\mca{G}}(\lambda+\wt_{H}\det i^{\ast}_{p}\mca{G},p)$ and $\mca{E}_{\mf{C},T^{1/2}_{\mca{G}},s}(\lambda,p)=\mca{E}_{\mf{C},T^{1/2},s+\det\mca{G}}(\lambda+\wt_{H}\det i^{\ast}_{p}\mca{G},p)$ for any $\mca{G}\in K_{\bb{T}}(X)$ and $(\lambda,p)\in\bb{F}$. In particular, we have $\bb{B}_{L,T^{1/2}_{\mca{G}},s}=\bb{B}_{L,T^{1/2},s+\det\mca{G}}$ and $\bb{B}_{X,T^{1/2}_{\mca{G}},s}=\bb{B}_{X,T^{1/2},s+\det\mca{G}}$. 
\end{lemma}

\begin{proof}
By Lemma~\ref{std_polarization}, we obtain $\beta^{K}_{\mf{C},T^{1/2}_{\mca{G}},s}=\beta^{K}_{\mf{C},T^{1/2},s+\det\mca{G}}$. Since $(\lambda',p')>_{\mf{C},s}(\lambda,p)$ is equivalent to $(\lambda'+\wt_{H}\det i^{\ast}_{p'}\mca{G},p')>_{\mf{C},s+\det\mca{G}}(\lambda+\wt_{H}\det i^{\ast}_{p}\mca{G},p)$ and 
\begin{align*}
&\mca{C}_{\mf{C},T^{1/2}_{\mca{G}},s}(\lambda,p)\in\mca{S}_{\mf{C},T^{1/2}_{\mca{G}},s}(\lambda,p)+\sum_{(\lambda',p')>_{\mf{C},s}(\lambda,p)}v^{-1}\bb{Z}[v^{-1}]\cdot\mca{S}_{\mf{C},T^{1/2}_{\mca{G}},s}(\lambda',p')\\
&=\mca{S}_{\mf{C},T^{1/2},s+\det\mca{G}}(\lambda+\wt_{H}\det i^{\ast}_{p}\mca{G},p)+\sum_{(\lambda',p')>_{\mf{C},s}(\lambda,p)}v^{-1}\bb{Z}[v^{-1}]\cdot\mca{S}_{\mf{C},T^{1/2},s+\det\mca{G}}(\lambda'+\wt_{H}\det i^{\ast}_{p'}\mca{G},p'),
\end{align*}
$\mca{C}_{\mf{C},T^{1/2}_{\mca{G}},s}(\lambda,p)$ also satisfies the characterizing properties of $\mca{C}_{\mf{C},T^{1/2},s+\det\mca{G}}(\lambda+\wt_{H}\det i^{\ast}_{p}\mca{G},p)$. The proof of $\mca{E}_{\mf{C},T^{1/2}_{\mca{G}},s}(\lambda,p)=\mca{E}_{\mf{C},T^{1/2},s+\det\mca{G}}(\lambda+\wt_{H}\det i^{\ast}_{p}\mca{G},p)$ is similar. 
\end{proof}

\begin{lemma}\label{can_slope}
Assume Conjecture~\ref{K-theoretic_canonical_basis}. For any $l\in P$ and $(\lambda,p)\in\bb{F}$, we have $\mca{C}_{\mf{C},T^{1/2},s+l}(\lambda,p)=\mf{L}(l)\otimes\mca{C}_{\mf{C},T^{1/2},s}(\lambda-\wt_{H}i^{\ast}_{p}\mf{L}(l),p)$ and $\mca{E}_{\mf{C},T^{1/2},s+l}(\lambda,p)=\mf{L}(l)\otimes\mca{E}_{\mf{C},T^{1/2},s}(\lambda-\wt_{H}i^{\ast}_{p}\mf{L}(l),p)$. In particular, we have $\bb{B}_{L,T^{1/2},s+l}=\mf{L}(l)\otimes\bb{B}_{L,T^{1/2},s}$ and $\bb{B}_{X,T^{1/2},s+l}=\mf{L}(l)\otimes\bb{B}_{X,T^{1/2},s}$.
\end{lemma}

\begin{proof}
By Lemma~\ref{std_slope}, we obtain $\beta^{K}_{\mf{C},T^{1/2},s+l}=\mf{L}(l)\circ\beta^{K}_{\mf{C},T^{1/2},s}\circ\mf{L}(l)^{-1}$, where $\mf{L}(l)$ is identified with the automorphism of $K_{\bb{T}}(X)$ given by the tensor product of $\mf{L}(l)$. Since $(\lambda',p')>_{\mf{C},s+l}(\lambda,p)$ is equivalent to $(\lambda'-\wt_{H}i^{\ast}_{p'}\mf{L}(l),p')>_{\mf{C},s}(\lambda-\wt_{H}i^{\ast}_{p}\mf{L}(l),p)$ and 
\begin{align*}
&\mca{C}_{\mf{C},T^{1/2},s+l}(\lambda,p)\in\mca{S}_{\mf{C},T^{1/2},s+l}(\lambda,p)+\sum_{(\lambda',p')>_{\mf{C},s+l}(\lambda,p)}v^{-1}\bb{Z}[v^{-1}]\cdot\mca{S}_{\mf{C},T^{1/2},s+l}(\lambda',p')\\
&=\mf{L}(l)\otimes\left(\mca{S}_{\mf{C},T^{1/2},s}(\lambda-\wt_{H}i^{\ast}_{p}\mf{L}(l),p)+\sum_{(\lambda',p')>_{\mf{C},s+l}(\lambda,p)}v^{-1}\bb{Z}[v^{-1}]\cdot\mca{S}_{\mf{C},T^{1/2},s}(\lambda'-\wt_{H}i^{\ast}_{p'}\mf{L}(l),p')\right),
\end{align*}
$\mf{L}(l)^{-1}\otimes\mca{C}_{\mf{C},T^{1/2},s+l}(\lambda,p)$ also satisfies the characterizing properties of $\mca{C}_{\mf{C},T^{1/2},s}(\lambda-\wt_{H}i^{\ast}_{p}\mf{L}(l),p)$. The proof of $\mca{E}_{\mf{C},T^{1/2},s+l}(\lambda,p)=\mf{L}(l)\otimes\mca{E}_{\mf{C},T^{1/2},s}(\lambda-\wt_{H}i^{\ast}_{p}\mf{L}(l),p)$ is similar. 
\end{proof}

For the behavior of $K$-theoretic canonical bases under the duality, we have the following. 

\begin{lemma}
Assume Conjecture~\ref{K-theoretic_canonical_basis} and Conjecture~\ref{positivity}. For any $(\lambda,p)\in\bb{F}$, we have $v^{-\dim X}\mca{C}_{\mf{C},T^{1/2},s}(\lambda,p)^{\vee}=\mca{C}_{\mf{C},T^{1/2}_{\mr{opp}},-s}(-\lambda^{-},p^{-})$ and $\mca{E}_{\mf{C},T^{1/2},s}(\lambda,p)^{\vee}=\mca{E}_{\mf{C},T^{1/2}_{\mr{opp}},-s}(-\lambda^{-},p^{-})$. Here, $(\lambda^{-},p^{-})$ is as in Proposition~\ref{degree_bound}. In particular, we have $v^{-\dim X}\cdot\bb{B}_{L,T^{1/2},s}^{\vee}=\bb{B}_{L,T^{1/2}_{\mr{opp}},-s}$ and $\bb{B}_{X,T^{1/2},s}^{\vee}=\bb{B}_{X,T^{1/2}_{\mr{opp}},-s}$. 
\end{lemma}

\begin{proof}
By Lemma~\ref{std_duality} and Proposition~\ref{degree_bound}, one can check that $v^{-\dim X}\mca{C}_{\mf{C},T^{1/2},s}(\lambda,p)^{\vee}\in K_{\bb{T}}(L)$ is contained in $\mca{S}_{\mf{C},T^{1/2}_{\mr{opp}},-s}(-\lambda^{-},p^{-})+\sum_{(\lambda',p')\in\bb{F}}\left(v^{-\frac{\dim X}{2}+1}\bb{Z}[v]\cap v^{-1}\bb{Z}[v^{-1}]\right)\cdot\mca{S}_{\mf{C},T^{1/2}_{\mr{opp}},-s}(\lambda',p')$. By Corollary~\ref{global_crystal_basis}, this implies $v^{-\dim X}\mca{C}_{\mf{C},T^{1/2},s}(\lambda,p)^{\vee}=\mca{C}_{\mf{C},T^{1/2}_{\mr{opp}},-s}(-\lambda^{-},p^{-})$. The second statement follows from the first by Lemma~\ref{inner_symmetry} and Proposition~\ref{dual_basis}. 
\end{proof}

Since the fixed point basis often has some representation theoretic and combinatorial meaning, we also give a formula for the transition matrix between the fixed point basis and $\bb{B}_{L,T^{1/2},s}$. For each $p\in X^{H}$, we denote by $\mca{O}_{p}\in K_{\bb{T}}(X)$ the $K$-theory class of the skyscraper sheaf at $p$. 

\begin{prop}\label{fixed_point_basis}
Assume Conjecture~\ref{K-theoretic_canonical_basis}. For any $p\in X^{H}$, we have 
\begin{align*}
\mca{O}_{p}=v^{\dim X}\cdot\sum_{p'\in X^{H}}\left(i^{\ast}_{p}\mca{E}_{\mf{C},T^{1/2},s}(p')\right)^{\vee}\cdot\mca{C}_{\mf{C},T^{1/2},s}(p').
\end{align*}
\end{prop}

\begin{proof}
This follows from Proposition~\ref{dual_basis} and 
\begin{align*}
(\mca{O}_{p}||\mca{E}_{\mf{C},T^{1/2},s}(p'))=(\mca{O}_{p}:\bb{D}_{X}\mca{E}_{\mf{C},T^{1/2},s}(p'))=v^{\dim X}\cdot\left(i^{\ast}_{p}\mca{E}_{\mf{C},T^{1/2},s}(p')\right)^{\vee}.
\end{align*}
\end{proof}

\begin{remark}
For example, when $X$ is the Hilbert scheme of points in the affine plane and the slope is sufficiently close to $1$, it turns out that $\mca{E}_{\mf{C},T^{1/2},s}(p)$'s are given by the indecomposable summands of the Procesi bundle. If we identify $K_{\bb{T}}(L)$ and the space of symmetric functions as in \cite{Ha}, then $\mca{C}_{\mf{C},T^{1/2},s}(p)$'s corresponds to the Schur functions and $\mca{O}_{p}$'s corresponds to the modified Macdonald polynomials. Transition matrix of these bases are given by the $q,t$-Kostka polynomials which are given by the characters of fibers of indecomposable summands of the Procesi bundle. 
\end{remark}

\subsection{Categorical lifts}

In this section, we state several conjectures about the categorical meaning of canonical and standard bases. We assume all the conjectures and assumptions stated in the previous sections without any comment. 

We first recall the notion of tilting bundle on $X$. Let $\mca{T}$ be a vector bundle on $X$. We say that $\mca{T}$ is a \emph{tilting bundle} on $X$ if it satisfies
\begin{itemize}
\item $\mca{T}$ is a weak generator for $D(\mr{QCoh}(X))$, i.e., $\RHom(\mca{T},\mca{F})=0$ implies $\mca{F}\cong0$ for $\mca{F}\in D(\mr{QCoh}(X))$. 
\item $\Ext^{i}(\mca{T},\mca{T})=0$ for $i\neq0$.
\end{itemize}

When $X$ is a Slodowy variety, Bezrukavnikov-Mirkovi\'c \cite{BM} proved that there exists an $\bb{T}$-equivariant tilting bundle on $X$ such that Lusztig's $K$-theoretic canonical basis for $X$ consists of indecomposable summands of the tilting bundle up to equivariant shifts. Moreover, the structure sheaf of $X$ is contained in the canonical basis. We also expect in general that $\bb{B}_{X,T^{1/2},s}$ is given by the classes of indecomposable summands of some tilting bundle on $X$. 

\begin{conj}\label{conj_general_tilting_bundle}
For each $p\in X^{H}$, there exists an $\bb{T}$-equivariant vector bundle lifting $\mca{E}_{\mf{C},T^{1/2},s}(p)$ (denoted by the same letter) such that $\mca{T}_{\mf{C},T^{1/2},s}\coloneqq\oplus_{p\in X^{H}}\mca{E}_{\mf{C},T^{1/2},s}(p)$ is a tilting bundle on $X$. Moreover, at least one of $\mca{E}_{\mf{C},T^{1/2},s}(p)$ is a line bundle. 
\end{conj}

For toric hyper-K\"ahler manifolds, this conjecture is proved in Corollary~\ref{cor_toric_tilting}. In particular, this conjecture implies that $\bb{B}_{X,T^{1/2},s}$ consists of $K$-theory classes of actual vector bundles on $X$ and hence our choice of the sign is geometrically natural. 

\begin{corollary}\label{cor_Koszulity}
If we assume Conjecture~\ref{conj_general_tilting_bundle}, then the ring $\mca{A}_{\mf{C},T^{1/2},s}\coloneqq\End(\mca{T}_{\mf{C},T^{1/2},s})^{\mr{opp}}$ is non-negatively graded with respect to the $\bb{S}$-action and its degree 0 part is semisimple. 
\end{corollary}

\begin{proof}
By the definition of tilting bundle, we have 
\begin{align*}
(\mca{E}_{\mf{C},T^{1/2},s}(p)||\mca{E}_{\mf{C},T^{1/2},s}(p'))&=[R\Gamma(\mca{E}_{\mf{C},T^{1/2},s}(p)\otimes^{\bb{L}}\bb{D}_{X}\mca{E}_{\mf{C},T^{1/2},s}(p'))]\\
&=[R\Gamma(\bb{D}_{X}R\mca{H}om(\mca{E}_{\mf{C},T^{1/2},s}(p),\mca{E}_{\mf{C},T^{1/2},s}(p')))]\\
&=[\Hom(\mca{E}_{\mf{C},T^{1/2},s}(p),\mca{E}_{\mf{C},T^{1/2},s}(p'))^{\vee}]
\end{align*}
as rational functions in equivariant parameters. Since the last term is also contained in $K_{H}(\mr{pt})(\!(v^{-1})\!)$, we obtain $[\Hom(\mca{E}_{\mf{C},T^{1/2},s}(p),\mca{E}_{\mf{C},T^{1/2},s}(p'))]\in\delta_{p,p'}+vK_{H}(\mr{pt})[\![v]\!]$ by Lemma~\ref{asymptotic_orthonormality}. This implies that it is non-negatively graded with respect to $\bb{S}$ and the degree zero part is semisimple.
\end{proof}

By Kaledin's argument in \cite{BM} section 5.5, this also implies that $\mca{A}_{\mf{C},T^{1/2},s}$ is Koszul. We denote by $\mca{B}_{\mf{C},T^{1/2},s}$ its Koszul dual. We note that $\mca{A}_{\mf{C},T^{1/2},s}$ and $\mca{B}_{\mf{C},T^{1/2},s}$ has a natural $H$-action. By the independence of $\bb{B}_{X,T^{1/2},s}$ on $\mf{C}$, $\mca{A}_{\mf{C},T^{1/2},s}$ and $\mca{B}_{\mf{C},T^{1/2},s}$ does not depend on the choice of $\mf{C}$ if we forget the $H$-actions. Since $X$ is smooth, $\mca{A}_{\mf{C},T^{1/2},s}$ has finite global dimension and hence $\mca{B}_{\mf{C},T^{1/2},s}$ is finite dimensional. 

By the standard properties of tilting bundle, there is a derived equivalence 
\begin{align}\label{derived_equiv}
\psi_{\mf{C},T^{1/2},s}:\DCohT(X)\cong D^{b}(\mca{A}_{\mf{C},T^{1/2},s}\mbox{-}\mr{gmod}^{H})
\end{align}
given by the functor $\RHom(\mca{T}_{\mf{C},T^{1/2},s},-)$. Here, we denote by $\mca{A}_{\mf{C},T^{1/2},s}\mbox{-}\mr{gmod}^{H}$ the category of finitely generated graded $H$-equivariant modules of $\mca{A}_{\mf{C},T^{1/2},s}$. We denote the $t$-structure on $\DCohT(X)$ corresponding to the standard $t$-structure on $D^{b}(\mca{A}_{\mf{C},T^{1/2},s}\mbox{-}\mr{gmod}^{H})$ by $\tau_{T^{1/2},s}$. By construction, $\mca{F}\in\DCohT(X)$ is contained in the heart if and only if $R^{\neq0}\mr{Hom}(\mca{T}_{\mf{C},T^{1/2},s},\mca{F})=0$. 

For $(\lambda,p)\in\bb{F}$, we will write $\mca{E}_{\mf{C},T^{1/2},s}(\lambda,p)\coloneqq[\lambda]\cdot\mca{E}_{\mf{C},T^{1/2},s}(p)$ as in the $K$-theoretic one. Each $\psi_{\mf{C},T^{1/2},s}(\mca{E}_{\mf{C},T^{1/2},s}(\lambda,p))$ defines a graded $H$-equivariant indecomposable projective module of $\mca{A}_{\mf{C},T^{1/2},s}$ denoted by $P^{\mca{A}}_{\mf{C},T^{1/2},s}(\lambda,p)$. It has a unique one-dimensional simple quotient denoted by $L^{\mca{A}}_{\mf{C},T^{1/2},s}(\lambda,p)$. We have
\begin{align*}
\partial(\psi_{\mf{C},T^{1/2},s}^{-1}(L^{\mca{A}}_{\mf{C},T^{1/2},s}(\lambda,p))||\mca{E}_{\mf{C},T^{1/2},s}(\lambda',p'))&=\left[\RHom(\mca{E}_{\mf{C},T^{1/2},s}(\lambda',p'),v^{\dim X}\psi_{\mf{C},T^{1/2},s}^{-1}(L^{\mca{A}}_{\mf{C},T^{1/2},s}(\lambda,p))[\dim X])^{H}\right]\\
&=v^{\dim X}\cdot\delta_{(\lambda,p),(\lambda',p')}.
\end{align*}
Hence the $K$-theory class of $v^{-\dim X}\cdot\psi_{\mf{C},T^{1/2},s}^{-1}(L^{\mca{A}}_{\mf{C},T^{1/2},s}(\lambda,p))$ coincides with $\mca{C}_{\mf{C},T^{1/2},s}(\lambda,p)$ by Proposition~\ref{dual_basis}. Using this lift, we sometimes regard $K$-theoretic canonical bases of $K_{\bb{T}}(L)$ as objects in $\DCohT(X)$. 

By the Koszul duality (cf. \cite{BGS,MOS}), we also obtain the following derived equivalence
\begin{align}\label{Koszul_duality}
\msc{K}: D^{b}(\mca{A}_{\mf{C},T^{1/2},s}\mbox{-}\mr{gmod}^{H})\cong D^{b}(\mca{B}_{\mf{C},T^{1/2},s}\mbox{-}\mr{gmod}^{H}).
\end{align}
We note that since $\mca{B}_{\mf{C},T^{1/2},s}$ is finite dimensional, the Koszul duality equivalence preserves the boundedness by \cite[Theorem~2.12.6]{BGS}. The standard $t$-structure on $D^{b}(\mca{B}_{\mf{C},T^{1/2},s}\mbox{-}\mr{gmod}^{H})$ induces a $t$-structure on $D^{b}(\mca{A}_{\mf{C},T^{1/2},s}\mbox{-}\mr{gmod}^{H})$ and its heart consists of linear complex of projective modules, that is, object quasi-isomorphic to a complex of the form 
\begin{align*}
0\rightarrow v^{N}P_{N}\rightarrow v^{N-1}P_{N-1}\rightarrow\cdots\rightarrow v^{M}P_{M}\rightarrow0,
\end{align*}
where each $P_{i}$ sits in the $(-i)$-th term and it is a direct sum of projective modules of the form $P^{\mca{A}}_{\mf{C},T^{1/2},s}(\lambda,p)$. By the construction of $\msc{K}$, we have $\msc{K}\circ v[1]=v^{-1}\circ\msc{K}$ and $L^{\mca{B}}_{\mf{C},T^{1/2},s}(\lambda,p)\coloneqq\msc{K}(P^{\mca{A}}_{\mf{C},T^{1/2},s}(\lambda,p))$ is a graded $H$-equivariant one-dimensional simple module of $\mca{B}_{\mf{C},T^{1/2},s}$. We note that any simple object in $\mca{B}_{\mf{C},T^{1/2},s}\mbox{-}\mr{gmod}^{H}$ is of the form $v^{m}L^{\mca{B}}_{\mf{C},T^{1/2},s}(\lambda,p)$ for some $m\in\bb{Z}$ and $(\lambda,p)\in\bb{F}$. Since $\mca{A}_{\mf{C},T^{1/2},s}$ is Koszul, $L^{\mca{A}}_{\mf{C},T^{1/2},s}(\lambda,p)$ is quasi-isomorphic to a linear complex of projective modules and hence we have $I^{\mca{B}}_{\mf{C},T^{1/2},s}(\lambda,p)\coloneqq\msc{K}(L^{\mca{A}}_{\mf{C},T^{1/2},s}(\lambda,p))\in\mca{B}_{\mf{C},T^{1/2},s}\mbox{-}\mr{gmod}^{H}$. This is the injective hull of $L^{\mca{B}}_{\mf{C},T^{1/2},s}(\lambda,p)$. 

For a categorical lift of standard basis, we conjecture the following. 

\begin{conj}\label{categorical_stable_basis}
For each $(\lambda,p)\in X^{H}$, there exists an object $\nabla^{\mca{B}}_{\mf{C},T^{1/2},s}(\lambda,p)\in\mca{B}_{\mf{C},T^{1/2},s}\mbox{-}\mr{gmod}^{H}$ such that $\Delta^{\mca{A}}_{\mf{C},T^{1/2},s}(\lambda,p)\coloneqq\msc{K}^{-1}(\nabla^{\mca{B}}_{\mf{C},T^{1/2},s}(\lambda,p))$ is contained in the standard heart of $D^{b}(\mca{A}_{\mf{C},T^{1/2},s}\mbox{-}\mr{gmod}^{H})$ and the $K$-theory class of $\nabla_{\mf{C},T^{1/2},s}(\lambda,p)\coloneqq\psi^{-1}_{\mf{C},T^{1/2},s}(\Delta^{\mca{A}}_{\mf{C},T^{1/2},s}(\lambda,p))\in\DCohT(X)$ coincides with $\beta^{K}_{\mf{C},T^{1/2},s}(\mca{S}_{\mf{C},T^{1/2},s}(\lambda,p))$. If we set $\Delta_{\mf{C},T^{1/2},s}(\lambda,p)\coloneqq v^{-\frac{\dim X}{2}}\nabla_{-\mf{C},T^{1/2},s}(\lambda,p)[-\frac{\dim X}{2}]\in\DCohT(X)$, then we have 
\begin{align}\label{derived_orthonormality}
\Hom_{\DCohT(X)}\left(v^{j}\Delta_{\mf{C},T^{1/2},s}(\lambda,p),\nabla_{\mf{C},T^{1/2},s}(\lambda',p')[i]\right)=\begin{cases}
\bb{C} &\mbox{ if }i=j=0\mbox{ and }(\lambda,p)=(\lambda',p'),\\
0 &\mbox{ otherwise}.
\end{cases}
\end{align}
\end{conj}

For toric hyper-K\"ahler manifolds, this conjecture is proved in Theorem~\ref{thm_ext_orthogonality}. We note that the $K$-theory class of $\Delta_{\mf{C},T^{1/2},s}(\lambda,p)$ coincides with $\mca{S}_{\mf{C},T^{1/2},s}(\lambda,p)$ and the equation (\ref{derived_orthonormality}) lifts the orthonormality of $K$-theoretic standard bases in Lemma~\ref{std_orthonormality}. We expect that $\nabla_{\mf{C},T^{1/2},s}(\lambda,p)$ and $\Delta_{\mf{C},T^{1/2},s}(\lambda,p)$ are essentially given by the theory of categorical stable envelope (c.f. \cite{O2}). Since $\Delta^{\mca{A}}_{\mf{C},T^{1/2},s}(\lambda,p)$ is a Koszul module of $\mca{A}_{\mf{C},T^{1/2},s}$, we have a resolution of the form
\begin{align}\label{proj_res_A_std}
0\rightarrow v^{\frac{\dim X}{2}}P_{\frac{\dim X}{2}}\rightarrow\cdots\rightarrow vP_{1}\rightarrow P^{\mca{A}}_{\mf{C},T^{1/2},s}(\lambda,p)\rightarrow \Delta^{\mca{A}}_{\mf{C},T^{1/2},s}(\lambda,p)\rightarrow0
\end{align}
in $\mca{A}_{\mf{C},T^{1/2},s}\mbox{-}\mr{gmod}^{H}$ by Corollary~\ref{std_to_canonical}, where each $P_{i}$ ($i=1,\ldots,\frac{\dim X}{2}$) is a direct sum of projective modules of the form $P^{\mca{A}}_{\mf{C},T^{1/2},s}(\lambda',p')$ satisfying $(\lambda',p')<_{\mf{C},s}(\lambda,p)$. By the construction of $\msc{K}$, it follows that $\nabla^{\mca{B}}_{\mf{C},T^{1/2},s}(\lambda,p)$ is non-positively graded and we have an inclusion $L^{\mca{B}}_{\mf{C},T^{1/2},s}(\lambda,p)\hookrightarrow\nabla^{\mca{B}}_{\mf{C},T^{1/2},s}(\lambda,p)$ such that any composition factor $v^{j}L^{\mca{B}}_{\mf{C},T^{1/2},s}(\lambda',p')$ of the quotient $\nabla^{\mca{B}}_{\mf{C},T^{1/2},s}(\lambda,p)/L^{\mca{B}}_{\mf{C},T^{1/2},s}(\lambda,p)$ satisfies $j<0$ and $(\lambda',p')<_{\mf{C},s}(\lambda,p)$.

Since the shift $v[1]$ preserves the linear complex of projective modules, $\psi_{\mf{C},T^{1/2},s}(\Delta_{\mf{C},T^{1/2},s}(\lambda,p))$ is quasi-isomorphic to a complex of the form 
\begin{align*}
0\rightarrow P^{\mca{A}}_{\mf{C},T^{1/2},s}(\lambda,p)\rightarrow v^{-1}P_{1}\rightarrow\cdots\rightarrow v^{-\frac{\dim X}{2}}P_{\frac{\dim X}{2}}\rightarrow0
\end{align*}
where $P^{\mca{A}}_{\mf{C},T^{1/2},s}(\lambda,p)$ sits in degree 0 and each $P_{i}$ is the same as (\ref{proj_res_A_std}). If we write $\Delta^{\mca{B}}_{\mf{C},T^{1/2},s}(\lambda,p)\coloneqq\msc{K}(\psi_{\mf{C},T^{1/2},s}(\Delta_{\mf{C},T^{1/2},s}(\lambda,p)))\in\mca{B}_{\mf{C},T^{1/2},s}\mbox{-}\mr{gmod}^{H}$, then this is non-negatively graded and we have a projection $\Delta^{\mca{B}}_{\mf{C},T^{1/2},s}(\lambda,p)\twoheadrightarrow L^{\mca{B}}_{\mf{C},T^{1/2},s}(\lambda,p)$ such that any composition factor $v^{j}L^{\mca{B}}_{\mf{C},T^{1/2},s}(\lambda',p')$ of the kernel satisfies $j>0$ and $(\lambda',p')<_{\mf{C},s}(\lambda,p)$. By (\ref{derived_orthonormality}), we obtain 
\begin{align}\label{ext_orthonormality}
\Ext^{i}_{\mca{B}_{\mf{C},T^{1/2},s}\mbox{-}\mr{gmod}^{H}}\left(v^{j}\Delta^{\mca{B}}_{\mf{C},T^{1/2},s}(\lambda,p),\nabla^{\mca{B}}_{\mf{C},T^{1/2},s}(\lambda',p')\right)=\begin{cases}
\bb{C} &\mbox{ if }i=j=0\mbox{ and }(\lambda,p)=(\lambda',p')\\
0 &\mbox{ otherwise}.
\end{cases}
\end{align}

Next we show that $\mca{B}_{\mf{C},T^{1/2},s}\mbox{-}\mr{gmod}^{H}$ has a structure of graded highest weight category assuming Conjecture~\ref{categorical_stable_basis}. We first recall the definition of graded highest weight category following \cite{CPS1,CPS2}. Let $\mb{C}$ be a $\bb{C}$-linear abelian category with a free $\bb{Z}$-action which is locally Artinian, contains enough injectives, and satisfies Grothendieck's condition AB5. The action of $j\in\bb{Z}$ on an object $M\in\mca{C}$ is denoted by $M\mapsto M\langle j\rangle$. For each $M,N\in\mb{C}$, we set $\hom_{\mb{C}}(M,N)\coloneqq\oplus_{j\in\bb{Z}}\Hom_{\mb{C}}(M\langle j\rangle,N)$ and $\ext^{i}_{\mb{C}}(M,N)\coloneqq\oplus_{j\in\bb{Z}}\Ext^{i}_{\mb{C}}(M\langle j\rangle,N)$. For a simple object $L$, we denote by $[M:L]$ the multiplicity of $L$ in the composition series of $M$. 

\begin{dfn}[\cite{CPS1,CPS2}]\label{graded_highest_weight}
A category $\mb{C}$ as above is called \emph{graded highest weight category} if there exists an interval finite poset $\Lambda$ satisfying the following conditions:
\begin{itemize}
\item For each $\lambda\in\Lambda$, we have a simple object $L(\lambda)$ such that $\{L(\lambda)\langle j\rangle\}_{j\in\bb{Z},\lambda\in\Lambda}$ is a complete set of non-isomorphic simple objects of $\mb{C}$. 
\item For each $\lambda\in\Lambda$, there is an object $\nabla(\lambda)$ (called \emph{costandard object}) with an inclusion $L(\lambda)\hookrightarrow\nabla(\lambda)$ such that any composition factor $L(\mu)\langle j\rangle$ of the quotient $\nabla(\lambda)/L(\lambda)$ satisfies $j<0$ and $\mu<\lambda$. Moreover, for each $\lambda,\mu\in\Lambda$, $\dim_{\bb{C}}\hom_{\mb{C}}(\nabla(\lambda),\nabla(\mu))$ and $\sum_{j\in\bb{Z}}[\nabla(\lambda):L(\mu)\langle j\rangle]$ are finite. 
\item For each $\lambda\in\Lambda$, injective hull $I(\lambda)$ of $L(\lambda)$ has an increasing filtration $0=F_{0}(\lambda)\subset F_{1}(\lambda)\subset\cdots$ with $\cup_{i}F_{i}(\lambda)=I(\lambda)$ such that $F_{1}(\lambda)\cong\nabla(\lambda)$ and $F_{i}(\lambda)/F_{i-1}(\lambda)\cong\nabla(\mu_{i})\langle j\rangle$ for some $j<0$ and $\mu_{i}>\lambda$ if $i>1$. Moreover, the set $\{i\mid \mu_{i}=\mu\}$ is finite for any $\mu\in\Lambda$. 
\end{itemize}
\end{dfn}

\begin{prop}
Assume Conjecture~\ref{categorical_stable_basis}. Then the category $\mca{B}_{\mf{C},T^{1/2},s}\mbox{-}\mr{gmod}^{H}$ has a structure of graded highest weight category with poset $(\bb{F},\leq_{\mf{C},s})$ and the costandard object parametrized by $(\lambda,p)\in\bb{F}$ is given by $\nabla^{\mca{B}}_{\mf{C},T^{1/2},s}(\lambda,p)$. 
\end{prop}

\begin{proof}
Since $\mca{B}_{\mf{C},T^{1/2},s}$ is finite dimensional over $\bb{C}$, the category $\mca{B}_{\mf{C},T^{1/2},s}\mbox{-}\mr{gmod}^{H}$ is Artinian. The free $\bb{Z}$-action on $\mca{B}_{\mf{C},T^{1/2},s}\mbox{-}\mr{gmod}^{H}$ is given by the grading shifts $M\langle j\rangle\coloneqq v^{j}M$. The poset $\bb{F}$ is interval finite by Lemma~\ref{partial_order}. For each $(\lambda,p)\in\bb{F}$, we associate the simple module $L^{\mca{B}}_{\mf{C},T^{1/2},s}(\lambda,p)$. Then the first two condition in Definition~\ref{graded_highest_weight} has been already checked above. 

On the level of $K$-theory, we have 
\begin{align*}
[I^{\mca{B}}_{\mf{C},T^{1/2},s}(\lambda,p)]=\sum_{(\lambda,p)\leq_{\mf{C},s}(\lambda',p')}P^{\lambda',p'}_{\lambda,p}(-v)\cdot[\nabla^{\mca{B}}_{\mf{C},T^{1/2},s}(\lambda',p')]
\end{align*}
by Proposition~\ref{degree_bound}. Since $P^{\lambda,p}_{\lambda,p}(-v)=1$ and $P^{\lambda',p'}_{\lambda,p}(-v)\in v^{-1}\bb{Z}[v^{-1}]$ if $(\lambda',p')\neq(\lambda,p)$, it is enough to prove that $I^{\mca{B}}_{\mf{C},T^{1/2},s}(\lambda,p)$ has a costandard filtration with $F_{1}(\lambda,p)=\nabla^{\mca{B}}_{\mf{C},T^{1/2},s}(\lambda,p)$. This follows from the following lemma and its proof.
\end{proof}

\begin{lemma}\label{lem_ext_criterion}
An object $M\in\mca{B}_{\mf{C},T^{1/2},s}\mbox{-}\mr{gmod}^{H}$ has a costandard filtration if and only if 
\begin{align}\label{ext_criterion}
\ext^{1}_{\mca{B}_{\mf{C},T^{1/2},s}\mbox{-}\mr{gmod}^{H}}\left(\Delta^{\mca{B}}_{\mf{C},T^{1/2},s}(\lambda,p),M\right)=0
\end{align}
for any $(\lambda,p)\in\bb{F}$. 
\end{lemma}

\begin{proof}
The only if part follows from (\ref{ext_orthonormality}). Let $M$ be an object satisfying (\ref{ext_criterion}). Since $M$ has finite length, we may prove the statement by induction on the length of $M$. If $M\neq0$, then one can take a minimal $(\lambda,p)\in\bb{F}$ such that $\hom\left(L^{\mca{B}}_{\mf{C},T^{1/2},s}(\lambda,p),M\right)\neq0$. For any $(\lambda',p')<_{\mf{C},s}(\lambda,p)$, let $K$ be the kernel of $\Delta^{\mca{B}}_{\mf{C},T^{1/2},s}(\lambda',p')\twoheadrightarrow L^{\mca{B}}_{\mf{C},T^{1/2},s}(\lambda',p')$. By the discussion above, any composition factor $v^{j}L^{\mca{B}}_{\mf{C},T^{1/2},s}(\lambda'',p'')$ of $K$ satisfies $(\lambda'',p'')<_{\mf{C},s}(\lambda',p')<_{\mf{C},s}(\lambda,p)$ and hence we have $\hom(K,M)=0$. By the exact sequence 
\begin{align*}
\hom(K,M)\rightarrow\ext^{1}\left(L^{\mca{B}}_{\mf{C},T^{1/2},s}(\lambda',p'),M\right)\rightarrow\ext^{1}\left(\Delta^{\mca{B}}_{\mf{C},T^{1/2},s}(\lambda',p'),M\right)
\end{align*}
we obtain $\ext^{1}\left(L^{\mca{B}}_{\mf{C},T^{1/2},s}(\lambda',p'),M\right)=0$ for any $(\lambda',p')<_{\mf{C},s}(\lambda,p)$. Since the composition factors of $\nabla^{\mca{B}}_{\mf{C},T^{1/2},s}(\lambda,p)/L^{\mca{B}}_{\mf{C},T^{1/2},s}(\lambda,p)$ are of the form $v^{j}L^{\mca{B}}_{\mf{C},T^{1/2},s}(\lambda',p')$ for $(\lambda',p')<_{\mf{C},s}(\lambda,p)$, we obtain $\ext^{1}\left(\nabla^{\mca{B}}_{\mf{C},T^{1/2},s}(\lambda,p)/L^{\mca{B}}_{\mf{C},T^{1/2},s}(\lambda,p),M\right)=0$ and $\hom\left(\nabla^{\mca{B}}_{\mf{C},T^{1/2},s}(\lambda,p)/L^{\mca{B}}_{\mf{C},T^{1/2},s}(\lambda,p),M\right)=0$. This implies that $\hom\left(\nabla^{\mca{B}}_{\mf{C},T^{1/2},s}(\lambda,p),M\right)\cong\hom\left(L^{\mca{B}}_{\mf{C},T^{1/2},s}(\lambda,p),M\right)$, hence we can lift the inclusion $L^{\mca{B}}_{\mf{C},T^{1/2},s}(\lambda,p)\hookrightarrow M$ to a homomorphism $f:\nabla^{\mca{B}}_{\mf{C},T^{1/2},s}(\lambda,p)\rightarrow M$. 

We claim that $f$ is injective. Otherwise, there is a simple submodule $v^{j}L^{\mca{B}}_{\mf{C},T^{1/2},s}(\lambda',p')$ in $\Ker(f)\subset\nabla^{\mca{B}}_{\mf{C},T^{1/2},s}(\lambda,p)$. Then there is a nontrivial homomorphism between $v^{j}\Delta^{\mca{B}}_{\mf{C},T^{1/2},s}(\lambda',p')$ and $\nabla^{\mca{B}}_{\mf{C},T^{1/2},s}(\lambda,p)$, hence we must have $j=0$ and $(\lambda',p')=(\lambda,p)$ by (\ref{ext_orthonormality}). This implies $f(L^{\mca{B}}_{\mf{C},T^{1/2},s}(\lambda,p))=0$ which contradicts the choice of $f$. 

Therefore, we obtain an inclusion $\nabla^{\mca{B}}_{\mf{C},T^{1/2},s}(\lambda,p)\hookrightarrow M$. Let $N$ be the cokernel of this inclusion. Then $N$ also satisfies the condition (\ref{ext_criterion}) by (\ref{ext_orthonormality}) and the length of $N$ is smaller than $M$. By induction hypothesis, $N$ has a costandard filtration and hence $M$ does.  
\end{proof}

Finally, we also conjecture the following statement which lifts the $K$-theoretic bar involutions to the derived category. Let $\mca{B}_{\mf{C},T^{1/2},s}=\oplus_{\lambda\in\bb{X}^{\ast}(H)}\mca{B}^{\lambda}_{\mf{C},T^{1/2},s}$ be the $H$-weight space decomposition. 

\begin{conj}\label{anti-involution}
There exists an anti-involution $\iota$ on $\mca{B}_{\mf{C},T^{1/2},s}$ which is identity on degree 0 part, compatible with the grading, and satisfies $\iota(\mca{B}^{\lambda}_{\mf{C},T^{1/2},s})=\mca{B}^{-\lambda}_{\mf{C},T^{1/2},s}$ for any $\lambda\in\bb{X}^{\ast}(H)$. 
\end{conj}

For toric hyper-K\"ahler manifolds, this conjecture is proved in Corollary~\ref{cor_toric_anti_involution}. Let $M=\oplus_{i\in\bb{Z},\lambda\in\bb{X}^{\ast}(H)}M_{i}^{\lambda}$ be a graded $H$-equivariant module of $\mca{B}_{\mf{C},T^{1/2},s}$, where $M_{i}^{\lambda}$ has degree $i$ and $H$-weight $\lambda$. We define $\mb{D}M=\oplus_{i\in\bb{Z},\lambda\in\bb{X}^{\ast}(H)}(\mb{D}M)^{\lambda}_{i}\in\mca{B}_{\mf{C},T^{1/2},s}\mbox{-}\mr{gmod}^{H}$ by setting $(\mb{D}M)^{\lambda}_{i}\coloneqq\Hom_{\bb{C}}(M^{\lambda}_{-i},\bb{C})$ equipped with a left $\mca{B}_{\mf{C},T^{1/2},s}$-module structure through $\iota$. The properties of $\iota$ in Conjecture~\ref{anti-involution} implies that $\mb{D}L^{\mca{B}}_{\mf{C},T^{1/2},s}(\lambda,p)\cong L^{\mca{B}}_{\mf{C},T^{1/2},s}(\lambda,p)$. In particular, $\mb{D}$ induces an involution on $\DCohT(X)$ which lifts the $K$-theoretic bar involution.



\subsection{Wall-crossings}

By varying the slope parameter, we obtain a family of $t$-structures $\tau_{T^{1/2},s}$ on $\msc{D}\coloneqq\DCoh_{L}(X)$, the derived category of coherent sheaves on $X$ set-theoretically supported in $L$. It may be natural to expect that this is part of a data defining real variations of stability condition in the sense of Anno-Bezrukavnikov-Mirkovi\'c \cite{ABM}. We first recall the definition of real variations of stability conditions in our situation. 

Let $\mr{Alc}_{K}$ be the set of connected components of $P_{\bb{R}}\setminus\cup_{\beta\in\overline{\Psi}}\{s\in P_{\bb{R}}\mid \langle s,\beta\rangle\in\bb{Z}\}$. We call an element of $\mr{Alc}_{K}$ K\"ahler alcove. For two K\"ahler alcoves $\mb{A}_{-}\neq \mb{A}_{+}\in\mr{Alc}_{K}$ sharing a codimension one face contained in $w_{\beta,n}\coloneqq\{s\in P_{\bb{R}}\mid \langle s,\beta\rangle=m\}$ for some $m\in\bb{Z}$ and $\beta\in\overline{\Psi}$ which is positive with respect to $\mf{A}$, we say that $\mb{A}_{+}$ is \emph{above} $\mb{A}_{-}$ if $\mb{A}_{-}\subset\{s\in P_{\bb{R}}\mid \langle s,\beta\rangle<m\}$ and $\mb{A}_{+}\subset\{s\in P_{\bb{R}}\mid \langle s,\beta\rangle>m\}$. Let $\msc{Z}:P_{\bb{R}}\rightarrow\Hom_{\bb{Z}}(K(\msc{D}),\bb{R})$ be a polynomial map and $\tau$ be a map from $\mr{Alc}$ to the set of bounded $t$-structures on $\msc{D}$. For $\mb{A}\in\mr{Alc}_{K}$, let $\msc{C}_{\mb{A}}$ be the heart of the $t$-structure $\tau(\mb{A})$ on $\msc{D}$. For a hyperplane $w\subset P_{\bb{R}}$ and $n\in\bb{Z}_{\geq0}$, let $\msc{C}^{n}_{\mb{A},w}\subset\msc{C}_{\mb{A}}$ be the full subcategory consisting of objects $M\in\msc{C}_{\mb{A}}$ such that the polynomial function on $P_{\bb{R}}$ defined by $s\mapsto \langle \msc{Z}(s),[M]\rangle$ has zero of order at least $n$ on $w$. This is a Serre subcategory of $\msc{C}_{\mb{A}}$ and let $\msc{D}^{n}_{\mb{A},w}\coloneqq\{\mca{F}\in\msc{D}\mid \forall j,H^{j}_{\tau(\mb{A})}(\mca{F})\in\msc{C}^{n}_{\mb{A},w}\}$ be a thick subcategory of $\msc{D}$. Here, $H^{j}_{\tau(\mb{A})}$ is the $j$-th cohomology functor with respect to the $t$-structure $\tau(\mb{A})$. We set $\mr{gr}^{n}_{\mb{A},w}(\msc{D})\coloneqq\msc{D}^{n}_{\mb{A},w}/\msc{D}^{n+1}_{\mb{A},w}$ and $\mr{gr}^{n}_{w}(\msc{C}_{\mb{A}})\coloneqq\msc{C}^{n}_{\mb{A},w}/\msc{C}^{n+1}_{\mb{A},w}$.

\begin{dfn}[\cite{ABM}]\label{real_variation}
A data $(\msc{Z},\tau)$ as above is called \emph{real variation of stability conditions} on $\msc{D}$ if it satisfies the following conditions: 
\begin{itemize}
\item For any $\mb{A}\in\mr{Alc}_{K}$ and nonzero $M\in\msc{C}_{\mb{A}}$, we have $\langle\msc{Z}(s),[M]\rangle>0$ for any $s\in \mb{A}$. 
\item For any $\mb{A}_{-}\neq \mb{A}_{+}\in\mr{Alc}_{K}$ sharing a codimension one face contained in a hyperplane $w$ with $\mb{A}_{+}$ being above $\mb{A}_{-}$, we have $\msc{D}^{n}_{\mb{A}_{-},w}=\msc{D}^{n}_{\mb{A}_{+},w}$ and $\mr{gr}^{n}_{w}(\msc{C}_{\mb{A}_{-}})=\mr{gr}^{n}_{w}(\msc{C}_{\mb{A}_{+}})[n]$ in $\mr{gr}^{n}_{\mb{A}_{-},w}(\msc{D})=\mr{gr}^{n}_{\mb{A}_{+},w}(\msc{D})$ for any $n\in\bb{Z}_{\geq0}$.\footnote{We change the sign here from \cite{ABM} since we mainly use the slope parameters, which are essentially opposite to the quantization parameters.}
\end{itemize}
The polynomial function $\msc{Z}$ is called \emph{central charge} for the real variation of stability conditions. 
\end{dfn}

\begin{remark}
A part of the conjecture of Bezrukavnikov-Okounkov stated in \cite{ABM} claims that there exists a real variation of stability conditions on $\msc{D}$ as above (we do not need to assume that the fixed point set $X^{H}$ is finite). Moreover, the $t$-structures $\tau$ are given by quantization of $X$ in positive characteristic.  

Let $\ell$ be a prime number and consider our conical symplectic resolutions over an algebraically closed field of characteristic $\ell$ temporarily. For $\lambda\in P$, one can consider Frobenius constant quantization of $\mca{O}_{X}$ with quantization parameter $\lambda$ which gives a sheaf of Azumaya algebras $\msc{A}_{\lambda}$ on the Frobenius twist $X^{(1)}$ of $X$. Then the conjecture says that when $\ell$ is sufficiently large and $-\frac{\lambda}{\ell}\in \mb{A}$, the Azumaya algebra $\msc{A}_{\lambda}$ splits on the formal neighborhood of $L^{(1)}$ and if the splitting bundles are chosen in a compatible way, their $\bb{S}$-equivariant lifts to $X\cong X^{(1)}$ gives a tilting bundle and its dual gives a $t$-structure compatible with $\tau(\mb{A})$ under base change to positive characteristic. See also \cite{W} for another approach to this tilting bundle. 

As an analogue of Lusztig's conjecture on modular representation theory, we expect that the set of vector bundles $\{\mca{E}_{\mf{C},T^{1/2},s}(p)\}_{p\in X^{H}}$ (considered in positive characteristic) give the set of indecomposable summands of the dual of a splitting bundle of $\msc{A}_{\lambda}$ up to equivariant parameter twists when $s=-\frac{\lambda}{\ell}$ and $\ell$ is sufficiently large. We note that if we change the polarization $T^{1/2}$ by $T^{1/2}_{\mca{G}}$ for some $\mca{G}\in K_{\bb{T}}(X)$, then the $K$-theoretic canonical bases will change by tensor product of some line bundle by Lemma~\ref{can_pol} and Lemma~\ref{can_slope}. Since a line bundle twist of a splitting bundle is also a splitting bundle, this change can be absorbed into the choice of splitting bundles.  
\end{remark}

Let $\mca{P}$ be a vector bundle on $X$ and $\msc{Z}_{\mca{P}}:P_{\bb{R}}\rightarrow\Hom_{\bb{Z}}(K(\msc{D}),\bb{R})$ be a polynomial function satisfying 
\begin{align*}
\langle \msc{Z}_{\mca{P}}(\mca{L}),\mca{F}\rangle=\frac{1}{\rank{\mca{P}}}\cdot\chi(X,\mca{F}\otimes\mca{P}\otimes\mca{L}^{-1})
\end{align*}
for any line bundle $\mca{L}\in P$ and $\mca{F}\in\msc{D}$. We note that $\msc{Z}_{\mca{P}^{\oplus n}}=\msc{Z}_{\mca{P}}$ for any $n\in\bb{Z}_{>0}$. In this paper, we only conjecture the following weaker statement. We will check this for toric hyper-K\"ahler manifolds in Corollary~\ref{cor_toric_real_variation}.  

\begin{conj}\label{conj_real_variation}
There exists a vector bundle $\mca{P}$ on $X$ such that $(\msc{Z}_{\mca{P}},\tau)$ gives a real variation of stability conditions on $\msc{D}$, where $\tau$ is given by $\tau(\mb{A})=\tau_{T^{1/2},s}$ for $s\in\mb{A}$. 
\end{conj}

\begin{remark}
The vector bundle $\mca{P}$ should be a line bundle when $X$ is a Slodowy variety by the result of \cite{ABM}, but it should be a vector bundle of higher rank in general. This vector bundle is expected to be given by looking at the asymptotic behavior under $\ell\rightarrow\infty$ of the multiplicity of indecomposable summands of the splitting bundle above for a fixed $\lambda$. More precisely, fix a generic $\lambda$ and let $m_{p}(\ell)$ be the multiplicity of $\mca{E}_{\mf{C},T^{1/2},-\frac{\lambda}{\ell}}(p)^{\vee}$ in the splitting bundle of $\msc{A}_{\lambda}$ (as non-equivariant vector bundles). We set $m_{p}\coloneqq\lim_{\ell\rightarrow\infty}\frac{1}{\ell^{\dim X/2}}m_{p}(\ell)$ and take $m\in\bb{Z}_{>0}$ such that $m\cdot m_{p}$ is an integer for any $p\in X^{H}$. Let $\mb{A}\in\mr{Alc}_{K}$ be the alcove containing $-\frac{\lambda}{\ell}$ for any sufficiently large $\ell$ and take $s\in\mb{A}$. Then we expect that one can take
\begin{align*}
\mca{P}=\sum_{p\in X^{H}}m\cdot m_{p}\cdot\mca{E}_{\mf{C},T^{1/2},s}(p)^{\vee}.
\end{align*}
We note that the central charge $\msc{Z}_{\mca{P}}$ does not depend on the choice of $m$. We also expect that this does not depend on the choice of $\lambda$ if we forget the equivariant structures. 
\end{remark}

We now describe the behavior of $K$-theoretic canonical bases under the wall-crossing of the slope parameters. As in the previous section, we expect that the information on the equivariant parameter $v$ has some information on the cohomological shifts appearing in Definition~\ref{real_variation}. The following conjecture comes from numerical experiments. 

\begin{conj}\label{conj_wall_crossing}
Let $\mb{A}_{-}\neq \mb{A}_{+}\in\mr{Alc}_{K}$ be two K\"ahler alcoves sharing a codimension one face contained in a hyperplane $w$ with $\mb{A}_{+}$ being above $\mb{A}_{-}$. For $s_{-}\in \mb{A}_{-}$ and $s_{+}\in \mb{A}_{+}$, there exists a sequence of integers $0\leq n_{0}<n_{1}<\cdots<n_{l}$ and decompositions $\mathbb{B}_{X,T^{1/2},s_{-}}=\sqcup_{i=0}^{l}\mathbb{B}^{i}_{s_{-},w}$ and $\mathbb{B}_{X,T^{1/2},s_{+}}=\sqcup_{i=0}^{l}\mathbb{B}^{i}_{s_{+},w}$ stable under equivariant parameter shifts for $H$ such that for any $\mathcal{E}'\in\mathbb{B}^{i}_{s_{+},w}$, there exists $\mathcal{E}\in\mathbb{B}^{i}_{s_{-},w}$ satisfying 
\begin{align*}
\mathcal{E}'=(-1)^{n_{i}-n_{0}-i}v^{n_{i}}\mathcal{E}+\sum_{\substack{j<i\\\mathcal{F}\in\mathbb{B}^{j}_{s_{-},\mca{H}}}}f_{\mathcal{E},\mathcal{F}}(v)\cdot\mathcal{F}, 
\end{align*}
where $f_{\mathcal{E},\mathcal{F}}(v)\in v^{n_{j}+1}\mathbb{Z}[v]\cap v^{n_{i}-1}\bb{Z}[v^{-1}]$ for any $\mathcal{F}\in\mathbb{B}^{j}_{s_{-},\mca{H}}$. Moreover, if we assume Conjecture~\ref{conj_real_variation} and let $\mathcal{C}\in\mathbb{B}_{L,T^{1/2},s}$ be the element dual to $\mathcal{E}$, then $n_{i}-n_{0}-i$ is the order of vanishing at $w$ of the polynomial function $s\mapsto\langle \msc{Z}_{\mca{P}}(s),\mca{C}\rangle$. 
\end{conj}

For toric hyper-K\"ahler manifolds, this conjecture follows from Proposition~\ref{prop_toric_wall_crossing_characterization}, Lemma~\ref{lem_toric_wall_crossing_formula}, and Corollary~\ref{cor_toric_central_charge}. This conjecture implies that if one can calculate the behavior of $K$-theoretic canonical bases under wall-crossing, one can find the order of zero at various hyperplanes for the central charge of every elements in $\mathbb{B}_{L,T^{1/2},s}$. This information is usually enough to determine $\mca{P}$ in practice.

One can also use this conjecture to compute $\mathbb{B}_{X,T^{1/2},s_{+}}$ from the knowledge of $\mathbb{B}_{X,T^{1/2},s_{-}}$ since this implies that $\beta^{K}_{T^{1/2},s_{+}}$ is triangular with respect to $\mathbb{B}_{X,T^{1/2},s_{-}}$ and the degree condition on $v$ enables us to calculate each $f_{\mathcal{E},\mathcal{F}}(v)$ by Kazhdan-Lusztig type algorithm as in the proof of Proposition~\ref{expansion}. For example, when $X$ is the Hilbert scheme of $n$-points in the affine plane, then $\mathbb{B}_{X,T^{1/2},s}$ is given by the indecomposable summands of the Procesi bundle up to equivariant parameter shifts when $s$ is sufficiently close to $1$. Using this algorithm, one can calculate $\mathbb{B}_{X,T^{1/2},s}$ for any $s$ in principle. We have checked the first part of Conjecture~\ref{conj_wall_crossing} (together with Conjecture~\ref{K-theoretic_canonical_basis}) in this case when $n\leq 8$ by using computer. 

\section{Elliptic bar involutions}

In this section, we discuss an elliptic analogue of the $K$-theoretic bar involutions. Since our definition of the $K$-theoretic bar involution only involves $K$-theoretic stable bases, it seems natural to define elliptic bar involution using elliptic stable bases defined by Aganagic-Okounkov \cite{AO}. We conjecture that under certain natural normalization (contrary to the seemingly ad hoc normalization in the $K$-theory case), this will give us an involution on a certain analytic completion of the localized (extended) equivariant $K$-theory. 

In order to define an elliptic analogue of the $K$-theoretic canonical bases, we need to find an elliptic version of some conditions characterizing canonical bases other than bar invariance. In this paper, we do not investigate this direction further. Instead, we will construct a family of elliptic bar invariant elements which are ``as simple as possible'' and have some nice properties to be called elliptic canonical bases in the case of toric hyper-K\"ahler manifolds in section 6. We leave the problem of finding a characterization of these elements which makes sense in general for a future direction of research. 

\subsection{Elliptic standard bases}

We first briefly recall the definition of elliptic stable basis defined by Aganagic-Okounkov. For more details, we refer to the original paper \cite{AO}. We follow the notations and assumptions of section 3. 

We fix an elliptic curve $E\coloneqq\bb{C}^{\times}/q^{\bb{Z}}$ over $\bb{C}$, where $q$ is a complex number satisfying $0<|q|<1$. Let $Ell_{\bb{T}}(X)$ be the $\bb{T}$-equivariant elliptic cohomology of $X$ associated with $E$. Since odd cohomology of $X$ vanishes, this is a scheme which is affine over the abelian variety $E_{\bb{T}}\coloneqq \bb{X}_{\ast}(\bb{T})\otimes_{\bb{Z}}E=Ell_{\bb{T}}(\mr{pt})$. For a construction of equivariant elliptic cohomology, see for example \cite{Ga}. We also set $E_{P}\coloneqq P\otimes_{\bb{Z}}E$ and $E^{\bb{T}}_{P}\coloneqq \Pic^{\bb{T}}(X)\otimes_{\bb{Z}}E$. We note that there is an exact sequence 
\begin{align*}
0\rightarrow E_{\bb{T}}^{\vee}\rightarrow E^{\bb{T}}_{P}\rightarrow E_{P}\rightarrow0,
\end{align*}
where $E_{\bb{T}}^{\vee}=\bb{X}^{\ast}(\bb{T})\otimes_{\bb{Z}}E$. We define the extended equivariant elliptic cohomology for $X$ by $\mb{E}(X)\coloneqq Ell_{\bb{T}}(X)\times E_{P}$, which is affine over $\mb{B}_{X}\coloneqq E_{\bb{T}}\times E_{P}$. We note that since we have $\bb{X}_{\ast}(H)\cong P^{!}$ and $P\cong\bb{X}_{\ast}(H^{!})$, we can identify $\mb{B}_{X}\cong\mb{B}_{X^{!}}$ by simply exchanging the equivariant and K\"ahler parameters. We also set $\widetilde{\mb{E}}(X)\coloneqq Ell_{\bb{T}}(X)\times E^{\bb{T}}_{P}$, $\widetilde{\mb{B}}_{X}\coloneqq E_{\bb{T}}\times E^{\bb{T}}_{P}$. The structure morphism $\widetilde{\mb{E}}(X)\rightarrow\widetilde{\mb{B}}_{X}$ will be denoted by $\widetilde{\pi}$. 

Let
\begin{align*}
\vartheta(x)\coloneqq(x^{1/2}-x^{-1/2})\prod_{m=1}^{\infty}(1-q^{m}x)(1-q^{m}x^{-1})
\end{align*}
be a theta function. This is a multivalued holomorphic function on $\bb{C}^{\times}$ which satisfies $\vartheta(e^{2\pi i}x)=-\vartheta(x)$ and $\vartheta(qx)=-x^{-1}q^{-1/2}\vartheta(x)$ and can be seen as a section of degree 1 line bundle on $E$. Using this line bundle, we identify $E$ and its dual abelian variety $E^{\vee}$. Then the Poincar\'e line bundle on $E\times E^{\vee}\cong E\times E$ can be defined by the quasi-periods of the function 
\begin{align*}
(x,y)\mapsto\psi(x,y)\coloneqq\frac{\vartheta(xy)}{\vartheta(x)\vartheta(y)},
\end{align*}
i.e., it is single valued on $\bb{C}^{\times}\times\bb{C}^{\times}$ and satisfies $\psi(qx,y)=y^{-1}\psi(x,y)$ and $\psi(x,qy)=x^{-1}\psi(x,y)$. For $r\in\bb{Z}_{>0}$, we define a line bundle $\mca{O}(D)$ on $S^{r}E$, the $r$-th symmetric product of $E$, by the factors of automorphy of the symmetric function $(x_{1},\ldots, x_{r})\mapsto\prod_{i=1}^{r}\vartheta(x_{i})$.  

Let $V$ be a $\bb{T}$-equivariant vector bundle on $X$. Its characteristic classes give us a morphism $c_{V}:Ell_{\bb{T}}(X)\rightarrow S^{r}E$ and $\Theta(V)\coloneqq c_{V}^{\ast}\mca{O}(D)\in\Pic(Ell_{\bb{T}}(X))$ is called the Thoms class of $V$. In particular, by considering the characteristic classes of line bundles, we obtain a morphism $Ell_{\bb{T}}(X)\rightarrow\Hom_{\bb{Z}}(\Pic^{\bb{T}}(X),E)\cong(E^{\bb{T}}_{P})^{\vee}$ and hence $\widetilde{\mb{E}}(X)\rightarrow E^{\bb{T}}_{P}\times(E^{\bb{T}}_{P})^{\vee}$. We denote by $\mca{U}_{X}$ the line bundle on $\widetilde{\mb{E}}(X)$ defined by pulling back the Poincar\'e line bundle on $E^{\bb{T}}_{P}\times(E^{\bb{T}}_{P})^{\vee}$. 

For $\lambda\in\Pic^{\bb{T}}(X)\cong\Hom(E,E^{\bb{T}}_{P})$ and $\mu\in\bb{X}^{\ast}(\bb{T})\cong\Hom(E_{\bb{T}},E)$, let $\tau(\lambda,\mu):\widetilde{\mb{B}}_{X}\rightarrow\widetilde{\mb{B}}_{X}$ be the shift of K\"ahler parameters $(t,z)\mapsto(t,z+\lambda(\mu(t)))$, where $t\in E_{\bb{T}}$ and $z\in E^{\bb{T}}_{P}$. We denote by the same letter for the shift of K\"ahler parameters on $\widetilde{\mb{E}}(X)$.

For each fixed point $p\in X^{H}$, we obtain a natural morphism $i_{p}:\widetilde{\mb{B}}_{X}\cong Ell_{\bb{T}}(p)\times E^{\bb{T}}_{P}\rightarrow\widetilde{\mb{E}}(X)$ coming from the inclusion $i_{p}:\{p\}\hookrightarrow X$ and the functoriality of elliptic cohomology. We set $\mca{U}_{p}\coloneqq i^{\ast}_{p}\mca{U}_{X}$. Recall that we take a chamber $\mf{C}$ and a polarization $T^{1/2}$. Take a sufficiently generic $\xi\in\mf{C}$ and decompose $T^{1/2}_{p}=\mr{ind}_{p}+\mr{ind}^{-}_{p}+T^{1/2}_{p,=0}$ into attracting, repelling, and fixed parts with respect to $\xi$, where we assume that $T^{1/2}_{p,=0}$ coincides with the $H$-fixed part of $T^{1/2}_{p}$. We note that this decomposition might depends on the choice of $\xi$, but the definition of elliptic stable basis does not depend on this choice. 

\begin{dfn}[\cite{AO}]
For each $p\in X^{H}$, the \emph{elliptic stable basis} $\Stab^{AO}_{\mf{C},T^{1/2}}(p)$ is a section of some line bundle on $\widetilde{\mb{E}}(X)$ characterized by the following conditions: 
\begin{itemize}
\item $\Stab^{AO}_{\mf{C},T^{1/2}}(p)$ is a section of $\mca{U}_{X}\otimes\Theta(T^{1/2})\otimes\widetilde{\pi}^{\ast}(\tau(\det\mr{ind}_{p},v^{-2})^{\ast}\mca{U}_{p}^{-1}\otimes\Theta(T^{1/2}_{p,=0})^{-1})\otimes\ldots$, where $\ldots$ is a certain line bundle pulled back from $\widetilde{\mb{B}}_{X}/E_{H}$ and the section is allowed to be meromorphic on this factor. Here, $E_{H}$ acts on $\widetilde{\mb{B}}_{X}$ by the translation on the factor $E_{\bb{T}}$. 
\item The support of $\Stab^{AO}_{\mf{C},T^{1/2}}(p)$ is contained in $\sqcup_{p'\preceq_{\mf{C}}p}\Attr_{\mf{C}}(p')$. 
\item We have $i_{p}^{\ast}\Stab^{AO}_{\mf{C},T^{1/2}}(p)=\vartheta(N_{p,-})\in\Gamma(\widetilde{\mb{B}}_{X},\Theta(N_{p,-}))$, where $\vartheta(N_{p,-})=\prod_{i}\vartheta(w_{i})$ if we write $N_{p,-}=\sum_{i}[w_{i}]\in K_{\bb{T}}(p)$, $w_{i}\in\bb{X}^{\ast}(\bb{T})$.
\end{itemize}
\end{dfn}

By \cite{AO}, this is unique if it exists and the existence is proved for the case where $X$ is a toric hyper-K\"ahler manifold or a quiver variety. We assume the existence for the conical symplectic resolutions we consider in this paper. Moreover, this is constant on the $E_{\bb{T}}^{\vee}$-orbits, hence defines a section of some line bundle on $\mb{E}(X)$ which is also denoted by $\Stab^{AO}_{\mf{C},T^{1/2}}(p)$. 

As in the case of $K$-theory, it might be better to change the normalization of the elliptic stable bases slightly for our purpose. Recall that we always assume the existence of dual conical symplectic resolution $X^{!}=(X^{!},\mf{C}^{!},\mf{A}^{!},\ldots)$ for $X=(X,\mf{C},\mf{A},\ldots)$. 

\begin{dfn}
For each $p\in X^{H}$, we define the \emph{elliptic standard basis} $\Stab^{ell}_{X}(p)$ by 
\begin{align*}
\Stab^{ell}_{X}(p)=\vartheta(N^{!}_{p^{!},-})\cdot\tau(\det T^{1/2},v)^{\ast}(\Stab^{AO}_{\mf{C},T^{1/2}}(p)).
\end{align*}
\end{dfn}

Next we describe the line bundle of which $\Stab^{ell}_{\mf{C}}(p)$ defines a section. More precisely, we give a formula for the factors of automorphy of the restriction $\mb{S}_{X,p',p}=\mb{S}_{p',p}\coloneqq i^{\ast}_{p'}\Stab^{ell}_{X}(p)$ for every $p,p'\in X^{H}$. We will consider $\mb{S}_{p',p}$ as a multivalued meromorphic function on $\mb{B}^{K}_{X}\coloneqq(\bb{X}_{\ast}(\bb{T})\times P)\otimes_{\bb{Z}}\bb{C}^{\times}$. As in section 3.1, we take a basis $\{a_{1},\ldots,a_{e},v\}$ of $\bb{X}^{\ast}(\bb{T})$ which will be considered as a system of coordinates on $\bb{X}_{\ast}(\bb{T})\otimes_{\bb{Z}}\bb{C}^{\times}$. Similarly, we take a basis $\{z_{1},\ldots,z_{r}\}$ of $P^{\vee}$ which will be considered as a system of coordinates on $P\otimes_{\bb{Z}}\bb{C}^{\times}$. For each $\gamma\in \bb{X}_{\ast}(\bb{T})\times P$, we denote by $\theta_{X,p}(\gamma)$ the factor of automorphy of the function $\vartheta(N_{p,-})$ under the translation by $q^{\gamma}\coloneqq\gamma\otimes q\in \mb{B}^{K}_{X}$. We also set $\theta_{X,p',p}(\gamma)=\theta_{p',p}(\gamma)\coloneqq\theta_{X,p'}(\gamma)\cdot\theta_{X^{!},p^{!}}(\gamma)$. Assuming the existence of the dual pair in the sense of Definition~\ref{dual_pair}, we prove the following. 

\begin{prop}\label{prop_factor_ell_std}
In the above situation, $\mb{S}_{p',p}$ satisfies the following:
\begin{itemize}
\item For each $c\in\bb{X}_{\ast}(H)$, we have 
\begin{align}\label{factor_equiv}
\mb{S}_{p',p}(a\mapsto q^{c}a)=i_{p^{!}}^{\ast}\mf{L}^{!}(c)^{-1}\cdot i_{p'^{!}}^{\ast}\mf{L}^{!}(c)\cdot\theta_{p',p}(c)\cdot\mb{S}_{p',p}.
\end{align}
\item For each $l\in P$, we have 
\begin{align}\label{factor_Kahler}
\mb{S}_{p',p}(z\mapsto q^{l}z)=i_{p}^{\ast}\mf{L}(l)\cdot i_{p'}^{\ast}\mf{L}(l)^{-1}\cdot\theta_{p',p}(l)\cdot\mb{S}_{p',p}.
\end{align}
\item For $\delta\in\bb{X}_{\ast}(\bb{S})$ satisfying $\langle\delta,v\rangle=1$, we have 
\begin{align}\label{factor_v}
\mb{S}_{p',p}(v\mapsto qv)=\frac{\det N_{p,-}}{\det N_{p',-}}\frac{\det N^{!}_{p'^{!},-}}{\det N^{!}_{p^{!},-}}\cdot q^{\wt_{\bb{S}}\frac{\det N_{p,-}}{\det N_{p',-}}}\cdot\theta_{p',p}(\delta)\cdot\mb{S}_{p',p}.
\end{align}
\item $\sqrt{\det N_{p',-}\cdot\det N^{!}_{p^{!},-}}^{-1}\cdot\mb{S}_{p',p}$ is single valued on $\mb{B}^{K}_{X}$.
\end{itemize}
\end{prop}

\begin{proof}
By Proposition 3.1 in \cite{AO} and $T^{1/2}_{p'}=N_{p',-}+\mr{ind}_{p'}-v^{-2}\mr{ind}_{p'}^{\vee}+T^{1/2}_{p',=0}$, $\mb{S}_{p',p}$ is a meromorphic section of the line bundle
\begin{align}\label{ell_line_bundle}
\frac{\tau(\det T^{1/2},v)^{\ast}\mca{U}_{p'}}{\tau(\det T^{1/2}\cdot\det\mr{ind}_{p}^{-2},v)^{\ast}\mca{U}_{p}}\cdot\frac{\Theta(\mr{ind}_{p'})\Theta(v^{2})^{\rank\mr{ind}_{p'}}}{\Theta(v^{-2}\cdot\mr{ind}_{p'}^{\vee})}\cdot\frac{\Theta(T^{1/2}_{p',=0})\Theta(v^{2})^{-\rank\mr{ind}_{p'}}}{\Theta(T^{1/2}_{p,=0})\Theta(v^{2})^{-\rank\mr{ind}_{p}}}\cdot\Theta(N_{p',-})\Theta(N^{!}_{p^{!},-}). 
\end{align}
In particular, this implies the single valuedness of $\sqrt{\det N_{p',-}\cdot\det N^{!}_{p^{!},-}}^{-1}\cdot\mb{S}_{p',p}$. We now describe the factors of automorphy for each factors in (\ref{ell_line_bundle}).  

Let $\{l_{1},\ldots,l_{r}\}\subset P$ be the dual basis of $\{z_{1},\ldots,z_{r}\}$. We set $\mf{L}_{i}\coloneqq\mf{L}(l_{i})\in\Pic^{\bb{T}}(X)$. Then $\{\mf{L}_{1},\ldots,\mf{L}_{r}, a_{1},\ldots,a_{e},v\}$ is a basis of $\Pic^{\bb{T}}(X)$. Let $\{z_{1},\ldots,z_{r},z_{a_{1}},\ldots,z_{a_{e}},z_{v}\}\subset\Pic^{\bb{T}}(X)^{\vee}$ be its dual basis. By definition, the line bundle $\mca{U}_{p}$ on $\widetilde{\mb{B}}_{X}$ is characterized by the factors of automorphy of the function 
\begin{align*}
\prod_{i=1}^{r}\psi(i^{\ast}_{p}\mf{L}_{i},z_{i})\cdot\prod_{i=1}^{e}\psi(a_{i},z_{a_{i}})\cdot\psi(v,z_{v})
\end{align*}
for each $p\in X^{H}$. Therefore, for $\mca{L}=\prod_{i}\mf{L}_{i}^{n_{i}}\cdot\prod_{i}a_{i}^{n_{a_{i}}}\cdot v^{n_{v}}\in\Pic^{\bb{T}}(X)$, the line bundle $\tau(\mca{L},v)^{\ast}\mca{U}_{p}$ is characterized by the factor of automorphy of the function 
\begin{align*}
\prod_{i=1}^{r}\psi(i^{\ast}_{p}\mf{L}_{i},z_{i}v^{n_{i}})\cdot\prod_{i=1}^{e}\psi(a_{i},z_{a_{i}}v^{n_{a_{i}}})\cdot\psi(v,z_{v}v^{n_{v}}).
\end{align*}
By using the formula 
\begin{align*}
\psi(q^{m}x,q^{n}y)=q^{-mn}x^{-n}y^{-m}\psi(x,y),
\end{align*} 
one can check that the factors of automorphy for $\tau(\mca{L},v)^{\ast}\mca{U}_{p}$ are given as follows:
\begin{itemize}
\item For $a\mapsto q^{c}a$, it is given by 
\begin{align*}
\prod_{i=1}^{r}z_{i}^{-\langle i^{\ast}_{p}\mf{L}_{i},c\rangle}\cdot\prod_{i=1}^{e}z_{a_{i}}^{-\langle a_{i},c\rangle}\cdot v^{-\langle i_{p}^{\ast}\mca{L},c\rangle};
\end{align*}
\item For $z\mapsto q^{l}z$, it is given by $i_{p}^{\ast}\mf{L}(l)^{-1}$;
\item For $v\mapsto qv$, it is given by
\begin{align*}
z_{v}^{-1}\prod_{i=1}^{r}z_{i}^{-\wt_{\bb{S}}i^{\ast}_{p}\mf{L}_{i}}\cdot(qv)^{-\wt_{\bb{S}}i_{p}^{\ast}\mca{L}}\cdot i_{p}^{\ast}\mca{L}^{-1}.
\end{align*}
\end{itemize} 
Hence the factors of automorphy for the first factor in (\ref{ell_line_bundle}) is given as follows: 
\begin{itemize}
\item For $a\mapsto q^{c}a$, it is given by 
\begin{align*}
\prod_{i=1}^{r}z_{i}^{\langle i^{\ast}_{p}\mf{L}_{i},c\rangle-\langle i^{\ast}_{p'}\mf{L}_{i},c\rangle}\cdot v^{\langle \det T^{1/2}_{p},c\rangle-\langle \det T^{1/2}_{p'},c\rangle-2\langle\det\mr{ind}_{p},c\rangle};
\end{align*}
\item For $z\mapsto q^{l}z$, it is given by $i_{p}^{\ast}\mf{L}(l)\cdot i_{p'}^{\ast}\mf{L}(l)^{-1}$;
\item For $v\mapsto qv$, it is given by
\begin{align*}
\prod_{i=1}^{r}z_{i}^{\wt_{\bb{S}}i^{\ast}_{p}\mf{L}_{i}-\wt_{\bb{S}}i^{\ast}_{p'}\mf{L}_{i}}\cdot(qv)^{\wt_{\bb{S}}\det T^{1/2}_{p}-\wt_{\bb{S}}\det T^{1/2}_{p'}-2\wt_{\bb{S}}\det\mr{ind}_{p}}\cdot \frac{\det T^{1/2}_{p}}{\det T^{1/2}_{p'}}\cdot(\det\mr{ind}_{p})^{-2}.
\end{align*}
\end{itemize} 
Since the second and the third factor in (\ref{ell_line_bundle}) does not depend on the K\"ahler parameters, the equation (\ref{factor_Kahler}) follows. 

One can also check that the factor of automorphy for the second factor in (\ref{ell_line_bundle}) is given as follows:
\begin{itemize}
\item For $a\mapsto q^{c}a$, it is given by $v^{2\langle\det\mr{ind}_{p'},c\rangle}$;
\item For $v\mapsto qv$, it is given by $(qv)^{2\wt_{\bb{S}}\det\mr{ind}_{p'}}\cdot(\det\mr{ind}_{p'})^{2}$.
\end{itemize}
Since we have 
\begin{align}\label{eqn_pol_normal}
\det N_{p,-}=v^{-2\rank\mr{ind}_{p}}\cdot\det T^{1/2}_{p}\cdot(\det\mr{ind}_{p})^{-2}\cdot(\det T^{1/2}_{p,=0})^{-1}
\end{align}
and the third factor in (\ref{ell_line_bundle}) does not depend on the equivariant parameters, the factor of automorphy of $\mb{S}_{p',p}$ under $a\mapsto q^{c}a$ is given by 
\begin{align*}
\prod_{i=1}^{r}z_{i}^{\langle i^{\ast}_{p}\mf{L}_{i},c\rangle-\langle i^{\ast}_{p'}\mf{L}_{i},c\rangle}\cdot v^{\langle \det N_{p,-},c\rangle-\langle \det N_{p',-},c\rangle}\cdot\theta_{p',p}(c).
\end{align*}
This and (\ref{eqn_line_bundle_dual}) imply (\ref{factor_equiv}). 

Since $T^{1/2}_{p,=0}+v^{-2}(T^{1/2}_{p,=0})^{\vee}=0$, we have $T^{1/2}_{p,=0}=\sum_{i}(v^{m_{i}}-v^{-2-m_{i}})$ for some $m_{i}\in\bb{Z}$. Using this, one can check that the factor of automorphy for the line bundle $\Theta(T^{1/2}_{p,=0})$ under $v\mapsto qv$ is given by $(qv^{2})^{\wt_{\bb{S}}\det T^{1/2}_{p,=0}}$ and hence the factor of automorphy of the third factor in (\ref{ell_line_bundle}) under $v\mapsto qv$ is given by 
\begin{align*}
(qv^{2})^{\wt_{\bb{S}}\det T^{1/2}_{p',=0}+2\rank\mr{ind}_{p'}-\wt_{\bb{S}}\det T^{1/2}_{p,=0}-2\rank\mr{ind}_{p}}
\end{align*}
After some simplification using (\ref{eqn_pol_normal}), one can check that the factor of automorphy of $\mb{S}_{p',p}$ under $v\mapsto qv$ is given by 
\begin{align*}
\prod_{i=1}^{r}z_{i}^{\wt_{\bb{S}}i^{\ast}_{p}\mf{L}_{i}-\wt_{\bb{S}}i^{\ast}_{p'}\mf{L}_{i}}\cdot(qv)^{\wt_{\bb{S}}\frac{\det N_{p,-}}{\det N_{p',-}}}\cdot\frac{\det N_{p,-}}{\det N_{p',-}}\cdot\theta_{p',p}(\delta).
\end{align*}
By (\ref{eqn_S_wt}), we obtain $\wt_{\bb{S}}\det N_{p,-}-\wt_{\bb{S}}\det N_{p',-}=\wt_{\bb{S}}\det N^{!}_{p'^{!},-}-\wt_{\bb{S}}\det N^{!}_{p^{!},-}$ and hence (\ref{eqn_line_bundle}) implies 
\begin{align*}
\prod_{i=1}^{r}z_{i}^{\wt_{\bb{S}}i^{\ast}_{p}\mf{L}_{i}-\wt_{\bb{S}}i^{\ast}_{p'}\mf{L}_{i}}\cdot v^{\wt_{\bb{S}}\frac{\det N_{p,-}}{\det N_{p',-}}}=\frac{\det N^{!}_{p'^{!},-}}{\det N^{!}_{p^{!},-}}. 
\end{align*}
This proves (\ref{factor_v}). 
\end{proof}

In particular, if we set $\mb{S}^{!}_{p^{!},p'^{!}}\coloneqq i^{\ast}_{p^{!}}\Stab^{ell}_{X^{!}}(p'^{!})$, then the line bundle on $\mb{B}_{X}$ defined by the factors of automorphy of $\mb{S}_{p'.p}$ is the same as the line bundle on $\mb{B}_{X^{!}}$ defined by $\mb{S}^{!}_{p^{!},p'^{!}}$ under the identification $\mb{B}_{X}\cong\mb{B}_{X^{!}}$. Following \cite{AO,RSVZ1,RSVZ2}, we conjecture that elliptic standard bases have certain symmetry under the symplectic duality. 

\begin{conj}[\cite{AO,RSVZ1,RSVZ2}]\label{conj_ell_stab_duality}
For any $p,p'\in X^{H}$, $\mb{S}_{p'.p}$ is holomorphic and $\mb{S}_{p'.p}=\pm\mb{S}^{!}_{p^{!},p'^{!}}$.
\end{conj}

For toric hyper-K\"ahler manifolds, this conjecture is proved in Corollary~\ref{cor_ell_toric_stab_symmetry}. 

\subsection{Flops}

Before considering an elliptic analogue of the bar involution, we state some conjectures about the elliptic stable bases for the maximal flop of $X$. Recall that our conical symplectic resolution is always equipped with some additional data as in section 3.1. For $X=(X,\mf{C},\mf{A},\Phi,\Psi,\mf{L})$, we define $-X\coloneqq(X,-\mf{C},\mf{A},\Phi,\Psi,\mf{L})$. We also assume that $-X^{!}$ has a dual conical symplectic resolution in the sense of Definition~\ref{dual_pair}, which is denoted by $X_{\mr{flop}}=(X_{\mr{flop}},\mf{C}_{\mr{flop}},\mf{A}_{\mr{flop}},\Phi_{\mr{flop}},\Psi_{\mr{flop}},\mf{L}_{\mr{flop}})$ and called a maximal flop of $X$. By definition, we have an identification $\bb{X}_{\ast}(H)\cong P^{!}\cong\bb{X}_{\ast}(H_{\mr{flop}})$, $P\cong\bb{X}_{\ast}(H^{!})\cong P_{\mr{flop}}$, and $(X^{H},\preceq_{\mf{C}})\cong(X_{\mr{flop}}^{H_{\mr{flop}}},\succeq_{\mf{C}_{\mr{flop}}})$. We simply denote by $p\in X^{H}_{\mr{flop}}$ the fixed point corresponding to $p\in X^{H}$. Under this identification, we also have $\mf{C}_{\mr{flop}}=\mf{C}$, $\mf{A}_{\mr{flop}}=-\mf{A}$, $\Phi_{\mr{flop}}(p)=\Phi(p)$, and $\Psi_{\mr{flop}}(p)=\Psi(p)$. By (\ref{eqn_line_bundle}), we also obtain $i^{\ast}_{p}\mf{L}_{\mr{flop}}(\lambda)=\overline{i^{\ast}_{p}\mf{L}(\lambda)}$ for any $\lambda\in P$. The relation $\Phi_{\mr{flop}}(p)=\Phi(p)$ implies that the tangent spaces $T_{p}X_{\mr{flop}}$ and $T_{p}X$ have the same multiset of $H$-weights, and in particular, we have $\dim X_{\mr{flop}}=\dim X$. For the $\bb{S}$-weights, we expect the following.

\begin{conj}\label{conj_flop_tangent}
For any $p\in X^{H}$, we have $T_{p}X_{\mr{flop}}=v^{-2}\cdot\overline{T_{p}X}$ as $\bb{T}$-modules. 
\end{conj} 

For toric hyper-K\"ahler manifolds, this conjecture is checked in Corollary~\ref{cor_toric_flop}. In particular, this conjecture implies $N^{\mr{flop}}_{p,-}=v^{-2}\cdot\overline{N}_{p,-}$, where $N^{\mr{flop}}_{p,-}$ is the repelling part of $T_{p}X_{\mr{flop}}$. This is compatible with the relation $\wt_{\bb{S}}\det N^{\mr{flop}}_{p,-}=-\wt_{\bb{S}}\det N_{p,-}-\dim X$ which comes from (\ref{eqn_S_wt}). We note that from our definition, the relation $\det N^{\mr{flop}}_{p,-}=v^{-\dim X}\cdot\det\overline{N}_{p,-}$ always holds without assuming Conjecture~\ref{conj_flop_tangent}. 

Let $\widetilde{\mca{M}}_{X}$ be the field of multivalued meromorphic functions on $\mb{B}^{K}_{X}$. We set $\mb{S}_{X}\coloneqq(\mb{S}_{X,p,p'})_{p,p'\in X^{H}}$. This is a matrix whose entries are elements in $\widetilde{\mca{M}}_{X}$. By the triangular property of elliptic stable bases, $\mb{S}_{X}$ is invertible. As in $K$-theory, the inverse for $\bb{S}$ induces an involution $\overline{(-)}:\widetilde{\mca{M}}_{X}\rightarrow\widetilde{\mca{M}}_{X}$. We conjeture the following formula expressing $\overline{\mb{S}}_{X}$ geometrically. 

\begin{conj}\label{conj_ell_flop}
$\overline{\mb{S}}_{X}=(-1)^{\frac{\dim X}{2}+\frac{\dim X^{!}}{2}}\cdot\mb{S}_{-X_{\mr{flop}}}$. 
\end{conj}

For toric hyper-K\"ahler manifolds, this conjecture is proved in Corollary~\ref{cor_ell_toric_flop}. If we assume Conjecture~\ref{conj_flop_tangent}, then we have $\overline{N}_{p,-}=v^{2}\cdot N^{\mr{flop}}_{p,-}=(N^{\mr{flop}}_{p,+})^{\vee}$ and hence for the diagonal entries, we obtain 
\begin{align*}
\overline{\mb{S}}_{X,p,p}=\vartheta(\overline{N}_{p,-})\vartheta(\overline{N}^{!}_{p^{!},-})=(-1)^{\frac{\dim X}{2}+\frac{\dim X^{!}}{2}}\cdot\vartheta(N^{\mr{flop}}_{p,+})\vartheta(N^{!,\mr{flop}}_{p^{!},+})=(-1)^{\frac{\dim X}{2}+\frac{\dim X^{!}}{2}}\cdot\mb{S}_{-X_{\mr{flop}},p,p}.
\end{align*}
As another evidence for Conjecture~\ref{conj_ell_flop}, one can also check that each corresponding entry has the same factors of automorphy. 

\begin{prop}
Assume Conjecture~\ref{conj_flop_tangent}. For any $p,p'\in X^{H}$, both $\overline{\mb{S}}_{X,p',p}$ and $\mb{S}_{-X_{\mr{flop}},p',p}$ can be considered as a section of the same line bundle on $\mb{B}_{X}\cong\mb{B}_{-X_{\mr{flop}}}$. 
\end{prop}

\begin{proof}
As above, one can check $\overline{\theta_{X,p',p}(\gamma)}=\theta_{-X_{\mr{flop}},p',p}(\gamma)$ for any $\gamma\in\bb{X}_{\ast}(H)\times P$ and $\overline{\theta_{X,p',p}(-\delta)}=\theta_{-X_{\mr{flop}},p',p}(\delta)$ for $\delta\in\bb{X}_{\ast}(\bb{S})$. By Proposition~\ref{prop_factor_ell_std}, the coincidence of factors of automorphy under $a\mapsto q^{c}a$ and $z\mapsto q^{l}z$ follows from $i^{\ast}_{p}\mf{L}^{!}_{\mr{flop}}(c)=\overline{i^{\ast}_{p}\mf{L}^{!}(c)}$ and $i^{\ast}_{p}\mf{L}_{\mr{flop}}(l)=\overline{i^{\ast}_{p}\mf{L}(l)}$. 

For $v\mapsto qv$, we note that by (\ref{eqn_S_wt}) and (\ref{factor_v}), we have
\begin{align*}
\overline{\mb{S}}_{X,p',p}(v\mapsto qv)=\overline{\mb{S}_{X,p',p}(v\mapsto q^{-1}v)}=\frac{\det \overline{N}_{p',-}}{\det \overline{N}_{p,-}}\frac{\det \overline{N}^{!}_{p^{!},-}}{\det \overline{N}^{!}_{p'^{!},-}}\cdot q^{\wt_{\bb{S}}\frac{\det N_{p,-}}{\det N_{p',-}}}\cdot\overline{\theta_{X,p',p}(-\delta)}\cdot\overline{\mb{S}}_{X,p',p}.\end{align*}
Therefore, the coincidence of the factor of automorphy under $v\mapsto qv$ together with the coincidence of monodromies follows from $\overline{N}_{p,-}=(N^{\mr{flop}}_{p,+})^{\vee}$. 
\end{proof}

We now state the main conjecture in this section. This conjecture together with Conjecture~\ref{conj_ell_flop} will imply that the elliptic bar involution defined in the next section is actually an involution. 

\begin{conj}\label{conj_ell_indep_chamber}
The matrix $\mb{M}_{X}\coloneqq\mb{S}_{X_{\mr{flop}}}\cdot\mb{S}_{X}^{-1}$ does not depend on the choice of the chamber $\mf{C}$. 
\end{conj}

For toric hyper-K\"ahler manifolds, this conjecture is proved in Corollary~\ref{cor_ell_toric_indep_flop}. As another evidence for this conjecture, we prove that each entry of $\mb{M}_{X}$ can be considered as a section of some line bundle on $\mb{B}_{X}$ which does not depend on the choice of $\mf{C}$. We write $\mb{U}_{p,p'}$ and $\mb{S}^{\mr{flop}}_{p,p'}$ the $(p,p')$-entry of $\mb{S}_{X}^{-1}$ and $\mb{S}_{X_{\mr{flop}}}$ respectively. We first calculate the factors of automorphy for $\mb{U}_{p,p'}$. 

\begin{lemma}\label{lem_factor_inverse}
$\mb{U}_{p,p'}$ satisfies the following:
\begin{itemize}
\item For each $c\in\bb{X}_{\ast}(H)$, we have 
\begin{align*}
\mb{U}_{p,p'}(a\mapsto q^{c}a)=i_{p^{!}}^{\ast}\mf{L}^{!}(c)\cdot i_{p'^{!}}^{\ast}\mf{L}^{!}(c)^{-1}\cdot\theta_{p',p}(c)^{-1}\cdot\mb{U}_{p,p'};
\end{align*}
\item For each $l\in P$, we have 
\begin{align*}\mb{U}_{p,p'}(z\mapsto q^{l}z)=i_{p}^{\ast}\mf{L}(l)^{-1}\cdot i_{p'}^{\ast}\mf{L}(l)\cdot\theta_{p',p}(l)^{-1}\cdot\mb{U}_{p,p'};
\end{align*}
\item For $\delta\in\bb{X}_{\ast}(\bb{S})$ satisfying $\langle\delta,v\rangle=1$, we have 
\begin{align*}
\mb{U}_{p,p'}(v\mapsto qv)=\frac{\det N_{p',-}}{\det N_{p,-}}\frac{\det N^{!}_{p^{!},-}}{\det N^{!}_{p'^{!},-}}\cdot q^{\wt_{\bb{S}}\frac{\det N_{p',-}}{\det N_{p,-}}}\cdot\theta_{p',p}(\delta)^{-1}\cdot\mb{U}_{p,p'};
\end{align*}
\item $\sqrt{\det N_{p',-}\cdot\det N^{!}_{p^{!},-}}\cdot\mb{U}_{p,p'}$ is single valued on $\mb{B}^{K}_{X}$.
\end{itemize}
\end{lemma}

\begin{proof}
Since $\mb{S}_{p,p'}=\mb{U}_{p,p'}=0$ unless $p\preceq_{\mf{C}}p'$, we prove the statements by induction on the number of $p''\in X^{H}$ satisfying $p\preceq_{\mf{C}}p''\preceq_{\mf{C}}p'$. We note that by Proposition~\ref{prop_factor_ell_std}, the factor of automorphy of $\mb{S}_{p',p}$ under the translation by $q^{\gamma}$ is of the form $f_{p}(\gamma)f_{p'}(\gamma)^{-1}\theta_{p',p}(\gamma)$ for any $\gamma\in\bb{X}_{\ast}(T)\times P$. It is enough to check that the factor of automorphy for $\mb{U}_{p,p'}$ is given by $f_{p}(\gamma)^{-1}f_{p'}(\gamma)\theta_{p',p}(\gamma)^{-1}$. If $p=p'$, then we have $\mb{U}_{p,p}=(\mb{S}_{p,p})^{-1}$ and the claim follows immediately. 

If $p\prec_{\mf{C}}p'$, then we have 
\begin{align*}
\mb{U}_{p,p'}=-(\mb{S}_{p',p'})^{-1}\cdot\sum_{p\preceq_{\mf{C}}p''\prec_{\mf{C}}p'}\mb{U}_{p,p''}\cdot\mb{S}_{p'',p'}. 
\end{align*}
By the induction hypothesis, the factor of automorphy of each term in the RHS is given by 
\begin{align*}
\theta_{p',p'}(\gamma)^{-1}\cdot f_{p}(\gamma)^{-1}f_{p''}(\gamma)\theta_{p'',p}(\gamma)^{-1}\cdot f_{p'}(\gamma)f_{p''}(\gamma)^{-1}\theta_{p'',p'}(\gamma)=f_{p}(\gamma)^{-1}f_{p'}(\gamma)\theta_{p',p}(\gamma)^{-1}.
\end{align*}
This proves the first three statements. The fourth statement can be proved similarly. 
\end{proof}

\begin{prop}
Each entry of $\mb{M}_{X}$ is a section of a line bundle on $\mb{B}_{X}$ which does not depend on the choice of $\mf{C}$.
\end{prop}

\begin{proof}
By Proposition~\ref{prop_factor_ell_std} and Lemma~\ref{lem_factor_inverse}, the factors of automorphy of $\mb{S}^{\mr{flop}}_{p,p''}\cdot\mb{U}_{p'',p'}$ for $p\succeq_{\mf{C}}p''\preceq_{\mf{C}}p'$ are given as follows:
\begin{itemize}
\item For $a\mapsto q^{c}a$, it is given by
\begin{align}\label{eqn_mon_factor_equiv}
\frac{i^{\ast}_{p^{!}}\mf{L}^{!}(c)\cdot\theta_{X_{\mr{flop}},p}(c)}{i^{\ast}_{p'^{!}}\mf{L}^{!}(c)\cdot\theta_{X,p'}(c)};
\end{align}
\item For $z\mapsto q^{l}z$, it is given by 
\begin{align}\label{eqn_mon_factor_Kahler}
\frac{i^{\ast}_{p'}\mf{L}(l)}{i^{\ast}_{p}\mf{L}_{\mr{flop}}(l)}\cdot\frac{i^{\ast}_{p''}\mf{L}_{\mr{flop}}(l)}{i^{\ast}_{p''}\mf{L}(l)}\cdot\frac{\theta_{-X^{!},p''^{!}}(l)}{\theta_{X^{!},p''^{!}}(l)};
\end{align}
\item For $v\mapsto qv$, it is given by 
\begin{align}\label{eqn_mon_factor_v}
\frac{\det N_{p',-}}{\det N^{\mr{flop}}_{p,-}}\cdot\frac{\det N^{!}_{p^{!},+}}{\det N^{!}_{p'^{!},-}}\cdot q^{\wt_{\bb{S}}\frac{\det N_{p',-}}{\det N^{\mr{flop}}_{p,-}}}\cdot\frac{\theta_{X_{\mr{flop}},p}(\delta)}{\theta_{X,p'}(\delta)}\cdot\frac{\det N^{\mr{flop}}_{p'',-}}{\det N_{p'',-}}\cdot\frac{\det N^{!}_{p''^{!},-}}{\det N^{!}_{p''^{!},+}}\cdot q^{\wt_{\bb{S}}\frac{\det N^{\mr{flop}}_{p'',-}}{\det N_{p'',-}}}\cdot\frac{\theta_{-X^{!},p''^{!}}(\delta)}{\theta_{X^{!},p''^{!}}(\delta)}. 
\end{align}
\end{itemize}

We first consider the case of $a\mapsto q^{c}a$. Since (\ref{eqn_mon_factor_equiv}) does not depend on $p''$, $\mb{M}_{X,p,p'}=\sum_{p''}\mb{S}^{\mr{flop}}_{p,p''}\cdot\mb{U}_{p'',p'}$ has (\ref{eqn_mon_factor_equiv}) as the factor of automorphy. In order to check the independence on $\mf{C}$, it is sufficient to check that $i^{\ast}_{p^{!}}\mf{L}^{!}(c)\cdot\theta_{X,p}(c)$ does not depend on the choice of $\mf{C}$. By (\ref{eqn_line_bundle_dual}), the $H$-weight of $i^{\ast}_{p^{!}}\mf{L}^{!}(c)$ does not depend on the choice of $\mf{C}$ and the $\bb{S}$-weight is given by $-\langle\det N_{p,-},c\rangle$. If we write $N_{p,-}=\sum_{i}w_{i}v^{m_{i}}$ for some $w_{i}\in\bb{X}^{\ast}(H)$ and $m_{i}\in\bb{Z}$, then it is enough to check the independence of 
\begin{align*}
v^{-\langle\det N_{p,-},c\rangle}\cdot\theta_{X,p}(c)=\prod_{i}(-1)^{\langle w_{i},c\rangle}q^{-\frac{\langle w_{i},c\rangle^{2}}{2}}w_{i}^{-\langle w_{i},c\rangle}v^{-(m_{i}+1)\langle w_{i},c\rangle}
\end{align*} 
on the choice of $\mf{C}$. If we change $\mf{C}$, then for some $i$, $w_{i}$ is replaced by $w_{i}^{-1}$ and $m_{i}$ is replaced by $-m_{i}-2$ in the formula of $N_{p,-}$. Since this replacement does not change each factor of the RHS, the independence follows. 

We next consider the case of $z\mapsto q^{l}z$. By (\ref{eqn_line_bundle}), we have 
\begin{align*}
\frac{i^{\ast}_{p''}\mf{L}_{\mr{flop}}(l)}{i^{\ast}_{p''}\mf{L}(l)}\cdot\frac{\theta_{-X^{!},p''^{!}}(l)}{\theta_{X^{!},p''^{!}}(l)}=v^{2\langle\det N^{!}_{p''^{!},-},l\rangle}\cdot v^{-2\langle\det N^{!}_{p''^{!},-},l\rangle}=1. 
\end{align*}
Therefore, (\ref{eqn_mon_factor_Kahler}) does not depend on $p''$ and hence it is the factor of automorphy of $\mb{M}_{X,p,p'}$. Its independence on $\mf{C}$ is clear. 

Now we consider the case of $v\mapsto qv$. We first check that (\ref{eqn_mon_factor_v}) does not depend on $p''$. One can check that 
\begin{align*}
\frac{\theta_{-X^{!},p''^{!}}(\delta)}{\theta_{X^{!},p''^{!}}(\delta)}=(qv)^{-2\wt_{\bb{S}}\det N^{!}_{p''^{!},-}-\dim X^{!}}\cdot v^{-\dim X^{!}}\cdot(\det N^{!}_{p''^{!},-})^{-2}.
\end{align*}
On the other hand, we have
\begin{align*}
\frac{\det N^{\mr{flop}}_{p'',-}}{\det N_{p'',-}}=v^{-\dim X-2\wt_{\bb{S}}\det N_{p'',-}}=v^{\dim X^{!}+2\wt_{\bb{S}}\det N^{!}_{p''^{!},-}}
\end{align*}
by (\ref{eqn_S_wt}) and 
\begin{align*}
\frac{\det N^{!}_{p''^{!},-}}{\det N^{!}_{p''^{!},+}}=v^{\dim X^{!}}\cdot(\det N^{!}_{p''^{!},-})^{2}.
\end{align*}
These equations imply that (\ref{eqn_mon_factor_v}) does not depend on $p''$. In order to prove the independence of  (\ref{eqn_mon_factor_v}) on the choice of $\mf{C}$, it is enough to check that 
\begin{align*}
\frac{\det N^{!}_{p^{!},-}}{\det N_{p,-}}\cdot q^{-\wt_{\bb{S}}\det N_{p,-}}\cdot\theta_{X,p}(\delta)
\end{align*}
does not depend on $\mf{C}$. We note that the $H^{!}$-weight of $\det N^{!}_{p^{!},-}$ does not depend on $\mf{C}$ and the $\bb{S}$-weight is $-\frac{\dim X^{!}}{2}-\frac{\dim X}{2}-\wt_{\bb{S}}\det N_{p,-}$ by (\ref{eqn_S_wt}). If we write $N_{p,-}=\sum_{i}w_{i}v^{m_{i}}$ for some $w_{i}\in\bb{X}^{\ast}(H)$ and $m_{i}\in\bb{Z}$, then we have 
\begin{align*}
\frac{v^{\wt_{\bb{S}}\det N^{!}_{p^{!},-}}}{\det N_{p,-}}\cdot q^{-\wt_{\bb{S}}\det N_{p,-}}\cdot\theta_{X,p}(\delta)=v^{-\frac{\dim X^{!}}{2}-\frac{\dim X}{2}}\cdot\prod_{i}(-1)^{m_{i}}(qv^{2})^{-\frac{m_{i}(m_{i}+2)}{2}}w_{i}^{-m_{i}-1}.
\end{align*}
Since each factor in the RHS does not change under $w_{i}\mapsto w_{i}^{-1}$ and $m_{i}\mapsto -m_{i}-2$, the independence on $\mf{C}$ follows. 

Finally, independence of monodromies of $\sqrt{\det N_{p,-}}$ on $\mf{C}$ easily implies the independence of monodromies of $\mb{M}_{X,p,p'}$. 
\end{proof}

\begin{remark}
We note that the matrix $\mb{M}_{X}$ is closely related to the monodromy operator appearing in \cite[Proposition~6.5]{AO} which intertwines the vertex function for $X$ and $X_{\mr{flop}}$. We expect that the dependence on $\mf{C}$ in \cite[Proposition~6.5]{AO} is eliminated by our choice of the normalization on the elliptic standard bases. Since the vertex functions do not depend on the choice of $\mf{C}$, one may hope that these kind of relations would imply Conjecture~\ref{conj_ell_indep_chamber}. We plan to investigate this approach in the future. 
\end{remark}

\subsection{Elliptic bar involutions}

In this section, we give a proposal for a definition of elliptic bar involution. Since our approach gives involution only after localization, we set $\mb{K}(X)_{\mr{loc}}\coloneqq\oplus_{p\in X^{H}}\mca{M}_{X}$, where $\mca{M}_{X}$ is the field of (single-valued) meromorphic functions on $\mb{B}^{K}_{X}$, and we will only work on $\mb{K}(X)_{\mr{loc}}$ in this paper. By restriction to the fixed points, we obtain a natural map $K_{\bb{T}}(X)\otimes_{\bb{Z}}\bb{Z}[P^{\vee}]\rightarrow\mb{K}(X)_{\mr{loc}}$. In order to consider $\Stab^{ell}_{X}(p)$ as an element of $\mb{K}(X)_{\mr{loc}}$ similarly, we need to kill the multi-valuedness. We use the data of polarization to fix this modification. 

We choose a splitting $\Pic^{H}(X)\cong P\oplus\bb{X}^{\ast}(H)$ by using the data $\mf{L}:P\rightarrow\Pic^{\bb{T}}(X)$ and take a polarization $T^{1/2}$ satisfying Assumption~\ref{assumption_polarization}. Let $\kappa\in P\oplus\bb{X}^{\ast}(H)$ be the element corresponding to $\det T^{1/2}\in\Pic^{H}(X)$. We also denote by $\mf{L}:P\oplus\bb{X}^{\ast}(H)\rightarrow\Pic^{\bb{T}}(X)$ the natural extension of $\mf{L}$. We note that $\mf{L}(\kappa)=v^{-w(\det T^{1/2})}\cdot\det T^{1/2}$. 
We also take similar data $T^{1/2,!}$ and $\kappa^{!}\in P^{!}\oplus\bb{X}^{\ast}(H^{!})$ on the dual conical symplectic resolution $X^{!}$. 

\begin{dfn}
For any $p\in X^{H}$, we set 
\begin{align*}
\mf{S}^{\kappa,\kappa^{!}}_{X}(p)\coloneqq\sqrt{\mf{L}(\kappa)\cdot i^{\ast}_{p^{!}}\mf{L}^{!}(\kappa^{!})}^{-1}\cdot\Stab^{ell}_{X}(p).
\end{align*}
\end{dfn}

\begin{lemma}
We have $\left(i^{\ast}_{p'}\mf{S}^{\kappa,\kappa^{!}}_{X}(p)\right)_{p'\in X^{H}}\in\mb{K}(X)_{\mr{loc}}$. 
\end{lemma}

\begin{proof}
By Assumption~\ref{assumption_polarization} for $T^{1/2}$ and $T^{1/2,!}$, we obtain $w(\det T^{1/2})\equiv w(\det T^{1/2,!})\mod 2$. Since we have $\sqrt{i^{\ast}_{p'}\mf{L}(\kappa)\cdot i^{\ast}_{p^{!}}\mf{L}^{!}(\kappa^{!})}=v^{-\frac{w(\det T^{1/2})}{2}-\frac{w(\det T^{1/2,!})}{2}}\sqrt{\det T^{1/2}_{p'}\cdot\det T^{1/2,!}_{p^{!}}}$, the statement follows from Proposition~\ref{prop_factor_ell_std}. 
\end{proof}

We will identify $\mf{S}^{\kappa,\kappa^{!}}_{X}(p)$ and $\left(i^{\ast}_{p'}\mf{S}^{\kappa,\kappa^{!}}_{X}(p)\right)_{p'\in X^{H}}\in\mb{K}(X)_{\mr{loc}}$. By the triangular property of the elliptic stable bases, $\{\mf{S}^{\kappa,\kappa^{!}}_{X}(p)\}_{p\in X^{H}}$ forms a basis of $\mb{K}(X)_{\mr{loc}}$ over $\mca{M}_{X}$. 

\begin{dfn}\label{dfn_elliptic_bar_inv}
We define the $\mca{M}_{X}$-semilinear map $\beta^{ell}_{X}=\beta^{ell}_{X,\kappa}:\mb{K}(X)_{\mr{loc}}\rightarrow\mb{K}(X)_{\mr{loc}}$ by 
\begin{align}\label{eqn_elliptic_bar_inv}
\beta^{ell}_{X}(\mf{S}^{\kappa,\kappa^{!}}_{X}(p))=(-1)^{\frac{\dim X}{2}}\mf{S}^{\kappa,\kappa^{!}}_{-X}(p) 
\end{align}
for any $p\in X^{H}$. Here, $\mca{M}_{X}$-semilinear means $\beta^{ell}_{X}(f\cdot m)=\bar{f}\cdot\beta^{ell}_{X}(m)$ for any $f\in\mca{M}_{X}$ and $m\in\mb{K}(X)_{\mr{loc}}$. We also identified $P^{!}_{\mr{flop}}\oplus\bb{X}^{\ast}(H^{!}_{\mr{flop}})$ and $P^{!}\oplus\bb{X}^{\ast}(H)$ in order to take $\kappa^{!}$ for $X^{!}_{\mr{flop}}$ in the RHS. 
\end{dfn} 

We note that $\beta^{ell}_{X}$ does not depend on the choice of $\kappa^{!}$ because of the relation $\mf{L}^{!}_{\mr{flop}}(\kappa^{!})=\overline{\mf{L}^{!}(\kappa^{!})}$. We also conjecture that $\beta^{ell}_{X}$ does not depend on the data $\mf{C}$. As an evidence for this conjecture, we check it by assuming Conjecture~\ref{conj_ell_flop} and Conjecture~\ref{conj_ell_indep_chamber}. 

\begin{prop}\label{prop_ell_indep_chamber}
Assume Conjecture~\ref{conj_ell_flop} and Conjecture~\ref{conj_ell_indep_chamber}. The map $\beta^{ell}_{X}$ does not depend on the choice of $\mf{C}$. 
\end{prop}

\begin{proof}
In this proof, we denote by $\beta_{\mf{C}}\coloneqq\beta^{ell}_{X}$. We take another chamber $\mf{C}'$ and define $\beta_{\mf{C}'}$ similarly by using this chamber. We also set $\mb{S}_{\mf{C}}\coloneqq\mb{S}_{X}$ and define $\mb{S}_{\mf{C}'}$ similarly by using $\mf{C}'$ instead of $\mf{C}$. Let us define the matrix $\mf{S}_{\mf{C}}\coloneqq(i^{\ast}_{p}\mf{S}^{\kappa,\kappa^{!}}_{X}(p'))_{p,p'\in X^{H}}$ and $\mf{S}_{\mf{C}'}$ similarly. If we identify $\mf{L}(\kappa)$ and the diagonal matrix $\mr{diag}((i^{\ast}_{p}\mf{L}(\kappa))_{p\in X^{H}})$, then we have 
\begin{align*}
\mf{S}_{\mf{C}}=\mf{L}(\kappa)^{-\frac{1}{2}}\cdot\mb{S}_{\mf{C}}\cdot\mf{L}^{!}_{\mf{C}}(\kappa^{!})^{-\frac{1}{2}}. 
\end{align*}
Here, $\mf{L}^{!}_{\mf{C}}(\kappa^{!})$ is the diagonal matrix defined similarly as $\mf{L}(\kappa)$ by using $\mf{L}^{!}(\kappa)$,  but since it depends on the choice of $\mf{C}$, we put the index to indicate the dependence. We note that $\mf{L}^{!}_{-\mf{C}}(\kappa^{!})=\overline{\mf{L}^{!}_{\mf{C}}(\kappa^{!})}$. By definition, we have 
\begin{align*}
\beta_{\mf{C}}(\mf{S}_{\mf{C}})=(-1)^{\frac{\dim X}{2}}\cdot\mf{S}_{-\mf{C}},
\end{align*}
where we understand that $\beta_{\mf{C}}$ is applied column by column. Since 
\begin{align*}
\beta_{\mf{C}'}(\mf{S}_{\mf{C}})=\beta_{\mf{C}'}(\mf{S}_{\mf{C}'}\cdot\mf{S}_{\mf{C}'}^{-1}\mf{S}_{\mf{C}})=(-1)^{\frac{\dim X}{2}}\cdot\mf{S}_{-\mf{C}'}\cdot\overline{\mf{S}}_{\mf{C}'}^{-1}\overline{\mf{S}}_{\mf{C}}, 
\end{align*}
the statement is equivalent to 
\begin{align*}
\overline{\mf{S}}_{\mf{C}}\cdot\mf{S}_{-\mf{C}}^{-1}=\overline{\mf{S}}_{\mf{C}'}\cdot\mf{S}_{-\mf{C}'}^{-1}.
\end{align*}
By (\ref{eqn_line_bundle_dual}) and Conjecture~\ref{conj_ell_flop}, we have 
\begin{align*}
\overline{\mf{S}}_{\mf{C}}\cdot\mf{S}_{-\mf{C}}^{-1}&=\overline{\mf{L}(\kappa)}^{-\frac{1}{2}}\cdot\overline{\mb{S}}_{\mf{C}}\cdot\overline{\mf{L}^{!}_{\mf{C}}(\kappa^{!})}^{-\frac{1}{2}}\cdot\mf{L}^{!}_{-\mf{C}}(\kappa^{!})^{\frac{1}{2}}\cdot\mb{S}_{-\mf{C}}^{-1}\cdot\mf{L}(\kappa)^{\frac{1}{2}}\\
&=(-1)^{\frac{\dim X}{2}+\frac{\dim X^{!}}{2}}\cdot\mf{L}_{\mr{flop}}(\kappa)^{-\frac{1}{2}}\cdot\mb{S}^{\mr{flop}}_{-\mf{C}}\cdot\mb{S}_{-\mf{C}}^{-1}\cdot\mf{L}(\kappa)^{\frac{1}{2}},
\end{align*}
where we set $\mb{S}^{\mr{flop}}_{-\mf{C}}=\mb{S}_{-X_{\mr{flop}}}$. Therefore, the statement follows from Conjecture~\ref{conj_ell_indep_chamber}. 
\end{proof}

We now explain the relation between elliptic and $K$-theoretic bar involutions. Take a generic slope $s\in P_{\bb{R}}$. By \cite[Proposition~3.6]{AO}, elliptic stable bases and $K$-theoretic stable bases are related by certain limit under $q\rightarrow0$. More precisely, we have 
\begin{align}\label{eqn_K_limit}
\lim_{q\rightarrow0}\left(\sqrt{\det T^{1/2}}^{-1}\cdot\Stab^{AO}_{\mf{C},T^{1/2}}(p)|_{z=q^{-s}}\right)=\Stab^{K}_{\mf{C},T^{1/2},s}(p),
\end{align}
where $z=q^{-s}$ means we specialize the K\"ahler parameter at $q^{-s}\in P\otimes_{\bb{Z}}\bb{C}^{\times}$. We note that (\ref{eqn_elliptic_bar_inv}) is equivalent to 
\begin{align*}
\beta^{ell}_{X}\left(\sqrt{\det T^{1/2}}^{-1}\cdot\Stab^{AO}_{\mf{C},T^{1/2}}(p)'\right)=(-1)^{\frac{\dim X}{2}}v^{w(\det T^{1/2})}\cdot\frac{\vartheta(N^{!,\mr{flop}}_{p^{!},-})}{\vartheta(\overline{N}^{!}_{p^{!},-})}\cdot\sqrt{\det T^{1/2}}^{-1}\cdot\Stab^{AO}_{-\mf{C},T^{1/2}}(p)',
\end{align*}
where we set $\Stab^{AO}_{\mf{C},T^{1/2}}(p)'\coloneqq\tau(\det T^{1/2},v)^{\ast}\Stab^{AO}_{\mf{C},T^{1/2}}(p)$. By using Conjecture~\ref{conj_ell_flop} and 
\begin{align*}
\lim_{q\rightarrow0}\frac{\vartheta(q^{\alpha}x)}{\vartheta(q^{\alpha})}=x^{-\lfloor\alpha\rfloor-\frac{1}{2}}
\end{align*}
for $\alpha\in\bb{R}\setminus\bb{Z}$, we obtain 
\begin{align*}
\lim_{q\rightarrow0}\left(\frac{\vartheta(N^{!,\mr{flop}}_{p^{!},-})|_{z=q^{-s}}}{\vartheta(\overline{N}^{!}_{p^{!},-})|_{z=q^{-s}}}\right)=\prod_{\beta\in\Psi_{+}(p)}v^{2\lfloor\langle s,\beta\rangle\rfloor+1}. 
\end{align*}
Since the limit in (\ref{eqn_K_limit}) does not change if we replace $\Stab^{AO}_{\mf{C},T^{1/2}}(p)$ by $\Stab^{AO}_{\mf{C},T^{1/2}}(p)'$, the $K$-theory limit $\beta^{K}_{T^{1/2},s}$ of $\beta^{ell}_{X}$ satisfies 
\begin{align*}
\beta^{K}_{T^{1/2},s}(\Stab^{K}_{\mf{C},T^{1/2},s}(p))=(-v)^{\frac{\dim X}{2}}\cdot v^{2a_{p}(T^{1/2},s)}\cdot\Stab^{K}_{-\mf{C},T^{1/2},s}(p). 
\end{align*} 
This is equivalent to our definition of $K$-theoretic bar involution and explains seemingly ad hoc normalizations in our definition of the $K$-theoretic standard bases. 

\section{Toric hyper-K\"ahler manifolds}

As an evidence for the main conjectures, we check all the conjectures stated in section 3 for the toric hyper-K\"ahler manifolds in this section. The conjectures stated in section 4 will be proved in the next section. 

\subsection{Preliminaries}

In this section, we briefly recall basic facts about toric hyper-K\"ahler manifolds introduced by Bielawski-Dancer \cite{BD}. For more detail, see for example \cite{BD,HS,Nag,P}. We first prepare some notations used in the following sections. We fix an integer $n\in\bb{Z}_{\geq1}$ and set $T\coloneqq(\bb{C}^{\times})^{n}$. We consider an exact sequence of algebraic tori of the form 
\begin{align}\label{eqn_exact_seq_tori}
1\rightarrow S\rightarrow T\rightarrow H\rightarrow 1,
\end{align}  
where $S\cong(\bb{C}^{\times})^{r}$ and $H\cong(\bb{C}^{\times})^{d}$ for some $r,d\in\bb{Z}_{\geq0}$ with $r+d=n$. Let  
\begin{align}\label{eqn_exact_seq_cocharacter}
0\rightarrow\bb{X}_{\ast}(S)\xrightarrow{^{t}\!\mf{b}}\bb{X}_{\ast}(T)\xrightarrow{\mf{a}}\bb{X}_{\ast}(H)\rightarrow 0
\end{align}
be the associated exact sequence. We fix an identification $\bb{X}_{\ast}(T)\cong\bb{Z}^{n}$ and take the standard basis $\{\varepsilon_{i}\}_{i=1}^{n}$. We set $\mf{a}_{i}\coloneqq \mf{a}(\varepsilon_{i})$ and assume that $\mf{a}_{i}\neq0$ for any $i=1,\ldots,n$. We also assume that $\mf{a}$ is unimodular, i.e., if $\{\mf{a}_{i_{1}},\ldots,\mf{a}_{i_{d}}\}$ is linearly independent, then they always generate $\bb{X}_{\ast}(H)$ over $\bb{Z}$. We set 
\begin{align*}
\bb{B}\coloneqq\left\{I\subset\{1,\ldots,n\}\mid\{\mf{a}_{i}\}_{i\in I}\mbox{ is a basis of }\bb{X}_{\ast}(H)\right\}.	
\end{align*}
We also consider the dual of the above exact sequence 
\begin{align}\label{eqn_exact_seq_character}
0\rightarrow\bb{X}^{\ast}(H)\xrightarrow{^{t}\!\mf{a}}\bb{X}^{\ast}(T)\xrightarrow{\mf{b}}\bb{X}^{\ast}(S)\rightarrow 0.
\end{align}
Let $\{\varepsilon_{i}^{\ast}\}_{i=1}^{n}\subset\bb{X}^{\ast}(T)$ be the dual basis of $\{\varepsilon_{i}\}_{i=1}^{n}$ and set $\mf{b}_{i}\coloneqq \mf{b}(\varepsilon_{i}^{\ast})$. We also assume that $\mf{b}_{i}\neq0$ for any $i=1,\ldots,n$. We note that the unimodularity of $\mf{a}$ is equivalent to the unimodularity of $\mf{b}$ and $\{\mf{b}_{j}\}_{j\in J}$ is a basis of $\bb{X}^{\ast}(S)$ if and only if $J^{c}\coloneqq\{1,\ldots,n\}\setminus J\in\bb{B}$. 

A subset equipped with a decomposition $C=C_{+}\sqcup C_{-}\subset\left\{1,\ldots, n\right\}$ is called signed circuit if $\left\{\mf{a}_{i}\right\}_{i\in C}$ is a minimal linearly dependent subset of $\left\{\mf{a}_{1},\ldots, \mf{a}_{n}\right\}$ and $\sum_{i\in C_{+}}\mf{a}_{i}-\sum_{i\in C_{-}}\mf{a}_{i}=0$. We note that by the unimodularity assumption, any minimal linear relations between $\mf{a}_{i}$'s can be written in this form up to scalar multiplication. For a signed circuit $C$, we denote by $\beta_{C}=(\beta_{1},\ldots, \beta_{n})\in\bb{X}_{\ast}(T)$ the element defined by 
\begin{align*}
\beta_{i}=\begin{cases}0\mbox{ if }i\notin C\\1\mbox{ if }i\in C_{+}\\-1\mbox{ if }i\in C_{-}.\end{cases}
\end{align*}
By definition, we have $\beta_{C}\in\Ker(\mf{a})=\bb{X}_{\ast}(S)$. We similarly define the notion of signed cocircuit using $\mf{b}_{i}$ instead of $\mf{a}_{i}$. For a signed cocircuit $C^{\vee}$, we define $\alpha_{C^{\vee}}\in\bb{X}^{\ast}(H)$ in the same way as $\beta_{C}$.  

We consider the natural $T$-action on $T^{\ast}\bb{C}^{n}$ given by 
\begin{align*}
(t_{1},\ldots,t_{n})\cdot (x_{1},\ldots,x_{n},y_{1},\ldots,y_{n})=(t_{1}x_{1},\ldots,t_{n}x_{n},t_{1}^{-1}y_{1},\ldots,t_{n}^{-1}y_{n}),
\end{align*}
where $(t_{1},\ldots,t_{n})\in T$, $(x_{1},\ldots,x_{n})$ is a point in $\bb{C}^{n}$, and $(y_{1},\ldots,y_{n})$ is a point in the cotangent fiber. This action is Hamiltonian with respect to the standard symplectic structure on $T^{\ast}\bb{C}^{n}$ and its moment map is given by $\mu_{n}(x,y)=\sum_{i=1}^{n}x_{i}y_{i}$. Let $\mu(x,y)\coloneqq\sum_{i=1}^{n}x_{i}y_{i}\mf{b}_{i}\in\mf{s}^{\ast}$ be the moment map for the $S$-action given by restriction, where $\mf{s}$ is the Lie algebra of $S$. We note that in the coordinate ring $\bb{C}[T^{\ast}\bb{C}^{n}]=\bb{C}[x_{1},\ldots,x_{n},y_{1},\ldots,y_{n}]$, the $S$-weights of $x_{i}$ and $y_{i}$ are given by $-\mf{b}_{i}$ and $\mf{b}_{i}$ respectively.  

We fix a generic element $\eta\in\bb{X}^{\ast}(S)$, where generic means for any circuit $C$, we have $\langle\eta,\beta_{C}\rangle\neq 0$. A point $(x,y)\in T^{\ast}\bb{C}^{n}$ is called $\eta$-semistable if there exists a positive integer $m$ and a polynomial $f\in\bb{C}[T^{\ast}\bb{C}^{n}]$ such that $f$ has $S$-weight $-m\eta$ and $f(p)\neq0$. Associated with these data, the toric hyper-K\"ahler manifold $X$ is defined by $X\coloneqq\mu^{-1}(0)^{\eta-\mr{ss}}/S$, where the superscript means the subset consisting of $\eta$-semistable points. We also define the Lawrence toric variety by $\mca{X}\coloneqq(T^{\ast}\bb{C}^{n})^{\eta-\mr{ss}}/S$. By the unimodularity assumption, these varieties are smooth and it is known that $\mathcal{X}$ is the universal Poisson deformation of $X$ in the sense of Namikawa, see \cite{Nag}. 

For $I\in\bb{B}$ and $j\in I^{c}$, we denote by $C^{I}_{j}$ the unique signed circuit contained in $I\cup\{j\}$ and $j\in C^{I}_{j,+}$ and we set $\beta^{I}_{j}\coloneqq\beta_{C^{I}_{j}}\in\bb{X}_{\ast}(S)$. We similarly define $\alpha^{I}_{i}\in\bb{X}^{\ast}(H)$ for $I\in\mathbb{B}$ and $i\in I$ using signed cocircuit contained in $I^{c}\cup\{i\}$. We note the following identity for any $i\in I$ and $j\in I^{c}$:
\begin{align}\label{eqn_alpha_beta}
\langle\alpha^{I}_{i},\mf{a}_{j}\rangle=-\langle\beta^{I}_{j},\mf{b}_{i}\rangle.
\end{align}
We decompose $I=I_{+}\sqcup I_{-}$ and $I^{c}=I^{c}_{+}\sqcup I^{c}_{-}$, where we set $I_{\pm}:=\{i\in I\mid \pm\langle\xi,\alpha^{I}_{i}\rangle>0\}$ and $I^{c}_{\pm}:=\{j\in I^{c}\mid \pm\langle\eta,\beta^{I}_{j}\rangle>0\}$. We note that these decompositions depend on the choice of $\xi$ and $\eta$, but we omit the dependence from the notation.  

\begin{lemma}\label{lem_toric_stability}
A point $(x,y)\in T^{\ast}\bb{C}^{n}$ is $\eta$-semistable if and only if there exists $I\in\bb{B}$ such that $x_{j}\neq0$ ($\forall j\in I^{c}_{+}$) and $y_{j}\neq0$ ($\forall j\in I^{c}_{-}$).
\end{lemma}

\begin{proof}
Let $p=(x,y)$ be a $\eta$-semistable point and $f\in\bb{C}[x,y]$ be a polynomial with $S$-weight $-m\eta$ and $f(p)\neq0$ for some $m\in\bb{Z}_{>0}$. We may assume that $f=\prod_{i}x_{i}^{m_{i}}y_{i}^{n_{i}}$ is a monomial. Then we have $\sum_{i}(m_{i}-n_{i})\mf{b}_{i}=m\eta$ and Lemma~\ref{lem_lp_1} implies that there exists $I\in\bb{B}$ such that $\pm(m_{j}-n_{j})>0$ for any $j\in I^{c}_{\pm}$. This implies that $x_{j}\neq0$ for any $j\in I^{c}_{+}$ and $y_{j}\neq0$ for any $j\in I^{c}_{-}$. Conversely, if $p\in T^{\ast}\bb{C}^{n}$ satisfies the latter condition, then $f=\prod_{j\in I^{c}_{+}}x_{j}^{\langle\eta,\beta^{I}_{j}\rangle}\cdot\prod_{j\in I^{c}_{-}}y_{j}^{-\langle\eta,\beta^{I}_{j}\rangle}$ has $S$-weight $-\eta$ and $f(p)\neq0$. 
\end{proof}

\begin{lemma}\label{lem_lp_1}
If we have $\sum_{i}m_{i}\mf{b}_{i}=\eta$ for some $m_{i}\in\bb{R}$, then there exists an $I\in\bb{B}$ such that $\pm m_{j}>0$ for any $j\in I^{c}_{\pm}$. 
\end{lemma}

\begin{proof}
We set $I_{0}\coloneqq\{i\mid m_{i}=0\}$. If $\{\mf{a}_{i}\}_{i\in I_{0}}$ is linearly dependent, then there exists a circuit $C\subset I_{0}$ and we obtain $\langle\eta,\beta_{C}\rangle=0$, which contradicts the genericity of $\eta$. Hence	$\{\mf{a}_{i}\}_{i\in I_{0}}$ is linearly independent. If $I_{0}\in\bb{B}$, then we have $m_{j}=\langle\eta,\beta^{I_{0}}_{j}\rangle$ and the statement follows. If $I_{0}\notin\bb{B}$, then there exists nonzero $\lambda\in\mf{h}_{\bb{R}}^{\ast}$ such that $\langle\lambda,\mf{a}_{i}\rangle=0$ for any $i\in I_{0}$. Since we have $\sum_{i}\langle\lambda,\mf{a}_{i}\rangle\mf{b}_{i}=0$, we may replace $m_{i}$ by $m_{j}^{t}\coloneqq m_{i}+t\langle\lambda,\mf{a}_{i}\rangle$ for sufficiently small $t\in\bb{R}$ without changing the sign of $m_{j}$ for $j\in I_{0}^{c}$. By taking a minimal $t$ such that some $m^{t}_{j}$ becomes 0, we obtain $t'\in\bb{R}$ such that $\pm m^{t'}_{j}\geq0$ for any $j$ with $\pm m_{j}>0$ and $I_{1}\coloneqq\{i\mid m^{t'}_{i}=0\}\supsetneq I_{0}$. By continuing this process, we obtain $I\in\bb{B}$ such that $I\supset I_{0}$ and $\pm m_{j}>0$ for any $j\in I^{c}_{\pm}$. 
\end{proof}

For $I\in\bb{B}$, we set $\widetilde{\mca{U}}_{I}\coloneqq\{(x,y)\in T^{\ast}\bb{C}^{n}\mid x_{j}\neq0\,(\forall j\in I^{c}_{+}),y_{j}\neq0\,(\forall j\in I^{c}_{-})\}$, $\mca{U}_{I}\coloneqq\widetilde{\mca{U}}_{I}/S\subset\mca{X}$, and $U_{I}\coloneqq(\mu^{-1}(0)\cap\widetilde{\mca{U}}_{I})/S\subset X$. By Lemma~\ref{lem_toric_stability}, $\{\mca{U}_{I}\}_{I\in\bb{B}}$ and $\{U_{I}\}_{I\in\bb{B}}$ are Zariski open affine coverings of $\mca{X}$ and $X$ respectively. For $i\in I$, we set 
\begin{align*}
x^{I}_{i}\coloneqq x_{i}\prod_{j\in I^{c}_{+}}x_{j}^{\langle\alpha^{I}_{i},\mf{a}_{j}\rangle}\prod_{j\in I^{c}_{-}}y_{j}^{-\langle\alpha^{I}_{i},\mf{a}_{j}\rangle}, \,y^{I}_{i}\coloneqq y_{i}\prod_{j\in I^{c}_{+}}x_{j}^{-\langle\alpha^{I}_{i},\mf{a}_{j}\rangle}\prod_{j\in I^{c}_{-}}y_{j}^{\langle\alpha^{I}_{i},\mf{a}_{j}\rangle}.
\end{align*}
These are elements of $\bb{C}[\mca{U}_{I}]=\bb{C}[\widetilde{\mca{U}}_{I}]^{S}$. It is easy to see that $\bb{C}[\mca{U}_{I}]=\bb{C}[x_{j}y_{j}\,(j\in I^{c}),x^{I}_{i},y^{I}_{i}\,(i\in I)]$ and $\bb{C}[U_{I}]=\bb{C}[x^{I}_{i},y^{I}_{i}\,(i\in I)]$. In particular, we have $U_{I}\cong\bb{C}^{2d}$ and $\dim X=2d$. Let $\mu_{\mca{X}}:\mca{X}\rightarrow\mf{s}^{\ast}$ be the morphism induced from the moment map $\mu$. This is well-defined since $\mu$ is $S$-invariant. The above description of open coverings gives the following. 

\begin{lemma}\label{lem_toric_deform}
The morphism $\mu_{\mca{X}}$ is flat and $X\cong\mu_{\mca{X}}^{-1}(0)$ as schemes.
\end{lemma}

Since $H$-weights of $x^{I}_{i}$ and $y^{I}_{i}$ (as functions on $U_{I}$) are $-\alpha^{I}_{i}$ and $\alpha^{I}_{i}$ respectively, $U_{I}$ has a unique $H$-fixed point which is denoted by $p_{I}$. In particular, we obtain a one-to-one correspondence between $\bb{B}$ and $X^{H}$ given by $I\mapsto p_{I}$.  

We consider the action of $\bb{S}\coloneqq\bb{C}^{\times}$ on $X$ or $\mca{X}$ induced by $\sigma\cdot (x,y)=(\sigma^{-1}x,\sigma^{-1}y)$, ($\sigma\in\bb{S}$, $(x,y)\in T^{\ast}\bb{C}^{n}$). With this $\bb{S}$-action, it is known that $X$ is a conical symplectic resolution. As in previous sections, we set $\bb{T}:=H\times\bb{S}$. We denote by $L$ the central fiber of $X\rightarrow\Spec(\bb{C}[X])$. 

We regard $\lambda\in\bb{X}^{\ast}(T)$ as a $1$-dimensional representation of $T\times\bb{S}$ with trivial $\mathbb{S}$-action and write the associated $\bb{T}$-equivariant line bundle on the quotients $X$ or $\mca{X}$ as $\mca{L}(\lambda)$ or $\widetilde{\mca{L}}(\lambda)$. We note that if $\lambda\in\bb{X}^{\ast}(H)$, then $\mca{L}(\lambda)$ is a trivial line bundle if we forget $H$-equivariant structure and the $H$-action is given by $\lambda$. This map gives an isomorphism $\Pic(X)\cong\bb{X}^{\ast}(S)$ under our assumption that $\mf{b}_{i}\neq0$. We also fix a splitting $\iota:\bb{X}^{\ast}(S)\rightarrow\bb{X}^{\ast}(T)$ of the natural surjection $\bb{X}^{\ast}(T)\twoheadrightarrow\bb{X}^{\ast}(S)$ and set $\mf{L}(l)\coloneqq\mca{L}(\iota(l))\in\Pic^{\bb{T}}(X)$ for any $l\in\bb{X}^{\ast}(S)$. We note that $\mf{L}(\eta)$ is an ample line bundle relative to the projective morphism $X\rightarrow\Spec(\bb{C}[X])$. For any $\lambda\in\mathbb{X}^{\ast}(H)$, we write $a^{\lambda}\in K_{H}(\mr{pt})$ the $K$-theory class corresponding to $\lambda$. The following result can be checked easily.

\begin{lemma}\label{lem_toric_line_general}
For any $\lambda\in\bb{X}^{\ast}(T)$ and $I\in\mathbb{B}$, we have the following identity in $K_{\bb{T}}(p_{I})$:
\begin{align*}
i^{\ast}_{p_{I}}\mca{L}(\lambda)=v^{\sum_{j\in I^{c}_{+}}\langle\lambda,\beta^{I}_{j}\rangle-\sum_{j\in I^{c}_{-}}\langle\lambda,\beta^{I}_{j}\rangle}\cdot a^{\lambda-\sum_{j\in I^{c}}\langle\lambda,\beta^{I}_{j}\rangle\varepsilon^{\ast}_{j}},
\end{align*}
where $\lambda-\sum_{j\in I^{c}}\langle\lambda,\beta^{I}_{j}\rangle\varepsilon^{\ast}_{j}$ is considered as an element of $\Ker(\mf{b})\cong\mathbb{X}^{\ast}(H)$. 
\end{lemma}

By using (\ref{eqn_alpha_beta}),  we obtain the following corollary of Lemma~\ref{lem_toric_line_general}.

\begin{corollary}\label{cor_toric_line_restr}
For $j\in I^{c}_{\pm}$, we have 
\begin{align*}
i^{\ast}_{p_{I}}\mca{L}(\varepsilon^{\ast}_{j})=v^{\pm 1}.
\end{align*}
For $i\in I$, we have 
\begin{align*}
i^{\ast}_{p_{I}}\mca{L}(\varepsilon^{\ast}_{i})=v^{-\langle\alpha^{I}_{i},\sum_{j\in I^{c}_{+}}\mf{a}_{j}-\sum_{j\in I^{c}_{-}}\mf{a}_{j}\rangle}a^{\alpha^{I}_{i}}.
\end{align*}
\end{corollary}

The equivariant $K$-theory class of the tangent bundle $T_{X}$ of $X$ is given by the following formula.
\begin{align*}
T_{X}=\sum_{i=1}^{n}v^{-1}\mca{L}(\varepsilon^{\ast}_{i})+\sum_{i=1}^{n}v^{-1}\mca{L}(-\varepsilon^{\ast}_{i})-r\cdot\mca{O}_{X}-rv^{-2}\cdot\mca{O}_{X}.
\end{align*}
By using Corollary~\ref{cor_toric_line_restr}, we obtain 
\begin{align}\label{eqn_toric_tangent_fiber}
i^{\ast}_{p_{I}}T_{X}=\sum_{i\in I}v^{-1-\langle\alpha^{I}_{i},\sum_{j\in I^{c}_{+}}\mf{a}_{j}-\sum_{j\in I^{c}_{-}}\mf{a}_{j}\rangle}\cdot a^{\alpha^{I}_{i}}+\sum_{i\in I}v^{-1+\langle\alpha^{I}_{i},\sum_{j\in I^{c}_{+}}\mf{a}_{j}-\sum_{j\in I^{c}_{-}}\mf{a}_{j}\rangle}\cdot a^{-\alpha^{I}_{i}}.
\end{align}
This implies that in the notation of section 3, the multiset of equivariant roots at $p_{I}$ is given by $\Phi(p_{I})=\{\pm\alpha^{I}_{i}\}_{i\in I}$ and $\overline{\Phi}=\{\alpha_{C^{\vee}}\mid C^{\vee}\mbox{ : signed cocircuit}\}$. We fix $\xi\in\bb{X}_{\ast}(H)$ satisfying $\langle\xi,\alpha_{C^{\vee}}\rangle\neq0$ for any cocircuit $C^{\vee}$ and take the chamber $\mf{C}$ to be the connected component of $\mf{h}_{\bb{R}}\setminus\cup_{\alpha\in\overline{\Phi}}\{x\in\mf{h}_{\bb{R}}\mid\langle x,\alpha\rangle=0\}$ containing $\xi$. With respect to this choice of chamber, we obtain 
\begin{align}\label{eqn_toric_normal_bundle}
N_{p_{I},-}=\sum_{i\in I_{-}}v^{-1-\langle\alpha^{I}_{i},\sum_{j\in I^{c}_{+}}\mf{a}_{j}-\sum_{j\in I^{c}_{-}}\mf{a}_{j}\rangle}\cdot a^{\alpha^{I}_{i}}+\sum_{i\in I_{+}}v^{-1+\langle\alpha^{I}_{i},\sum_{j\in I^{c}_{+}}\mf{a}_{j}-\sum_{j\in I^{c}_{-}}\mf{a}_{j}\rangle}\cdot a^{-\alpha^{I}_{i}}.
\end{align}
For the multiset of K\"ahler roots, we take $\Psi(p_{I})\coloneqq\{\pm\beta^{I}_{j}\}_{j\in I^{c}}$. We have $\overline{\Psi}=\{\beta_{C}\mid C\mbox{ : signed circuit}\}$. It is known (see for example \cite{Nag}) that the ample cone $\mf{A}\subset\Pic(X)\otimes_{\bb{Z}}\bb{R}\cong\mf{s}^{\ast}_{\bb{R}}$ is given by the connected component of $\mf{s}^{\ast}_{\bb{R}}\setminus\cup_{\beta\in\overline{\Psi}}\{x\in\mf{s}^{\ast}_{\bb{R}}\mid\langle x,\beta\rangle=0\}$ containing $\eta$. As in previous sections, $X$ is always equipped with these additional data $(X,\mf{C},\mf{A},\Phi,\Psi,\mf{L})$.

\subsection{Dual pairs}

In this section, we give a dual pair $(X^{!},\mf{C}^{!},\mf{A}^{!},\Phi^{!},\Psi^{!},\mf{L}^{!})$ for $(X,\mf{C},\mf{A},\Phi,\Psi,\mf{L})$. For $X^{!}$, this is given by a symplectic dual of $(X,\mf{C})$ in the sense of Braden-Licata-Proudfoot-Webster \cite{BLPW3}. I.e., $X^{!}$ is the toric hyper-K\"ahler manifolds defined by the exact sequence of algebraic tori which is dual to (\ref{eqn_exact_seq_tori}): 
\begin{align}\label{eqn_exact_seq_dual_tori}
1\rightarrow H^{\vee}\rightarrow T^{\vee}\rightarrow S^{\vee}\rightarrow 1.
\end{align} 
Here, the GIT parameter $\eta^{!}\in\bb{X}^{\ast}(H^{\vee})\cong\bb{X}_{\ast}(H)$ for $X^{!}$ is taken to be $\xi$. We note that the exact sequence of cocharacter lattices associated with (\ref{eqn_exact_seq_dual_tori}) 
\begin{align*}
0\rightarrow\bb{X}_{\ast}(H^{\vee})\xrightarrow{^{t}\!\mf{a}}\bb{X}_{\ast}(T^{\vee})\xrightarrow{\mf{b}}\bb{X}_{\ast}(S^{\vee})\rightarrow 0.
\end{align*}
is naturally isomorphic to the exact sequence (\ref{eqn_exact_seq_character}) and the exact sequence of character lattices
\begin{align*}
0\rightarrow\bb{X}^{\ast}(S^{\vee})\xrightarrow{^{t}\!\mf{b}}\bb{X}^{\ast}(T^{\vee})\xrightarrow{\mf{a}}\bb{X}^{\ast}(H^{\vee})\rightarrow 0.
\end{align*}
is isomorphic to (\ref{eqn_exact_seq_cocharacter}). In particular, the roles of $\mf{a}_{i}$ and $\mf{b}_{i}$ are exchanged. Hence the set parametrizing the $H^{!}\coloneqq S^{\vee}$-fixed points of $X^{!}$ is given by 
\begin{align*}
\bb{B}^{!}\coloneqq\left\{J\subset\{1,\ldots,n\}\mid\{\mf{b}_{j}\}_{j\in J}\mbox{ is a basis of }\bb{X}^{\ast}(S)\right\}.	
\end{align*}
The map $I\mapsto I^{c}$ gives a natural bijection $\bb{B}\cong\bb{B}^{!}$ and hence gives a bijection $X^{H}\cong (X^{!})^{H^{!}}$. We denote by $p^{!}_{I}$ the fixed point of $X^{!}$ corresponding to $I^{c}\in\bb{B}^{!}$ for any $I\in\bb{B}$. Under the natural identification $\bb{X}^{\ast}(H^{!})\cong\bb{X}_{\ast}(S)$, we obtain $\Phi^{!}(p^{!}_{I})=\{\pm\beta^{I}_{j}\}_{j\in I^{c}}$. For the chamber $\mf{C}^{!}\subset\mf{s}^{\ast}_{\bb{R}}$, we take $\mf{C}^{!}\coloneqq\mf{A}$. For the K\"ahler roots, we take $\Psi^{!}(p^{!}_{I})\coloneqq\{\pm\alpha^{I}_{i}\}_{i\in I}$. By our choice of GIT parameter, the ample cone $\mf{A}^{!}\subset\mf{h}_{\bb{R}}$ is given by $\mf{C}$. Therefore, the second condition of Definition~\ref{dual_pair} is satisfied. The order reversing property of the bijection $X^{H}\cong (X^{!})^{H^{!}}$ will be checked in the next section, see Corollary~\ref{cor_toric_order_reversing}. 

As in the case of $X$, we have a natural map $\mca{L}^{!}:\bb{X}^{\ast}(T^{\vee})\rightarrow\Pic^{\bb{T}^{!}}(X^{!})$. In order to define the lift $\mf{L}^{!}$, we need to take a splitting $\iota^{!}:\Pic(X^{!})\cong\bb{X}_{\ast}(H)\rightarrow\bb{X}_{\ast}(T)$ which is compatible with the splitting $\iota:\bb{X}^{\ast}(S)\rightarrow\bb{X}^{\ast}(T)$. Here, the compatibility means that for any $\lambda\in\bb{X}^{\ast}(S)$ and $\lambda^{!}\in\bb{X}_{\ast}(H)$, we have 
\begin{align}\label{eqn_toric_iota}
\langle\iota(\lambda),\iota^{!}(\lambda^{!})\rangle=0.
\end{align}
Existence of such a splitting is clear. We define $\mf{L}^{!}:\bb{X}_{\ast}(H)\rightarrow\Pic^{\bb{T}^{!}}(X^{!})$ by $\mf{L}^{!}(\lambda^{!})=\mca{L}^{!}(\iota^{!}(\lambda^{!}))$. 

\begin{prop}\label{prop_toric_dual_pair}
The pair $(X,\mf{C},\mf{A},\Phi,\Psi,\mf{L})$ and $(X^{!},\mf{C}^{!},\mf{A}^{!},\Phi^{!},\Psi^{!},\mf{L}^{!})$ forms a dual pair in the sense of Definition~\ref{dual_pair}. 
\end{prop}

\begin{proof}
We need to check (\ref{eqn_dual_pair_H}), (\ref{eqn_dual_pair_S}), (\ref{eqn_dual_pair_S_dual}), and (\ref{eqn_S_wt}) in our situation. By (\ref{eqn_toric_normal_bundle}), we obtain 
\begin{align*}
\det N_{p_{I},-}=v^{-d+\langle\sum_{i\in I_{+}}\alpha^{I}_{i}-\sum_{i\in I_{-}}\alpha^{I}_{i},\sum_{j\in I^{c}_{+}}\mf{a}_{j}-\sum_{j\in I^{c}_{-}}\mf{a}_{j}\rangle}\cdot a^{-\sum_{i\in I_{+}}\alpha^{I}_{i}+\sum_{i\in I_{-}}\alpha^{I}_{i}}. 
\end{align*}
In particular, we have 
\begin{align*}
\wt_{\bb{S}}\det N_{p_{I},-}+\frac{\dim X}{2}=\langle\sum_{i\in I_{+}}\alpha^{I}_{i}-\sum_{i\in I_{-}}\alpha^{I}_{i},\sum_{j\in I^{c}_{+}}\mf{a}_{j}-\sum_{j\in I^{c}_{-}}\mf{a}_{j}\rangle. 
\end{align*}
Similarly, we obtain 
\begin{align*}
\wt_{\bb{S}}\det N^{!}_{p^{!}_{I},-}+\frac{\dim X^{!}}{2}=\langle\sum_{j\in I^{c}_{+}}\beta^{I}_{j}-\sum_{j\in I^{c}_{-}}\beta^{I}_{j},\sum_{i\in I_{+}}\mf{b}_{i}-\sum_{i\in I_{-}}\mf{b}_{i}\rangle. 
\end{align*}
Hence the equation (\ref{eqn_S_wt}) follows from (\ref{eqn_alpha_beta}). 

Since we have 
\begin{align*}
\wt_{H^{!}}\det N^{!}_{p^{!}_{I},-}=-\sum_{j\in I^{c}_{+}}\beta^{I}_{j}+\sum_{j\in I^{c}_{-}}\beta^{I}_{j}, 
\end{align*}
the equation (\ref{eqn_dual_pair_S}) follows from Lemma~\ref{lem_toric_line_general}. The equation (\ref{eqn_dual_pair_S_dual}) can be proved similarly. Now we check (\ref{eqn_dual_pair_H}). By Lemma~\ref{lem_toric_line_general} applied for both $X$ and $X^{!}$, we have to check 
\begin{align}\label{eqn_compatibility_iota}
\langle\iota(\lambda)-\sum_{j\in I^{c}}\langle\lambda,\beta^{I}_{j}\rangle\varepsilon^{\ast}_{j},\lambda^{!}\rangle=-\langle\iota^{!}(\lambda^{!})-\sum_{i\in I}\langle\lambda^{!},\alpha^{I}_{i}\rangle\varepsilon_{i},\lambda\rangle
\end{align}
for any $\lambda\in\bb{X}^{\ast}(S)$, $\lambda^{!}\in\bb{X}_{\ast}(H)
$, and $I\in\bb{B}$. Since $\{\mf{a}_{i}\}_{i\in I}$ is a basis of $\bb{X}_{\ast}(H)$, it suffices to check (\ref{eqn_compatibility_iota}) for any $\lambda^{!}=\mf{a}_{i}$, $i\in I$. Since we have $\iota(\lambda)-\sum_{j\in I^{c}}\langle\lambda,\beta^{I}_{j}\rangle\varepsilon^{\ast}_{j}\in\bb{X}^{\ast}(H)$, we can calculate LHS of (\ref{eqn_compatibility_iota}) by using any lift of $\mf{a}_{i}$ to $\bb{X}_{\ast}(T)$. Therefore, we obtain 
\begin{align*}
\langle\iota(\lambda)-\sum_{j\in I^{c}}\langle\lambda,\beta^{I}_{j}\rangle\varepsilon^{\ast}_{j},\mf{a}_{i}\rangle=\langle\iota(\lambda)-\sum_{j\in I^{c}}\langle\lambda,\beta^{I}_{j}\rangle\varepsilon^{\ast}_{j},\varepsilon_{i}\rangle=\langle\iota(\lambda),\varepsilon_{i}\rangle.
\end{align*}
On the other hand, we may replace $\lambda$ by $\iota(\lambda)$ in the RHS of (\ref{eqn_compatibility_iota}) since $\iota^{!}(\lambda^{!})-\sum_{i\in I}\langle\lambda^{!},\alpha^{I}_{i}\rangle\varepsilon_{i}\in\bb{X}_{\ast}(S)$. Therefore, we obtain
\begin{align*}
-\langle\iota^{!}(\mf{a}_{i})-\sum_{i'\in I}\langle\mf{a}_{i},\alpha^{I}_{i'}\rangle\varepsilon_{i'},\lambda\rangle=-\langle\iota^{!}(\mf{a}_{i}),\iota(\lambda)\rangle+\langle\varepsilon_{i},\iota(\lambda)\rangle. 
\end{align*}
Hence the equation (\ref{eqn_compatibility_iota}) follows from (\ref{eqn_toric_iota}). This proves (\ref{eqn_line_bundle}). The proof of (\ref{eqn_line_bundle_dual}) is similar. 
\end{proof}

Proposition~\ref{prop_toric_dual_pair} implies that a maximal flop $X_{\mr{flop}}$ in the sense of section 4.2 is obtained in the same way as $X$ by replacing $\eta$ by $-\eta$. Since this exchanges $I^{c}_{+}$ and $I^{c}_{-}$, the formula (\ref{eqn_toric_tangent_fiber}) implies Conjecture~\ref{conj_flop_tangent}. 

\begin{corollary}\label{cor_toric_flop}
Conjecture~\ref{conj_flop_tangent} holds for toric hyper-K\"ahler manifolds. 
\end{corollary}

\subsection{$K$-theoretic standard bases}

In this section, we recall the description of $K$-theoretic stable bases for toric hyper-K\"ahler manifolds. In this paper, we always take the following polarization for the toric hyper-K\"ahler manifold $X$:
\begin{align}\label{eqn_toric_polarization}
T^{1/2}\coloneqq\sum_{i=1}^{n}v^{-1}\mca{L}(\varepsilon^{\ast}_{i})-r\cdot\mathcal{O}_{X}. 
\end{align}
In particular, we have $\det T^{1/2}=v^{-n}\mca{L}(\varepsilon^{\ast}_{1}+\cdots+\varepsilon^{\ast}_{n})$ and hence we obtain $w(\det T^{1/2})=-n$. Therefore, Assumption~\ref{assumption_polarization} is satisfied. Since we will not use other polarization below, we will omit $T^{1/2}$ from the notations. We set 
\begin{align*}
\mf{s}^{\ast}_{\mr{reg}}\coloneqq\left\{x\in\mf{s}^{\ast}_{\bb{R}}\mid \langle x,\beta_{C}\rangle\notin\bb{Z},\;\forall C\mbox{: circuit}\right\}
\end{align*}
and we take a slope parameter $s\in\mf{s}^{\ast}_{\mr{reg}}$. As in section 3.2, we consider the fractional line bundle $\mf{L}(s)$. By Lemma~\ref{lem_toric_line_general}, we obtain 
\begin{align}\label{eqn_toric_diff_slope}
\wt_{H}i^{\ast}_{p_{I}}\mf{L}(s)-\wt_{H}i^{\ast}_{p_{J}}\mf{L}(s)=\sum_{i\in J^{c}}\langle s,\beta^{J}_{i}\rangle\varepsilon^{\ast}_{i}-\sum_{j\in I^{c}}\langle s,\beta^{I}_{j}\rangle\varepsilon^{\ast}_{j}=-\sum_{j\in I^{c}\cap J}\langle s,\beta^{I}_{j}\rangle\alpha^{J}_{j}
\end{align}
for any $I,J\in\bb{B}$. Here, we have used $\beta^{J}_{i}=\sum_{j\in I^{c}}\langle\beta^{J}_{i},\mf{b}_{j}\rangle\beta^{I}_{j}$ and $\alpha^{J}_{j}=\varepsilon^{\ast}_{j}-\sum_{i\in I^{c}}\langle\beta^{J}_{i},\mf{b}_{j}\rangle\varepsilon^{\ast}_{i}$ in the second equality. In particular, this is not contained in $\bb{X}^{\ast}(H)$ if $I\neq J$. This implies the first part of Assumption~\ref{ass_stab_existence} and hence the uniqueness of $K$-theoretic stable bases. The existence of $K$-theoretic stable bases is proved in Proposition~\ref{prop_toric_K_stab}. 

We note that the coordinate function $x_{i}$ (resp. $y_{i}$) can be considered as a section of $v^{-1}\mca{L}(\varepsilon^{\ast}_{i})$ (resp. $v^{-1}\mca{L}(-\varepsilon^{\ast}_{i})$) on $X$. For $I\in\bb{B}$, let $L_{I}$ be the subvariety of $X$ defined by the equations $x_{i}=0$ ($i\in I_{-}$) and $y_{i}=0$ ($i\in I_{+}$). One can check that these defining equations form a regular sequence and the Koszul resolution gives the following. 
\begin{lemma}\label{lem_toric_resol_L_I}
Let $\mca{V}_{I}\coloneqq\bigoplus_{i\in I_{-}}v^{-1}\mca{L}(\varepsilon^{\ast}_{i})\oplus\bigoplus_{i\in I_{+}}v^{-1}\mca{L}(-\varepsilon^{\ast}_{i})$ be a vector bundle on $X$. We have the following exact sequence 
\begin{align*}
0\rightarrow{\textstyle\bigwedge^{d}}\mca{V}_{I}^{\vee}\rightarrow\ldots\rightarrow{\textstyle\bigwedge^{2}}\mca{V}_{I}^{\vee}\rightarrow\mca{V}_{I}^{\vee}\rightarrow\mca{O}_{X}\rightarrow\mca{O}_{L_{I}}\rightarrow0.
\end{align*}
\end{lemma}

In particular, we have 
\begin{align}\label{eqn_str_sheaf_L_I}
\mca{O}_{L_{I}}=\prod_{i\in I_{-}}(1-v\mca{L}(-\varepsilon^{\ast}_{i}))\prod_{i\in I_{+}}(1-v\mathcal{L}(\varepsilon^{\ast}_{i}))
\end{align}
in the equivariant $K$-theory of $X$. Moreover, one can easily check the following. 

\begin{lemma}\label{lem_toric_attr_closure}
For any $I\in\bb{B}$, we have $L_{I}=\overline{\Attr_{\mf{C}}(p_{I})}$. In particular, $p_{J}\in\overline{\Attr_{\mf{C}}(p_{I})}$ is equivalent to $I_{+}\cap J^{c}_{-}=I_{-}\cap J^{c}_{+}=\emptyset$ for any $I,J\in\bb{B}$.
\end{lemma}

\begin{corollary}\label{cor_toric_order_reversing}
For any $I,J\in\bb{B}$, $p_{J}\preceq_{\mf{C}}p_{I}$ is equivalent to $p^{!}_{I}\preceq_{\mf{C}^{!}}p^{!}_{J}$. 
\end{corollary}

\begin{proof}
Lemma~\ref{lem_toric_attr_closure} implies that $p_{J}\in\overline{\Attr_{\mf{C}}(p_{I})}$ is equivalent to $p^{!}_{I}\in\overline{\Attr_{\mf{C}^{!}}(p^{!}_{J})}$.
\end{proof}

Now we give an explicit formula for the $K$-theoretic stable bases for $X$. This is an explicit version of Exercise 9.1.15 in \cite{O1}.

\begin{prop}\label{prop_toric_K_stab}
For any $I\in\bb{B}$, we have 
\begin{align}\label{eqn_toric_stable_basis}
\Stab^{K}_{\mf{C},s}(p_{I})=v^{\sum_{j\in I^{c}_{-}}\lceil\langle s,\beta^{I}_{j}\rangle\rceil-\sum_{j\in I^{c}_{+}}\lfloor\langle s,\beta^{I}_{j}\rangle\rfloor}\cdot\mathcal{L}\left(-\sum_{i\in I_{+}}\varepsilon^{\ast}_{i}+\sum_{j\in I^{c}}\lfloor\langle s,\beta^{I}_{j}\rangle\rfloor\varepsilon^{\ast}_{j}\right)\otimes\mathcal{O}_{L_{I}}.
\end{align}
\end{prop}

\begin{proof}
Let us denote by $\Stab_{I}$ the RHS of (\ref{eqn_toric_stable_basis}). We check the three conditions in Definition~\ref{stab_def} for $\Stab_{I}$. Since we have $\Supp(\Stab_{I})=L_{I}=\overline{\Attr_{\mf{C}}(p_{I})}$ by Lemma~\ref{lem_toric_attr_closure}, the first condition in Definition~\ref{stab_def} is satisfied. 

We note that by Corollary~\ref{cor_toric_line_restr}, we have 
\begin{align*}
T^{1/2}_{p_{I}}=\sum_{i\in I}v^{-1-\langle\alpha^{I}_{i},\sum_{j\in I^{c}_{+}}\mf{a}_{j}-\sum_{j\in I^{c}_{-}}\mf{a}_{j}\rangle}\cdot a^{\alpha^{I}_{i}}+|I^{c}_{-}|(v^{-2}-1).
\end{align*}
Hence by (\ref{eqn_toric_normal_bundle}) and Corollary~\ref{cor_toric_line_restr}, we obtain 
\begin{align*}
\sqrt{\frac{\det N_{p_{I},-}}{\det T^{1/2}_{p_{I}}}}=v^{|I^{c}_{-}|}\cdot\prod_{i\in I_{+}}v^{\langle\alpha^{I}_{i},\sum_{j\in I^{c}_{+}}\mf{a}_{j}-\sum_{j\in I^{c}_{-}}\mf{a}_{j}\rangle}\cdot a^{-\alpha^{I}_{i}}=v^{|I^{c}_{-}|}\cdot i^{\ast}_{p_{I}}\mca{L}\left(-\sum_{i\in I_{+}}\varepsilon^{\ast}_{i}\right).
\end{align*}
On the other hand, we have $\bigwedge^{\bullet}_{-}(N^{\vee}_{-,p_{I}})=i^{\ast}_{p_{I}}\mca{O}_{L_{I}}$ by (\ref{eqn_str_sheaf_L_I}). Therefore, the second condition in Definition~\ref{stab_def} follows from Corollary~\ref{cor_toric_line_restr}.  

Finally, we check the third condition of Definition~\ref{stab_def}. For $I,J\in\bb{B}$, let us assume that $i^{\ast}_{p_{J}}\Stab_{I}\neq0$. By Lemma~\ref{lem_toric_attr_closure}, we have $I_{+}\cap J^{c}_{-}=I_{-}\cap J^{c}_{+}=\emptyset$. Hence up to the factor of $v$ and sign, we obtain
\begin{align*}
i^{\ast}_{p_{J}}\Stab_{I}=\pm v^{?}(1-v^2)^{|I_{\pm}\cap J^{c}_{\pm}|}\prod_{j\in I^{c}\cap J}a^{\lfloor \langle s,\beta^{I}_{j}\rangle\rfloor\alpha^{J}_{j}}\cdot\prod_{i\in I_{\pm}\cap J}\left(1-v^{\mp 1+\langle\alpha^{J}_{i},\sum_{j\in J^{c}_{+}}\mf{a}_{j}-\sum_{j\in J^{c}_{-}}\mf{a}_{j}\rangle}\cdot a^{-\alpha^{J}_{i}}\right).
\end{align*}
By (\ref{eqn_toric_diff_slope}), we obtain
\begin{align*}
\deg_{H}\left(i^{\ast}_{p_{J}}\Stab_{I}\cdot\,i^{\ast}_{p_{I}}\mf{L}(s)\cdot i^{\ast}_{p_{J}}\mf{L}(s)^{-1}\right)=\sum_{j\in I^{c}\cap J}(\lfloor \langle s,\beta^{I}_{j}\rangle\rfloor-\langle s,\beta^{I}_{j}\rangle)\cdot\alpha^{J}_{j}+\sum_{i\in I\cap J}\deg_{H}(1-a^{-\alpha^{J}_{i}}).
\end{align*} 
Here, the sum means the Minkowski sum. On the other hand, we have 
\begin{align*}
\deg_{H}\left(i^{\ast}_{p_{J}}\Stab_{J}\right)=\sum_{i\in J}\deg_{H}(1-a^{-\alpha^{J}_{i}}).
\end{align*}
Therefore, the third condition in Definition~\ref{stab_def} follows from $(\lfloor \langle s,\beta^{I}_{j}\rangle\rfloor-\langle s,\beta^{I}_{j}\rangle)\cdot\alpha^{J}_{j}\in\deg_{H}(1-a^{-\alpha^{J}_{j}})$ for each $j\in I^{c}\cap J$.
\end{proof}

We next determine the $K$-theoretic standard bases. We note that in the notation of section 3.1, we have $\Psi_{+}(p_{I})=\{\beta^{I}_{j}\}_{j\in I^{c}_{+}}\cup\{-\beta^{I}_{j}\}_{j\in I^{c}_{-}}$. Since we have $w(\det T^{1/2})=-n=-r-d$, we obtain 
\begin{align*}
a_{p_{I}}(s)=\sum_{j\in I^{c}_{+}}\lfloor\langle s,\beta^{I}_{j}\rangle\rfloor-\sum_{j\in I^{c}_{-}}\lceil\langle s,\beta^{I}_{j}\rangle\rceil-d.
\end{align*}
Hence we obtain the following corollary of Proposition~\ref{prop_toric_K_stab}.

\begin{corollary}\label{cor_toric_standard_basis}
For any $I\in\bb{B}$, the standard basis $\mca{S}_{\mf{C},s}(p_{I})$ is given by 
\begin{align*}
\mca{S}_{\mf{C},s}(p_{I})=(-v)^{-d}\cdot\mca{L}\left(-\sum_{i\in I_{+}}\varepsilon^{\ast}_{i}+\sum_{j\in I^{c}}\lfloor\langle s,\beta^{I}_{j}\rangle\rfloor\varepsilon^{\ast}_{j}\right)\otimes\mca{O}_{L_{I}}. 
\end{align*}
\end{corollary}

By using (\ref{eqn_str_sheaf_L_I}), we also obtain 
\begin{align}\label{eqn_toric_std_plus}
\mca{S}_{\mf{C},s}(p_{I})=\sum_{K\subset I}(-v)^{-|K|}\cdot\mca{L}\left(-\sum_{i\in I_{+}\cap K}\varepsilon^{\ast}_{i}-\sum_{i\in I_{-}\cap K^{c}}\varepsilon^{\ast}_{i}+\sum_{j\in I^{c}}\lfloor\langle s,\beta^{I}_{j}\rangle\rfloor\varepsilon^{\ast}_{j}\right).
\end{align}
Here, the sum runs over all subsets of $I$. By exchanging $I_{+}$ and $I_{-}$, we also obtain 
\begin{align}\label{eqn_toric_std_minus}
\mca{S}_{-\mf{C},s}(p_{I})=\sum_{K\subset I}(-v)^{-|K|}\cdot\mca{L}\left(-\sum_{i\in I_{-}\cap K}\varepsilon^{\ast}_{i}-\sum_{i\in I_{+}\cap K^{c}}\varepsilon^{\ast}_{i}+\sum_{j\in I^{c}}\lfloor\langle s,\beta^{I}_{j}\rangle\rfloor\varepsilon^{\ast}_{j}\right).
\end{align}
We note that these formulas depend only on the K\"ahler alcove containing the slope $s$. 

\subsection{Alcove model}

In this section, we reformulate Corollary~\ref{cor_toric_standard_basis} by using certain combinatorics of alcoves in $\mf{h}^{\ast}_{\bb{R}}$ and determine the $K$-theoretic canonical bases for toric hyper-K\"ahler manifolds. 

We write $\iota(s)=(s_{1},\ldots,s_{n})\in\bb{X}^{\ast}(T)\otimes_{\bb{Z}}\bb{R}\cong\bb{R}^{n}$. Using this data, we consider a periodic hyperplane arrangement in $\mf{h}^{\ast}_{\bb{R}}$ defined by $\mca{H}^{s}_{i,m}=\mca{H}_{i,m}\coloneqq\{x\in \mf{h}^{\ast}_{\bb{R}}\mid\langle x,\mf{a}_{i}\rangle+s_{i}=m\}$ for any $i=1,\ldots,n$ and $m\in\bb{Z}$. We note that if we use a different choice of the lift $\iota$, then the resulting hyperplane arrangement is given by a translation of the original one. We also remark that by the condition $s\in\mf{s}^{\ast}_{\mr{reg}}$, we have $\cap_{i\in C}\mca{H}_{i,m_{i}}=\emptyset$ for any circuit $C$ and any choice of $m_{i}\in\bb{Z}$. We denote by $\mr{Alc}_{s}$ the set of connected components of $\mf{h}^{\ast}_{\bb{R}}\setminus\cup_{i,m}\mca{H}_{i,m}$. We simply call an element of $\mr{Alc}_{s}$ alcove. 
  
By using the data $\mf{C}$, one can give a bijection between $\mr{Alc}_{s}$ and $\bb{F}$, where $\bb{F}$ is defined in (\ref{label_set}). For any $A\in\mr{Alc}_{s}$, the closure $\overline{A}$ is a polytope and the linear form $\xi:\mf{h}^{\ast}_{\bb{R}}\rightarrow\bb{R}$ restricted to $\overline{A}$ takes its minimum at a vertex $x_{A}$ by the genericity of $\xi$. We set $\mca{H}^{s,\pm}_{i,m}=\mca{H}^{\pm}_{i,m}\coloneqq\{x\in \mf{h}^{\ast}_{\bb{R}}\mid\pm(\langle x,\mf{a}_{i}\rangle+s_{i}-m)>0\}$. If we write $\{x_{A}\}=\cap_{i\in I}\mca{H}_{i,m_{i}}$ for some $m_{i}\in\bb{Z}$, then we have $I\in\bb{B}$  and $A\subset\bigcap_{i\in I_{+}}\mca{H}^{+}_{i,m_{i}}\cap\bigcap_{i\in I_{-}}\mca{H}^{-}_{i,m_{i}}$. We define a map $\varphi_{\mf{C},s}:\mr{Alc}_{s}\rightarrow\bb{F}$ by $\varphi_{\mf{C},s}(A)\coloneqq(\sum_{i\in I}m_{i}\alpha^{I}_{i},p_{I})$. We note that this does not depend on the choice of $\xi\in\mf{C}$. It is easy to check that $\varphi_{\mf{C},s}$ is a bijection by the genericity of $s$. We denote by $\leq_{\mf{C}}$ the partial order on $\mr{Alc}_{s}$ induced from the partial order $\leq_{\mf{C},s}$ on $\bb{F}$ via $\varphi_{\mf{C},s}$. By using Lemma~\ref{lem_toric_line_general}, we obtain 
\begin{align}\label{eqn_toric_x_A}
\sum_{i\in I}m_{i}\alpha^{I}_{i}-\wt_{H}i^{\ast}_{p_{I}}\mf{L}(s)=x_{A}.
\end{align}
This implies the following lemma.   

\begin{lemma}\label{lem_toric_poset_identification}
For any $A,B\in\mr{Alc}_{s}$, $A\leq_{\mf{C}}B$ if and only if $\langle x_{A},\xi\rangle\leq\langle x_{B},\xi\rangle$ for any $\xi\in\mf{C}$. 
\end{lemma}

For any $A\in\mr{Alc}_{s}$, we set 
\begin{align}\label{eqn_mu_A}
\mu_{A}\coloneqq\sum_{i=1}^{n}\lfloor \langle x,\mf{a}_{i}\rangle+s_{i}\rfloor\varepsilon^{\ast}_{i}\in\bb{X}^{\ast}(T)
\end{align}
for some $x\in A$ and consider the $\bb{T}$-equivariant line bundle $\mca{E}(A)\coloneqq\mca{L}\left(\mu_{A}\right)$. Note that this definition does not depend on the choice of $x\in A$ and also the choice of chamber $\mf{C}$. If $\varphi_{\mf{C},s}(A)=(\sum_{i\in I}m_{i}\alpha^{I}_{i},p_{I})$, then we have $A\subset\bigcap_{i\in I_{+}}\mca{H}^{+}_{i,m_{i}}\cap\bigcap_{i\in I_{-}}\mca{H}^{-}_{i,m_{i}}$. Hence, we obtain $\mca{E}(A)=\mca{L}\left(-\sum_{i\in I_{-}}\varepsilon^{\ast}_{i}+\sum_{j=1}^{n}\lfloor\langle x_{A},\mf{a}_{j}\rangle+s_{j}\rfloor\varepsilon^{\ast}_{j}\right)$. Since $\langle x_{A},\mf{a}_{i}\rangle+s_{i}=m_{i}\in\bb{Z}$ for any $i\in I$ and $\mf{a}_{j}=-\sum_{i\in I}\langle\beta^{I}_{j},\mf{b}_{i}\rangle \mf{a}_{i}$ for any $j\in I^{c}$, we have
\begin{align*}
\lfloor\langle x_{A},\mf{a}_{j}\rangle+s_{j}\rfloor&=\lfloor-\sum_{i\in I}\langle x_{A},\mf{a}_{i}\rangle\langle\beta^{I}_{j},\mf{b}_{i}\rangle+s_{j}\rfloor\\
&=\lfloor-\sum_{i\in I}m_{i}\langle\beta^{I}_{j},\mf{b}_{i}\rangle+s_{j}+\sum_{i\in I}s_{i}\langle\beta^{I}_{j},\mf{b}_{i}\rangle\rfloor\\
&=\sum_{i\in I}m_{i}\langle\alpha^{I}_{i},\mf{a}_{j}\rangle+\lfloor \langle s,\beta^{I}_{j}\rangle\rfloor.
\end{align*}
Therefore, we obtain
\begin{align}\label{eqn_toric_E(A)}
\mca{E}(A)=\prod_{i\in I} a^{m_{i}\alpha^{I}_{i}}\cdot\mca{L}\left(-\sum_{i\in I_{-}}\varepsilon^{\ast}_{i}+\sum_{j\in I^{c}}\lfloor \langle s,\beta^{I}_{j}\rangle\rfloor\varepsilon^{\ast}_{j}\right).
\end{align}
We set $\mca{E}_{\mf{C},s}(p_{I})\coloneqq\mca{L}\left(-\sum_{i\in I_{-}}\varepsilon^{\ast}_{i}+\sum_{j\in I^{c}}\lfloor \langle s,\beta^{I}_{j}\rangle\rfloor\varepsilon^{\ast}_{j}\right)$. Moreover, if $B\in\mr{Alc}_{s}$ is an alcove such that $x_{A}\in \overline{B}$ and $K\subset I$ is a subset such that $\{\mathcal{H}_{k,m_{k}}\}_{k\in K}$ is the set of hyperplanes separating $A$ and $B$, then we have
\begin{align}\label{eqn_toric_E(B)}
\mca{E}(B)=\prod_{i\in I} a^{m_{i}\alpha^{I}_{i}}\cdot
\mca{L}\left(-\sum_{i\in I_{+}\cap K}\varepsilon^{\ast}_{i}-\sum_{i\in I_{-}\cap K^{c}}\varepsilon^{\ast}_{i}+\sum_{j\in I^{c}}\lfloor\langle s,\beta^{I}_{j}\rangle\rfloor\varepsilon^{\ast}_{j}\right).
\end{align}
In particular, the RHS of (\ref{eqn_toric_E(B)}) is of the form $\mca{E}_{\mf{C},s}(p_{I'})$ for some $I'\in\bb{B}$ up to some $H$-equivariant parameter shift.

For any $A,B\in\mr{Alc}_{s}$, we write $\ell(A,B)$ the number of hyperplanes separating $A$ and $B$. We set $N(A)\coloneqq\{B\in\mr{Alc}_{s}\mid x_{A}\in\overline{B}\}$ and define $\mca{S}(A)\in K_{\bb{T}}(X)$ by the following formula:
\begin{align}\label{eqn_toric_S(A)}
\mca{S}(A)\coloneqq\sum_{B\in N(A)}(-v)^{-\ell(A,B)}\mca{E}(B).
\end{align} 
We note that by Lemma~\ref{lem_toric_poset_identification}, we have $B\leq_{\mf{C}}A$ for any $B\in N(A)$. By (\ref{eqn_toric_std_plus}) and (\ref{eqn_toric_E(B)}), we obtain the following formula expressing $\mca{S}_{\mf{C},s}(\varphi_{\mf{C},s}(A))$ defined in Definition~\ref{general_standard_basis}.

\begin{lemma}\label{lem_toric_S(A)}
For any $A\in\mr{Alc}_{s}$, we have $\mathcal{S}(A)=\mca{S}_{\mf{C},s}(\varphi_{\mf{C},s}(A))$. 
\end{lemma}

We now prove $\mca{E}(A)=\mca{E}_{\mf{C},s}(\varphi_{\mf{C},s}(A))$ in the notation of Conjecture~\ref{K-theoretic_canonical_basis}. By (\ref{eqn_toric_S(A)}), we obtain 
\begin{align}\label{eqn_exp_E(A)_to_S(B)_rough}
\mca{E}(A)\in\mca{S}(A)+\sum_{B<_{\mf{C}}A}v^{-1}\bb{Z}[v^{-1}]\cdot\mca{S}(B)
\end{align}
under certain completion as in section 3.3. Therefore, it is enough to check the following. 

\begin{prop}\label{prop_toric_bar_invariance_E(A)}
For any $A\in\mr{Alc}_{s}$, we have $\beta^{K}_{\mf{C},s}(\mca{E}(A))=\mca{E}(A)$. 
\end{prop}

\begin{proof}
Since $\{\mca{S}_{\mf{C},s}(p_{I})\}_{I\in\bb{B}}$ is a basis of $K_{\bb{T}}(X)_{\mr{loc}}$ over $\Frac(K_{\bb{T}}(\mr{pt}))$, $\{\mca{E}_{\mf{C},s}(p_{I})\}_{I\in\bb{B}}$ is also a basis of $K_{\bb{T}}(X)_{\mr{loc}}$ over $\Frac(K_{\bb{T}}(\mr{pt}))$ by (\ref{eqn_toric_E(A)}) and (\ref{eqn_toric_S(A)}). We define $\Frac(K_{H}(\mr{pt}))$-linear involution $\beta'$ on $K_{\bb{T}}(X)_{\mr{loc}}$ by 
\begin{itemize}
\item $\beta'(vm)=v^{-1}\beta'(m)$ for any $m\in K_{\bb{T}}(X)_{\mr{loc}}$,
\item $\beta'(\mca{E}_{\mf{C},s}(p_{I}))=\mca{E}_{\mf{C},s}(p_{I})$ for any $I\in\bb{B}$.
\end{itemize}
By (\ref{eqn_toric_E(A)}), we have $\beta'(\mca{E}(A))=\mca{E}(A)$ for any $A\in\mr{Alc}_{s}$ and hence
\begin{align*}
\beta'(\mca{S}_{\mf{C},s}(p_{I}))&=\sum_{K\subset I}(-v)^{|K|}\mca{L}\left(-\sum_{i\in I_{-}\cap K^{c}}\varepsilon^{\ast}_{i}-\sum_{i\in I_{+}\cap K}\varepsilon^{\ast}_{i}+\sum_{j\in I^{c}}\lfloor \langle s,\beta^{I}_{j}\rangle\rfloor\varepsilon^{\ast}_{j}\right)\\
&=(-v)^{d}\sum_{K\subset I}(-v)^{-|K|}\mathcal{L}\left(-\sum_{i\in I_{-}\cap K}\varepsilon^{\ast}_{i}-\sum_{i\in I_{+}\cap K^{c}}\varepsilon^{\ast}_{i}+\sum_{j\in I^{c}}\lfloor \langle s,\beta^{I}_{j}\rangle\rfloor\varepsilon^{\ast}_{j}\right)\\
&=(-v)^{d}\mca{S}_{-\mf{C},s}(p_{I})
\end{align*}
for any $I\in\mathbb{B}$ by (\ref{eqn_toric_std_plus}), (\ref{eqn_toric_std_minus}), and (\ref{eqn_toric_E(B)}). This implies that $\beta'=\beta^{K}_{\mf{C},s}$ and hence $\mca{E}(A)$ is bar invariant. 
\end{proof}

Since the set $\{\mca{E}(A)\}_{A\in\mr{Alc}_{s}}$ does not depend on the choice of $\mf{C}$, we obtain Conjecture~\ref{K-theoretic_bar_involution} for toric hyper-K\"ahler manifolds. 

\begin{corollary}\label{cor_toric_indep_chamber_K}
The $K$-theoretic bar involution $\beta^{K}_{\mf{C},s}$ does not depend on the choice of $\mf{C}$. 
\end{corollary}

Next, we determine $\mca{C}_{\mf{C},s}(\varphi_{\mf{C},s}(A))$ for any $A\in\mr{Alc}_{s}$. Recall the map $\partial:K_{\bb{T}}(\mr{pt})\rightarrow\bb{Z}[v,v^{-1}]$ defined in section 3.5. By Lemma~\ref{std_orthonormality} and Lemma~\ref{lem_toric_S(A)}, we have $\partial(\mca{S}(A)||\mca{S}(B))=\delta_{A,B}$ for any $A,B\in\mr{Alc}_{s}$. By induction, this and (\ref{eqn_exp_E(A)_to_S(B)_rough}) implies that for any $A\in\mr{Alc}_{s}$, there exists a unique element
\begin{align*}
\mca{C}(A)\in\mca{S}(A)+\sum_{B>_{\mf{C}}A}v^{-1}\bb{Z}[v^{-1}]\cdot\mca{S}(B)
\end{align*}
such that $\partial(\mca{C}(A)||\mca{E}(B))=\delta_{A,B}$ for any $A,B\in\mr{Alc}_{s}$. Since we have  
\begin{align*}
\partial(\mca{C}(A)||\mca{S}(B))&=\sum_{C\in N(B)}(-v)^{\ell(C,B)}\partial(\mca{C}(A)||\mca{E}(C))\\&=\begin{cases}(-v)^{\ell(A,B)}&\mbox{ if }A\in N(B)\\0&\mbox{ otherwise, }\end{cases}
\end{align*}
we obtain 
\begin{align}\label{eqn_toric_C(A)}
\mca{C}(A)=\sum_{B\in N^{-}(A)}(-v)^{-\ell(A,B)}\mca{S}(B)
\end{align}
where we set $N^{-}(A)\coloneqq\{B\in\mr{Alc}_{s}\mid A\in N(B)\}$. We note that $N^{-}(A)$ is a finite set. 

\begin{lemma}\label{lem_toric_bar_invariance_C(A)}
For any $A\in\mr{Alc}_{s}$, we have $\beta^{K}_{\mf{C},s}(\mca{C}(A))=v^{\dim X}\cdot\mca{C}(A)$. 
\end{lemma}

\begin{proof}
By Lemma~\ref{inner_bar} and Proposition~\ref{prop_toric_bar_invariance_E(A)}, we have 
\begin{align*}
\partial(v^{-\dim X}\cdot\beta^{K}_{\mf{C},s}(\mca{C}(A))||\mca{E}(B))&=v^{-\dim X}\partial(\beta^{K}_{\mf{C},s}(\mca{C}(A))||\beta^{K}_{\mf{C},s}(\mca{E}(B)))\\
&=\overline{\partial(\mca{C}(A)||\mca{E}(B))}\\
&=\delta_{A,B}
\end{align*}
for any $A,B\in\mr{Alc}_{s}$. Hence we obtain $\beta^{K}_{\mf{C},s}(\mca{C}(A))=v^{\dim X}\cdot\mca{C}(A)$. 
\end{proof}

Following the notation in section 3.4, we set $\bb{B}_{X,s}\coloneqq\{\mca{E}(A)\}_{A\in\mr{Alc}_{s}}$ and $\bb{B}_{L,s}\coloneqq\{\mca{C}(A)\}_{A\in\mr{Alc}_{s}}$ and call them $K$-theoretic canonical bases for $K_{\bb{T}}(X)$ and $K_{\bb{T}}(L)$. We will prove later (Corollary~\ref{cor_toric_canonical_basis}) that $\bb{B}_{X,s}$ (resp. $\bb{B}_{L,s}$) is actually a basis of $K_{\bb{T}}(X)$ (resp. $K_{\bb{T}}(L)$). 

\subsection{Wall-crossings}

In this section, we study the behavior of $K$-theoretic canonical bases under the variation of $s\in\mf{s}^{\ast}_{\mr{reg}}$. We fix a signed circuit $C=C_{+}\sqcup C_{-}$ satisfying $\langle\eta,\beta_{C}\rangle>0$ and consider a union of hyperplanes $w_{C}\coloneqq\{x\in\mf{s}^{\ast}_{\bb{R}}\mid\langle x,\beta_{C}\rangle\in\bb{Z}\}$. We take a generic element $s_{0}\in w_{C}$ such that $s_{0}$ does not lie in any $w_{C'}$ for some circuit $C'\neq C$. We consider two slopes $s_{-},s_{+}\in\mf{s}^{\ast}_{\mr{reg}}$ sufficiently close to $s_{0}$ such that they lie in the same connected component of $\mf{s}^{\ast}_{\bb{R}}\setminus\cup_{C'\neq C}w_{C'}$ as $s_{0}$ and satisfy $\langle s_{-},\beta_{C}\rangle<\langle s_{0},\beta_{C}\rangle<\langle s_{+},\beta_{C}\rangle$. We study the difference between $K$-theoretic canonical bases $\bb{B}_{X,s_{-}}$ and $\bb{B}_{X,s_{+}}$. We fix a path $\gamma$ connecting $s_{-}$ and $s_{+}$ in a neighborhood of $s_{0}$ and passing through $s_{0}$. 

For $(\lambda,p_{I})\in\bb{F}$, we consider the vertex $x_{\lambda,I}(s)$ corresponding to $\varphi^{-1}_{\mf{C},s}(\lambda,p_{I})\in\mr{Alc}_{s}$ as in the previous section. By (\ref{eqn_toric_x_A}), we have $x_{\lambda,I}(s)=\lambda-\wt_{H}i^{\ast}_{p_{I}}\mf{L}(s)$. If $s$ goes to $s_{0}$, then it can happen that $\lim_{s\rightarrow s_{0}}x_{\lambda,I}(s)=\lim_{s\rightarrow s_{0}}x_{\mu,J}(s)$ for some $(\mu,p_{J})\neq(\lambda,p_{I})\in\bb{F}$. If this does not happen, then hyperplanes other than $\mca{H}^{s}_{i,m_{i}}$ for $i\in I$ will be away from $x_{\lambda,I}(s)$ along $s\in\gamma$. Hence we obtain $\mca{E}_{\mf{C},s_{-}}(\lambda,p_{I})=\mca{E}_{\mf{C},s_{+}}(\lambda,p_{I})$. 

\begin{lemma}
For $I\neq J\in\bb{B}$, $\lim_{s\rightarrow s_{0}}x_{\lambda,I}(s)=\lim_{s\rightarrow s_{0}}x_{\mu,J}(s)$ for some $\lambda,\mu\in\bb{X}^{\ast}(H)$ if and only if $|I^{c}\cap C|=|J^{c}\cap C|=1$ and $I\cap C^{c}=J\cap C^{c}$.
\end{lemma}

\begin{proof}
Assume that $\lim_{s\rightarrow s_{0}}x_{\lambda,I}(s)=\lim_{s\rightarrow s_{0}}x_{\mu,J}(s)$ for some $(\mu,p_{J})\neq(\lambda,p_{I})\in\bb{F}$. By (\ref{eqn_toric_diff_slope}), this implies that $\sum_{j\in I^{c}\cap J}\langle s_{0},\beta^{I}_{j}\rangle\alpha^{J}_{j}\in\bb{X}^{\ast}(H)$ and hence $\langle s_{0},\beta^{I}_{j}\rangle\in\bb{Z}$ for any $j\in I^{c}\cap J$. By the choice of $s_{0}$, we must have $|I^{c}\cap J|=1$ and $\beta^{I}_{j}=\pm\beta_{C}$ for $j\in I^{c}\cap J$. In particular, we have $j\in C\subset I\cup\{j\}$ and hence we obtain $|I^{c}\cap C|=1$. By exchanging the role of $I$ and $J$, we also obtain $I\cap J^{c}=\{i\}$ for some $i$ and $i\in C\subset J\cap\{i\}$. This implies that $I\cap C^{c}=(I\setminus\{i\})\cap C^{c}=(J\setminus\{j\})\cap C^{c}=J\cap C^{c}$.

Conversely, we assume that $I\neq J\in\bb{B}$ satisfy $|I^{c}\cap C|=|J^{c}\cap C|=1$ and $I\cap C^{c}=J\cap C^{c}\eqqcolon K$. If we set $I^{c}\cap C=\{j\}$ and $J^{c}\cap C=\{i\}$, then we obtain $I=(C\setminus\{j\})\cup K$ and $J=(C\setminus\{i\})\cup K$ and hence $I^{c}\cap J=\{j\}$ and $I\cap J^{c}=\{i\}$. Since we have $j\in C\subset I\cup\{j\}$ and $i\in C\subset J\cup\{i\}$, we obtain $\beta^{I}_{j}=\pm\beta_{C}$ and $\beta^{J}_{i}=\pm\beta_{C}$. By (\ref{eqn_toric_diff_slope}), we obtain $\lim_{s\rightarrow s_{0}}x_{\lambda,I}(s)=\lim_{s\rightarrow s_{0}}x_{\mu,J}(s)$ for some $\lambda,\mu\in\bb{X}^{\ast}(H)$. 
\end{proof}

Now we study the behavior of $\mca{E}_{\mf{C},s}(\lambda,p_{I})$ under the wall-crossing, where $I$ satisfies $|I^{c}\cap C|=1$ and $I\cap C^{c}=K$ for a fixed subset $K\subset C^{c}$. We may assume that $\{\mf{a}_{i}\}_{i\in K}$ is linearly independent and $\{\mf{a}_{i}\}_{i\in C\cup K}$ spans $\bb{X}_{\ast}(H)$. In this case, the number of such $I\in\bb{B}$ is given by $|C|$. We fix $m_{i}\in\bb{Z}$ for each $i\in C\cup K$ such that $\cap_{i\in C\cup K}\mca{H}^{s_{0}}_{i,m_{i}}\eqqcolon\{x_{0}\}$ is not empty. We note that this implies that $\sum_{i\in C_{+}}m_{i}-\sum_{i\in C_{-}}m_{i}=\langle s_{0},\beta_{C}\rangle$. We consider all $(\lambda,p_{I})\in\bb{F}$ such that $\lim_{s\rightarrow s_{0}}x_{\lambda,I}(s)=x_{0}$. The number of such $(\lambda,p_{I})\in\bb{F}$ is also given by $|C|$. By choosing $s_{+},s_{-}$ sufficiently close to $s_{0}$, we may assume that there exists a convex neighborhood $U$ of $x_{0}\in\mf{h}^{\ast}_{\bb{R}}$ containing all $x_{\lambda,I}(s)$ such that any hyperplane of the form $\mca{H}^{s}_{i,m}$ does not intersect $U$ unless $i\in C\cup K$ and $m=m_{i}$ for any $s\in\gamma$. It is enough to consider the alcoves which intersect with $U$. We note that such an alcove $A$ is characterized by the sign $\epsilon:C\cup K\rightarrow\{\pm\}$ such that $A\subset\cap_{i\in C\cup K}\mca{H}^{s,\epsilon(i)}_{i,m_{i}}$. 

Let $\mf{h}_{C}$ be the $\bb{R}$-span of $\{\mf{a}_{i}\}_{i\in C}$ and $\mf{h}_{K}$ be the $\bb{R}$-span of $\{\mf{a}_{i}\}_{i\in K}$. By the choice of $K$, we obtain a decomposition $\mf{h}_{\bb{R}}\cong\mf{h}_{C}\times\mf{h}_{K}$ and this induces a decomposition $\mf{h}^{\ast}_{\bb{R}}\cong\mf{h}_{C}^{\ast}\times\mf{h}_{K}^{\ast}$ such that $\langle\mf{h}_{C},\mf{h}_{K}^{\ast}\rangle=0$ and $\langle\mf{h}_{K},\mf{h}_{C}^{\ast}\rangle=0$. We may also assume that $U\cong U_{C}\times U_{K}$ for some convex open subsets $U_{C}\subset\mf{h}^{\ast}_{C}$ and $U_{K}\subset\mf{h}^{\ast}_{K}$. Since we have $\mca{H}^{s}_{i,m_{i}}=(\mca{H}^{s}_{i,m_{i}}\cap\mf{h}^{\ast}_{C})\times\mf{h}_{K}^{\ast}$ for each $i\in C$ and $\mca{H}^{s}_{i,m_{i}}=\mf{h}^{\ast}_{C}\times(\mca{H}^{s}_{i,m_{i}}\cap\mf{h}_{K}^{\ast})$ for each $i\in K$, they induce hyperplane arrangements on $\mf{h}^{\ast}_{C}$ and $\mf{h}^{\ast}_{K}$. For each alcove $A$ with $A\cap U\neq\emptyset$, there is an alcove $A_{C}$ in $\mf{h}^{\ast}_{C}$ and an alcove $A_{K}$ in $\mf{h}^{\ast}_{K}$ such that $A\cap U\cong (A_{C}\cap U_{C})\times(A_{K}\cap U_{K})$. In particular, we can consider alcoves for $\mf{h}^{\ast}_{C}$ and $\mf{h}^{\ast}_{K}$ separately in $U$.

Since $\{\mf{a}_{i}\}_{i\in K}$ is linearly independent, $\cap_{i\in K}\mca{H}^{s,\epsilon(i)}_{i,m_{i}}\cap U_{K}$ is not empty and its volume is away from 0 along $s\in\gamma$ for any sign $\epsilon:K\rightarrow\{\pm\}$. For each sign $\epsilon:C\rightarrow\{\pm\}$, we set $\Delta_{\epsilon}(s)\coloneqq\cap_{i\in C}\mca{H}^{s,\epsilon(i)}_{i,m_{i}}\cap\mf{h}^{\ast}_{C}$. We define $\epsilon_{\pm}:C\rightarrow\{\pm\}$ by $\epsilon(i)=\pm$ for any $i\in C_{+}$ and $\epsilon(i)=\mp$ for any $i\in C_{-}$. 

\begin{lemma}\label{lem_simplex}
If $\pm\langle s-s_{0},\beta_{C}\rangle>0$, then $\Delta_{\epsilon}(s)\neq\emptyset$ unless $\epsilon=\epsilon_{\mp}$. Moreover, the volume of $\Delta_{\epsilon}(s)\cap U_{C}$ is away from 0 along $s\in\gamma$ unless $\epsilon=\epsilon_{\pm}$ and the volume of $\Delta_{\epsilon_{\pm}}(s)\cap U_{C}$ is proportional to $|\langle s-s_{0},\beta_{C}\rangle|^{|C|-1}$. 
\end{lemma}

\begin{proof}
By definition, $x\in\Delta_{\epsilon}(s)$ if and only if $\epsilon(i)\cdot(\langle x,\mf{a}_{i}\rangle+s_{i}-m_{i})>0$ for any $i\in C$. Since we have 
\begin{align*}
\sum_{i\in C_{+}}(\langle x,\mf{a}_{i}\rangle+s_{i}-m_{i})-\sum_{i\in C_{-}}(\langle x,\mf{a}_{i}\rangle+s_{i}-m_{i})=\langle s-s_{0},\beta_{C}\rangle,
\end{align*}
such an $x$ does not exist for $\epsilon=\epsilon_{\mp}$. If $\epsilon\neq\epsilon_{\pm}$ and $\epsilon\neq\epsilon_{\mp}$, then $\Delta_{\epsilon}(s_{0})$ is also not empty and the volume of $\Delta_{\epsilon}(s_{0})\cap U_{C}$ is positive. This implies the second statement. 

If $\epsilon=\epsilon_{\pm}$, then $\Delta_{\epsilon_{\pm}}(s)$ is a $(|C|-1)$-dimensional simplex. By our choice of $U$, every vertex of $\Delta_{\epsilon_{\pm}}(s)$ is contained in $U_{C}$ and hence $\Delta_{\epsilon_{\pm}}(s)\subset U_{C}$. Since each edge of $\Delta_{\epsilon_{\pm}}(s)$ has length proportional to $|\langle s-s_{0},\beta_{C}\rangle|$, its volume is proportional to $|\langle s-s_{0},\beta_{C}\rangle|^{|C|-1}$. 
\end{proof}

For each sign $\epsilon:C\cup K\rightarrow\{\pm\}$ satisfying $\cap_{i\in C\cup K}\mca{H}^{s,\epsilon(i)}_{i,m_{i}}\neq\emptyset$, we denote by $A_{\epsilon,x_{0}}(s)\in\mr{Alc}_{s}$ the unique alcove which is contained in $\cap_{i\in C\cup K}\mca{H}^{s,\epsilon(i)}_{i,m_{i}}$ and intersect with $U$. By definition, we obtain
\begin{align*}
\mca{E}(A_{\epsilon,x_{0}}(s))=\mca{L}\left(-\sum_{i\in\epsilon^{-1}(-)}\varepsilon^{\ast}_{i}\right)\otimes\mca{L}_{x_{0}}
\end{align*}
for any $s\in\gamma$, where we set 
\begin{align*}
\mca{L}_{x_{0}}\coloneqq\mca{L}\left(\sum_{i\in C\cup K}m_{i}\varepsilon^{\ast}_{i}+\sum_{i\notin C\cup K}\lfloor\langle x_{0},\mf{a}_{i}\rangle+\iota(s_{0})_{i}\rfloor\varepsilon^{\ast}_{i}\right).
\end{align*}
Lemma~\ref{lem_simplex} implies that for any subset $C'\subset C$ and $K'\subset K$, $\mca{L}(-\sum_{i\in C'\cup K'}\varepsilon^{\ast}_{i})\otimes\mca{L}_{x_{0}}$ is contained in $\bb{B}_{X,s_{\pm}}$ if and only if $C'\neq C_{\pm}$. In particular, $\mca{L}(-\sum_{i\in C_{-}\cup K'}\varepsilon^{\ast}_{i})\otimes\mca{L}_{x_{0}}$ is contained in $\bb{B}_{X,s_{+}}$ but not in $\bb{B}_{X,s_{-}}$, and $\mca{L}(-\sum_{i\in C_{+}\cup K'}\varepsilon^{\ast}_{i})\otimes\mca{L}_{x_{0}}$ is contained in $\bb{B}_{X,s_{-}}$ but not in $\bb{B}_{X,s_{+}}$. In summary, we obtained the following.

\begin{prop}\label{prop_toric_wall_crossing_characterization}
For any element $\mca{E}\in\bb{B}_{X,s_{+}}$, $\mca{E}$ is not contained in $\bb{B}_{X,s_{-}}$ if and only if the volume $\mr{Vol}(A(s))$ of $A(s)$ vanishes under the limit $s\rightarrow s_{0}$, where $A(s)\in\mr{Alc}_{s}$ is the alcove satisfying $\mca{E}(A(s))=\mca{E}$ for any $s\in\mf{s}^{\ast}_{\bb{R},\mr{reg}}$ contained in the same K\"ahler alcove as $s_{+}$. If this holds, then the order of vanishing of $\mr{Vol}(A(s))$ at $s=s_{0}$ is given by $|C|-1$ and $\mca{L}\left(\sum_{i\in C_{-}}\varepsilon^{\ast}_{i}-\sum_{i\in C_{+}}\varepsilon^{\ast}_{i}\right)\otimes\mca{E}$ is contained in $\bb{B}_{X,s_{-}}$ but not in $\bb{B}_{X,s_{+}}$. Moreover, $\mca{L}\left(\sum_{i\in C'\cap C_{-}}\varepsilon^{\ast}_{i}-\sum_{i\in C'\cap C_{+}}\varepsilon^{\ast}_{i}\right)\otimes\mca{E}$ is contained in both $\bb{B}_{X,s_{+}}$ and $\bb{B}_{X,s_{-}}$ for any subset $\emptyset\neq C'\subsetneq C$.
\end{prop}

Now the wall-crossing formula relating $\bb{B}_{X,s_{+}}$ and $\bb{B}_{X,s_{-}}$ follows from the following lemma.

\begin{lemma}\label{lem_toric_wall_crossing_formula}
For any signed circuit $C$ satisfying $\langle\eta,\beta_{C}\rangle>0$, there exists an exact sequence
\begin{align}\label{eqn_toric_wall_crossing_exact_seq}
0\rightarrow\mca{W}_{|C|}\rightarrow\cdots\rightarrow\mca{W}_{1}\rightarrow\mca{W}_{0}\rightarrow0
\end{align}
of vector bundles on $X$, where 
\begin{align*}
\mca{W}_{k}\coloneqq\bigoplus_{\substack{C'\subset C\\|C'|=k}}v^{k}\mca{L}\left(\sum_{i\in C'\cap C_{-}}\varepsilon^{\ast}_{i}-\sum_{i\in C'\cap C_{+}}\varepsilon^{\ast}_{i}\right).
\end{align*}
\end{lemma}

\begin{proof}
Since we have $\langle\eta,\beta_{C}\rangle=\sum_{j\in C\cap I^{c}}\langle\eta,\beta^{I}_{j}\rangle\langle\mf{b}_{j},\beta_{C}\rangle>0$, we must have $C_{+}\cap I^{c}_{+}\neq\emptyset$ or $C_{-}\cap I^{c}_{-}\neq\emptyset$ for any $I\in\bb{B}$. This implies that the subvariety of $X$ defined by $x_{i}=0$ for any $i\in C_{+}$ and $y_{i}=0$ for any $i\in C_{-}$ is empty. By considering the Koszul complex for these equations, we obtain an exact sequence of the form (\ref{eqn_toric_wall_crossing_exact_seq}). 
\end{proof}

In particular, Proposition~\ref{prop_toric_wall_crossing_characterization} and Lemma~\ref{lem_toric_wall_crossing_formula} imply the first part of Conjecture~\ref{conj_wall_crossing} by taking $l=1$, $n_{0}=0$, $n_{1}=|C|$, $\bb{B}^{0}_{s_{\pm},w}=\bb{B}_{X,s_{-}}\cap\bb{B}_{X,s_{+}}$, and $\bb{B}^{1}_{s_{\pm},w}=\bb{B}_{X,s_{\pm}}\setminus\bb{B}^{0}_{s_{\pm},w}$. In section 5.11, we will show that the central charge of $\mca{C}(A(s))$ is given by $\mr{Vol}(A(s))$ for any $A(s)\in\mr{Alc}_{s}$. This implies the second part of Conjecture~\ref{conj_wall_crossing}. 

As another application of these results, we obtain the following. 

\begin{corollary}\label{cor_weak_generator}
For any $s\in\mf{s}^{\ast}_{\mr{reg}}$, the vector bundle $\mca{T}_{\mf{C},s}\coloneqq\oplus_{I\in\bb{B}}\mca{E}_{\mf{C},s}(p_{I})$ weakly generate $D(\mr{QCoh}(X))$.
\end{corollary}

\begin{proof}
We note that any line bundle of the form $\mca{L}(\lambda)$ for $\lambda\in\bb{X}^{\ast}(T)$ is contained in $\bb{B}_{X,s'}$ for some $s'\in\mf{s}^{\ast}_{\mr{reg}}$ by Lemma~\ref{can_slope}. By connecting $s$ and $s'$ by a generic path and applying Proposition~\ref{prop_toric_wall_crossing_characterization} and Lemma~\ref{lem_toric_wall_crossing_formula} each time when the path crosses a wall, we obtain that $\mca{L}(\lambda)$ is contained in the full triangulated subcategory of $\DCoh(X)$ generated by $\{\mca{E}_{\mf{C},s}(p_{I})\}_{I\in\bb{B}}$ for any $\lambda\in\bb{X}^{\ast}(T)$. In particular, $\RHom(\mca{T}_{\mf{C},s},\mca{F})=0$ implies that $R\Gamma(\mca{F}\otimes\mca{L})=0$ for any sufficiently ample line bundles $\mca{L}$. This implies $\mca{F}\cong0$. 
\end{proof}

\subsection{Linear programming}
 
In this section, we collect some results about linear programming which will be used in the proof of Conjecture~\ref{conj_general_tilting_bundle} and Conjecture~\ref{categorical_stable_basis} for toric hyper-K\"ahler manifolds. Our reference for the theory of linear programming is \cite{BaKe}. 

We first prepare some notations about sign vectors. For $x\in\bb{R}$, we set 
\begin{align*}
\sigma(x)\coloneqq\begin{cases}
+ &\mbox{ if } x>0\\
- &\mbox{ if } x<0\\
0 &\mbox{ if } x=0
\end{cases}
\end{align*}
and for $x=(x_{1},\ldots,x_{n})\in\bb{R}^{n}$, we set $\sigma(x)\coloneqq(\sigma(x_{1}),\ldots,\sigma(x_{n}))\in\{+,-,0\}^{n}$. We will be interested in the sign patterns $\sigma(V)\coloneqq\{\sigma(x)\mid x\in V\}\subset\{+,-,0\}^{n}$ for a vector subspace $V\subset\bb{R}^{n}$.  

We set $E\coloneqq\{1,\ldots,n\}$. For a sign vector $Y\in\{+,-,0\}^{E}$, we define $Y^{+}\coloneqq\{i\in E\mid Y_{i}=+\}$, $Y^{-}\coloneqq\{i\in E\mid Y_{i}=-\}$, and $Y^{0}\coloneqq\{i\in E\mid Y_{i}=0\}$. We define its support by $\Supp(Y)\coloneqq Y^{+}\cup Y^{-}$. For a subset $I\subset E$, we write $Y_{I}\geq0$ if $Y_{i}\in\{+,0\}$, $Y_{I}\leq0$ if $Y_{i}\in\{-,0\}$, and $Y_{I}=0$ if $Y_{i}=0$ for any $i\in I$. 

\begin{dfn}
Two sign vectors $Y,Z\in\{+,-,0\}^{E}$ are called \emph{orthogonal} if 
\begin{align*}
(Y^{+}\cap Z^{+})\cup(Y^{-}\cap Z^{-})\neq\emptyset\Leftrightarrow(Y^{+}\cap Z^{-})\cup(Y^{-}\cap Z^{+})\neq\emptyset.
\end{align*}
This is denoted by $Y\perp Z$. For a subset $\mca{F}\in\{+,-,0\}^{E}$, the set 
\begin{align*}
\mca{F}^{\perp}\coloneqq\{Y\in\{+,-,0\}^{E}\mid Y\perp Z\mbox{ for any }Z\in\mca{F}\}
\end{align*}
is called the \emph{orthogonal complement} of $\mca{F}$. 
\end{dfn}

For a vector subspace $V\subset\bb{R}^{n}$, we denote by $V^{\perp}\subset\bb{R}^{n}$ the orthogonal complement of $V\subset\bb{R}^{n}$ with respect to the standard inner product on $\bb{R}^{n}$.

\begin{prop}\label{prop_om_orth_comp}
For any vector subspace $V\subset\bb{R}^{n}$, we have $\sigma(V)^{\perp}=\sigma(V^{\perp})$.  
\end{prop}

\begin{proof}
See Corollary 5.42 in \cite{BaKe}. 
\end{proof}

\begin{dfn}
For a subset $\mca{F}\subset\{+,-,0\}^{E}$ and disjoint subsets $I,J\subset E$, we set 
\begin{align*}
\mca{F}\setminus I/J\coloneqq\{Y\in\{+,-,0\}^{E\setminus(I\cup J)}\mid \exists Z\in\mca{F}\mbox{ s.t. }Z_{i}=0\mbox{ for }i\in I\mbox{ and }Z_{e}=Y_{e}\mbox{ for }e\in E\setminus(I\cup J) \}. 
\end{align*}
This is called the \emph{minor} of $\mca{F}$ obtained by deleting $I$ and contracting $J$. 
\end{dfn}

\begin{lemma}\label{lem_om_minor}
For any vector subspace $V\subset\bb{R}^{n}$ and disjoint subsets $I,J\subset E$, we have 
\begin{align*}
(\sigma(V)\setminus I/J)^{\perp}=\sigma(V)^{\perp}\setminus J/I.
\end{align*}
\end{lemma}

\begin{proof}
This follows from Proposition~\ref{prop_om_orth_comp}. See also Lemma 5.51 and Lemma 5.52 in \cite{BaKe}. 
\end{proof}

\begin{dfn}
For a subset $\mca{F}\subset\{+,-,0\}^{E}$, nonzero $Y\in\mca{F}$ is called \emph{elementary} sign vector of $\mca{F}$ if $\emptyset\neq\Supp(Z)\subset\Supp(Y)$ implies $\Supp(Z)=\Supp(Y)$ for any $Z\in\mca{F}$. The set of all elementary sign vectors of $\mca{F}$ is denoted by $\mr{elem}(\mca{F})$. 
\end{dfn}

\begin{prop}\label{prop_om_elem}
For any vector subspace $V\subset\bb{R}^{n}$, we have $\mr{elem}(\sigma(V))^{\perp}=\sigma(V)^{\perp}$. 
\end{prop}

\begin{proof}
See Corollary 5.37 in \cite{BaKe}.
\end{proof}

\begin{dfn}
For $Y,Z\in\{+,-,0\}^{E}$, we write $Y\preceq Z$ and say that $Y$ conforms to $Z$ if $Y^{+}\subset Z^{+}$ and $Y^{-}\subset Z^{-}$. This relation defines a partial order $\preceq$ on $\{+,-,0\}^{E}$.  
\end{dfn}

\begin{lemma}\label{lem_om_conform}
For any vector subspace $V\subset\bb{R}^{n}$, $\mr{elem}(\sigma(V))$	 coincides with the set of minimal nonzero elements of $\sigma(V)$ with respect to the partial order $\preceq$. 
\end{lemma}

\begin{proof}
See Lemma 5.30 in \cite{BaKe}.
\end{proof}

\begin{prop}[Minty's Lemma]\label{prop_Minty}
For any vector subspace $V\subset\bb{R}^{n}$ and every partition $E=R\sqcup G\sqcup B\sqcup W$ with $e\in R\sqcup G$, exactly one of the following holds:
\begin{itemize}
\item There exists $Y\in\sigma(V)$ such that $e\in\Supp(Y)$, $Y_{R}\geq0$, $Y_{G}\leq0$, and $Y_{W}=0$.
\item There exists $Z\in\sigma(V^{\perp})$ such that $e\in\Supp(Z)$, $Z_{R}\geq0$, $Z_{G}\leq0$, and $Z_{B}=0$. 
\end{itemize}
\end{prop}

\begin{proof}
See Proposition 5.12 in \cite{BaKe}. 
\end{proof}

In the below, we will consider the case $V=\Ker(\mf{b})\otimes_{\bb{Z}}\bb{R}\subset\bb{X}^{\ast}(T)\otimes_{\bb{Z}}\bb{R}\cong\bb{R}^{n}$, where the identification $\bb{X}^{\ast}(T)\otimes_{\bb{Z}}\bb{R}\cong\bb{R}^{n}$ is given by the fixed basis $\{\varepsilon^{\ast}_{1},\ldots,\varepsilon^{\ast}_{n}\}$. We note that the sign vectors $\sigma(V)$ does not change without tensoring $\bb{R}$. In this case, we have $V^{\perp}=\Ker(\mf{a})\otimes_{\bb{Z}}\bb{R}$ and hence $\mr{elem}(\sigma(V^{\perp}))=\{\sigma(\beta_{C})\mid C=C_{+}\sqcup C_{-}\mbox{ : signed circuit}\}$. 

\subsection{Toric varieties}

In this section, we recall a description of cohomology of line bundles on toric varieties. Since we only consider semi-projective toric varieties in this paper, we restrict our attention to these cases. Let $0\leq r\leq N$ be nonnegative integers and consider an exact sequence of tori
\begin{align*}
1\rightarrow T^{r}\rightarrow T^{m}\rightarrow T^{m-r}\rightarrow1,
\end{align*}
where $T^{k}\coloneqq(\bb{C}^{\times})^{k}$ for $k\in\bb{Z}_{\geq0}$. Let  
\begin{align*}
0\rightarrow\bb{X}_{\ast}(T^{r})\xrightarrow{^{t}\!\mb{B}}\bb{X}_{\ast}(T^{m})\xrightarrow{\mb{A}}\bb{X}_{\ast}(T^{m-r})\rightarrow0\\
0\rightarrow\bb{X}^{\ast}(T^{m-r})\xrightarrow{^{t}\!\mb{A}}\bb{X}^{\ast}(T^{m})\xrightarrow{\mb{B}}\bb{X}^{\ast}(T^{r})\rightarrow0
\end{align*}
be the associated exact sequence of cocharacter and character lattices. We set  $\mb{a}_{i}\coloneqq\mb{A}(\varepsilon_{i})$ and $\mb{b}_{i}\coloneqq\mb{B}(\varepsilon^{\ast}_{i})$ for a fixed basis $\{\varepsilon_{1},\ldots,\varepsilon_{m}\}$ of $\bb{X}_{\ast}(T^{m})$ and its dual basis $\{\varepsilon^{\ast}_{1},\ldots,\varepsilon^{\ast}_{m}\}$ of $\bb{X}^{\ast}(T^{m})$. We fix a generic element $\eta\in\sum_{i=1}^{m}\bb{Z}_{\geq0}\mb{b}_{i}$ such that if $\eta$ is contained in a cone of the form $\bb{R}_{\geq0}\mb{b}_{i_{1}}+\cdots\bb{R}_{\geq0}\mb{b}_{i_{l}}$ for $\{i_{1},\ldots, i_{l}\}\subset\{1,\ldots, m\}$, then $\{\mb{b}_{i_{1}},\ldots,\mb{b}_{i_{l}}\}$ generates $\bb{X}^{\ast}(T^{r})\otimes_{\bb{Z}}\bb{R}$. We set 
\begin{align*}
\Omega_{\eta}\coloneqq\{I\subset\{1,\ldots,m\}\mid \eta\in\sum_{i\in I}\bb{R}_{\geq0}\mb{b}_{i}\}
\end{align*}
and define a fan $\Sigma$ in $\bb{X}_{\ast}(T^{m-r})\otimes_{\bb{Z}}\bb{R}$ by $\Sigma\coloneqq\left\{\sigma_{I}\mid I\in \Omega_{\eta}\right\}$, where $\sigma_{I}\coloneqq\sum_{j\in I^{c}}\bb{R}_{\geq0}\mb{a}_{j}$. Let $\mca{X}(\Sigma)$ be the toric variety associated with the fan $\Sigma$. By Theorem 2.4 in \cite{HS}, $\mca{X}(\Sigma)$ is isomorphic to the GIT quotient $(\bb{C}^{m})^{\eta-\mr{ss}}/\!/T^{r}$. For simplicity, we assume that $\mca{X}(\Sigma)$ is smooth, i.e., $\{\mb{a}_{j}\}_{j\in I^{c}}$ generates $\bb{X}_{\ast}(T^{m-r})$ over $\bb{Z}$ for any $I\in\Omega_{\eta}$ such that $|I|=r$. In this case, the action of $T^{r}$ on $(\bb{C}^{m})^{\eta-\mr{ss}}$ is free and hence one can define a $T^{m-r}$-equivariant line bundle $\mca{L}(\lambda)$ associated with each character $\lambda\in\bb{X}^{\ast}(T^{m})$.

We set $R(\Sigma)\coloneqq\bb{C}[x_{1},\ldots,x_{m}]$ with an $\bb{X}^{\ast}(T^{m})$-grading given by $\deg(x_{i})=\varepsilon^{\ast}_{i}$.  Let $B(\Sigma)=\left(\prod_{i\in I}x_{i}\mid I\in\Omega_{\eta}\right)$ be a monomial ideal of $R(\Sigma)$. Since $B(\Sigma)$ is generated by homogeneous elements, the local cohomology $H^{i}_{B(\Sigma)}(R(\Sigma))$ of $R(\Sigma)$ with supports in $B(\Sigma)$ is also $\bb{X}^{\ast}(T^{m})$-graded. We denote by $R(\Sigma)_{\lambda}$ and $H^{i}_{B(\Sigma)}(R(\Sigma))_{\lambda}$ the weight $\lambda$ parts of $R(\Sigma)$ and $H^{i}_{B(\Sigma)}(R(\Sigma))$ for any $\lambda\in\bb{X}^{\ast}(T^{m})$.  One can relate them and the cohomology of the line bundle $\mca{L}(\lambda)$ as follows. 

\begin{lemma}\label{lem_coh_local}
For any $\lambda\in\bb{X}^{\ast}(T^{m})$ and $i\geq1$, we have $H^{i}(\mca{X}(\Sigma),\mca{L}(\lambda))^{T^{m}}\cong H^{i+1}_{B(\Sigma)}(R(\Sigma))_{\lambda}$. 
\end{lemma}

\begin{proof}
See for example	Theorem 9.5.7 in \cite{CLS}. 
\end{proof}

Next we recall a description of $H^{i}_{B(\Sigma)}(R(\Sigma))_{\lambda}$ in terms of simplicial cohomology due to Musta\c{t}\v{a} \cite{M}. For each $i=1,\ldots,m$, let $\Delta_{i}\coloneqq\{\mca{I}\subset\Omega_{\eta}\mid i\notin\cup_{I\in\mca{I}}I\}$ be a simplicial complex on $\Omega_{\eta}$. For a subset $M\subset\{1,\ldots,m\}$, we set $\Delta_{M}\coloneqq\cup_{i\in M}\Delta_{i}$. Here, we understand that if $M=\emptyset$, then $\Delta_{M}$ is the void complex which has trivial reduced cohomology. For $\lambda=(\lambda_{1},\ldots,\lambda_{m})\in\bb{Z}^{m}$, we define $\mr{neg}(\lambda)\coloneqq\{i\in\{1,\ldots,m\}\mid\lambda_{i}<0\}$. For a simplicial complex $\Delta$, we denote by $\tilde{H}^{i}(\Delta)$ the $i$-th reduced cohomology group of $\Delta$. 

\begin{lemma}[\cite{M}]\label{lem_local_simplicial}
For each $\lambda\in\bb{Z}^{m}$ and $i\in\bb{Z}_{\geq0}$, we have $H^{i}_{B(\Sigma)}(R(\Sigma))_{\lambda}\cong\tilde{H}^{i-2}(\Delta_{\mr{neg}(\lambda)})$. 
\end{lemma}

\begin{proof}
See Theorem 2.1 in \cite{M}.
\end{proof}

We will also need the following special case of Demazure vanishing theorem. 

\begin{lemma}\label{lem_Demazure_vanishing}
We have $H^{>0}(\mca{X}(\Sigma),\mca{O})=0$. 
\end{lemma}

\begin{proof}
Since $|\Sigma|=\sum_{i=1}^{n}\bb{R}_{\geq0}\mb{a}_{i}$ is convex, this follows from Demazure vanishing theorem, see for example Theorem 9.2.3 in \cite{CLS}.
\end{proof}

For example, we may apply the above results for Lawrence toric varieties. For this, we take $m=2n$ and change the index set $\{1,\ldots,m\}$ by $E_{n}\coloneqq\{\pm1,\ldots,\pm n\}$. We take $\mb{b}_{\pm i}=\pm\mf{b}_{i}$ for $i=1,\ldots,n$. In this case, we can take $\mb{a}_{i}=(\mf{a}_{i},e_{i})\in\bb{Z}^{d}\oplus\bb{Z}^{n}$ and $\mb{a}_{-i}=(0,e_{i})\in\bb{Z}^{d}\oplus\bb{Z}^{n}$, where $\{e_{1},\ldots,e_{n}\}$ is a basis of $\bb{Z}^{n}$. We also take the same $\eta$. By definition, the toric variety $\mca{X}(\Sigma)$ associated with these data is the Lawrence toric variety $\mca{X}$. By our assumption that $\mf{a}_{i}\neq0$ for any $i=1,\ldots,n$, we have $E_{n}\setminus\{e\}\in\Omega_{\eta}$ for any $e\in E_{n}$. In particular, the set of one dimensional cones $\Sigma(1)$ in $\Sigma$ is given by $\Sigma(1)=\{\bb{R}_{\geq0}\mb{a}_{e}\}_{e\in E}$ and hence the ring $R(\Sigma)=\bb{C}[x_{1},\ldots,x_{n},y_{1},\ldots, y_{n}]$ coincides with the Cox's homogeneous coordinate ring \cite{C} for the Lawrence toric variety $\mca{X}$, where we write $y_{i}$ for the variable corresponding to $-i\in E_{n}$. Therefore, we obtain a ring isomorphism 
\begin{align}\label{eqn_Cox_ring}
R(\Sigma)\cong\oplus_{\lambda\in\bb{X}^{\ast}(T^{2n})}\Gamma(\mca{X},\widetilde{\mca{L}}(\lambda))^{T^{2n}}
\end{align}
by Proposition 1.1 in \cite{C}. If we restrict the torus action of $T^{2n}$ on $R(\Sigma)$ to $T$ via $T\rightarrow T^{2n}$ given by $t\mapsto(t,t^{-1})$, then the weight $\mu\in\bb{X}^{\ast}(T)$ part of $R(\Sigma)$ is given by $\bb{C}[u_{1},\ldots,u_{n}]\cdot x^{\mu}$, where we set $u_{i}\coloneqq x_{i}y_{i}$ and $x^{\mu}\coloneqq\prod_{i,\mu_{i}>0}x_{i}^{\mu_{i}}\prod_{i,\mu_{i}<0}y_{i}^{-\mu_{i}}$. Combined with (\ref{eqn_Cox_ring}), we obtain 
\begin{align}\label{eqn_Lawrence_global_section}
\Gamma(\mca{X},\widetilde{\mca{L}}(\mu))^{T}\cong\bb{C}[u_{1},\ldots,u_{n}]\cdot x^{\mu}.
\end{align}
We note that the degree of $x^{\mu}$ with respect to $\bb{S}$-action is given by $\sum_{i}|\mu_{i}|$. For each $i=1,\ldots,n$, the section $x_{i}\in\Gamma(\mca{X},\widetilde{\mca{L}}(\varepsilon^{\ast}_{i}))^{T}$ gives a $\bb{C}[u_{1},\ldots,u_{n}]$-module homomorphism $x_{i}\cdot:\Gamma(\mca{X},\widetilde{\mca{L}}(\mu-\varepsilon^{\ast}_{i}))^{T}\rightarrow\Gamma(\mca{X},\widetilde{\mca{L}}(\mu))^{T}$　of degree 1. In terms of the isomorphism (\ref{eqn_Lawrence_global_section}), we have 
\begin{align}\label{eqn_mult_x_i}
x_{i}\cdot x^{\mu-\varepsilon^{\ast}_{i}}=\begin{cases}
x^{\mu} &\mbox{ if }\mu_{i}>0,\\
u_{i}\cdot x^{\mu} &\mbox{ if }\mu_{i}\leq0.
\end{cases}
\end{align}
Similarly, the section $y_{i}\in\Gamma(\mca{X},\widetilde{\mca{L}}(-\varepsilon^{\ast}_{i}))^{T}$ gives a $\bb{C}[u_{1},\ldots,u_{n}]$-module homomorphism $y_{i}\cdot:\Gamma(\mca{X},\widetilde{\mca{L}}(\mu+\varepsilon^{\ast}_{i}))^{T}\rightarrow\Gamma(\mca{X},\widetilde{\mca{L}}(\mu))^{T}$ given by 
\begin{align}\label{eqn_mult_y_i}
y_{i}\cdot x^{\mu+\varepsilon^{\ast}_{i}}=\begin{cases}
x^{\mu} &\mbox{ if }\mu_{i}<0,\\
u_{i}\cdot x^{\mu} &\mbox{ if }\mu_{i}\geq0.
\end{cases}
\end{align}


\subsection{Tilting bundles}

In this section, we check Conjecture~\ref{conj_general_tilting_bundle} for toric hyper-K\"ahler manifolds, i.e., $\mca{T}_{\mf{C},s}\coloneqq\oplus_{I\in\bb{B}}\mca{E}_{\mf{C},s}(p_{I})$ is a tilting bundle on $X$. The main result of this section is proved via different methods by McBreen-Webster \cite{MW} and \v{S}penko-Van den Bergh \cite{SV} independently. We give still another direct proof of it. 

We note that for any $t,t'\in\bb{R}$, we have 
\begin{align*}
\lfloor t\rfloor+\lfloor t' \rfloor\leq\lfloor t+t'\rfloor\leq\lfloor t\rfloor+\lfloor t' \rfloor+1,\\
\lfloor t\rfloor-\lfloor t' \rfloor-1\leq\lfloor t-t'\rfloor\leq\lfloor t\rfloor-\lfloor t' \rfloor.
\end{align*} 
These inequalities easily imply that for any $A\in\mr{Alc}_{s}$ and signed circuit $C$, we have 
\begin{align*}
\lfloor\langle s,\beta_{C}\rangle\rfloor-|C_{+}|+1\leq\langle\mu_{A},\beta_{C}\rangle\leq\lfloor\langle s,\beta_{C}\rangle\rfloor+|C_{-}|,
\end{align*}
where $\mu_{A}$ is defined as in (\ref{eqn_mu_A}). In particular, we have 
\begin{align}\label{eqn_ineq_mu}
-|C|+1\leq\langle\mu_{A}-\mu_{A'},\beta_{C}\rangle\leq|C|-1
\end{align}
for any $A,A'\in\mr{Alc}_{s}$. 

Recall that one can associate a $\bb{T}$-equivariant line bundle $\widetilde{\mca{L}}\left(\lambda\right)$ on the Lawrence toric variety $\mca{X}$ for each $\lambda\in\bb{X}^{\ast}(T)$. We define $\pi:\bb{X}^{\ast}(T^{2n})\rightarrow\bb{X}^{\ast}(T)$ by $\pi(\lambda_{1},\ldots,\lambda_{n},\lambda_{-1},\ldots,\lambda_{-n})=\sum_{i=1}^{n}(\lambda_{i}-\lambda_{-i})\varepsilon^{\ast}_{i}$. By Lemma~\ref{lem_coh_local} and Lemma~\ref{lem_local_simplicial}, we obtain 
\begin{align}\label{eqn_Lawrence_simplicial}
H^{>0}(\mca{X},\widetilde{\mca{L}}(\lambda))^{T}\cong\oplus_{\tilde{\lambda}\in\pi^{-1}(\lambda)}\tilde{H}^{\geq0}(\Delta_{\mr{neg}(\tilde{\lambda})}).	
\end{align}
We first prove the vanishing of higher extensions on the level of Lawrence toric variety. 

\begin{prop}\label{prop_toric_higher_vanishing_tilting}
For any $A,A'\in\mr{Alc}_{s}$, we have $H^{>0}(\mca{X},\widetilde{\mca{L}}(\mu_{A}-\mu_{A'}))=0$.
\end{prop}

\begin{proof}
It is enough to prove $H^{>0}(\mca{X},\widetilde{\mca{L}}(\mu_{A}-\mu_{A'}))^{T}=0$ since we have
\begin{align*}
H^{>0}(\mca{X},\widetilde{\mca{L}}(\mu_{A}-\mu_{A'}))\cong\oplus_{\alpha\in\bb{X}^{\ast}(H)}H^{>0}(\mca{X},\widetilde{\mca{L}}(\mu_{A+\alpha}-\mu_{A'}))^{T}
\end{align*}
We set $\mu\coloneqq\mu_{A}-\mu_{A'}\in\bb{X}^{\ast}(T)$ and $Y\coloneqq\sigma(\mu)\in\{+,-,0\}^{n}$. We note that if $\lambda\in\bb{X}^{\ast}(T)$ satisfies $\mf{b}(\lambda)=0$, then Lemma~\ref{lem_Demazure_vanishing} and (\ref{eqn_Lawrence_simplicial}) imply that $\tilde{H}^{\geq0}(\Delta_{\mr{neg}(\tilde{\lambda})})=0$ for any lift $\tilde{\lambda}\in\pi^{-1}(\lambda)$. Therefore, it is enough to prove that for any lift $\tilde{\mu}\in\pi^{-1}(\mu)$, there exists $\lambda\in\Ker(\mf{b})$ and $\tilde{\lambda}\in\pi^{-1}(\lambda)$ such that $\mr{neg}(\tilde{\mu})=\mr{neg}(\tilde{\lambda})$. If $\pm\mu_{i}>0$, then the possibilities for the set $\mr{neg}(\tilde{\mu}_{i},\tilde{\mu}_{-i})$ is $\emptyset$, $\{\mp i\}$, or $\{i,-i\}$. If $\mu_{i}=0$, then the possibilities for the set $\mr{neg}(\tilde{\mu}_{i},\tilde{\mu}_{-i})$ is $\emptyset$ or $\{i,-i\}$, which is contained in the possibilities for $\mu_{i}\neq0$. This implies that it is enough to prove the existence of $\lambda\in\Ker(\mf{b})$ such that $\sigma(\mu_{i})=\sigma(\lambda_{i})$ if $\mu_{i}\neq0$, i.e., $Y|_{\Supp(Y)}\in\sigma(\Ker(\mf{b}))/Y^{0}$. By Proposition~\ref{prop_om_orth_comp}, Lemma~\ref{lem_om_minor}, and Proposition~\ref{prop_om_elem}, we have 
\begin{align*}
\sigma(\Ker(\mf{b}))/Y^{0}=\mr{elem}(\sigma(\Ker(\mf{a}))\setminus Y^{0})^{\perp}.
\end{align*}
Since $\mr{elem}(\sigma(\Ker(\mf{a}))\setminus Y^{0})=\{\sigma(\beta_{C})\mid C\mbox{ : signed circuit},C\subset\Supp(Y)\}$, if $Y|_{\Supp(Y)}\notin\sigma(\Ker(\mf{b}))/Y^{0}$, then there exists a signed circuit $C\subset\Supp(Y)$ such that $Y$ is not orthogonal to $\sigma(\beta_{C})$. We may assume that $(Y^{+}\cap C_{-})\cup(Y^{-}\cap C_{+})=\emptyset$. This and $C\subset\Supp(Y)$ imply that for each $i\in C_{\pm}$, we have $\pm\mu_{i}\geq1$ and hence $\langle\mu,\beta_{C}\rangle\geq|C|$ which contradicts (\ref{eqn_ineq_mu}). 
 \end{proof}

\begin{corollary}\label{cor_toric_tilting}
The vector bundle $\mca{T}_{\mf{C},s}$ is a tilting bundle on $X$.
\end{corollary}

\begin{proof}
By Corollary~\ref{cor_weak_generator}, it is enough to prove $\Ext^{>0}(\mca{T}_{\mf{C},s},\mca{T}_{\mf{C},s})=0$, i.e., $H^{>0}(X,\mca{L}(\mu_{A}-\mu_{A'}))=0$ for any $A,A'\in\mr{Alc}_{s}$. Let $i:\{0\}\hookrightarrow\mf{s}^{\ast}$ be the inclusion and recall the morphism $\mu_{\mca{X}}:\mca{X}\rightarrow\mf{s}^{\ast}$ defined before Lemma~\ref{lem_toric_deform}. Since $\mu_{\mca{X}}$ is flat, $i$ and $\mu_{\mca{X}}$ are Tor independent and hence the base change formula implies 
\begin{align}\label{eqn_tor_indep_bc1}
R\Gamma(X,\mca{L}(\mu_{A}-\mu_{A'}))\cong Li^{\ast}R\mu_{\mca{X}\ast}\widetilde{\mca{L}}(\mu_{A}-\mu_{A'}).
\end{align}
Since $R\mu_{\mca{X}\ast}\widetilde{\mca{L}}(\mu_{A}-\mu_{A'})$ is concentrated on cohomological degree 0 by Proposition~\ref{prop_toric_higher_vanishing_tilting}, the RHS of (\ref{eqn_tor_indep_bc1}) has vanishing cohomology at positive degree. This proves $H^{>0}(X,\mca{L}(\mu_{A}-\mu_{A'}))=0$. 
\end{proof}

Combined with Proposition~\ref{prop_toric_bar_invariance_E(A)} and Lemma~\ref{lem_toric_bar_invariance_C(A)}, we obtain Conjecture~\ref{K-theoretic_canonical_basis} for toric hyper-K\"ahler manifolds. 

\begin{corollary}\label{cor_toric_canonical_basis}
$\bb{B}_{X,s}$ (resp. $\bb{B}_{L,s}$) is a $\bb{Z}[v,v^{-1}]$-basis of $K_{\bb{T}}(X)$ (resp. $K_{\bb{T}}(L)$).  
\end{corollary}

\begin{proof}
Corollary~\ref{cor_toric_tilting} implies that $\bb{B}_{X,s}$ generates $K_{\bb{T}}(X)$ over $\bb{Z}[v,v^{-1}]$ and hence it is a basis. Since the pairing $(-:-)$ defined in (\ref{usualinnerprod}) induces a perfect pairing between $K_{\bb{T}}(X)$ and $K_{\bb{T}}(L)$, the pairing $(-||-)$ also gives a perfect pairing between $K_{\bb{T}}(X)$ and $K_{\bb{T}}(L)$. Therefore, the dual basis $\bb{B}_{L,s}$ of $\bb{B}_{X,s}$ is a $\bb{Z}[v,v^{-1}]$-basis of $K_{\bb{T}}(L)$. 
\end{proof}

As a module over $\bb{C}[u_{1},\ldots,u_{n}]\cong\bb{C}[\mf{t}^{\ast}]$, we have $R(\Sigma)\cong\oplus_{\mu\in\bb{X}^{\ast}(T)}\bb{C}[\mf{t}^{\ast}]\cdot x^{\mu}$ and the $\bb{C}[\mf{s}^{\ast}]$-module structure coming from the morphism $\mu_{\mca{X}}:\mca{X}\rightarrow\mf{s}^{\ast}$ is the one induced from the natural inclusion $\bb{C}[\mf{s}^{\ast}]\hookrightarrow\bb{C}[\mf{t}^{\ast}]$. We also denote by $x^{\mu}\in\Gamma(X,\mca{L}(\mu))^{H}$ the section coming from $x^{\mu}\in\Gamma(\mca{X},\widetilde{\mca{L}}(\mu))^{T}$. As a corollary of (\ref{eqn_tor_indep_bc1}), we obtain the following. 

\begin{lemma}\label{lem_toric_global_section}
If $H^{>0}(\mca{X},\widetilde{\mca{L}}(\mu))=0$ for $\mu\in\bb{X}^{\ast}(T)$, then we have $\Gamma(X,\mca{L}(\mu))^{H}\cong\bb{C}[\mf{h}^{\ast}]\cdot x^{\mu}$. 
\end{lemma}

As in section 3.6, we set $\mca{A}_{\mf{C},s}\coloneqq\End(\mca{T}_{\mf{C},s})^{\mr{opp}}$. By Corollary~\ref{cor_toric_tilting}, we obtain a derived equivalence
\begin{align}\label{eqn_toric_derived_equiv}
\psi_{\mf{C},s}:\DCohT(X)\cong D^{b}(\mca{A}_{\mf{C},s}\mbox{-}\mr{gmod}^{H})
\end{align}
given by $\psi_{\mf{C},s}(\mca{F})=\RHom(\mca{T}_{\mf{C},s},\mca{F})$ for $\mca{F}\in\DCohT(X)$. 

We next give a presentation of the ring $\mca{A}_{\mf{C},s}$. We set 
\begin{align}\label{eqn_mu_I}
\mu_{I}\coloneqq-\sum_{i\in I_{-}}\varepsilon^{\ast}_{i}+\sum_{j\in I^{c}}\lfloor\langle s,\beta^{I}_{j}\rangle\rfloor\varepsilon^{\ast}_{j}
\end{align}
so that $\mca{T}_{\mf{C},s}=\oplus_{I\in\bb{B}}\mca{L}(\mu_{I})$ for any $I\in\bb{B}$. Let $e_{I}\in\mca{A}_{\mf{C},s}$ be the idempotent corresponding to the identity map in $\Hom(\mca{L}(\mu_{I}),\mca{L}(\mu_{I}))$. For $I,J\in\bb{B}$, the $H$-weight space of $e_{I}\mca{A}_{\mf{C},s}e_{J}$ of weight $\alpha\in\bb{X}^{\ast}(H)$ is given by $\Hom(\mca{L}(\mu_{I}+\alpha),\mca{L}(\mu_{J}))^{H}\cong\bb{C}[\mf{h}^{\ast}]\cdot x^{\mu_{J}-\mu_{I}-\alpha}$ by Lemma~\ref{lem_toric_global_section}. We set $m^{\alpha}_{IJ}\coloneqq x^{\mu_{J}-\mu_{I}-\alpha}\in e_{I}\mca{A}_{\mf{C},s}e_{J}$. We note that $e_{I}=m^{0}_{II}$ and $\bb{C}[\mf{h}^{\ast}]=\bb{C}[X]^{H}$ is contained in the center of $\mca{A}_{\mf{C},s}$. For $I,J,J',K\in\bb{B}$ and $\alpha,\alpha'\in\bb{X}^{\ast}(H)$, we obtain 
\begin{align}\label{eqn_A_relation}
m^{\alpha}_{IJ}\cdot m^{\alpha'}_{J'K}=\delta_{J,J'}\prod_{i:\mu_{i}\mu'_{i}<0}u_{i}^{\min\{|\mu_{i}|,|\mu'_{i}|\}}\cdot m^{\alpha+\alpha'}_{IK},
\end{align}
where we set $\mu=\mu_{J}-\mu_{I}-\alpha$ and $\mu'=\mu_{K}-\mu_{J'}-\alpha'$. In summary, we obtained the following. 

\begin{lemma}
The $\bb{C}[\mf{h}^{\ast}]$-algebra $\mca{A}_{\mf{C},s}$ is isomorphic to 
\begin{align*}
\bigoplus_{\substack{I,J\in\bb{B}\\\alpha\in\bb{X}^{\ast}(H)}}\bb{C}[\mf{h}^{\ast}]\cdot m^{\alpha}_{IJ},
\end{align*}
where the multiplication rule is given by (\ref{eqn_A_relation}). Moreover, $\mf{h}\subset\bb{C}[\mf{h}^{\ast}]$ have $H$-weight 0 and degree 2, and $m^{\alpha}_{IJ}$ has $H$-weight $\alpha$ and degree $\sum_{i}|\mu_{i}|$, where $\mu=\mu_{J}-\mu_{I}-\alpha$. 
\end{lemma}

For another presentation of the algebra $\mca{A}_{\mf{C},s}$ which is apparently quadratic and a presentation of its Koszul dual $\mca{B}_{\mf{C},s}$, see \cite{MW}. This presentation is enough to check Conjecture~\ref{anti-involution}. 

\begin{corollary}\label{cor_toric_anti_involution}
The algebra $\mca{A}_{\mf{C},s}$ and $\mca{B}_{\mf{C},s}$ has an anti-involution which is identity on degree 0 part, compatible with the grading, and reversing the $H$-weights. 
\end{corollary}

\begin{proof}
For the algebra $\mca{A}_{\mf{C},s}$, we define the $\bb{C}[\mf{h}^{\ast}]$-algebra anti-involution by sending $m^{\alpha}_{IJ}$ to $m^{-\alpha}_{JI}$. It is easy to check that this preserves the relation (\ref{eqn_A_relation}) and satisfies the required conditions. It is also easy to check that this anti-involution induces a similar anti-involution on its quadratic dual which is isomorphic to $\mca{B}_{\mf{C},s}$. 
\end{proof}

\subsection{Positivity}

In this section, we prove Conjecture~\ref{positivity} and the first part of Conjecture~\ref{categorical_stable_basis} for toric hyper-K\"ahler manifolds. For $A\in\mr{Alc}_{s}$ with $\varphi_{\mf{C},s}(A)=(\lambda,p_{I})$, we define a categorical lift $\Delta_{\mf{C},s}(A)\in\DCohT(X)$ of $\mca{S}(A)$ by the formula 
\begin{align*}
\Delta_{\mf{C},s}(A)\coloneqq v^{-d}a^{\lambda}\cdot\mca{L}\left(-\sum_{i\in I_{+}}\varepsilon^{\ast}_{i}+\sum_{j\in I^{c}}\lfloor\langle s,\beta^{I}_{j}\rangle\rfloor\varepsilon^{\ast}_{j}\right)\otimes\mca{O}_{L_{I}}[-d].
\end{align*}
We set $\nabla_{\mf{C},s}(A)\coloneqq v^{d}\Delta_{-\mf{C},s}(A)[d]$. By definition and (\ref{eqn_toric_E(A)}), we have 
\begin{align*}
\nabla_{\mf{C},s}(A)=\mca{E}(A)\otimes\mca{O}_{L^{-}_{I}},
\end{align*}
where $L^{-}_{I}$ is the subvariety of $X$ defined by $x_{i}=0$ ($i\in I_{+}$) and $y_{i}=0$ ($i\in I_{-}$). For any subset $K\subset I$, we set $\mu_{A,K}\coloneqq\mu_{A}-\sum_{j\in K\cap I_{+}}\varepsilon^{\ast}_{j}+\sum_{j\in K\cap I_{-}}\varepsilon^{\ast}_{j}$. We note that by (\ref{eqn_toric_E(B)}), we have $\mca{L}(\mu_{A,K})\in\bb{B}_{X,s}$ for any $K\subset I$. By Lemma~\ref{lem_toric_resol_L_I}, we have the following exact sequence: 
\begin{align}\label{eqn_toric_resol_nabla}
0\rightarrow v^{d}\mca{L}(\mu_{A,I})\rightarrow \bigoplus_{\substack{K\subset I\\|K|=d-1}}v^{d-1}\mca{L}(\mu_{A,K})\rightarrow\cdots\rightarrow\bigoplus_{i\in I}v\mca{L}(\mu_{A,\{i\}})\rightarrow\mca{L}(\mu_{A})\rightarrow\nabla_{\mf{C},s}(A)\rightarrow0.
\end{align}
We recall that this is given by Koszul resolution. 

\begin{prop}\label{prop_nabla_heart}
For any $A,B\in\mr{Alc}_{s}$, we have $R^{\neq0}\mr{Hom}(\mca{E}(B),\nabla_{\mf{C},s}(A))=0$ and 
\begin{align}\label{eqn_Hom_tilt_std}
\Hom(\mca{E}(B),\nabla_{\mf{C},s}(A))^{T}=\begin{cases}
\bb{C}\cdot x^{\mu} &\mbox{ if }\pm\mu_{i}\leq0\mbox{ for any }i\in I_{\pm},\\
0 &\mbox{ otherwise},
\end{cases}
\end{align}
where $\mu=\mu_{A}-\mu_{B}$ and $x^{\mu}$ is the image of $x^{\mu}\in\Gamma(X,\mca{L}(\mu))^{T}$ under the natural map $\Gamma(X,\mca{L}(\mu))^{T}\rightarrow\Gamma(X,\nabla_{\mf{C},s}(p_{I})\otimes\mca{L}(-\mu_{B}))^{T}$ coming from (\ref{eqn_toric_resol_nabla}). 
\end{prop}

\begin{proof}
By Proposition~\ref{prop_toric_higher_vanishing_tilting}, we have $H^{>0}(\mca{X},\widetilde{\mca{L}}(\mu_{A,K}-\mu_{B}))=0$ for any $K\subset I$. Therefore, Lemma~\ref{lem_toric_global_section} and (\ref{eqn_toric_resol_nabla}) imply that $\RHom(\mca{E}(B),\nabla_{\mf{C},s}(A))^{T}$ is given by 
\begin{align*}
0\rightarrow\bb{C}[\mf{h}^{\ast}]\cdot x^{\mu_{A,I}-\mu_{B}}\rightarrow\cdots\rightarrow\bigoplus_{i\in I_{+}}\bb{C}[\mf{h}^{\ast}]\cdot x^{\mu-\varepsilon^{\ast}_{i}}\oplus\bigoplus_{i\in I_{-}}\bb{C}[\mf{h}^{\ast}]\cdot x^{\mu+\varepsilon^{\ast}_{i}}\rightarrow\bb{C}[\mf{h}^{\ast}]\cdot x^{\mu}\rightarrow0,
\end{align*}
where the complex is given by Koszul type complex for $x_{i}$ ($i\in I_{+}$) and $y_{i}$ ($i\in I_{-}$). If $\pm\mu_{i}\leq0$ for any $i\in I_{\pm}$, then by (\ref{eqn_mult_x_i}) and (\ref{eqn_mult_y_i}), this complex is isomorphic to the Koszul complex of $\bb{C}[\mf{h}^{\ast}]$ with respect to the regular sequence $\{u_{i}\}_{i\in I}$. This implies that its $0$-th cohomology is one dimensional and other cohomologies vanish. If $\pm\mu_{i}>0$ for some $i\in I_{\pm}$, then this complex is isomorphic to the Koszul complex of $\bb{C}[\mf{h}^{\ast}]$ with respect to a sequence containing $1$. Therefore, all the cohomologies vanish in this case.
\end{proof}

\begin{corollary}
For any $A\in\mr{Alc}_{s}$, $\psi_{\mf{C},s}(\nabla_{\mf{C},s}(A))\in D^{b}(\mca{A}_{\mf{C},s}\mbox{-}\mr{gmod}^{H})$ is a Koszul module of $\mca{A}_{\mf{C},s}$. 
\end{corollary}

\begin{proof}
Proposition~\ref{prop_nabla_heart} implies that $\psi_{\mf{C},s}(\nabla_{\mf{C},s}(A))$ is contained in the standard heart of $D^{b}(\mca{A}_{\mf{C},s}\mbox{-}\mr{gmod}^{H})$. The exact sequence (\ref{eqn_toric_resol_nabla}) implies that $\psi_{\mf{C},s}(\nabla_{\mf{C},s}(A))$ is a Koszul module of $\mca{A}_{\mf{C},s}$. 
\end{proof}

Next we give a formula expressing $\mca{S}(A)$ (resp. $\mca{E}(A)$) in terms of $\{\mca{C}(B)\}_{B\in\mr{Alc}_{s}}$ (resp. $\{\mca{S}(B)\}_{B\in\mr{Alc}_{s}}$). For $A\in\mr{Alc}_{s}$, let $M(A)$ be the set of alcove $B$ such that for any hyperplane $\mca{H}_{i,m}$ passing through $x_{A}$, $A$ and $B$ are on the same side with respect to $\mca{H}_{i,m}$. We note that $B\in M(A)$ implies $B\geq_{\mf{C}}A$. We also set $M^{-}(A):=\{B\in\mr{Alc}_{s}\mid A\in M(B)\}$. In terms of the combinatorics of alcoves, the condition appearing in (\ref{eqn_Hom_tilt_std}) can be written as $B\in M(A)$ or not. Moreover, the degree of $x^{\mu}$ is given by $\ell(A,B)$.

\begin{corollary}\label{cor_toric_positivity}
For any $A\in\mr{Alc}_{s}$, we have 
\begin{align}
\mca{S}(A)&=\sum_{B\in M(A)}v^{-\ell(A,B)}\mathcal{C}(B),\label{eqn_toric_StoC}\\
\mathcal{E}(A)&=\sum_{B\in M^{-}(A)}v^{-\ell(A,B)}\mathcal{S}(B).\label{eqn_toric_EtoS}
\end{align}	
\end{corollary}

\begin{proof}
As in the proof of Corollary~\ref{cor_Koszulity}, Proposition~\ref{prop_nabla_heart} implies that 
\begin{align*}
\partial(\mca{E}(B)||\mca{S}(A))&=\left[\Hom(\mca{E}(B),\nabla_{\mf{C},s}(A))^{T}\right]^{\vee}\\
&=\begin{cases}
v^{-\ell(A,B)}&\mbox{ if }B\in M(A)\\
0&\mbox{ if }B\notin M(A).\end{cases}
\end{align*}
This implies (\ref{eqn_toric_StoC}) and (\ref{eqn_toric_EtoS}). 
\end{proof}

\begin{corollary}\label{cor_toric_positivity2}
Conjecture~\ref{positivity} holds for toric hyper-K\"ahler manifolds. 
\end{corollary}

\begin{proof}
This follows from (\ref{eqn_toric_C(A)}) and (\ref{eqn_toric_EtoS}). 
\end{proof}

\subsection{Ext-orthogonality}

In this section, we prove the second half of Conjecture~\ref{categorical_stable_basis}. 

\begin{thm}\label{thm_ext_orthogonality}
For any $A,B\in\mr{Alc}_{s}$, we have 
\begin{align*}
\RHom(\Delta_{\mf{C},s}(B),\nabla_{\mf{C},s}(A))^{T}\cong\begin{cases}
\bb{C}&\mbox{ if }A=B\\
0&\mbox{ if }A\neq B
\end{cases}
\end{align*}
as $\bb{S}$-modules, where $\bb{C}$ is considered as a trivial $\bb{S}$-module sitting in cohomological degree 0. In particular, Conjecture~\ref{categorical_stable_basis} holds for toric hyper-K\"ahler manifolds. 
\end{thm}

\begin{proof}
Let $A,B\in\mr{Alc}_{s}$ be two alcoves with $\varphi_{\mf{C},s}(A)=(\lambda_{A},p_{I})$ and $\varphi_{\mf{C},s}(B)=(\lambda_{B},p_{J})$. We recall that $\mu_{A}=\lambda_{A}-\sum_{i\in I_{-}}\varepsilon^{\ast}_{i}+\sum_{j\in I^{c}}\lfloor\langle s,\beta^{I}_{j}\rangle\rfloor\varepsilon^{\ast}_{j}$ and $\mu_{B}=\lambda_{B}-\sum_{i\in J_{-}}\varepsilon^{\ast}_{i}+\sum_{j\in J^{c}}\lfloor\langle s,\beta^{J}_{j}\rangle\rfloor\varepsilon^{\ast}_{j}$. We set $\mu\coloneqq\mu_{A}-\mu_{B}$ and $K\coloneqq\{k\in I\cap J\mid x_{A},x_{B}\in\mca{H}_{k,m_{k}}\mbox{ for some }m_{k}\in\bb{Z}\}$. We note that we have
\begin{align*}
\mu_{k}=\begin{cases}
1 &\mbox{ if }k\in K\cap I_{+}\cap J_{-},\\
-1 &\mbox{ if }k\in K\cap I_{-}\cap J_{+},\\
0 &\mbox{ if }k\in K\cap ((I_{+}\cap J_{+})\cup(I_{-}\cap J_{-})).
\end{cases}
\end{align*} 
By Lemma~\ref{lem_toric_resol_L_I}, $\Delta_{\mf{C},s}(B)$ is quasi-isomorphic to 
\begin{align*}
0\rightarrow\mca{E}(B)\rightarrow\bigoplus_{\substack{B'\in N(B)\\\ell(B,B')=1}}v^{-1}\mca{E}(B')\rightarrow\bigoplus_{\substack{B'\in N(B)\\\ell(B,B')=2}}v^{-2}\mca{E}(B')\rightarrow\cdots,
\end{align*}
where $\mca{E}(B)$ sits in cohomological degree 0. If $\Hom(\mca{E}(B'),\nabla_{\mf{C},s}(A))^{T}\neq0$ for $B'\in N(B)$, then we have $B'\in M(A)$ by Proposition~\ref{prop_nabla_heart}. If $\pm\mu_{i}>0$ for some $i\in I_{\pm}\cap K^{c}$, then we have $M(A)\cap N(B)=\emptyset$ and hence $\RHom(\Delta_{\mf{C},s}(B),\nabla_{\mf{C},s}(A))^{T}=0$. Therefore, we may assume $\pm\mu_{i}\leq0$ for any $i\in I_{\pm}\cap K^{c}$. 

Let $K'\subset J$ be the subset satisfying $\mu_{B'}=\mu_{B}-\sum_{j\in K'\cap J_{+}}\varepsilon^{\ast}_{j}+\sum_{j\in K'\cap J_{-}}\varepsilon^{\ast}_{j}$. We set $\mu_{K'}\coloneqq\mu_{A}-\mu_{B'}$. The condition $B'\in M(A)$ implies that we must have $K'\supset K_{0}\coloneqq K\cap ((I_{-}\cap J_{+})\cup(I_{+}\cap J_{-}))$ and $K'\cap K\cap((I_{+}\cap J_{+})\cup(I_{-}\cap J_{-}))=\emptyset$, i.e., $K'\subset K_{1}\coloneqq (J\cap K^{c})\cup K_{0}$. Conversely, the condition $K_{0}\subset K'\subset K_{1}$ implies $B'\in M(A)$ by the assumption that $\pm\mu_{i}\leq0$ for any $i\in I_{\pm}\cap K^{c}$. Therefore, $\RHom(\Delta_{\mf{C},s}(B),\nabla_{\mf{C},s}(A))^{T}$ is quasi-isomorphic to a Koszul type complex
\begin{align}\label{eqn_Koszul_complex_intersection}
0\rightarrow\bb{C}\cdot x^{\mu_{K_{1}}}\rightarrow\cdots\rightarrow\bigoplus_{j\in K_{1}\cap K^{c}_{0}}\bb{C}\cdot x^{\mu_{K_{0}\cup\{j\}}}\rightarrow\bb{C}\cdot x^{\mu_{K_{0}}}\rightarrow0.
\end{align}
Here, each map is induced from $x_{j}\cdot:\bb{C}[\mf{h}^{\ast}]\cdot x^{\mu_{K'}-\varepsilon^{\ast}_{j}}\rightarrow\bb{C}[\mf{h}^{\ast}]\cdot x^{\mu_{K'}}$ for $j\in J_{-}\cap K^{c}$ and $y_{j}\cdot:\bb{C}[\mf{h}^{\ast}]\cdot x^{\mu_{K'}+\varepsilon^{\ast}_{j}}\rightarrow\bb{C}[\mf{h}^{\ast}]\cdot x^{\mu_{K'}}$ for $j\in J_{+}\cap K^{c}$. By (\ref{eqn_mult_x_i}) and (\ref{eqn_mult_y_i}), the map $x_{j}\cdot:\bb{C}\cdot x^{\mu'-\varepsilon^{\ast}_{j}}\rightarrow\bb{C}\cdot x^{\mu'}$ is given by 
\begin{align*}
x_{j}\cdot x^{\mu'-\varepsilon^{\ast}_{j}}=\begin{cases}
x^{\mu'} &\mbox{ if }\mu'_{j}>0\\
0 &\mbox{ if }\mu'_{j}\leq0
\end{cases}
\end{align*}
and $y_{j}\cdot:\bb{C}\cdot x^{\mu'+\varepsilon^{\ast}_{j}}\rightarrow\bb{C}\cdot x^{\mu'}$ is given by 
\begin{align*}
y_{j}\cdot x^{\mu'+\varepsilon^{\ast}_{j}}=\begin{cases}
x^{\mu'} &\mbox{ if }\mu'_{j}<0\\
0 &\mbox{ if }\mu'_{j}\geq0
\end{cases}
\end{align*}
We set $\nu\coloneqq\mu_{K_{0}}$. If $\nu_{j}>0$ for some $j\in J_{-}\cap K^{c}$ or $\nu_{j}<0$ for some $j\in J_{+}\cap K^{c}$, then the complex (\ref{eqn_Koszul_complex_intersection}) is isomorphic to a Koszul complex of $\bb{C}$ with respect to a sequence containing $1$ and hence it is acyclic. Therefore, if $\RHom(\Delta_{\mf{C},s}(B),\nabla_{\mf{C},s}(A))^{T}\ncong0$, then we must have $\pm\nu_{j}\geq0$ for any $j\in J_{\pm}\cap K^{c}$. We also note that $\nu_{k}=0$ for any $k\in K$ and $\pm\nu_{i}\leq0$ for any $i\in I_{\pm}\cap K^{c}$. We claim that these conditions imply $I=J$. If $I=J$, then we have $\nu_{i}=0$ for any $i\in I$ and $K_{0}=\emptyset$. This implies that $\mu_{i}=\nu_{i}=0$ for any $i\in I$ and hence $A=B$. In this case, $K_{1}=\emptyset$ and the complex (\ref{eqn_Koszul_complex_intersection}) is isomorphic to $\bb{C}$ which sits in cohomological degree 0. 

Now we assume $I\neq J$. We first refine the inequalities (\ref{eqn_ineq_mu}). We note that 
\begin{align*}
\nu=\lambda_{A}-\lambda_{B}-\sum_{i\in I_{-}}\varepsilon^{\ast}_{i}+\sum_{i\in J_{-}}\varepsilon^{\ast}_{i}+\sum_{i\in K_{0}\cap J_{+}}\varepsilon^{\ast}_{i}-\sum_{i\in K_{0}\cap J_{-}}\varepsilon^{\ast}_{i}+\sum_{j\in I^{c}}\lfloor\langle s,\beta^{I}_{j}\rangle\rfloor\varepsilon^{\ast}_{j}-\sum_{j\in J^{c}}\lfloor\langle s,\beta^{J}_{j}\rangle\rfloor\varepsilon^{\ast}_{j}
\end{align*}
Since $K_{0}\cap J_{+}=K\cap I_{-}\cap J_{+}$ and $K_{0}\cap J_{-}=K\cap I_{+}\cap J_{-}$, we have 
\begin{align*}
-\sum_{i\in I_{-}}\varepsilon^{\ast}_{i}+\sum_{i\in J_{-}}\varepsilon^{\ast}_{i}+\sum_{i\in K_{0}\cap J_{+}}\varepsilon^{\ast}_{i}-\sum_{i\in K_{0}\cap J_{-}}\varepsilon^{\ast}_{i}&=-\sum_{i\in I_{-}\cap (J_{+}\sqcup J^{c})\cap K^{c}}\varepsilon^{\ast}_{i}+\sum_{J_{-}\cap (I_{+}\sqcup I^{c})\cap K^{c}}\varepsilon^{\ast}_{i}\end{align*}
For any signed circuit $C=C_{+}\sqcup C_{-}$, we have $\beta_{C}=\sum_{j\in C_{+}\cap I^{c}}\beta^{I}_{j}-\sum_{j\in C_{-}\cap I^{c}}\beta^{I}_{j}$. This implies as in (\ref{eqn_ineq_mu}) that 
\begin{align*}
-|C_{+}\cap I^{c}|+1+\lfloor\langle s,\beta_{C}\rangle\rfloor\leq\sum_{j\in C_{+}\cap I^{c}}\lfloor\langle s,\beta^{I}_{j}\rangle\rfloor-\sum_{j\in C_{-}\cap I^{c}}\lfloor\langle s,\beta^{I}_{j}\rangle\rfloor\leq|C_{-}\cap I^{c}|+\lfloor\langle s,\beta_{C}\rangle\rfloor.
\end{align*}
Therefore, we obtain 
\begin{align*}
\langle\nu,\beta_{C}\rangle&\leq-|C_{+}\cap I_{-}\cap (J_{+}\sqcup J^{c})\cap K^{c}|+|C_{-}\cap I_{-}\cap (J_{+}\sqcup J^{c})\cap K^{c}|+|C_{-}\cap I^{c}|\\
&\hspace{2em}+|C_{+}\cap J_{-}\cap(I_{+}\sqcup I^{c})\cap K^{c}|-|C_{-}\cap J_{-}\cap(I_{+}\sqcup I^{c})\cap K^{c}|+|C_{+}\cap J^{c}|-1\\
&\leq-|C_{+}\cap I_{-}\cap J^{c}|+|C_{-}\cap I_{-}\cap J_{+}|+|C_{-}\cap I_{-}\cap J^{c}|+|C_{-}\cap I^{c}|\\
&\hspace{2em}+|C_{+}\cap J_{-}\cap I_{+}|+|C_{+}\cap J_{-}\cap I^{c}|-|C_{-}\cap J_{-}\cap I^{c}|+|C_{+}\cap J^{c}|-1\\
&=|C_{-}\cap(I_{-}\setminus J_{-})|+|C_{-}\setminus(I\cup J_{-})|+|C_{+}\cap (J_{-}\setminus I_{-})|+|C_{+}\setminus(J\cup I_{-})|-1.
\end{align*}
In particular, if we assume the existence of a signed circuit $C$ such that $C_{+}\subset I_{-}\cup J_{+}$ and $C_{-}\subset I_{+}\cup J_{-}$, then the above inequality implies $\langle\nu,\beta_{C}\rangle\leq-1$. On the other hand, we have $\nu_{i}\geq0$ for any $i\in I_{-}\cup J_{+}$ and $\nu_{i}\leq0$ for any $i\in I_{+}\cup J_{-}$. This implies that $\langle\nu,\beta_{C}\rangle\geq0$ which gives a contradiction. 

Therefore, it is enough to prove the existence of a signed circuit as above. We set $V=\Ker(\mf{b})\otimes\bb{R}$ and $E=\{1,\ldots,n\}$ as in section 5.6. We also define a partition $E=R\sqcup G\sqcup B\sqcup W$ by $W=(I_{-}\cup J_{+})\cap(I_{+}\cup J_{-})=(I_{+}\cap J_{+})\cup(I_{-}\cap J_{-})$, $R=(I_{-}\cup J_{+})\setminus W=(I_{-}\setminus J_{-})\cup(J_{+}\setminus I_{+})$, $G=(I_{+}\cup J_{-})\setminus W=(I_{+}\setminus J_{+})\cup(J_{-}\setminus I_{-})$, and $B=I^{c}\cap J^{c}$. We note that the assumption $I\neq J$ implies $R\sqcup G\neq\emptyset$. By Lemma~\ref{lem_om_conform}, it is enough to prove the existence of nonzero $Z\in\sigma(V^{\perp})$ such that $Z_{R}\geq0$, $Z_{G}\leq0$, and $Z_{B}=0$. If we assume that such $Z$ does not exists, then Proposition~\ref{prop_Minty} implies that there exists a nonzero $Y\in\sigma(V)$ such that $Y_{R}\geq0$, $Y_{G}\leq0$, and $Y_{W}=0$. By Lemma~\ref{lem_om_conform} again, there exists signed cocircuit $C^{\vee}=C^{\vee}_{+}\sqcup C^{\vee}_{-}$ such that $C^{\vee}_{+}\subset R\sqcup B$ and $C^{\vee}_{-}\subset G\sqcup B$. Since we have $I\cap C^{\vee}_{+}\subset I_{-}$, $I\cap C^{\vee}_{-}\subset I_{+}$, and $I\cap C^{\vee}\neq\emptyset$, we obtain
\begin{align*}
\langle\xi,\alpha_{C^{\vee}}\rangle=\sum_{i\in I\cap C^{\vee}_{+}}\langle\xi,\alpha^{I}_{i}\rangle-\sum_{i\in I\cap C^{\vee}_{-}}\langle\xi,\alpha^{I}_{i}\rangle<0.
\end{align*}
On the other hand, $J\cap C^{\vee}_{+}\subset J_{+}$, $J\cap C^{\vee}_{-}\subset J_{-}$, and $J\cap C^{\vee}\neq\emptyset$ imply that 
\begin{align*}
\langle\xi,\alpha_{C^{\vee}}\rangle=\sum_{i\in J\cap C^{\vee}_{+}}\langle\xi,\alpha^{J}_{i}\rangle-\sum_{i\in J\cap C^{\vee}_{-}}\langle\xi,\alpha^{J}_{i}\rangle>0.
\end{align*}
This is a contradiction and hence the required circuit exists. This completes the proof of Theorem~\ref{thm_ext_orthogonality} and hence Conjecture~\ref{categorical_stable_basis} for toric hyper-K\"ahler manifolds.
\end{proof}

\subsection{Central charges}

In this section, we prove Conjecture~\ref{conj_real_variation} and the second part of Conjecture~\ref{conj_wall_crossing}. In order to prove them, we need to construct the central charge $\msc{Z}:\mf{s}^{\ast}_{\bb{R}}\rightarrow\Hom_{\bb{Z}}(K(L),\bb{R})$. We claim that for a K\"ahler alcove $\mb{A}$ and $s\in\mb{A}$, the central charge of canonical bases $\mca{C}(A(s))$ corresponding to $A(s)\in\mr{Alc}_{s}$ is given by the volume of the polytope $A(s)$ which is a polynomial function in $s$. In order to check that these polynomial functions do not depend on the choice of $\mb{A}$, we consider their equivariant lifts. 

Recall that we write $\iota(s)=(s_{1},\ldots,s_{n})$ for $s\in\mf{s}^{\ast}_{\bb{R}}$. For $I\in\bb{B}$, we set 
\begin{align*}
\square_{I}(s)\coloneqq\{x\in\mf{h}^{\ast}_{\bb{R}}\mid 0\leq\langle x,\mf{a}_{i}\rangle+s_{i}\leq1\mbox{ for any } i\in I\}.
\end{align*}
Take a generic $c\in\mf{h}$ such that $\langle c,\alpha\rangle\notin\bb{Z}$ for any equivariant root $\alpha\in\bb{X}^{\ast}(H)$. We consider $\bb{C}$ as a module over $K_{H}(\mr{pt})$ by $a^{\lambda}\mapsto e^{2\pi\sqrt{-1}\langle c,\lambda\rangle}$ and define a map $\msc{Z}_{c}:\mf{s}^{\ast}_{\bb{R}}\rightarrow\Hom_{K_{H}(\mr{pt})}(K_{H}(L),\bb{C})$ by assigning 
\begin{align*}
\langle \msc{Z}_{c}(s),\mca{O}_{p_{I}}\rangle=\int_{\square_{I}(s)}e^{2\pi \sqrt{-1}\langle c,x\rangle}dx=\prod_{i\in I}\frac{e^{2\pi\sqrt{-1}\langle c,\alpha^{I}_{i}\rangle}-1}{2\pi \sqrt{-1}\langle c,\alpha^{I}_{i}\rangle}\cdot e^{-2\pi \sqrt{-1}s_{i}\langle c,\alpha^{I}_{i}\rangle}
\end{align*}
for each $I\in\bb{B}$. Here, $\mca{O}_{p_{I}}$ is the skyscraper sheaf at $p_{I}\in X^{H}$. Since $\{\mca{O}_{p_{I}}\}_{I\in\bb{B}}$ is a basis of $K_{H}(L)$ after localization, the genericity of $c$ implies that this extends to a $K_{H}(\mr{pt})$-linear map $\msc{Z}_{c}(s):K_{H}(L)\rightarrow\bb{C}$. We note that these functions are analytic in $s$.

\begin{lemma}\label{lem_Euler_integral}
For any $s\in\mf{s}^{\ast}_{\mr{reg}}$ and $A(s)\in\mr{Alc}_{s}$, we have 
\begin{align}\label{eqn_Euler_integral}
\langle \msc{Z}_{c}(s),\mca{C}(A(s))\rangle=\int_{A(s)}e^{2\pi\sqrt{-1}\langle c,x\rangle}dx.
\end{align}
\end{lemma}

\begin{proof}
We note that if we change $A(s)$ by $A(s)+\alpha$ for some $\alpha\in\bb{X}^{\ast}(H)$, then both sides of (\ref{eqn_Euler_integral}) are multiplied by $e^{2\pi\sqrt{-1}\langle c,\alpha\rangle}$. For any $A(s)\in\mr{Alc}_{s}$, there exists unique $\alpha\in\bb{X}^{\ast}(H)$ such that $A(s)+\alpha\subset\square_{I}$ since $\{\alpha^{I}_{i}\}_{i\in I}$ is a basis of $\bb{X}^{\ast}(H)$. If $A(s)\subset\square_{I}(s)$, then we have $\wt_{H}i^{\ast}_{p_{I}}\mca{E}(A(s))=0$ by Corollary~\ref{cor_toric_line_restr}. By Proposition~\ref{fixed_point_basis}, we obtain 
\begin{align}\label{eqn_toric_fix_to_C}
\mca{O}_{p_{I}}=\sum_{\substack{A(s)\in\mr{Alc}_{s}\\A(s)\subset\square_{I}(s)}}\mca{C}(A(s))
\end{align}
as elements in $K_{H}(L)$ for any $I\in\bb{B}$, i.e., after specializing $v=1$. On the other hand, we obviously have 
\begin{align}\label{eqn_integral_fix_to_C}
\langle \msc{Z}_{c}(s),\mca{O}_{p_{I}}\rangle=\sum_{\substack{A(s)\in\mr{Alc}_{s}\\A(s)\subset\square_{I}(s)}}\int_{A(s)}e^{2\pi\sqrt{-1}\langle c,x\rangle}dx
\end{align}
for any $I\in\bb{B}$. One can solve (\ref{eqn_toric_fix_to_C}) to express $\mca{C}(A(s))$ in terms of $\mca{O}_{p_{I}}$. Since $c$ is generic, one can also solve (\ref{eqn_integral_fix_to_C}) to express $\int_{A(s)}e^{2\pi\sqrt{-1}\langle c,x\rangle}dx$ in terms of $\langle \msc{Z}_{c}(s),\mca{O}_{p_{I}}\rangle$ in the same way. Therefore, $\langle \msc{Z}_{c}(s),\mca{C}(A(s))\rangle$ and $\int_{A(s)}e^{2\pi\sqrt{-1}\langle c,x\rangle}dx$ should coincide for any $A(s)\subset\square_{I}(s)$. This proves (\ref{eqn_Euler_integral}).
\end{proof}

\begin{remark}
One can also consider an equivariant lift of $\msc{Z}$ to $K_{\bb{T}}(L)$ by replacing the integral (\ref{eqn_Euler_integral}) by certain Euler type integrals \cite{GKZ} appearing in the theory of Gelfand-Kapranov-Zelevinsky's hypergeometric differential equations. We note that these differential equations come from quantum differential equations of toric hyper-K\"ahler manifolds by \cite{MS}. We do not pursue this direction further here since we do not need it. 
\end{remark}

\begin{corollary}\label{cor_toric_central_charge}
There exists a polynomial map $\msc{Z}:\mf{s}^{\ast}_{\bb{R}}\rightarrow\Hom_{\bb{Z}}(K(L),\bb{R})$ such that $\langle \msc{Z}(s),\mca{C}(A(s))\rangle=\mr{Vol}(A(s))$ for any $s\in\mf{s}^{\ast}_{\mr{reg}}$ and $A(s)\in\mr{Alc}_{s}$. Moreover, there exists a vector bundle $\mca{P}$ such that $\msc{Z}=\msc{Z}_{\mca{P}}$. 
\end{corollary}

\begin{proof}
Since the RHS of (\ref{eqn_Euler_integral}) is holomorphic in $c$ and $\{\mca{C}(A(s))\}_{A(s)\in\mr{Alc}_{s}}$ forms a basis of $K_{H}(L)$, $\langle \msc{Z}_{c}(s),\mca{C}\rangle$ is holomorphic in $c$ for any $\mca{C}\in K_{H}(L)$. Therefore, one can substitute $c=0$ to obtain a map $\msc{Z}\coloneqq \msc{Z}_{c=0}:\mf{s}^{\ast}_{\bb{R}}\rightarrow\Hom_{\bb{Z}}(K(L),\bb{R})$. Lemma~\ref{lem_Euler_integral} implies that for any $A(s)\in\mr{Alc}_{s}$, we have $\langle \msc{Z}(s),\mca{C}(A(s))\rangle=\mr{Vol}(A(s))$. Since this is a polynomial function in $s$, $\msc{Z}$ is also a polynomial function. 

We next consider the value of the central charges at $s=0$. We take $s_{o}\in\mf{s}^{\ast}_{\mr{reg}}$ in a neighborhood of $0$ and consider $\mca{P}\coloneqq m\sum_{A(s_{o})\in\mr{Alc}_{s_{o}}/\bb{X}^{\ast}(H)}\mr{Vol}(A(0))\cdot\mca{E}(A(s_{o}))^{\vee}$, where $A(0)$ is the limit of $A(s_{o})$ as $s_{o}\rightarrow0$ and $m\in\bb{Z}_{>0}$ is taken so that $m\mr{Vol}(A(0))\in\bb{Z}$ for any $A(s_{o})\in\mr{Alc}_{s_{o}}$. We note that this does not depend on the choice of $s_{o}$ by Proposition~\ref{prop_toric_wall_crossing_characterization}. We note that $\rank\mca{P}=m$ since we have $\sum_{A(s)\in\mr{Alc}_{s}/\bb{X}^{\ast}(H)}\mr{Vol}(A(s))=\mr{Vol}(\square_{I}(s))=1$. 

We note that by Proposition~\ref{dual_basis}, we have $\chi(X,\mca{C}(A)\otimes\mca{E}(A')^{\vee})=\delta_{A,A'}$ for any $A,A'\in\mr{Alc}_{s}/\bb{X}^{\ast}(H)$. This implies $\langle \msc{Z}_{\mca{P}}(0),\mca{C}(A(s_{o}))\rangle=\mr{Vol}(A(0))=\langle \msc{Z}(0),\mca{C}(A(s_{o}))\rangle$ for any $A(s_{o})\in\mr{Alc}_{s_{o}}$ and hence $\msc{Z}_{\mca{P}}(0)=\msc{Z}(0)$. For any $l\in\Pic(X)\cong\bb{X}^{\ast}(S)$, the periodic hyperplane arrangements in $\mf{h}^{\ast}_{\bb{R}}$ defined in section 5.4 does not change if we change $s$ by $s+l$.  This and Lemma~\ref{can_slope} imply that for any $A(s_{o})\in\mr{Alc}_{s_{o}}$, there exists $A(s_{o}+l)\in\mr{Alc}_{s_{o}+l}$ such that $\mr{Vol}(A(s_{o}))=\mr{Vol}(A(s_{o}+l))$ and $\mca{C}(A(s_{o}+l))=\mf{L}(l)\otimes\mca{C}(A(s_{o}))$. Therefore, we obtain $\langle \msc{Z}_{\mca{P}}(l),\mca{C}(A(s_{o}+l))\rangle=\mr{Vol}(A(0))=\mr{Vol}(A(l))=\langle \msc{Z}(l),\mca{C}(A(s_{o}+l))\rangle$ for any $A(s_{o}+l)\in\mr{Alc}_{s_{o}+l}$. This implies that $\msc{Z}_{\mca{P}}(l)=\msc{Z}(l)$ for any $l\in\Pic(X)$. Since both of them are polynomial functions in $s$, we obtain $\msc{Z}_{\mca{P}}(s)=\msc{Z}(s)$ for any $s\in\mf{s}^{\ast}_{\bb{R}}$. 
\end{proof}

We now prove Conjecture~\ref{conj_real_variation} for toric hyper-K\"ahler manifolds. Recall that for any $\mb{A}\in\mr{Alc}_{K}$, we associate a $t$-structure $\tau(\mb{A})$ on $\msc{D}\coloneqq\DCoh_{L}(X)\subset\DCoh(X)$ defined by the tilting bundle $\mca{T}_{\mf{C},s}$ for $s\in\mb{A}$. 

\begin{corollary}\label{cor_toric_real_variation}
The pair $(\msc{Z},\tau)$ gives a real variation of stability conditions on $\msc{D}$. 
\end{corollary}

\begin{proof}
The first condition in Definition~\ref{real_variation} follows from Corollary~\ref{cor_toric_central_charge} since the volume of a full dimensional polytope is positive. Hence it is enough to check the second condition in Definition~\ref{real_variation}. 

For any $s\in\mb{A}\in\mr{Alc}_{K}$ and a wall $w=\{x\in\mf{s}^{\ast}_{\bb{R}}\mid\langle x,\beta_{C}\rangle=m\}$ of $\mb{A}$, we set $\mr{Alc}^{0}_{s,w}\coloneqq\{A(s)\in\mr{Alc}_{s}\mid \mr{Vol}(A(s))\mbox{ does not vanish on }w\}$ and $\mr{Alc}^{1}_{s,w}\coloneqq\mr{Alc}_{s}\setminus\mr{Alc}^{0}_{s,w}$. By Proposition~\ref{prop_toric_wall_crossing_characterization} and Corollary~\ref{cor_toric_central_charge}, the Serre subcategories $\msc{C}^{n}_{\mb{A}}$ of the heart $\msc{C}_{\mb{A}}$ of $\tau(\mb{A})$ defined in section 3.7 are given by $\msc{C}_{\mb{A}}$ if $n=0$, generated by objects $\{\mca{C}(A)\mid A\in\mr{Alc}^{1}_{s,w}\}$ if $0<n<|C|$, and 0 if $n\geq|C|$. Here, we identified $\mca{C}(A)\in K_{\bb{T}}(X)$ as an object of $\msc{D}$ as in section 3.6 by forgetting equivariant structures. We note that we have $\msc{C}^{1}_{\mb{A}}=\msc{C}^{|C|-1}_{\mb{A}}=\{\mca{F}\in\msc{C}_{\mb{A}}\mid\Hom(\mca{E}(A),\mca{F})=0\mbox{ for any }A\in\mr{Alc}^{0}_{s,w}\}$. This easily implies that $\msc{D}^{1}_{\mb{A},w}=\msc{D}^{|C|-1}_{\mb{A},w}=\{\mca{F}\in\msc{D}\mid\RHom(\mca{E}(A),\mca{F})=0\mbox{ for any }A\in\mr{Alc}^{0}_{s,w}\}$. 

Let $\mb{A}_{+}\neq\mb{A}_{-}\in\mr{Alc}_{K}$ be two K\"ahler alcove sharing the same wall $w$ such that $\mb{A}_{+}$ is above $\mb{A}_{-}$. We take $s_{\pm}\in\mb{A}_{\pm}$. Since $\{\mca{E}(A)\mid A\in\mr{Alc}^{0}_{s_{+},w}\}=\{\mca{E}(A)\mid A\in\mr{Alc}^{0}_{s_{-},w}\}$ by Proposition~\ref{prop_toric_wall_crossing_characterization}, we obtain $\msc{D}^{n}_{\mb{A}_{+},w}=\msc{D}^{n}_{\mb{A}_{-},w}$ for any $n\in\bb{Z}_{\geq0}$. By the definition of $\tau(\mb{A}_{\pm})$, we have $\msc{C}^{|C|-1}_{\mb{A}_{\pm},w}=\{\mca{F}\in\msc{D}^{|C|-1}_{\mb{A}_{\pm},w}\mid\mr{RHom}^{\neq0}(\mca{E}(A),\mca{F})=0\mbox{ for any }A\in\mr{Alc}^{1}_{s_{\pm},w}\}$. By Proposition~\ref{prop_toric_wall_crossing_characterization} and Lemma~\ref{lem_toric_wall_crossing_formula}, we have $\RHom(\mca{E}(A_{+}),\mca{F})\cong\RHom(\mca{E}(A_{-})[|C|-1],\mca{F})$ for any $\mca{F}\in\msc{D}^{|C|-1}_{\mb{A}_{\pm},w}$ and $A_{\pm}\in\mr{Alc}_{s_{\pm}}$ such that $\mca{E}(A_{-})\cong\mca{L}\left(\sum_{i\in C_{-}}\varepsilon^{\ast}_{i}-\sum_{i\in C_{+}}\varepsilon^{\ast}_{i}\right)\otimes\mca{E}(A_{+})$. This implies that $\gr^{|C|-1}_{w}(\msc{C}_{\mb{A}_{-}})=\gr^{|C|-1}_{w}(\msc{C}_{\mb{A}_{+}})[|C|-1]$ in $\gr^{|C|-1}_{\mb{A}_{\pm},w}\msc{D}$. Similarly, we have $\gr^{0}_{w}(\msc{C}_{\mb{A}_{\pm}})=\{\mca{F}\in\gr^{0}_{\mb{A}_{\pm},w}(\msc{D})\mid \mr{RHom}^{\neq0}(\mca{E}(A),\mca{F})=0\mbox{ for any }A\in\mr{Alc}^{0}_{s_{\pm},w}\}$ and hence $\gr^{0}_{w}(\msc{C}_{\mb{A}_{-}})=\gr^{0}_{w}(\msc{C}_{\mb{A}_{+}})$ in $\gr^{0}_{\mb{A}_{\pm},w}\msc{D}$. This proves the second condition in Defnition~\ref{real_variation} and hence $(\msc{Z},\tau)$ gives a real variation of stability conditions. 
\end{proof}

\section{Elliptic canonical bases for toric hyper-K\"ahler manifolds}

In this final section, we define what we call elliptic canonical bases for toric hyper-K\"ahler manifolds and prove some basic properties of them. As a corollary, we prove Conjecture~\ref{conj_ell_indep_chamber} for toric hyper-K\"ahler manifolds. We will follow the notations of section 4 and 5. 

\subsection{Elliptic standard bases}

First we recall the description of elliptic stable bases for toric hyper-K\"ahler manifolds given in \cite{AO,S}. Recall that we fix the polarization $T^{1/2}$ in (\ref{eqn_toric_polarization}) which satisfies $\det T^{1/2}=v^{-n}\mca{L}(\kappa)$, where we set $\kappa\coloneqq\varepsilon^{\ast}_{1}+\cdots+\varepsilon^{\ast}_{n}\in\bb{X}^{\ast}(T)$. We also take similar polarization for $X^{!}$ and $\kappa^{!}\coloneqq\varepsilon_{1}+\cdots+\varepsilon_{n}\in\bb{X}_{\ast}(T)$.

\begin{prop}[\cite{AO,S}]\label{prop_toric_stab_AO}
For any $I\in\bb{B}$, we have 
\begin{align*}
\Stab^{AO}_{\mf{C},T^{1/2}}(p_{I})=\prod_{i\in I_{+}}\vartheta(v^{-1}\mca{L}(-\varepsilon^{\ast}_{i}))\cdot\prod_{i\in I_{-}}\vartheta(v^{-1}\mca{L}(\varepsilon^{\ast}_{i}))\cdot\prod_{j\in I^{c}_{+}}\frac{\vartheta(v^{l_{j}-1}z^{\beta^{I}_{j}}\mca{L}(\varepsilon^{\ast}_{j}))}{\vartheta(v^{l_{j}}z^{\beta^{I}_{j}})}\cdot\prod_{j\in I^{c}_{-}}\frac{\vartheta(v^{-l_{i}-1}z^{-\beta^{I}_{j}}\mca{L}(-\varepsilon^{\ast}_{j}))}{\vartheta(v^{-l_{i}}z^{-\beta^{I}_{j}})},
\end{align*}
where $l_{j}\coloneqq-\langle\beta^{I}_{j},\kappa+\sum_{i\in I_{+}}\mf{b}_{i}-\sum_{i\in I_{-}}\mf{b}_{i}\rangle\pm1$ for any $j\in I^{c}_{\pm}$.
\end{prop}

\begin{proof}
By Theorem 5 in \cite{S}, we have 
\begin{align}\label{eqn_toric_ell_stab_rough}
\Stab^{AO}_{\mf{C},T^{1/2}}(p_{I})=\prod_{i\in I_{+}}\vartheta(v^{-1}\mca{L}(-\varepsilon^{\ast}_{i}))\cdot\prod_{i\in I_{-}}\vartheta(v^{-1}\mca{L}(\varepsilon^{\ast}_{i}))\cdot\prod_{j\in I^{c}_{+}}\frac{\vartheta(v^{l_{j}-1}z^{\beta_{j}}\mca{L}(\varepsilon^{\ast}_{j}))}{\vartheta(v^{l_{j}}z^{\beta_{j}})}\cdot\prod_{j\in I^{c}_{-}}\frac{\vartheta(v^{-l_{i}-1}z^{-\beta_{j}}\mca{L}(-\varepsilon^{\ast}_{j}))}{\vartheta(v^{-l_{i}}z^{-\beta_{j}})}
\end{align}
for some $\beta_{j}\in\bb{X}_{\ast}(S)$ and $l_{j}\in\bb{Z}$ for $j\in I^{c}$. We only need to determine $\beta_{j}$ and $l_{j}$. For any $l\in\bb{X}^{\ast}(S)$, the factor of automorphy of $i^{\ast}_{p_{J}}\Stab^{AO}_{\mf{C},T^{1/2}}(p_{I})$ under $z\mapsto q^{l}z$ is given by $i^{\ast}_{p_{I}}\mf{L}(l)\cdot i^{\ast}_{p_{J}}\mf{L}(l)^{-1}$ by (\ref{factor_Kahler}). On the other hand, the factor of automorphy of the RHS of (\ref{eqn_toric_ell_stab_rough}) is given by $\prod_{j\in I^{c}}i^{\ast}_{p_{I}}\mca{L}(\varepsilon^{\ast}_{j})^{\langle l,\beta_{j}\rangle}\cdot i^{\ast}_{p_{J}}\mca{L}(\varepsilon^{\ast}_{j})^{-\langle l,\beta_{j}\rangle}$ by Corollary~\ref{cor_toric_line_restr}. This implies that $\sum_{j\in I^{c}}\langle l,\beta_{j}\rangle\mf{b}_{j}=l=\sum_{j\in I^{c}}\langle l,\beta^{I}_{j}\rangle\mf{b}_{j}$ for any $l\in P$. Hence we have $\beta_{j}=\beta^{I}_{j}$ for any $j\in I^{c}$. 

We next consider the factor of automorphy of $\tau(\det T^{1/2},v)^{\ast}(i^{\ast}_{p_{J}}\Stab^{AO}_{\mf{C},T^{1/2}}(p_{I}))$ under $a\mapsto q^{c}a$ for $c\in\bb{X}_{\ast}(H)$. By (\ref{factor_equiv}) and Corollary~\ref{cor_toric_line_restr}, this is given by 
\begin{align*}
v^{-\sum_{i\in I_{+}}\langle c,\alpha^{I}_{i}\rangle+\sum_{i\in I_{-}}\langle c,\alpha^{I}_{i}\rangle}\cdot z^{-\sum_{j\in J}\langle c,\alpha^{J}_{j}\rangle\varepsilon_{j}+\sum_{i\in I}\langle c,\alpha^{I}_{i}\rangle\varepsilon_{i}}\cdot i^{\ast}_{p_{J}}\mca{L}\left(-\sum_{j\in J}\langle c,\alpha^{J}_{j}\rangle\varepsilon^{\ast}_{j}\right).
\end{align*}
On the other hand, (\ref{eqn_toric_ell_stab_rough}) implies that this should be equal to 
\begin{align*}
v^{-\langle c,\sum_{i\in I_{+}}\alpha^{J}_{i}\rangle+\langle c,\sum_{i\in I_{-}}\alpha^{J}_{i}\rangle-\langle c,\sum_{j\in I^{c}_{+}}(l'_{j}-1)\alpha^{J}_{j}\rangle-\langle c,\sum_{j\in I^{c}_{-}}(l'_{j}+1)\alpha^{J}_{j}\rangle}\cdot z^{-\sum_{j\in I^{c}}\langle c,\alpha^{J}_{j}\rangle\beta^{I}_{j}}\cdot i^{\ast}_{p_{J}}\mca{L}\left(-\sum_{j\in J}\langle c,\alpha^{J}_{j}\rangle\varepsilon^{\ast}_{j}\right)
\end{align*}
for any $J\in\bb{B}$, where we set $l'_{j}=l_{j}+\langle\beta^{I}_{j},\kappa\rangle$. By comparing the exponent of $v$, we obtain 
\begin{align*}
-\sum_{i\in I_{+}}\alpha^{J}_{i}+\sum_{i\in I_{-}}\alpha^{J}_{i}-\sum_{j\in I^{c}_{+}}(l'_{j}-1)\alpha^{J}_{j}-\sum_{j\in I^{c}_{-}}(l'_{j}+1)\alpha^{J}_{j}=-\sum_{i\in I_{+}}\alpha^{I}_{i}+\sum_{i\in I_{-}}\alpha^{I}_{i}
\end{align*}
for any $J\in\bb{B}$. For each $j\in I^{c}_{\pm}$, by taking $J\in\bb{B}$ such that $j\in J$ and considering the pairing with $\mf{a}_{j}$, we obtain $l'_{j}=\langle\sum_{i\in I_{+}}\alpha^{I}_{i}-\sum_{i\in I_{-}}\alpha^{I}_{i},\mf{a}_{j}\rangle\pm1$. This implies $l_{j}=-\langle\beta^{I}_{j},\kappa+\sum_{i\in I_{+}}\mf{b}_{i}-\sum_{i\in I_{-}}\mf{b}_{i}\rangle\pm1$ by (\ref{eqn_alpha_beta}).
\end{proof}

\begin{corollary}\label{cor_toric_ell_std}
For any $I\in\bb{B}$, we have
\begin{align*}
\Stab^{ell}_{X}(p_{I})=(-1)^{|I_{+}|+|I^{c}_{+}|}\cdot\prod_{i=1}^{n}\vartheta(\mca{L}(\varepsilon^{\ast}_{i})\cdot i^{\ast}_{p^{!}_{I}}\mca{L}^{!}(\varepsilon_{i})).
\end{align*}
\end{corollary}

\begin{proof}
We note that by Corollary~\ref{cor_toric_line_restr} applied for $X^{!}$, we have $i^{\ast}_{p^{!}_{I}}\mca{L}^{!}(\varepsilon_{j})=v^{-\langle\beta^{I}_{j},\sum_{i\in I_{+}}\mf{b}_{i}-\sum_{i\in I_{-}}\mf{b}_{i}\rangle}z^{\beta^{I}_{j}}$ for any $j\in I^{c}$ and $i^{\ast}_{p^{!}_{I}}\mca{L}^{!}(\varepsilon_{i})=v^{\pm1}$ for $i\in I_{\pm}$. This and Proposition~\ref{prop_toric_stab_AO} imply that 
\begin{align*}
\tau(\det T^{1/2},v)^{\ast}\Stab^{AO}_{\mf{C},T^{1/2}}(p_{I})=(-1)^{|I_{+}|}\prod_{i=1}^{n}\vartheta(\mca{L}(\varepsilon^{\ast}_{i})\cdot i^{\ast}_{p^{!}_{I}}\mca{L}^{!}(\varepsilon_{i}))\cdot\prod_{j\in I^{c}_{+}}\vartheta(v\cdot i^{\ast}_{p^{!}_{I}}\mca{L}^{!}(\varepsilon_{i}))^{-1}\cdot\prod_{i\in I^{c}_{-}}\vartheta(v^{-1}i^{\ast}_{p^{!}_{I}}\mca{L}^{!}(\varepsilon_{i}))^{-1}.
\end{align*}
Now the statement of the corollary follows from 
\begin{align*}
\vartheta(N^{!}_{p^{!}_{I},-})=\prod_{i\in I^{c}_{+}}\vartheta(v^{-1}i^{\ast}_{p^{!}_{I}}\mca{L}^{!}(-\varepsilon_{i}))\cdot\prod_{i\in I^{c}_{-}}\vartheta(v^{-1}i^{\ast}_{p^{!}_{I}}\mca{L}^{!}(\varepsilon_{i}))
\end{align*}
obtained by applying (\ref{eqn_toric_normal_bundle}) to $X^{!}$.
\end{proof}

In particular, this implies Conjecture~\ref{conj_ell_stab_duality} for toric hyper-K\"ahler manifolds. Recall that we wrote $\mb{S}_{X,p_{J},p_{I}}\coloneqq i^{\ast}_{p_{J}}\Stab^{ell}_{X}(p_{I})$. 

\begin{corollary}\label{cor_ell_toric_stab_symmetry}
For any $I,J\in\bb{B}$, we have $(-1)^{|I_{+}|+|I^{c}_{+}|}\cdot\mb{S}_{X,p_{J},p_{I}}=(-1)^{|J^{c}_{+}|+|J_{+}|}\mb{S}_{X^{!},p^{!}_{I},p^{!}_{J}}$ and they are holomorphic sections of the line bundle on $\mb{B}_{X}\cong\mb{B}_{X^{!}}$ described in Proposition~\ref{prop_factor_ell_std}. 
\end{corollary}

Moreover, Corollary~\ref{cor_toric_ell_std} also implies Conjecture~\ref{conj_ell_flop} for toric hyper-K\"ahler manifolds. 

\begin{corollary}\label{cor_ell_toric_flop}
For any $I,J\in\bb{B}$, we have $\overline{\mb{S}}_{X,p_{J},p_{I}}=(-1)^{n}\cdot\mb{S}_{-X_{\mr{flop}},p_{J},p_{I}}$. 
\end{corollary}

\begin{proof}
Let $\mca{L}_{\mr{flop}}(\lambda)$ be the $\bb{T}$-equivariant line bundle on $X_{\mr{flop}}$ associated with $\lambda\in\bb{X}^{\ast}(T)$. By Corollary~\ref{cor_toric_ell_std} and $i^{\ast}_{p_{I}}\mca{L}_{\mr{flop}}(\lambda)=\overline{i^{\ast}_{p_{I}}\mca{L}(\lambda)}$, we obtain
\begin{align*}
(-1)^{|I_{-}|+|I^{c}_{-}|}\cdot\mb{S}_{-X_{\mr{flop}},p_{J},p_{I}}&=\prod_{i=1}^{n}\vartheta(i^{\ast}_{p_{J}}\mca{L}_{\mr{flop}}(\varepsilon^{\ast}_{i})\cdot i^{\ast}_{p^{!}_{I}}\mca{L}^{!}_{\mr{flop}}(\varepsilon_{i}))\\
&=\prod_{i=1}^{n}\vartheta(\overline{i^{\ast}_{p_{J}}\mca{L}(\varepsilon^{\ast}_{i})}\cdot \overline{i^{\ast}_{p^{!}_{I}}\mca{L}^{!}(\varepsilon_{i})})\\
&=(-1)^{|I_{+}|+|I^{c}_{+}|}\cdot\overline{\mb{S}_{X,p_{J},p_{I}}}.
\end{align*}
\end{proof}

\subsection{Elliptic canonical bases}

Now we construct the elliptic canonical bases for toric hyper-K\"ahler manifolds. From now on, we identify $\bb{X}^{\ast}(T)\cong\bb{Z}^{n}$ and $\bb{X}_{\ast}(T)\cong\bb{Z}^{n}$ by using the standard inner product $(-,-):\bb{Z}^{n}\times\bb{Z}^{n}\rightarrow\bb{Z}$ given by $((\lambda_{1},\ldots,\lambda_{n}),(\mu_{1},\ldots,\mu_{n}))=\sum_{i=1}^{n}\lambda_{i}\mu_{i}$. We note that under this identification, we have $\kappa^{!}=\kappa$. Using this identification, we may consider $\bb{X}_{\ast}(S)\subset\bb{X}^{\ast}(T)$ and $\bb{X}^{\ast}(H)\subset\bb{X}_{\ast}(T)$. We recall that $(q;q)_{\infty}=\prod_{m\geq1}(1-q^{m})$. 

\begin{dfn}
For any $\lambda\in\bb{Z}^{n}$, we define 
\begin{align*}
\Theta_{X}(\lambda)&\coloneqq(q;q)_{\infty}^{-r}\sum_{\beta\in\bb{X}_{\ast}(S)}(-1)^{(\kappa,\beta)}q^{\frac{1}{2}(\beta,\beta+\kappa)+(\lambda,\beta)}\mca{L}(\lambda+\beta)z^{\beta},\\
\Theta_{X^{!}}(\lambda)&\coloneqq(q;q)_{\infty}^{-d}\sum_{\alpha\in\bb{X}^{\ast}(H)}(-1)^{(\kappa^{!},\alpha)}q^{\frac{1}{2}(\alpha,\alpha+\kappa^{!})+(\lambda,\alpha)}\mca{L}^{!}(\lambda+\alpha)a^{\alpha}.
\end{align*}
We can consider $\{\Theta_{X}(\lambda)\}_{\lambda\in\bb{Z}^{n}}$ and $\{\Theta_{X^{!}}(\lambda)\}_{\lambda\in\bb{Z}^{n}}$ as elements of $\mb{K}(X)_{\mr{loc}}\cong\mb{K}(X^{!})_{\mr{loc}}$ and call them \emph{elliptic canonical bases} for $X$ and $X^{!}$ respectively.
\end{dfn}

We note that for any $\alpha\in\bb{X}^{\ast}(H)$ and $\beta\in\bb{X}_{\ast}(S)$, we have $(\alpha,\beta)=0$. Using this, one can easily check the following relations.

\begin{lemma}\label{lem_ell_can_transformation_rule}
For any $\alpha\in\bb{X}^{\ast}(H)$ and $\beta\in\bb{X}_{\ast}(S)$, we have 
\begin{align*}
\Theta_{X}(\lambda+\alpha)&=a^{\alpha}\Theta_{X}(\lambda),\\
\Theta_{X}(\lambda+\beta)&=(-1)^{(\kappa,\beta)}q^{-\frac{1}{2}(\beta,\beta+\kappa)-(\lambda,\beta)}z^{-\beta}\Theta_{X}(\lambda).
\end{align*}
\end{lemma}

In particular, the number of linearly independent elements in $\{\Theta_{X}(\lambda)\}_{\lambda\in\bb{Z}^{n}}$ over $\mca{M}_{X}$ is less that or equal to the number of elements of $\Xi\coloneqq\bb{Z}^{n}/(\bb{X}^{\ast}(H)+\bb{X}_{\ast}(S))\cong\bb{X}^{\ast}(S)/\bb{X}_{\ast}(S)\cong\bb{X}_{\ast}(H)/\bb{X}^{\ast}(H)$. The unimodularity of $\mf{a}$ and $\mf{b}$ implies the following. 

\begin{lemma}\label{lem_Xi}
We have $|\Xi|=|\bb{B}|$. 
\end{lemma}

\begin{proof}
By taking a basis of $\bb{X}_{\ast}(H)$, we consider each $\mf{a}_{i}\in\bb{X}_{\ast}(H)$ as a column vector and $\mf{a}=(\mf{a}_{1},\ldots,\mf{a}_{n})$ as a $(d\times n)$-matrix. We note that $|\bb{X}_{\ast}(H)/\bb{X}^{\ast}(H)|=|\det(\mf{a}\cdot ^{t}\!\mf{a})|$. For each subset $I=\{i_{1},\ldots,i_{d}\}\subset\{1,\ldots, n\}$ with $|I|=d$, we denote by $\mf{a}_{I}=(\mf{a}_{i_{1}},\ldots,\mf{a}_{i_{d}})$ the $(d\times d)$-matrix obtained by removing certain columns from $\mf{a}$. By the unimodularity of $\mf{a}$, we have $\det(\mf{a}_{I})=\pm1$ if $I\in\bb{B}$ and $\det(\mf{a}_{I})=0$ otherwise. Therefore, we obtain 
\begin{align*}
\det(\mf{a}\cdot ^{t}\!\mf{a})=\sum_{\substack{I\subset\{1,\ldots,n\}\\|I|=d}}\det(\mf{a}_{I}\cdot ^{t}\!\mf{a}_{I})=\sum_{I\in\bb{B}}\det(\mf{a}_{I})^{2}=|\bb{B}|.
\end{align*}
\end{proof}

\begin{corollary}\label{cor_ell_Xi}
For any data $\mf{C}$ and $s\in\mf{s}^{\ast}_{\mr{reg}}$, the map $I\mapsto\mu_{I}$ gives a bijection $\bb{B}\cong\Xi$, where $\mu_{I}$ is defined as in (\ref{eqn_mu_I}).
\end{corollary}

\begin{proof}
It is enough to prove that the map is injective. If there exists $\alpha\in\bb{X}^{\ast}(H)$, $\beta\in\bb{X}_{\ast}(S)$ and $I\neq J\in\bb{B}$ such that $\mu_{J}=\mu_{I}+\alpha+\beta$, then (\ref{eqn_ineq_mu}) implies that $(\beta,\beta_{C})\leq|C|-1$ for any signed circuit $C$. Since $\{\mca{L}(\mu_{I})\}_{I\in\bb{B}}$ forms a basis of $K_{\bb{T}}(X)$ over $K_{\bb{T}}(\mr{pt})$, we must have $\beta\neq0$. By Lemma~\ref{lem_om_conform}, there exists a signed circuit $C=C_{+}\sqcup C_{-}$ such that $\pm\beta_{i}>0$ for any $i\in C_{\pm}$. This implies that $(\beta,\beta_{C})\geq|C|$ and hence gives a contradiction.
\end{proof}

Now we prove the main result of this paper. Recall the Jacobi triple product formula: 
\begin{align}\label{eqn_Jacobi_triple_product}
(q;q)_{\infty}\vartheta(x)=\sum_{m\in\bb{Z}}(-1)^{m}q^{\frac{m(m+1)}{2}}x^{m+\frac{1}{2}}.
\end{align}
We also recall that $\mf{S}_{X}(p_{I})\coloneqq\sqrt{\mca{L}(\kappa)\cdot i^{\ast}_{p^{!}_{I}}\mca{L}^{!}(\kappa^{!})}^{-1}\cdot\Stab^{ell}_{X}(p_{I})$.

\begin{thm}\label{thm_Schur_Weyl}
For any $I\in\bb{B}$, we have
\begin{align}\label{eqn_ell_Schur_Weyl}
(-1)^{|I_{+}|+|I^{c}_{+}|}\cdot\mf{S}_{X}(p_{I})=\sum_{\lambda\in\Xi}(-1)^{(\kappa,\lambda)}q^{\frac{1}{2}(\lambda,\lambda+\kappa)}i^{\ast}_{p^{!}_{I}}\Theta_{X^{!}}(\lambda)\cdot\Theta_{X}(\lambda).
\end{align}
Here, we fix a lift $\Xi\rightarrow\bb{Z}^{n}$ and consider $\lambda\in\Xi$ as an element of $\bb{Z}^{n}$. 
\end{thm}

\begin{proof}
We first note that by Lemma~\ref{lem_ell_can_transformation_rule}, each term in the RHS of (\ref{eqn_ell_Schur_Weyl}) does not depend on the choice of a lift of $\lambda\in\Xi$ to $\bb{Z}^{n}$. By Corollary~\ref{cor_toric_ell_std} and (\ref{eqn_Jacobi_triple_product}), we have
\begin{align*}
(-1)^{|I_{+}|+|I^{c}_{+}|}\cdot(q;q)_{\infty}^{n}\Stab^{ell}_{X}(p_{I})&=\prod_{i=1}^{n}(q;q)_{\infty}\vartheta(\mca{L}(\varepsilon^{\ast}_{i})\cdot i^{\ast}_{p^{!}_{I}}\mca{L}^{!}(\varepsilon_{i}))\\
&=\sum_{\mu\in\bb{Z}^{n}}(-1)^{(\mu,\kappa)}q^{\frac{1}{2}(\mu,\mu+\kappa)}i^{\ast}_{p^{!}_{I}}\mca{L}^{!}(\mu+\frac{1}{2}\kappa^{!})\cdot\mca{L}(\mu+\frac{1}{2}\kappa)
\end{align*}
We note that any element of $\bb{Z}^{n}$ can be uniquely written as $\lambda+\alpha+\beta$ for $\lambda\in\Xi$, $\alpha\in\bb{X}^{\ast}(H)$, and $\beta\in\bb{X}_{\ast}(S)$. By using $\mca{L}(\alpha)=a^{\alpha}$, $\mca{L}^{!}(\beta)=z^{\beta}$, and $(\alpha,\beta)=0$ for $\alpha\in\bb{X}^{\ast}(H)$ and $\beta\in\bb{X}_{\ast}(S)$, we obtain
\begin{align*}
(-1)^{|I_{+}|+|I^{c}_{+}|}\cdot\mf{S}_{X}(p_{I})&=(q;q)_{\infty}^{-n}\sum_{\substack{\lambda\in\Xi\\ \alpha\in\bb{X}^{\ast}(H)\\ \beta\in\bb{X}_{\ast}(S)}}(-1)^{(\lambda+\alpha+\beta,\kappa)}q^{\frac{1}{2}(\lambda+\alpha+\beta,\lambda+\alpha+\beta+\kappa)}i^{\ast}_{p^{!}_{I}}\mca{L}^{!}(\lambda+\alpha+\beta)\cdot\mca{L}(\lambda+\alpha+\beta)\\
&=\sum_{\lambda\in\Xi}(-1)^{(\kappa,\lambda)}q^{\frac{1}{2}(\lambda,\lambda+\kappa)}\cdot(q;q)_{\infty}^{-d}\sum_{\alpha\in\bb{X}^{\ast}(H)}(-1)^{(\kappa^{!},\alpha)}q^{\frac{1}{2}(\alpha,\alpha+\kappa^{!})+(\lambda,\alpha)}i^{\ast}_{p^{!}_{I}}\mca{L}^{!}(\lambda+\alpha)a^{\alpha}\\
&\hspace{2em}\times(q;q)_{\infty}^{-r}\sum_{\beta\in\bb{X}_{\ast}(S)}(-1)^{(\kappa,\beta)}q^{\frac{1}{2}(\beta,\beta+\kappa)+(\lambda,\beta)}\mca{L}(\lambda+\beta)z^{\beta}\\
&=\sum_{\lambda\in\Xi}(-1)^{(\kappa,\lambda)}q^{\frac{1}{2}(\lambda,\lambda+\kappa)}i^{\ast}_{p^{!}_{I}}\Theta_{X^{!}}(\lambda)\cdot\Theta_{X}(\lambda).
\end{align*}
\end{proof}

Since $\{\mf{S}_{X}(p_{I})\}_{I\in\bb{B}}$ is a basis of $\mb{K}(X)_{\mr{loc}}$ over $\mca{M}_{X}$, Lemma~\ref{lem_Xi} and Theorem~\ref{thm_Schur_Weyl} implies that $\{\Theta_{X}(\lambda)\}_{\lambda\in\Xi}$ is also a a basis of $\mb{K}(X)_{\mr{loc}}$ over $\mca{M}_{X}$. By applying Theorem~\ref{thm_Schur_Weyl} for $-X$, we obtain
\begin{align*}
(-1)^{|I_{-}|+|I^{c}_{+}|}\cdot\mf{S}_{-X}(p_{I})&=\sum_{\lambda\in\Xi}(-1)^{(\kappa,\lambda)}q^{\frac{1}{2}(\lambda,\lambda+\kappa)}i^{\ast}_{p^{!}_{I}}\Theta_{X^{!}_{\mr{flop}}}(\lambda)\cdot\Theta_{-X}(\lambda)\\
&=\sum_{\lambda\in\Xi}(-1)^{(\kappa,\lambda)}q^{\frac{1}{2}(\lambda,\lambda+\kappa)}\overline{i^{\ast}_{p^{!}_{I}}\Theta_{X^{!}}(\lambda)}\cdot\Theta_{X}(\lambda).
\end{align*}
This and (\ref{eqn_ell_Schur_Weyl}) implies that if we define $\mca{M}_{X}$-semilinear map $\beta'_{X}:\mb{K}(X)_{\mr{loc}}\rightarrow\mb{K}(X)_{\mr{loc}}$ by 
\begin{align}\label{eqn_ell_bar_inv}
\beta'_{X}(\Theta_{X}(\lambda))=\Theta_{X}(\lambda)
\end{align}
for any $\lambda\in\Xi$, then we have $\beta'_{X}(\mf{S}_{X}(p_{I}))=(-1)^{d}\mf{S}_{-X}(p_{I})$ for any $I\in\bb{B}$ i.e., $\beta'_{X}=\beta^{ell}_{X}$. This proves the following result which is the main observation in this paper and partly justify our definition of elliptic canonical bases for toric hyper-K\"ahler manifolds.  

\begin{corollary}\label{cor_ell_bar_invariance}
For each $\lambda\in\bb{Z}^{n}$, we have $\beta^{ell}_{X}(\Theta_{X}(\lambda))=\Theta_{X}(\lambda)$ and $\beta^{ell}_{X^{!}}(\Theta_{X^{!}}(\lambda))=\Theta_{X^{!}}(\lambda)$. Moreover, the elliptic bar involution $\beta^{ell}_{X}$ does not depend on the choice of chamber $\mf{C}$. 
\end{corollary}

\begin{proof}
For the independence on $\mf{C}$, it is enough to note that $\Theta_{X}(\lambda)$ does not depend on the choice of $\mf{C}$ and (\ref{eqn_ell_bar_inv}) uniquely characterize the map $\beta^{ell}_{X}$.  
\end{proof}

\begin{remark}
If one try to prove $\beta^{ell}_{X}(\Theta_{X}(\lambda))=\Theta_{X}(\lambda)$ directly from the definition of $\beta^{ell}_{X}$, then certain nontrivial identities of various theta functions will be needed. Our proof mimics that of Proposition~\ref{prop_toric_bar_invariance_E(A)} and does not involve any nontrivial calculations. In fact, our definition of elliptic canonical bases is designed so that this kind of proof works nicely. 
\end{remark}

\begin{corollary}\label{cor_ell_toric_indep_flop}
Conjecture~\ref{conj_ell_indep_chamber} holds for toric hyper-K\"ahler manifolds. 
\end{corollary}

\begin{proof}
This follows from Corollary~\ref{cor_ell_bar_invariance} by reversing the argument of the proof of Proposition~\ref{prop_ell_indep_chamber}.
\end{proof}

This completes the proof of all the conjectures stated in section 3 and 4 in the case of toric hyper-K\"ahler manifolds. 

\subsection{$K$-theory limits}

In this section, we check that the elliptic canonical bases for toric hyper-K\"ahler manifolds lift $K$-theoretic canonical bases for any slopes. This result gives another justification of our definition of elliptic canonical bases. 

Let $\lambda\in\bb{Z}^{n}$ and $s\in\mf{s}^{\ast}_{\mr{reg}}$. Since $\Theta_{X}(\lambda)|_{z=q^{-s}}$ might not have well-definded limit under $q\rightarrow0$, we consider its leading term $\mr{LT}_{s}(\Theta_{X}(\lambda))\coloneqq\mr{LT}(\Theta_{X}(\lambda)|_{z=q^{-s}})$. Here, we write $\mr{LT}(f)\coloneqq f_{t_{0}}$ for $f=\sum_{t\in\bb{R}}f_{t}q^{t}$ and $t_{0}\coloneqq\min\{t\mid f_{t}\neq0\}$ if it exists. In order to calculate $\mr{LT}_{s}(\Theta_{X}(\lambda))$, we need to know when the function $\frac{1}{2}(\beta,\beta+\kappa)+(\lambda-s,\beta)$ on $\beta\in\bb{X}_{\ast}(S)$ takes its minimum. We note that $\mr{LT}_{s}(\Theta_{X}(\lambda+\alpha))=a^{\alpha}\cdot\mr{LT}_{s}(\Theta_{X}(\lambda))$ and $\mr{LT}_{s}(\Theta_{X}(\lambda+\beta))=(-1)^{(\kappa,\beta)}\mr{LT}_{s}(\Theta_{X}(\lambda))$ for any $\alpha\in\bb{X}^{\ast}(H)$ and $\beta\in\bb{X}_{\ast}(S)$. In particular, we only need to consider $\mr{LT}_{s}(\Theta_{X}(\lambda))$ for $\lambda\in\Xi$. We will identify $\Xi=\{\mu_{I}\}_{I\in\bb{B}}$ by using Corollary~\ref{cor_ell_Xi}.  

We note that for any $t\in\bb{R}$, the function from $\bb{Z}$ to $\bb{R}$ defined by $\frac{1}{2}m(m+1)-tm$ for $m\in\bb{Z}$ takes its minimum at $m=\lfloor t\rfloor$ if $t\notin\bb{Z}$ and $m=t,t-1$ if $t\in\bb{Z}$. Therefore, the function from $\bb{Z}^{n}$ to $\bb{R}$ given by  
\begin{align}\label{eqn_quadratic_form}
\frac{1}{2}(\mu,\mu+\kappa)-(s,\wt_{H^{!}}i^{\ast}_{p^{!}_{I}}\mca{L}^{!}(\mu))=\sum_{i=1}^{n}\frac{1}{2}\mu_{i}(\mu_{i}+1)-\sum_{j\in I^{c}}\mu_{j}\langle s,\beta^{I}_{j}\rangle
\end{align}
for $\mu\in\bb{Z}^{n}$ takes its minimum when $\mu_{j}=\lfloor\langle s,\beta^{I}_{j}\rangle\rfloor$ for any $j\in I^{c}$ and $\mu_{i}=0,-1$ for any $i\in I$. We remark that for any such $\mu$, we have $\mca{L}(\mu)\in\bb{B}_{X,s}$ by (\ref{eqn_toric_E(B)}). In particular, (\ref{eqn_quadratic_form}) takes its minimum when $\mu=\mu_{I}$. Therefore, we have 
\begin{align*}
\frac{1}{2}(\mu_{I},\mu_{I}+\kappa)-(s,\wt_{H^{!}}i^{\ast}_{p^{!}_{I}}\mca{L}^{!}(\mu_{I}))&\leq\frac{1}{2}(\mu_{I}+\beta,\mu_{I}+\beta+\kappa)-(s,\beta+\wt_{H^{!}}i^{\ast}_{p^{!}_{I}}\mca{L}^{!}(\mu_{I}))\\
&=\frac{1}{2}(\mu_{I},\mu_{I}+\kappa)-(s,\wt_{H^{!}}i^{\ast}_{p^{!}_{I}}\mca{L}^{!}(\mu_{I}))+\frac{1}{2}(\beta,\beta+\kappa)+(\mu_{I}-s,\beta)
\end{align*} 
for any $\beta\in\bb{X}_{\ast}(S)$. I.e., the function $\frac{1}{2}(\beta,\beta+\kappa)+(\mu_{I}-s,\beta)$ on $\bb{X}_{\ast}(S)$ takes its minimum at $\beta=0$. On the other hand, if this function takes its minimum at $0\neq\beta\in\bb{X}_{\ast}(S)$, then the function (\ref{eqn_quadratic_form}) takes its minimum at $\mu_{I}+\beta$ and hence we obtain $\mca{L}(\mu_{I}), \mca{L}(\mu_{I}+\beta)\in\bb{B}_{X,s}$. This contradicts the inequality (\ref{eqn_ineq_mu}) as in the proof of Corollary~\ref{cor_ell_Xi}. Therefore, we obtain $\mr{LT}_{s}(\Theta_{X}(\mu_{I}))=\mca{L}(\mu_{I})\in\bb{B}_{X,s}$. In summary, we obtained the following formula. 

\begin{prop}\label{prop_K_limit}
For any $\lambda\in\bb{Z}^{n}$, there exists unique $A\in\mr{Alc}_{s}$ and $\beta\in\bb{X}_{\ast}(S)$ such that $\lambda=\mu_{A}+\beta$, where $\mu_{A}$ is defined as in (\ref{eqn_mu_A}). Moreover, we have $\mr{LT}_{s}(\Theta_{X}(\lambda))=(-1)^{(\kappa,\beta)}\mca{E}(A)$.  
\end{prop}

Let us take $s_{+},s_{-}\in\mf{s}^{\ast}_{\mr{reg}}$ as in section 5.5 which are separated by a wall $w=\{x\in\mf{s}^{\ast}_{\bb{R}}\mid\langle x,\beta_{C}\rangle=m\}$ for some signed circuit $C$ and $m\in\bb{Z}$ with $\langle\eta,\beta_{C}\rangle>0$. For any $\lambda\in\bb{Z}^{n}$, there exists unique $A_{\pm}\in\mr{Alc}_{s_{\pm}}$ and $\beta_{\pm}\in\bb{X}_{\ast}(S)$ such that $\lambda=\mu_{A_{\pm}}+\beta_{\pm}$ by Proposition~\ref{prop_K_limit}. If $\mca{L}(\mu_{A_{+}})\in\bb{B}_{X,s_{-}}$, then the uniqueness implies $\mu_{A_{-}}=\mu_{A_{+}}$ and $\beta_{-}=\beta_{+}$. Therefore, we obtain $\mr{LT}_{s_{+}}(\Theta_{X}(\lambda))=\mr{LT}_{s_{-}}(\Theta_{X}(\lambda))\in\bb{B}^{0}_{s_{\pm},w}=\bb{B}_{X,s_{+}}\cap\bb{B}_{X,s_{-}}$. On the other hand, if $\mca{L}(\mu_{A_{+}})\in\bb{B}^{1}_{s_{+},w}=\bb{B}_{X,s_{+}}\setminus\bb{B}^{0}_{s_{+},w}$, then the uniqueness and Proposition~\ref{prop_toric_wall_crossing_characterization} implies that we have $\mu_{A_{-}}=\mu_{A_{+}}-\beta_{C}$ and $\beta_{-}=\beta_{+}+\beta_{C}$. Therefore, Lemma~\ref{lem_toric_wall_crossing_formula} implies that $\mr{LT}_{s_{+}}(\Theta_{X}(\lambda))=-v^{|C|}\cdot\mr{LT}_{s_{-}}(\Theta_{X}(\lambda))$ modulo lower terms spanned by $\bb{B}^{0}_{s_{\pm},w}$. In some sense, the elliptic canonical bases organize part of wall-crossing phenomenon of $K$-theoretic canonical bases in Conjecture~\ref{conj_wall_crossing} in a beautiful way. 


\end{document}